\documentclass[12pt]{amsart}
\usepackage{a4wide,amsfonts,amsmath,amssymb,amsthm,epsfig,cite,graphicx,hyperref,
color, esint,fancyhdr, enumerate, latexsym,amsrefs}
\usepackage{amsmath}
\usepackage{mathtools}
\usepackage{accents}
\DeclareMathOperator{\spn}{span}

\usepackage{verbatim}
\newtheorem{thm}{Theorem}[section]
\newtheorem{cor}[thm]{Corollary}
\newtheorem{lem}[thm]{Lemma} 
\newtheorem{prop}[thm]{Proposition}
\theoremstyle{definition} 
\newtheorem{defn}[thm]{Definition}
\theoremstyle{remark}
\newtheorem{rem}[thm]{Remark}
\newtheorem{ex}[thm]{Example}
\numberwithin{equation}{section}

\usepackage[OT2,OT1]{fontenc}

\DeclareRobustCommand\cyr{%
  \renewcommand\rmdefault{wncyr}%
  \renewcommand\sfdefault{wncyss}%
  \renewcommand\encodingdefault{OT2}%
  \normalfont
  \selectfont}
\DeclareTextFontCommand{\textcyr}{\cyr}

\hypersetup{
    colorlinks,
    citecolor=blue,
    filecolor=black,
    linkcolor=red,
    urlcolor=black
}

\newcommand{\pa}{\partial}

\begin{document}

\keywords{Holomorphic Retracts, Balanced pseudoconvex domains, Schwarz lemma, Biholomorphic inequivalences, Union Problem, Kobayashi hyperbolic, Caratheodory geodesics}
\subjclass{Primary: 32F45; Secondary: 32H15, 32H20}
\author{G. P. Balakumar, Jiju Mammen}

\address{G. P. Balakumar: Department of Mathematics, Indian Institute of Technology Palakkad, 678557, India}
\email{gpbalakumar@iitpkd.ac.in}
\address{Jiju Mammen: Department of Mathematics, Indian Institute of Technology Palakkad, 678557, India}
\email{211914001@smail.iitpkd.ac.in}

\title{Exploring Holomorphic Retracts}

\begin{abstract}
The purpose of this article is towards {\it systematically} characterizing holomorphic retracts of 
domains of holomorphy; to begin with,  
bounded balanced pseudoconvex domains $B \subset \mathbb{C}^N$. Specifically, 
we show that every (holomorphic) retract of $B$ passing through 
its center (i.e., the origin), is the 
graph of a holomorphic map over a linear subspace of $B$. If $B$ 
is further assumed to be convex with $\mathbb{C}$-extremal boundary, then every 
such retract is linear -- this is a well-known result due to Vesentini but we give here an 
elementary proof without making use of complex geodesics which works when
convexity is dropped and $\mathbb{C}$-extremality is replaced
by holomorphic extremality. The converse 
of Vesentini's theorem is also established and
general results characterizing the {\it non-linear}
retracts through the origin are attained (for fairly general domains in this context) 
when the boundary fails to have extreme points.
As for retracts 
not passing through the origin, we 
obtain the following result: if $B$ is a strictly convex ball with respect to some norm on $\mathbb{C}^N$,
and $\rho$ any holomorphic
retraction map on $B$ whose derivative at the center is a norm-one projection, then $Z=\rho(B)$ is the graph of a 
holomorphic map over a linear subspace of 
$B$. To deal with a case when $\partial B$ may fail to have sufficiently many extreme points,
we consider products of bounded balanced domains of holomorphy with holomorphic extreme boundaries and obtain a complete description 
of retracts passing through the origin. 
The conditions of this result can indeed be verified to be applied to solve a special 
case of the union problem with a degeneracy, namely:
to characterize those Kobayashi corank one complex manifolds $M$ which can be expressed as 
an increasing union of
submanifolds which are biholomorphic to a prescribed homogeneous 
bounded balanced domain in $\mathbb{C}^N$ (a description of all of its retracts being attained by the aforementioned result). Along the 
way, we discuss examples of non-convex but pseudoconvex bounded balanced domains in 
which most one-dimensional linear subspaces fail to be holomorphic retracts, in good 
contrast to the convex case. Indeed, for the simplest possible case here
namely the `$\ell^q$-ball' for $0<q<1$ as well as its `anisotropic' analogues, we obtain a complete list of 
all retracts through the origin of each possible dimension; whether $q$ is less than $1$ or not, retracts through the origin are shown to be
biholomorphic copies of lower dimensional $\ell^q$-balls. The same is done also for a class of domains lying at the other 
extreme in their Levi geometry but nevertheless share with $\ell^q$-balls the feature of their boundaries being devoid of complex line segments; namely, balanced $\mathbb{C}$-extremal analytic
polyhedra. In the course of achieving this, we prove a generalization of Mazet's Schwarz lemma for pseudoconvex balanced domains.
The above mentioned determination of retracts enables an illustration of applying 
retracts to establishing biholomorphic inequivalences.
 To go beyond balanced domains, we then first obtain a complete characterization of retracts
of the Hartogs triangle and `analytic complements' thereof. Thereafter, 
we drop the assumption of boundedness as well, to attain
similar characterization results for domains which are neither 
bounded nor topologically trivial. We conclude 
with remarks about retracts of $\mathbb{C}^2$ including an alternative proof about the characterization of its polynomial
retracts.
\end{abstract}

\maketitle
\section{Introduction}
\noindent We begin 
with remarks for motivation and background, possibly  somewhat elaborate; this owes to the
unavailability of a survey or expository article on {\it holomorphic} retracts to refer the reader to. 
This will also simultaneously help us set up notations and fix terminologies. While the main purpose of this article is to contribute original results, we wish to put our results systematically in connection with what is known for reasons already indicated. In this regard, we emphasize that this article is a survey in parts as well,
striving to give references for the known
results (to the
best of our knowledge).
By a {\it retraction} we shall mean 
an idempotent {\it map}, whereas by a {\it retract}, we mean the image of such a map. More precisely, 
we shall be interested in 
{\it holomorphic} retracts i.e., the image of a holomorphic retraction which by definition, is a holomorphic 
self-map $\rho$ of a domain 
$D \subset \mathbb{C}^N$ (or more generally, a complex manifold) satisfying the idempotent condition 
$\rho^{\circ 2} := \rho \circ \rho = \rho$. This means that $\rho$ acts as 
the identity map on $Z=\rho(D)$ which shall be our notation for a retract 
throughout this article; 
a trivia worth noting is that $Z$ is precisely the set of points in $D$ fixed 
by $\rho$. A pair of extreme or
trivial examples of retractions (resp. retracts) of any complex manifold are the identity map 
and constant maps (resp. $D$ and its points); obviously, we would not like to 
talk of them any further 
and thereby shall be referred to 
as {\it trivial retracts} henceforth. For instance, with this terminology and understanding, we shall
say that in dimension $N=1$, there are no non-trivial retracts. This means that for
any domain $D$ in $\mathbb{C}$ or any Riemann surface, any holomorphic self-map
$\rho$ satisfying $\rho^{\circ 2} = \rho$ is either a constant map or else, 
the identity map on $D$. Indeed, this follows from the open mapping theorem 
together with the 
idempotency condition which implies that $\rho(D)$ is a closed subset 
of $D$. In other words, holomorphic retracts
are interesting only when the dimension $N$ gets bigger than one. Now in 
multidimensional complex analysis, a lot 
of effort has gone into 
the study of biholomorphic and more generally, proper 
(holomorphic\footnote{Unless otherwise mentioned, all 
maps in this article are to be understood
to be holomorphic.}) maps, as is natural. To 
mention just 
one or two of the first basic facts about proper maps between a pair of domains in $\mathbb{C}^N$, they 
are known to be always surjective and for many a class of domains $D$
with `nice' boundaries, proper self-maps 
turn out to be automorphisms of $D$. Though absolutely trivial to check, let us mention 
for contrast with the just mentioned 
facts about proper maps, that: non-trivial retractions are never surjective 
(and never are they, automorphisms!); as one another sharp contrast,
proper maps never decrease dimension, whereas non-trivial retractions always 
decrease dimension. In other words, retractions
could be viewed as being, loosely speaking, at the `opposite end of the spectrum' when compared to proper maps.\\ 

\noindent While much work has been carried out on proper holomorphic maps and their images, far less 
investigation has been done about retracts. Infact, in addition to (the foregoing remarks and) their role 
in certain fundamental mapping problems indicated above,
retracts are of significance in complex dynamics as well. To just be specific by providing a standard 
result in this regard, we mention the following theorem which shows how retracts arise in the 
study of iterates of a holomorphic self-map of taut manifolds.
Needless to say, this theorem has been applied to the study of the 
special case of {\it proper} holomorphic self-maps. In light of the the stark contrasts 
between proper maps and retraction maps mentioned above, the connection amongst these maps, is thereby all the more
interesting.
	
\begin{thm}[Bedford, \cite{Bdfrd}; Abate, \cite{Abate_Horosphre}]\label{Iterat-Dynamics}
Let $X$ be a taut manifold, $f$ a holomorphic self-map of $X$. Suppose that the 
sequence $\left\{f^{\circ k}\right\}$ of iterates of $f$ is not compactly divergent. Then 
there exist a submanifold $M$ of $X$ and a holomorphic 
retraction $\rho: X \rightarrow M$ such that every limit point $h \in \operatorname{Hol}(X, X)$ 
of $\left\{f^{\circ k}\right\}$ is of the form
\[
h=\gamma \circ \rho
\]
where $\gamma$ is an automorphism of $M$. Infact, $\rho$ is itself a limit point of 
the sequence $\left\{f^{\circ k}\right\}$ and $f$ acts as an automorphism of $M$.
\end{thm}

\noindent The manifold $M$ occurring in the above 
theorem -- referred to as the {\it limit manifold} -- is, as is evident, a holomorphic 
retract of $X$. Determining all holomorphic retracts of taut manifolds in general 
is way too difficult. We shall therefore be taking up first, the question of describing the holomorphic 
retracts in what may seem 
to be a rather small class, namely that of bounded balanced pseudoconvex 
domains $D$ in $\mathbb{C}^N$; not all such domains
need be taut, but all bounded balanced domains which are taut, are contained within this class
(this owes to results of Barth \cite{Barth}). Whether or not $D$ is taut, a first general fact 
to be noted is that every retract of $D$ is a Stein manifold. This follows from 
the basic fact that  retracts of (connected) complex manifolds
are closed {\it connected} complex submanifolds and if the manifold is a Stein manifold, then so is
every of its retracts; as our class of domains of interest is contained within that comprising 
domains of holomorphy in $\mathbb{C}^N$ which are the simplest instances of Stein manifolds,
the retracts in this article, all have this property. Perhaps before the just mentioned inheritance
of the property of being a Stein manifold, we must recall in a sentence here the
elementary fact that various topological properties such as simply connectedness or more
generally the triviality of any of the homotopy or homology groups (of the ambient space)
is inherited under passage 
to a retract. Indeed in general, if $R \xhookrightarrow{\iota} X $ denotes a (continuous) retraction
map on a topological space $X$,
then the homomorphism induced by the inclusion $\iota$ at the level of each homotopy or homology group, is injective
since $r_* \circ \iota_* \equiv {\rm identity}$, in standard notations; this may be paraphrased as the 
general observation that topological complexity cannot increase 
for a retract (and our results will
establish certain finer analogues of this in several special cases in the 
complex analytic category). While this means that all such groups must 
be trivial for every retract of a contractible space $X$, we cannot deduce as readily out of this, the fact that 
contractibility is inherited as well. Indeed, this inheritance follows directly from the 
definition of contractibility 
and the only reason we are spending words here to spell this out is because we shall be dealing with 
balanced domains for the most part, and all such domains are evidently contractible. It may 
be worth remarking once here that as soon as we restrict our considerations to the category of smooth
manifolds, we gain easily defined properties that can be used to distinguish a pair of retracts, namely for instance, the dimension 
of a retract i.e., as soon as we know that a given pair of retracts of a contractible manifold are 
of different dimensions, we are ascertained that they are diffeomorphically inequivalent. This is just 
to {\it indicate} that classifying retracts of a manifold which is topologically 
trivial (here, contractible) is not an intractable or
meaningless exercise. We shall in-fact be focussing our attention in 
good measure on bounded 
balanced domains of holomorphy.
Let us only remark here as to why we do not replace boundedness by the apparently more
general assumption of Kobayashi hyperbolicity: a {\it balanced} 
domain is hyperbolic if and only if it is bounded (\cite{Kodama}) whereby in particular, the class of hyperbolic {\it balanced} convex
domains (in $\mathbb{C}^N$) is precisely the class of bounded balanced convex domains. Now, boundedness implies
hyperbolicity and 
it is known 
that a convex domain is taut if and only if it is hyperbolic \footnote{In such a case the convex domain is biholomorphic to a 
bounded domain but not necessarily to a bounded {\it convex} domain, as shown by Zimmer \cite{Zim}.} (\cite{Bracci-Saracco})). It follows that
the foregoing theorem in conjunction with our results would be useful towards obtaining a description 
of the limit manifolds of (proper) holomorphic self-maps of
those bounded balanced convex domains that we consider in this article, in particular.
Even within this class 
this question of determining {\it all} retracts has indeed been of interest but has been 
rather a difficult one and a complete description of all retracts is 
available only for the Euclidean ball $\mathbb{B}$, the polydisc $\Delta^{N}$ and 
perhaps their Cartesian products (herein and henceforth, all balls are centered at the origin and of unit radius with
respect to some norm on $\mathbb{C}^N$; when the norm is the $l^2$ norm, the corresponding ball 
referred to as the Euclidean ball is denoted $\mathbb{B}^N$ or $\mathbb{B}$ for short;
$\Delta$ denotes the standard unit disc in $\mathbb{C}$). \\

\noindent Let us also mention here for clarity that our focus will be to determine 
the form of the retracts $Z=F(X)$ rather 
than retractions mappings $F$. This owes to the fact that for a given retract $Z$, there 
are in general plenty of retractions all of whose images are the same $Z$. This 
is manifestly brought out already in the simplest of cases for $X$ 
namely, when $X$ is the bidisc $\Delta^{2}$, a description of whose retracts was attained (indeed
for $\Delta^N$)
in 1981 by Heath and Suffridge in \cite{Heath_Suff}. The case of the Euclidean ball dealt 
with in 1974 by Suffridge in \cite{Suffridge-Ball} is exposited 
well in Rudin's book \cite{Rud_book_Fn_theory_unitball}; we shall not detail 
the timeline of the results, beyond
the most basic cases of the ball and polydisc. As highlighted in the introduction 
of this article by Heath -- Suffridge,
there is a wide range of holomorphic retractions of the bidisc with the same image, 
which we recall here for 
some later significance for our article. Let $0 < t < 1$ and pick an arbitrary 
holomorphic function $f: \Delta^2 \to \mathbb{C}$ which is bounded above in modulus by $t(1-t)/2$. 
Then the mapping on $\Delta^2$, defined by 
\[
 (z,w) \mapsto \left[(1-t) z+t e^{i \theta} w +
 \left(z-e^{i \theta} w\right)^{2} f(z, w)\right]\left(1, e^{-i \theta}\right), 
\]
for any fixed $\theta \in \mathbb{R}$, is a holomorphic retraction of $\Delta^{2}$ 
onto $\left\{z\left(1, e^{-i \theta}\right):|z|<1\right\}$, as is checked by a 
straightforward computation which we skip here, to get ahead with building a
perspective in this introductory section. Note that 
though such a retraction may well be highly non-linear and have a complicated 
expression as just exemplified, its image (regardless of what $f$ is), is 
just a linear copy of the disc. This finishes the demonstration that for a 
given retract $Z$, there may very well be plenty of 
(as remarked in \cite{Heath_Suff} `quite complicated'!) retractions 
all of whose images are the same $Z$. At this juncture, it may be tempting to ask as to 
whether this phenomenon of the existence of rather complicated holomorphic retractions with the same image, 
owes to the polydisc being a domain that reduces to a product of one-dimensional discs, thereby 
in-particular, giving rise to a plethora of analytic discs in the boundary. That this is not 
the case is borne out by the fact that the same is true (in-fact, in some sense all the more! -- see below)
for Euclidean balls, as 
is recorded in chapter-8
of Rudin's book \cite{Rud_book_Fn_theory_unitball} wherein complicated retraction maps 
of the Euclidean ball $\mathbb{B}^2 \subset \mathbb{C}^2$, all with the same image, 
are exemplified. Indeed to recall this briefly here,
pick any $f \in H^\infty(\mathbb{B}^2)$ of the form
\[
f(z,w) = z + w^2 g_1(z,w) + w^4 g_2(z,w) + w^6 g_3(z,w) + \ldots
\]
where the $g_j$'s are {\it arbitrary} elements of $H^\infty(\mathbb{B}^2)$; then the 
associated self-map $F$ of $\mathbb{B}^2$ given by $F(z,w)=\left(f(z,w),0 \right)$ is a
retraction map, as simply follows from the fact that $f(z,0) \equiv z$.
Its image is $\pi_1(\mathbb{B}^2) = \Delta$,
where $\pi_1$ is the standard linear projection onto the first factor.
Hence again, while $f$ and thereby $F$ may well be 
be very complicated, the retract here is just a disc. We are therefore motivated to focus on retracts,
first more than retraction {\it mappings}.\\

\noindent Let us then begin with the obvious observation that retracts are always fixed point sets 
of some holomorphic map; the converse however is not true in general. But then 
again, an easy but noteworthy consequence of theorem \ref{Iterat-Dynamics} is that
the fixed point set of every holomorphic self-map on any taut manifold $X$, is {\it contained}
within a non-trivial retract{\footnote{these little claims will be justified in the preliminaries and subsequent sections.}. Furthermore, there are results for an exact equality
rather than just containment:
as shown in \cite{Mazet_Vigue_fxdpnt}, \cite{Thai_Fxdpnt_cnvx} among others, the 
fixed point set of any holomorphic self map of any hyperbolic convex 
domain in $\mathbb{C}^N$ -- whenever it is non-empty -- is also a retract of the 
domain. Whether or not fixed point sets of a holomorphic self map is the retract of 
possibly some other holomorphic self map of 
$X$, a definitive basic feature that {\it all} retracts have in common is that every 
retract $Z$ has the extension property in the sense that every holomorphic function 
on $Z$ admits (just by composition with the retraction mapping $X$ onto $Z$) a 
holomorphic extension to $X$. When the given function on $Z$ is bounded, we also 
have (by the same means of composition with the retraction map), a holomorphic 
extension on $X$ with the same sup-norm and thus retracts form the simplest class 
of subvarieties of $X$ with the norm-preserving extension property to $X$. Whether 
retracts are the only subvarieties with this property was examined and shown to be so, in 
many special cases by Kosiński and McCarthy 
in the relatively recent work \cite{Kosnski} wherein counterexamples have also been given. This is the only sentence that we spend in recalling the recent history about the works on retracts in general, all other citations being actually used in one way or the other in the sequel.
All this is only towards indicating the significance and 
motivation for the study of holomorphic retracts, {\it in general}. Our 
first goal for this article is rather specific and it is: to obtain a 
description of the holomorphic retracts of pseudoconvex bounded balanced domains passing through 
its center, namely the origin in $\mathbb{C}^N$. This itself is noteworthy; if not 
for anything else, this suffices for an important application to attain (employing the above 
theorem \ref{Iterat-Dynamics}), the fairly recent result of \cite{JJ} together with \cite{Pranav_JJ}
that: every proper holomorphic self-map of any smoothly
bounded balanced pseudoconvex domain of finite type is an automorphism. We shall not however elaborate upon this here.\\
	
\noindent As mentioned in the abstract, we intend a systematic presentation; so,
to not be amiss about addressing the simplest of questions in the aforementioned first 
goal then, we recall that
in \cite{Vesntn}, E. Vesentini already attained a fundamental result on the form of 
retracts through the center of certain convex balls. As pointed out
by J. P. Vigue in \cite{Vigue_fxdpnt_cnvx}, Vesentini's approach makes `heavy' use of complex geodesics.
We shall to begin with, give a relatively elementary proof of a generalization of this result 
without using complex geodesics. We however
we emphasize that complex geodesics are conceptually significant 
here; they are particularly advantageous
when the boundary satisfies extremality conditions, but
in the presence of analytic varieties in the boundary,  
it seems fruitful to take a more direct approach. 
For instance, this approach led us among other things, to attain theorem \ref{Charac-retract-1diml} giving
a fairly precise description of the retracts when extremality conditions are completely dropped in the setting of Vesentini's theorem. To see what we mean by this further, we request the reader to see its proof 
in section \ref{Vesentini}
as well as that of
proposition \ref{JJ-improved} below (theorem \ref{MainThm} is also based on this, though the
details are longer). 
 Perhaps more importantly, we first note that all  Caratheodory geodesics (abbreviated complex $C$-geodesic) are retracts but not conversely \footnote{To mention a simple example consider the copy of $\Delta^*$ in $\Delta \times \Delta^*$, which is evidently a retract but not a complex $C$-geodesic}. It may be noted in Vesentini's theorem that the boundary  is assumed 
to be $\mathbb{C}$-extremal which is weaker than the condition of being $\mathbb{R}$-extremal;
thereby, this theorem not only applies to 
$\mathbb{R}$-extremal domains such as $l^p$-balls for $1<p<\infty$, but also more domains 
such as the standard $l^1$-ball in $\mathbb{C}^N$ which is 
{\it not} $\mathbb{R}$-extremal but indeed $\mathbb{C}$-extremal. The notions
of $\mathbb{R}$-extremality, $\mathbb{C}$-extremality and the relationship between them 
is recalled in the following section \ref{preliminaries} about preliminaries. The necessity of such extremality 
hypotheses is borne out by the fact that the assertion fails even for the polydisc {\it and}
for certain simple irreducible domains as well.

\begin{thm}[Vesentini] \label{Vesntn}
Let $D$ be a bounded balanced convex domain in $\mathbb{C}^N$ with $\mathbb{C}$-extremal boundary. Then 
every retract of $D$ passing through origin is a linear subspace of $D$. 
\end{thm}

\noindent Here and throughout the article, by a linear subspace $D_L$ of a balanced 
domain $D \subset \mathbb{C}^N$,
we mean that $D_L$ is of the form $D_L=D \cap L$, where $L$ is a linear subspace 
of $\mathbb{C}^N$. Likewise, by an affine linear subspace $A'$ of $D$, we mean 
$A'=A \cap D$ where $A$ is an affine linear subspace of $\mathbb{C}^N$. To recall
a relevant fact in this language, every retract of a ball (centered at the origin, as always in this article)
linearly equivalent to
the standard Euclidean ball $\mathbb{B}^N$,
is an affine linear subspace. A question which arises here is whether 
linear retracts $D_L$ can fail to be realizable as the 
image of a linear (retraction) map. While it is easy 
to ascertain that this cannot happen, it may be 
interesting to note that this can happen if `linear'
is replaced by `affine-linear' i.e., there exist affine-linear
retracts for certain
balanced domains wherein they cannot be realized as images
of affine-linear retraction maps.
We 
postpone rendering clarity about such ambiguities and questions
to section \ref{Polyballs&mainThms-sectn}; likewise,
for definitions of a couple of other terms as well, in order to focus more in this introductory section, 
on the narrative of our results and work about
retracts in perspective of what is already known.\\

\noindent Now, generalizations of 
Vesentini's theorem have certainly been obtained
in the intervening years. Let us only mention 
a relatively recent, best result in this direction
extended a bit more for our applications. 
In discussing
this, we also attain a generalization of Mazet's Schwarz lemma,
namely theorem \ref{extr-Schw-lem-holextrbdy} of the sequel, that is a key tool for some of the main results
about retracts as well as the following proposition.

\begin{prop}(A small modification of proposition $4. 5$ of \cite{JJ}): \label{JJ-improved}
Let $D$ be any balanced pseudoconvex domain in $\mathbb{C}^N$. Suppose $D \cap T_0Z$ is bounded and 
$Z$ is a retract through 
the origin with the property that all points of $\partial D \cap T_0Z$ 
are holomorphically extreme boundary points of $D$. Then, $Z$ is linear.
\end{prop}
\noindent We shall employ this in 
determining retracts through the origin of all possible dimensions in $\ell_q$-`balls' as in theorem
\ref{ell_q_ball_C3} below. The $\ell_q$-balls  
give the simplest domains of holomorphy with holomorphically
extreme but non-convex boundaries, at-least for the case where global $C^1$-smoothness of
boundary fails. 
Now, the above proposition also helps in establishing biholomorphic inequivalences, 
particularly with product domains; for instance, a bounded balanced pseudoconvex domain with holomorphically extreme boundary, and the polydisc -- this is only to give a quick example 
here. We now discuss other possible results about biholomorphic inequivalences, but only briefly as the focus of this article is more on attaining descriptions of retracts. Towards this, we  consider domains with greater complexity to proceed systematically towards which, 
we deal with the simplest class of domains
generalizing the polydisc other than products, that 
are also rich in analytic varieties in their boundaries, namely analytic 
polyhedra. 
Now, stating the essence of the foregoing differently, we may use knowledge about retracts of 
a domain to detect whether it is irreducible i.e., not being biholomorphic to a product domain.  
The key idea here is the
the following common feature amongst all product domains: through {\it every} point of a product domain, there are sufficiently many directions
along which existence of non-trivial retracts is guaranteed. Let us just state here the 
simplest result to this effect to indicate what this means; leaving more general formulations 
to be laid out in section \ref{Irreducibity} wherein we shall have other remarks to tie up with and furnish applications.

\begin{prop}\label{retr-of-prod-in-C2}
Let $D_1, D_2$ be domains in $\mathbb{C}$ with one of them bounded, say $D_1$. Let
$p$ be an arbitrary point of $D=D_1 \times D_2$.
Then there exists $\epsilon > 0$ such that: for every vector
$v_2 \in T_{p_2} D_2 \simeq \mathbb{C}$ with $\vert v_2 \vert < \epsilon$, there exists 
a retract $Z$ of the product domain $D$ passing through $p = (p_1,p_2)$ with $ (1,v_2) \in T_p Z$.
\end{prop}

\noindent The version of this for arbitrary (finite) dimensions is
formulated in proposition \ref{Retrct_Dirctn} which as already alluded to, paves the way for establishing the following.

\begin{thm}\label{prod_anal_poly}
 Let $D$ be a (bounded) balanced $\mathbb{C}$-extremal analytic polyhedron. Let $G$ be any domain in $\mathbb{C}^N$ admitting retracts for some open piece of directions at each 
 $w \in G$.
 Then $D$ is not biholomorphic to $G$. The
 same is true when $D$ is replaced by $D \setminus A$ as well, where $A$ is any non-trivial
 analytic set (possibly of codimension one) not containing the origin. In particular, $D$ and 
 $D \setminus A$ are irreducible i.e., not biholomorphic to a product domain.
\end{thm}

\noindent The precise meaning of what it means for a domain $G$ to admit retracts through a point $w \in G$ for some open piece of directions (based at the origin in $T_w G$) is laid out in definition \ref{open-piece-retracts}. The condition of 
$\mathbb{C}$-extremality means that there are no (non-trivial) complex line segments in the boundary.
As the foregoing contains results involving the conditions
of $\mathbb{C}$-extremality and holomorphic extremality, let us mention here that: while the two notions are equivalent in the 
presence of convexity, this is not so in general, as is manifestly seen
in the class of analytic polyhedra. 
The inequivalence of analytic polyhedra with strictly pseudoconvex domains (one of the simplest classes of domains with holomorphically extreme boundaries) is well-known due to Henkin's 1973 article \cite{henkin}. Product domains also share the same property: none of them is biholomorphic to any domain with holomorphically extreme boundary; the techniques of Remmert -- Stein attribute such inequivalences to the lack of analytic varieties in the boundaries of domains in the latter class. 
The question about the irreducibility of analytic polyhedra $D$ and associated `co-analytic domains' $D \setminus A$ for analytic sets $A \subset D$, particularly when the codimension of $A$ is one, is therefore interesting. One cannot hope for results without imposing conditions as there are many analytic polyhedra which are completely reducible, the polydisc being the simplest example. The above result shows that irreducibility of $D$ is guaranteed, as soon as the milder $\mathbb{C}$-extremality condition holds at-least within the class of balanced domains. 
We hasten to mention that while there may well be shorter proofs than the one via retracts to prove the irreducibility of those analytic polyhedra $D$ as in the above theorem, it gets somewhat more difficult firstly for the associated co-analytic domains
$D \setminus A$ therein; and all the more difficult for more general cases. 
As our purpose in this article is to attain results characterizing retracts more than various applications, we do not strive to include beyond a couple of results of this kind.
We also hasten to mention that there are better techniques to detect the 
occurrence of a product structure for domains far beyond the class of balanced domains owing
particularly to the very recent work \cite{Bh-Bo-Su}, which however requires some quantitative knowledge 
about the squeezing function. However, we believe that
the approach via retracts is worth recording, as there are no universal techniques to establish biholomorphic inequivalences. We shall indeed show that there are plenty of 
balanced polyhedra as in the above theorem, in chapter \ref{Irreducibity}.\\

\noindent Getting back to retracts, let us note that while Vesentini's theorem 
says that {\it all} 
retracts of $D$ as in that theorem, passing through the origin, are linear, it does not say that 
all linear subspaces are retracts. Infact, this is utterly false as soon as we consider 
any such $D$ as in the above theorem other than balls linearly equivalent to the standard Euclidean ball; this is 
due to the main result of Bohnenblust in \cite{Bhnblst} (prior to which the real version was 
obtained by Kakutani in \cite{Kakutni}). The question of {\it precisely which} linear 
subspaces are realizable as retracts of $D$ as 
in Vesentini's theorem is significant for us. For instance: while
every retract of the polydisc $\Delta^N$ is a graph, not every graph in $\Delta^N$ is 
a retract; only those which are graphs on linear retracts of $\Delta^N$ are, making it
important to first precisely pin down which linear subspaces are retracts. Indeed, this 
is quite an interesting question in its own right and is tied to a classical question 
in the geometry of Banach spaces which in our context
of finite dimensional (complex) Euclidean spaces reads as: which are those linear subspaces 
of $\mathbb{C}^N$ equipped with some norm $\mu$, that are realizable as images of norm-one (linear) 
projections when we measure the norm of the projection on 
$\mathbb{C}^N$ with respect to the operator norm induced by $\mu$?
There are differences 
in the answers for the cases $p=1,\infty$
even if we exclude $p=2$, the well-known case
distinguished as arising from an inner product,
thereby not surprising at all.
Before getting to such
differences, let us mention 
that the result is uniform for all
finite $p$ other than two and just state 
this unified (known) result. As our 
interest lies in retracts, we shall phrase it in the language of retracts rather 
than norm-one projections.
Specifically, the characterization of linear retracts (thereby those holomorphic retracts which pass through the origin, owing to Vesentini's theorem) of the standard $l^p$ 
ball denoted $B_{\ell^p}$, can be 
readily deduced from those of norm-one projections in $\mathbb{C}^N$ equipped 
with $l^p$-norm recorded in the well-known book on classical Banach spaces by 
Lindenstrauss -- Tzafiri and it reads as follows for the finite dimensional case
wherein the result holds for $p=1$ as well, unlike its infinite dimensional counterpart.

\begin{thm}\label{Lindenstrauss -- Tzafiri}  \rm (Theorem 2.a.4 of \cite{Lindstraus}):
Let $Y$ be a linear subspace of $\mathbb{C}^N$ equipped with $\|\cdot\|_{l^p}$ 
where $1 \leq p < \infty, p \neq 2$. Then $Y \cap B_{\ell^p}$ is a linear retract of $B_{\ell^p}$ 
if and only if there exist vectors $\{u_j\}_{j=1}^m$ of $\ell^p$-norm 1 in $\mathbb{C}^N$ of the form
\begin{equation}\label{supp}
u_j = \sum_{i \in \sigma_j} \lambda_i e_i, ~ 1\leq j \leq m, ~ with~ \sigma_j \cap \sigma_k = \phi ~ for ~ k \neq j
\end{equation}
such that $Y = span\{u_j\}_{j=1}^m$, where $\sigma_j := supp(u_j) = \{1 \leq i \leq N: u_j(i) \neq 0\}$
(and $\lambda_i$'s are some complex numbers such that the $\ell^p$-norm of $u_j$ are equal to $1$). Here $e_i$ denotes the standard basis of $\mathbb{C}^N$.
\end{thm}

\noindent We shall derive
in section \ref{ell_q_ball},
alternative formulations (some of which are known) of the above 
result, which are more geometrically 
revealing 
to show that the linear
retracts of $\ell^p$-balls are precisely
isometric copies of lower dimensional
standard $\ell^p$-balls. Stated differently, 
only those linear subspaces of $\ell^p$-balls 
which are biholomorphic to lower dimensional
$\ell^p$-balls, turn out to be their holomorphic retracts through the origin. 
Next, we note that the case $p=2$ has been excluded only because in that case $B_{\ell^2}=\mathbb{B}$, wherein all 
linear subspaces are retracts. We would rather like to draw attention to the 
more important fact that this result is {\it not} true for $p = \infty$ to 
demonstrate which (by a counterexample as simple as possible), we include 
an example in section \ref{Polyballs&mainThms-sectn}. 
Nevertheless, the aforementioned
geometric characterization of the linear 
retracts of $\ell^p$-balls for finite $p$, as 
lower dimensional versions of themselves, holds true
for the polydisc as well.
Indeed, this can derived without much difficulty from \cite{Heath_Suff}, though it is not 
stated in these explicit terms.

\begin{prop} \label{Polydisk}
A linear subspace $L$ of the polydisc $\Delta^N$ is a holomorphic retract of $\Delta^N$ if and only if $L = Y \cap \Delta^N$, 
where $Y$ is a linear subspace of $\mathbb{C}^N$ which can be expressed as the span of vectors 
$v_1,\ldots,v_r$ (for some integer $r$ between $1$ and $N$) all from the boundary
of the polydisc, satisfying the following conditions:
\begin{enumerate}
\item [(i)] there exists a subset of indices $J = \{j_1,\ldots j_r\} \subset \{1,2,\ldots,N\}$, such that 
for each $k=1, \ldots,r$, the vectors $v_k$ may be expressed in terms of the standard basis 
vectors $\{e_i\}_{i=1}^{N}$ of $\mathbb{C}^N$, as:
\[ 
v_k = \left(\sum_{i \notin J}c_{i,j_k}e_i\right) + e_{j_k}
\]
\item [(ii)] For each $m \in  \{1,2,\ldots,N\} \setminus J $, the 
scalars $c_{m,j}$  satisfy $\sum_{j \in J}|c_{m,j}| \leq 1$. 
\end{enumerate}
Consequently, non-trivial linear retracts of the polydisc $\Delta^N$ are precisely 
those linear subspaces which are isometric copies of 
lower dimensional polydiscs $\Delta^k$ (for $1 \leq k \leq N$).
\end{prop}       

\noindent It follows that all holomorphic retracts of 
$\Delta^N$ (not necessarily passing through the origin) are biholomorphic copies of lower
dimensional polydiscs.
It is immediate that through any given point of the polydisc, there exist
retracts of all possible dimensions ($k=0,1,\ldots, N$), as its automorphism group acts transitively;
a little contemplation shows that the transitive action is unnecessary to confirm the
existence of such retracts. Indeed, for every $1 \leq p < \infty$, the $\ell^p$-ball fails to be homogeneous unlike the 
polydisc as soon as $p \neq 2$; but it is easy to derive as a corollary of theorem \ref{Lindenstrauss -- Tzafiri}
that through each point of the $\ell^p$-ball, there exist
retracts of all possible dimensions, as the retracts through the origin serve to confirm this. It is too much to ask for retracts along any given set of directions
because as the above proposition shows: not every two-dimensional subspace is a retract, a fact which holds
not only for $\Delta^N$ but also for every ball that is not linearly equivalent to the 
ball with respect to the $l^2$-norm,
$\mathbb{B}^N$. We must not be amiss to note however that {\it every} one dimensional 
linear subspace 
of a bounded balanced {\it convex} domain  in $\mathbb{C}^N$ is a retract -- this 
follows by the Hahn--Banach theorem. It is natural to 
ask whether this holds if we relax the convexity condition on the domain to pseudoconvexity. The answer is in 
the negative, as was noted by Lempert in \cite{Lempert-FundArt} by indicating an example to complement his 
passing observation therein,
that the existence of $2$-dimensional retracts is rather restricted i.e., to the effect that when convexity 
is relaxed, even $1$-dimensional retracts can fail to exist in abundance. 
The only rigorous record of the just-mentioned Lempert's example 
is in \cite{Jrncki_invrnt_dst} (specifically in example 11.3.10 therein). 
While this relies on the usage of complex geodesics and is done only for one particular 
direction, 
we found that we can record a simpler example wherein we can begin to concretely address the 
question about retracts of higher dimensions as well; moreover this
can be
achieved entirely through elementary means. In-fact, our example is as simple 
as the standard $\ell^q$-`ball' for $q<1$ and we record this in the following proposition; an 
interesting feature of the $\ell^q$-`ball' in $\mathbb{C}^N$ is that apart from failing to 
be convex, it fails to be concave unlike its real counter-part.
Let us also note that while the above-mentioned example of Lempert
is a pseudoconvex complete Reinhardt domain much as ours, we believe recording the aforementioned $\ell^q$-`ball' is 
worthwhile due to the fact that the literature about retracts of dimension {\it higher than} one, is 
scant at-least compared 
to complex geodesics and one-dimensional retracts. In that regard, we would like to mention the very 
recent article \cite{GZ} \footnote{This appeared on the arXiv, a couple of months after the earlier versions of
our article}. Ghosh -- Zwonek attain in this article \cite{GZ}, a complete characterization of the two dimensional retracts 
(and thereby all retracts) of the three dimensional Lie ball and the tetrablock; this article also has
noteworthy general results about retracts of Lempert domains. We now get to the special case of
the $\ell^q$-`ball' for $q<1$ in arbitrary (finite) dimension, which is 
amongst the simplest examples of (pseudoconvex) domains which fail to be a Lempert domain (see proposition
\ref{Lemp_cnvx} below as well); indeed, we deal with large collections of domains that fail to be 
Lempert domains. 

\begin{thm} \label{ell_q_ball_C3}
Let $ q = (q_1,\ldots,q_N)$, where $q_j$'s are arbitrary positive numbers, all less than one. Consider the $\ell^q$-`ball' given by 
\begin{equation}\label{egg-decouple}
D_q = \{(z_1,\ldots,z_N) \in \mathbb{C}^N \; :\; |z_1|^{q_1} + \ldots + |z_N|^{q_N} < 1\}.
\end{equation}  
The holomorphic retracts through the origin of $D_q$ are 
precisely the orthogonal projections of $D$ onto a
span of a finite subset of the standard basis vectors of $\mathbb{C}^N$; thereby, the retracts are all copies of lower-dimensional $\ell^q$-balls.
\end{thm}
The techniques used to obtain this theorem can also be used to characterize retracts through the origin of some (bounded balanced pseudoconvex) \textit{non-decoupled domains} as well (see corollary \ref{non-decoupled}).
Note that the boundaries of these domains are holomorphically extreme, in particular $\mathbb{C}$-extremal,
as is the case for $\ell^p$-balls for $1 \leq p < \infty$ as well; in such cases proposition 
\ref{JJ-improved} reduces 
considerations from holomorphic retracts to linear ones. 
We also include a couple of examples of certain balanced {\it non-convex} pseudoconvex domains
to illustrate the case where  
extremality condition fails (in the same section \ref{ell_q_ball} where we detail the proof of the above theorem). Indeed,
we shall show in such examples that retracts of all possible 
dimension exists through each point of the domain other than the origin as well. That retracts cannot
also be expected in any given direction for the non-convex case, is substantiated by $D_q$ as 
in the above theorem, owing to the spectacular failure even at the origin (of $1$-dimensional retracts,
along any direction other than the coordinate directions). \\

\noindent Getting back to analysing a bit further, examples
arising out of $\ell^p$-balls (where $p \geq 1$), let us 
remark a word or two, about generating examples of domains satisfying 
the $\mathbb{C}$-extremality 
condition in Vesentini's result, before we pass out of it. Firstly, while the
product operation fails to generate further examples of balls 
with $\mathbb{C}$-extremal property, intersection does i.e., the intersection of a pair of 
balls, possibly with respect to different norms, with $\mathbb{C}$-extremal boundary 
is again a ball (with respect to some norm) with {\it $\mathbb{C}$-extremal boundary};
Vesentini's theorem applies to pin-down the retracts through the center of 
such a `ball of intersection'. Now towards non-convex cases, Vesentini's theorem cannot in general be modified easily so as to adapt to other classes of domains despite proposition 
\ref{JJ-improved}; for instance, to deal with those analytic polyhedra whose boundaries
are $\mathbb{C}$-extremal. Such domains can never 
be holomorphically extreme (so proposition \ref{JJ-improved} cannot be applied), as their boundaries are
Levi flat on an open dense subset \footnote{Indeed, this is the maximum 
possible Levi flatness that can be expected in this context; for instance, while it is 
possible for a smoothly bounded domain $D$ to be holomorphically extreme
all along $\partial D$, it is impossible for $\partial D$ to be Levi flat everywhere.}.
However, we were able to attain the following result. 
\begin{thm}\label{max_noncnvx}
There do not exist any (non-trivial) retracts through the origin of (bounded) balanced $\mathbb{C}$-extremal analytic polyhedron $D$ in $\mathbb{C}^N$ along the direction of any of the points in the open faces of $\partial D$. To be more precise, suppose that $D$ can be described 
as
\[
D = \{ z \in U \; : \;  \vert f_k(z) \vert <1, \text { for all } 1 \leq k \leq m \},
\]
for some holomorphic $f_k$'s in a neighbourhood $U$ of $\overline{D}$ (and that $D$ happens to be an open, connected,
bounded, balanced subset of $\mathbb{C}^N$ with $\mathbb{C}$-extremal boundary). Suppose also that the system of defining functions $(f_k)$ is minimal i.e., none of the functions $f_k$ can be removed without changing $D$. Let $p$ be 
a boundary point on any of the open faces of $\partial D$, i.e. $p \in \tilde{\sigma}_j$ for some $j$ with $1 \leq j \leq m$, where 
\[
\tilde{\sigma}_j := \{z \in \partial D: |f_j(z)| =1,~|f_l(z)| < 1~\text{for}~1 \leq l \leq m~\text{with}~l \neq j\}.
\] 
Then there does not exist any retract $Z$ of $D$ passing through the origin tangent to $p$ thereat
i.e., with  $p \in T_0Z$.
\end{thm}

\noindent We know that $T_0 Z$ is itself a retract, indeed a linear retract. What the above theorem means for linear retracts of $D$ is that they cannot intersect the open faces $\tilde{\sigma}_j$. Thereby, every (non-trivial) linear retract is spanned by (at-most $N-1$) points from the lower dimensional ribs of the polyhedron. Consequently, owing to theorem \ref{Graph} below, every (non-trivial holomorphic) retract through the origin is a biholomorphic copy of a lower dimensional linear slice of the same polyhedron spanned by vectors from its ribs. \\

\noindent Existence of domains as in the above theorem in abundance will be shown in section \ref{Irreducibity}. A complete determination of retracts (even through the origin) of such domains seems possible only when either $f_j$'s are explicitly given or when sufficient data about them is given. We shall do this exercise only for what is perhaps the simplest of such polyhedrons namely,  
$D_h:=\{ (z,w) \in  \mathbb{C}^2 : \vert z^2 - w^2 \vert <1, \; \vert zw \vert <2 \}$. It may be noted that this itself is interesting to start with, as its boundary is {\it not} entirely complex non-degenerate, (i.e. its ribs has singular points within it).\\


\noindent The next main result of this paper, from which we deduce both the above theorem as well as theorem \ref{ell_q_ball_C3}, is the following.

\begin{thm}\label{cnvx_hull}
Let $D$ be a bounded balanced pseudoconvex domain in $\mathbb{C}^N$ and $C$ be the convex hull of $\overline{D}$. Let $L$ be any linear subspace which cannot be precisely spanned by vectors from $\partial D \cap \partial C$. Then $D_L:=D \cap L$ is not a linear retract of $D$.    
\end{thm}

\noindent We now lay down the aforementioned application of theorem \ref{ell_q_ball_C3} and theorem \ref{max_noncnvx} about retracts of $D$ to establish irreducibility, i.e., biholomorphic inequivalence with product domains. As a consequence of the foregoing pair of theorems, we see that a holomorphic retract $Z$ through the origin of an analytic polyhedron as in the above theorem \ref{max_noncnvx} has the property that $T_0Z$ intersects $\partial D$ only in its closed nowhere dense subset comprising ribs (formed by intersections of at-least two of the faces) of $\partial D$. Thereby, such retracts are sparse: $D$ does not admit any retract along any open piece of directions at the origin. This leads to theorem \ref{prod_anal_poly}; though this was already stated, the point here is to highlight that the  proof is very similar for theorem \ref{ell_q_ball_C3} we well. Indeed, to minimize repetition,
we prove both cases simultaneously in section \ref{Irreducibity}.
 
\begin{cor}\label{gen-anal-poly}
Let $D$ be a non-convex decoupled egg domain $D$ as in (\ref{egg-decouple}) of theorem \ref{ell_q_ball_C3}. Then $D$ is not biholomorphic to a product domain. If $A$ is a non-trivial analytic subset of $D$ not passing through the origin then the same conclusion holds for $D \setminus A$.
\end{cor}

\noindent Among others, the above corollary alongwith theorem \ref{prod_anal_poly} deals with three classes of domains: product domains, 
analytic polyhedra and `co-analytic' domains \footnote{This is not standard terminology; however, this helps prevent repeating long phrases to describe a class of domains that will occur oftentimes in this article. We shall also refer to them as `analytic complements' at times.} i.e., domains of the form $D \setminus A$ wherein $A$ is a non-trivial analytic subset of $D$ (we only deal with domains of holomorphy $D$ assuming ${\dim A}>0$ needless to say). Members of all these classes, share the
common feature that their boundaries have non-trivial analytic varieties in them. 
Corollary \ref{gen-anal-poly} is to the effect that nevertheless it is easy to get hold of mutually inequivalent members from these classes. In-fact, as we shall see along the way in section \ref{Irreducibity} (one of the major sections of this article), an analytic polyhedron is never equivalent to a co-analytic domain $D \setminus A$, whatever $D$ may be. As can be guessed, such inequivalences stem from the differences in the location of the 
analytic boundary pieces (for the former, these are not situated in the interior as is 
evidently the case with the latter). An analytic polyhedron can be equivalent to a product domain, the simplest example being $\Delta^2$; however, proposition \ref{bih_Delta^2} will confirm that this is the only bounded {\it balanced} domain in $\mathbb{C}^2$ for which such an equivalence occurs.
Finally and more importantly, to bring out the role of the extremality condition both in Vesentini's theorem and allied
results (such as the foregoing theorem \ref{cnvx_hull} and theorem \ref{MainThm} below), let us 
mention that the converse of Vesentini's theorem is also true: if all retracts through the origin of 
a bounded balanced convex domain are linear, then the domain is $\mathbb{C}$-extremal. In-fact, it is
enough to just have linearity of the one-dimensional retracts and furthermore, a general and local version can 
be formulated, namely as follows. 

\begin{thm} \label{converse-of-Vesentini}
Let $D$ be a bounded balanced pseudoconvex domain in $\mathbb{C}^N$ and $p \in \partial D$. Assume 
that $\partial D$ is 
convex near $p$ or more generally, convex at $p$ with the Minkowski functional of $D$ being continuous
near $p$. Suppose the only one-dimensional retract through the origin in the direction of $p$ 
is the linear one, then $p$
is $\mathbb{C}$-extremal i.e., there exist non-linear retracts, as soon as $p$ 
fails to be a  $\mathbb{C}$-extremal boundary point. 
\end{thm}

\begin{rem} \label{rem-cnvrse-Vesen}
\begin{itemize}
\item[(i)] We note that the hypotheses that $\partial D$ is convex near $p$ implies convexity of 
$\partial D $ {\it at} $p$ together-with the continuity of the Minkowski functional near $p$; this 
will be discussed during the course of the proof. 
When $\partial D$ is convex
near $p$, it is enough to check for the existence of
some holomorphic variety through $p$
rather than a linear one, in order to ascertain 
existence of non-linear retracts; this is due to the equivalence of 
$\mathbb{C}$-extremality with 
holomorphic extremality in the setting of the 
above theorem, which is discussed in 
proposition \ref{hol_C_extrm}.
It may happen that $\partial D$ is 
convex at some boundary point but not in any open piece of the boundary about it; to
mention a simple example, we may consider $l^q$-`balls' for $q<1$ in $\mathbb{C}^2$ and examine 
points $p$ on the coordinate axes. This gives an example wherein convexity near $p$ fails but
the weaker condition of convexity at $p$ together-with the continuity of the Minkowski functional near $p$ as in the statement of the above theorem,
indeed holds.
Finally here let us 
recall that if $p$ is a non-convex boundary point then there are no one-dimensional retracts through
the origin in the direction of $p$.\\

\item[(ii)] More importantly, a precise description of all the 
one-dimensional non-linear retracts through
the origin in a bounded balanced pseudoconvex domain along the direction of a boundary 
point $p$ at which $\partial D$ is convex but fails to be $\mathbb{C}$-extremal -- with 
no assumptions about the continuity of the Minkowski functional near $p$ as in the statement of the above theorem -- is  
attained in proposition \ref{non-linear-for-convrs-Vesen}. That leads to 
the following theorem about retracts of all dimensions.
\end{itemize}
\end{rem}

\begin{thm}\label{Charac-retract-1diml}
Let $D$ be a bounded balanced pseudoconvex domain in $\mathbb{C}^N$. 
Let $Z$ be any (non-trivial) retract of 
$D$ through the origin with the property that
$\partial D$ is convex near points 
where $L:=T_0Z$ intersects $\partial D$.
Then $Z$ can 
be realized as the 
graph of a holomorphic map over the linear subspace $D_L$ (which itself is a retract of $D$), taking 
values in  
$\bigcup_{p \in \partial D_L}F_\mathbb{C}(p, 
\partial D) \cap D_M$ where $F_\mathbb{C}(p, \partial D)$ 
is the $\mathbb{C}$-linear face of $\overline{D}$ 
corresponding to the boundary point $p$ and $M$ is 
a certain linear subspace complementary to $L$ (and as usual $D_M:=D \cap M$).


\end{thm}

\noindent In proving this theorem, we first derive from an idea from \cite{Abate_Isometry}, a general result without any convexity hypothesis, thus going beyond the purview of Vesentini's result. 
Namely, the following result wherein we assume boundedness. It must be noted that 
dropping boundedness leads to the failure \footnote{see proposition 
 \ref{not_graph} of section \ref{sectn-polyretract-of-C^2}.} of the result whereas replacing boundedness by the apparently 
more general condition of Kobayashi hyperbolicity is insignificant here,
owing to  
the already mentioned result of Kodama in \cite{Kodama},  
according to which any hyperbolic {\it balanced} domain is bounded.

\begin{thm}\label{Graph}
Let $D$ be any bounded balanced pseudoconvex domain in $\mathbb{C}^N$, $Z$ any (non-trivial) retract passing through the 
origin with $\rho: D \longrightarrow Z$, a retraction map. Let $L$ denote the tangent space to $Z$ 
at the origin. Then $\rho_{|_{D_L}}$ is a biholomorphism mapping $D_L = D \cap L$ onto $Z$ and upto 
a linear change of coordinates, $Z$ is the graph of a holomorphic map over $D_L$ taking values in $D_M:=\ker\left(D\rho(0)\right)$. 
\end{thm}

\noindent It follows immediately from this theorem that $Z$ is biholomorphic to a (lower dimensional)
the linear retract $D_L$. In the cases where $\partial D$ is holomorphically extreme, $Z=D_L$. Whatever be the case for $D$ and $Z$ \textit{as above}, several inheritance properties extending the list mentioned earlier follow: if $D$ is $\mathbb{C}$-extremal analytic polyhedron or, strongly pseudoconvex or more generally (pseudoconvex and) of finite type, then $Z$ is biholomorphic to a domain of the same kind. While it is true that when $D$ is an $\ell^p$-ball, $Z$ is biholomorphic to a lower dimensional $\ell^p$-ball, this is nowhere close to being as trivial as the just
mentioned inheritance properties.
Next, we note that it does not follow that any two retracts of the same
dimension through the
origin, are biholomorphically 
equivalent; indeed, even a pair of linear retracts
of the same dimension may fail to be biholomorphic. To mention a concrete
example, we may consider $D=\mathbb{B}^2 \times \Delta^2$ wherein both factors can be seen as two-dimensional retracts through 
the origin in $D$.
Relegating further relevant observations here to section \ref{Polyballs&mainThms-sectn}, 
we record here a corollary for an important class of domains.
Indeed, when $D$ as in the above theorem (without any extremality condition about any boundary point), is homogeneous
i.e., ${\rm Aut}(D)$ acts transitively on $D$, then of-course we 
may glean a characterization of {\it all}
retracts of such a domain.
 
\begin{cor}\label{homogeneous}
Let $D$ be any homogeneous bounded balanced domain in $\mathbb{C}^N$. Then, 
every non-trivial retract $Z$ (not necessarily passing through the origin in $D$) is 
biholomorphic via a member of ${\rm Aut}(D)$ to the
graph of a holomorphic map over a linear retract $D_L$ 
(which as usual is a ball with respect to some norm in the lower dimensional complex Euclidean space $L = T_0Z$), of $D$.  
\end{cor}

\noindent Owing to the assumption of being {\it balanced}, we recall that the 
domains as in this corollary are same as the (much studied) bounded balanced symmetric domains which are biholomorphic 
to (Cartesian) products of the standard Cartan domains of types I -- IV, 
all of whose automorphism groups are well-known (Hua \cite{Hua}); this also means that domains as in the 
above theorem are convex. This is a good place to draw attention of 
the reader to the fact
that a precise determination of even the linear retracts of a (convex)
ball -- as in theorem \ref{Lindenstrauss -- Tzafiri} for the
$l^p$-ball for $p<\infty$ and proposition \ref{Polydisk} for the $p=\infty$ 
case much earlier above, and few other relevant results as 
in the already mentioned very recent work \cite{GZ} 
which deals with among other things, Cartan domains of type IV -- is in general quite 
difficult. Moreover, it must be remembered that an orthogonal projection of 
such $l^p$-balls onto a linear subspace
in general is {\it not} a retract, except for the case $p=2$. 
As these are classical facts, the significance of {\it linear} retracts from the viewpoint of recent
research work may need a remark; let us only mention that it was shown by Mok in his fairly recent
work \cite{Mok}, that every linear totally geodesic submanifold through the origin of any bounded symmetric domain
$\Omega$ with respect to a ${\rm Aut}(\Omega)$-invariant K\"ahler metric (for instance, the Bergman metric)
is realizable as a retract just by orthogonal projection.
As this has been dealt with deeply in \cite{Mok}, we shall proceed forth next, to look at another direction (but still 
related to the foregoing corollary in a sense as shall be seen), in relaxing the assumptions 
in Vesentini's theorem; wherein which, we retain convexity but 
drop the extremality condition and ask if finer information can be obtained 
about the
retracts than the general one in the foregoing theorem \ref{Graph}.
Recall that Vesentini's result fails if we drop the extremality 
condition even if we retain convexity, as is manifestly seen 
in the simplest case of polydiscs. This motivates us to first enquire about the 
retracts of domains 
which {\it fail} to satisfy extremality conditions at {\it most} of its boundary points 
in the simplest of test cases such as domains decomposable into a
product of balls with respect to some norms (called polyballs in \cite{Kuczm_Rtrct_polybll}) satisfying an extremality 
condition, much as the polydisc. Pursuing this line 
of enquiry has led us to the result laid down next, which 
generalizes the result for polydiscs due to Heath and Suffridge, to a much larger class of domains and 
with a simpler proof following the recent work \cite{BBMV-union}. 
As indicated above, this theme of retracts of polyballs has indeed been considered 
earlier for instance in \cite{Kuczm_Rtrct_polybll}; it was shown therein that fixed point 
sets of holomorphic self-maps of 
such $N$-fold product of balls are always retracts as well, leaving to be desired however, as to what 
the retracts of polyballs look like. We 
shall address this here, in the following next main result of this paper.

\begin{thm} \label{MainThm}
\begin{itemize}
\item[(A)] Let $B_1,B_2$ be a pair of bounded balanced domains of 
holomorphy with holomorphically extreme boundaries,
possibly in 
different complex Euclidean spaces.
Then,
every non-trivial retract of $D=B_1 \times B_2$ through the origin is the graph, either of a $B_2$-valued holomorphic 
map over a (complex) linear subspace of $B_1$ or, of a $B_1$-valued holomorphic map
over a linear subspace of $B_2$.\\
More specifically and precisely, 
let $B_1 \subset \mathbb{C}^{N_1}$ and 
$B_2 \subset \mathbb{C}^{N_2}$  be as above i.e., not necessarily convex but pseudoconvex bounded balanced 
domains with holomorphically extreme boundaries. Let $D=B_1\times B_2$ and suppose without loss
of generality that $\min\{N_1,N_2\}=N_1$. Let $\mu_1,\mu_2$ denote the Minkowski 
functionals of $B_1,B_2$ respectively.
Let $Z$ be any non-trivial retract of $D$, $L=T_0Z$, the tangent space to $Z$ at the origin, $D_L=L \cap D$ and 
$\partial D_L = L \cap \partial D$. Then:\\

\begin{itemize}
\item[(i)] if $\partial D_L \subset \partial B_1 \times \partial B_2$, then $Z$ is a complex 
linear subspace of $D$ which can be realized as the graph of a complex linear map over 
the linear subspace $\pi_1(D_L)$ of $B_1$:
\[
Z=\{ \big( w, \beta_1(w), \ldots, \beta_{N_2}(w)\big) \; : \; w \in \pi_1(D_L) \}
\]
where the mapping
\[
w \longmapsto \big( \beta_1(w), \ldots, \beta_{N_2}(w) \big)
\]
maps $B_1^L:= \pi_1(D_L)$ into $B_2$,\\
 
\item[(ii)] if $\partial D_L \subset \partial B_1 \times \overline{B}_2$, then $Z$ is the 
    graph of a $B_2$-valued holomorphic map over a linear subspace $(B_1)_{L_1}$ of $B_1$, given by $(B_1)_{L_1} := B_1 \cap L_1$, where $L_1 = \pi_1(L)$.

\item[(iii)] if $\partial D_L \subset \overline{B}_1 \times \partial B_2$, then $Z$ is 
    the graph of a $B_1$-valued holomorphic map over a linear subspace $(B_2)_{L_2}$ of $B_2$, given by $(B_2)_{L_2} := B_2 \cap L_2$, where $L_2 = \pi_2(L)$.\\

\item[(iv)] if $\partial D_L$ intersects 
both $\partial B_1 \times B_2$ and
$B_1 \times \partial B_2$  -- which are disjoint open pieces of the boundary of $D$, indeed connected components of 
$\partial D \setminus (\partial B_1 \times \partial B_2)$ --  then $Z$ is a linear retract of $D$ of complex dimension 
$\geq 2$, which is neither contained in $B_1$ nor in $B_2$.\\
\end{itemize}
\item[(B)] Suppose now that 
$B_1 \subset \mathbb{C}^{N_1},B_2 \subset \mathbb{C}^{N_2}$ are bounded 
balanced domains of holomorphy (i.e., with the 
holomorphic extremal assumption on their boundaries 
dropped). Then statements (ii) and 
(iii) of (A) above, hold true. Statements (i) and (iv) hold
true if $Z$ is replaced by $D_L$ (thus they hold when $Z$ is linear, in which case $Z=D_L$; if $Z$ is not
linear, then $Z$ is the graph over $D_L$).
\end{itemize}
\end{thm}

\noindent As the first assertion about retracts through the 
origin is already contained in theorem \ref{Graph}, the real
matter of here is the finer information in the statments (i) -- (iv) of (A) which we note
are not mutually exclusive; it may also be noted that (ii) and (iii) 
are actually exhaustive. 
We further note that it may 
very well happen that neither $\pi_1(Z)$ nor
$\pi_2(Z)$ is a retract of the 
respective factors
$B_1,B_2$ (see the example in remark \ref{pi_1(Z)}).
Next, we note that we cannot {\it trivially} extend all 
of part-(A) to  extended products 
$D=B_1 \times B_2 \times \ldots \times B_N$ in the following manner  
even in the special case wherein all the $B_j$'s to be bounded balanced 
convex domains each of whose boundaries satisfy the strict convexity 
condition. Namely: by replacing one of the factors in the above theorem by a 
product of $(N-1)$ of these balls i.e., writing for 
instance $D=(B_1 \times B_2 \times \ldots \times B_{N-1}) \times B_N$,
appealing to the above theorem and using an inductive argument; 
this is due to the fact that 
$B_1 \times B_2 \times \ldots \times B_{N-1}$ fails to satisfy the extremality  condition 
stipulated in part-(A) as soon as $N>2$. However, such an extension is indeed possible and not 
difficult either; provided only that we repeat the same arguments multiple times by projecting 
onto the individual factors $B_j$ to harness the strict convexity  condition 
satisfied 
by each of these factor balls of the product $D$. As the complete formulation of the 
end result of this, turns out to 
be lengthy, we shall just lay it down in section \ref{Polyballs&mainThms-sectn}. 
Here we record an application of part-(B) of the above theorem (in-fact, its essence namely theorem \ref{Graph}, suffices)
to: the `union problem' for a nice class of polyballs which is interesting owing to it being a case with 
some degeneracy; this is based on the ideas of Forn\ae ss - Sibony in their article \cite{Fost}. Recall that the 
union problem comprises the following: let $M$ be a complex manifold which is 
the union of an increasing sequence of open subsets $M_j$ each of which is biholomorphic to a fixed 
domain $\Omega \subset \mathbb{C}^N$; describe $M$ in terms of $\Omega$.\\

\begin{thm}\label{union-prob-polyballs}
Take $\Omega$ in the union problem to be any homogeneous bounded balanced domain; as
is well-known such a domain is a polyball. If 
$M$ is non-hyperbolic and the corank of its (infinitesimal) Kobayashi 
metric is one, then $M$ is biholomorphic to 
$\Omega_L \times \mathbb{C}$ where $\Omega_L$ is a linear retract of the polyball $\Omega$.
\end{thm}

\noindent Next, as most balls fail to be homogeneous, the problem of
characterizing retracts of bounded balanced pseudoconvex domains 
{\it not} passing through the origin, is in general more difficult and significant. Towards this, we report the 
following result.

\begin{thm}\label{not-thro-origin}
Let $D$ be a balanced pseudoconvex domain in 
$\mathbb{C}^N$, $\rho: D \to Z$ a (non-trivial) retraction map and $p$ the point 
in $Z$ which is the image of the 
origin under $\rho$; so, $Z$
is a retract not necessarily passing
through the origin.
Suppose that: (i) the slice
$D_L:=D \cap L$ is a bounded convex
domain in $L:=T_pZ$ with $\mathbb{C}$-extremal
boundary
and (ii) $D\rho_{|_{0}}$ is a projection which maps $D$ onto $D_L$, where $L = T_pZ$.
Then $Z = \rho(D)$ is the graph of a holomorphic map over 
the linear subspace $D_L$ of $D$. 
\end{thm}

\noindent Examples wherein all the
hypotheses of the theorem are satisfied will be provided in the section where we lay down the proof. To move forward, we note that we have for the most part focussed attention
on bounded balanced pseudoconvex domains which being contractible, may be regarded as being 
topologically trivial; moreover, they 
have other special properties such as Stein neighbourhood bases, and plurisubharmonic barriers all along their boundary (i.e., are hyperconvex) as soon as the Minkowski function is continuous. To turn next to topologically non-trivial cases failing to be hyperconvex,
one simple case to begin with in this direction, is that
of analytic complements of domains with a
Stein neighbourhood basis; specifically, see the remarks 
following lemma \ref{Retr-of-anal-complements}.   
Next, to consider the simplest domain lacking a Stein neighbourhood basis {\it and} plurisubharmonic barriers throughout the boundary then,
we are led to considering the standard domain in this context namely the 
Hartogs triangle given by 
\[
\Omega_H = \{(z, w) \in \mathbb{C}^2 \;: \;\vert w \vert < \vert z \vert < 1  \}.
\]
We also contend that this is perhaps the simplest topologically non-trivial case
to consider because it is biholomorphic to the product 
$\Delta \times  \left( \Delta \setminus \{0\} \right)$;
indeed the map $\phi(z,w)=(w/z, z)$, effects the aforementioned biholomorphic correspondence.
Thus it suffices to determine the retracts of $\Delta \times \Delta^*$, which 
we do in the following theorem giving a complete characterization of
all retracts of the Hartogs triangle as well as its `analytic complements'; this will be deduced as a consequence of a far more general lemma about retracts of 'co-analytic' domains namely lemma \ref{Retr-of-anal-complements}.

\begin{thm}\label{Hartogs triangle}
Every non-trivial retract of $\Delta \times \Delta^*$ is a holomorphic graph over one of 
the factors; more precisely, either of the form
$\{(z,f(z)): z \in \Delta\}$, where $f$ is a holomorphic map from $\Delta$ to $\Delta^*$,  or
of the form $\{(g(w),w): w \in \Delta^*\}$ for some holomorphic map self map $g$ of $\Delta$. 
Moreover, upto conjugation by an automorphism of the bidisc $\Delta^2$,  
every non-trivial retraction map is of the form 
\[
\left(z, w\right) \rightarrow\left(z, f\left(z\right)\right)
\] 
for some holomorphic $f:\Delta \to \Delta^*$, or of the form
\[
\left(z, w\right) \rightarrow\left[t z 
+ (1-t) e^{-i \alpha} w+\left(e^{-i \alpha} w-z\right)^2 h\left(z, w\right)\right]\left(1, e^{i \alpha}\right)
\] 
for some real $\alpha$, $t \in (0,1)$ and holomorphic function $h:\Delta^2 \to \Delta$. \\
Finally, retracts of $\Omega_H^A:=\Omega_H \setminus A$ where $A$ is any analytic subset of $\Omega_H$ 
are precisely given by $Z \cap \Omega_H^A$ where $Z$ is a retract of $\Omega_H$.
\end{thm}

\noindent Next, it may be noted that an overarching hypothesis that we have imposed on 
all domains considered 
thus far, is their boundedness or, equivalently in our major context of balanced domains here, hyperbolicity. 
To deal with the {\it simplest} case of unbounded and non-hyperbolic domains then,
we first consider a product domain in which one of the factors is still a bounded domain 
but the other unbounded.
We thought it better to separately record the 
result for the simplest case of this kind, namely of the form $\mathbb{C} \times D$ where $D$ is 
a bounded planar domain; the characterization of its 
retracts is laid down in the following.

\begin{thm} \label{1st-unbdd & top-nontrivial}
If $D$ is a bounded domain in $\mathbb{C}$, then every non-trivial retract of $\mathbb{C} \times D$
is either of the form $\mathbb{C} \times \{c\}$ for some $c \in D$ or, of the form
$\{\left( f(w),w \right)\; : \; w \in D\}$ for some holomorphic function $f$ on $D$; and, conversely
every such graph is a retract.
\end{thm}

\noindent The above theorem allows for domains 
with a variety of fundamental groups; more such variety is facilitated by 
the next set of results.

\begin{prop}\label{discrete-removed}
Let $X = \mathbb{C} \setminus A$ where $A$ is any discrete set (allowed to be empty as well), and $D$ any 
bounded domain in $\mathbb{C}^N$. Then, every retract $\widetilde{Z}$ of $X \times D$ is 
precisely of the form 
\[ 
\widetilde{Z} = \{(f(w),w) : w \in Z\},
\] 
where $Z$ is a retract of $D$ and $f$ is a holomorphic function from $Z$ to $\mathbb{C}\setminus A$, or of the 
form $X \times Z$
where again $Z$ is a retract of $D$.
\end{prop}

\noindent While this article is focussed on characterization of 
 retracts rather than 
 retraction {\it mappings}, we thought it noteworthy to record in 
passing, a complete 
characterization
of the retraction maps on the most special domain among the domains 
in the above theorem namely, $\mathbb{C} \times \Delta$. 

\begin{thm}\label{C-minus-Delta}
Every non-trivial retraction map $F$ of $\mathbb{C} \times \Delta$, upto 
conjugation by an automorphism (so that $F$ fixes the origin), is precisely:
either of the form $F(z,w) = \left(z+wh(z,w),0\right)$ where
$h \in \mathcal{O}(\mathbb{C} \times \Delta)$, or of the 
form $F(z,w) = (h(w),w)$ where $h \in \mathcal{O}(\Delta)$. 
\end{thm}

\noindent While this article also focusses attention for the most part on
domains in $\mathbb{C}^N$, we would like to interject here a result about
retracts of manifolds, as it is a by-product of the essential ideas 
in the proof of the foregoing theorem \ref{1st-unbdd & top-nontrivial} and 
 includes cases with far greater topological complexity. For instance, fundamental groups of planar domains are countably generated and in fact always freely generated, whereas for Riemann surfaces this is in general not the case.

\begin{thm} \label{RiemSurf x HyperbolicMfld}
Consider the following two cases for a product manifold $X \times Y$.
\begin{itemize}
\item[(i)] $X$ is any compact Riemann surface and $Y$ is a non-compact Riemann surface,
\item[(ii)] $X$ is a non-hyperbolic Riemann surface and $Y$ is a (Kobayashi) hyperbolic complex manifold. 
\end{itemize}
Then, every holomorphic retract $\widetilde{Z}$ of $X \times Y$ is precisely: either 
of the form
\[ 
\widetilde{Z} = \{(f(w),w) : w \in Z\},
\]
where $Z$ is a retract of $Y$ and $f$ is a holomorphic map from $Z$ to $X$
or, of the form $X \times Z$ where again $Z$ is a retract of $Y$.
\end{thm} 

\noindent It must be noted that this theorem does not subsume proposition \ref{discrete-removed} above, for both factors of the product therein may very well be hyperbolic. We hasten to mention that there are indeed similar ideas in the proofs of the foregoing four results on products, which are therefore dealt with together in section \ref{Products-unbdd-topnontrivial} towards reducing repetition of details. \\

\noindent Finally, as one concrete but `extreme' example of an unbounded domain, we may simply 
take the domain to be $\mathbb{C}^N$ itself. While it is rather hard to 
precisely characterize all retracts of $\mathbb{C}^N$, and
we shall exemplify the 
possibility of non-hyperbolic domains possessing hyperbolic retracts, it is easy to observe a 
qualitative general feature common to all retracts of $\mathbb{C}^N$ apart from their 
topological triviality namely: they are non-hyperbolic to the same (maximal) 
extent as $\mathbb{C}^N$ i.e., their 
Kobayashi metric vanishes identically. It is also easy to see that the 
one-dimensional retracts of $\mathbb{C}^N$ are
biholomorphic to $\mathbb{C}$.
For more precise results, we consider $\mathbb{C}^2$ -- indeed, this is the simplest of balanced domains of holomorphy
or for that matter, the simplest of domains with non-trivial 
retracts -- and ask for a complete characterization of 
{\it all} of its holomorphic retracts (which seems unknown). Despite
the apparent simplicity of the domain, this
seems rather non-trivial; indeed, while graphs (in $\mathbb{C}^2$) of 
entire functions on $\mathbb{C}$
are retracts and the fact that retracts of products of planar domains in the foregoing, have 
literally turned out 
to be graphs of holomorphic maps over one of the factors, 
it is not the case that: every retract of 
$\mathbb{C}^2=\mathbb{C} \times \mathbb{C}$ is the graph of some entire function
on either of its factors,
as we shall show by providing a concrete example in the last section \ref{sectn-polyretract-of-C^2}. On the 
other hand, it is possible to give a complete characterization of polynomial retracts of $\mathbb{C}^2$
and this in essence is  
attained by Shpilrain and Yu
in \cite{Shpilrain -- Yu} 
wherein it is deduced from other 
known
results about polynomial rings and their retracts,
specifically, the Abhyankar -- Moh theorem and a
result due to Costa in \cite{Cost}. One may even wonder as to why is this result
not a direct corollary of the Abhyankar -- Moh theorem;
we address this in the final section. We shall however
give a self-contained and elementary proof
not relying on such theorems, 
which does not seem to be recorded in the 
literature to the best of our knowledge; we believe such a proof will help clarify where exactly such theorems are actually necessary. We emphasize for one last time
that this
result is not new, and is due to various major contributors as clarified above.

\begin{thm}
\label{polyretract-of-C^2} 
    If  $Z$ is a non-trivial polynomial retract of $\mathbb{C}^2$, then 
    there exists a polynomial automorphism  $\Phi$ on $\mathbb{C}^2$ such 
    that  $\Phi(Z) = \{(\zeta,0)\in \mathbb{C}^2 \; :\; \zeta\in \mathbb{C} \}$.
\end{thm}
\noindent Such a `rectifiability' 
assertion cannot in general be made for all holomorphic retracts of $\mathbb{C}^N$
for any $N\geq 2$, as follows by a simple 
analysis of the proof of the main result 
in the work \cite{Forstn}
due to Forstneric--Globevnik--Rosay.\\

\noindent \textbf{Acknowledgments}: 	
We thank Gautam Bharali, C. R. Jayanarayanan, Rohith Varma, J. Jaikrishnan and Kaushal Verma. 
The second named author would like to
acknowledge the support of the CSIR-JRF fellowship provided by the Council of Scientific 
and Industrial Research, India.

\section{Preliminaries} \label{preliminaries}
\noindent We recall certain basic and general facts about retracts which would 
needless to say, apply to our specific cases and put it
in good perspective, as well. While we would like to minimize repetition, we prefer to 
put the details on record to ensure a rigorous exposition. We believe many results here (if not all) are probably folklore or even if new, are relatively straightforward but essential. 
However, we repeat here a few generalities from the introductory section for one last time, to cast this section and the rest,
systematically.
Let us begin with the most general definition relevant to us:  
a self-map $\rho$ of a {\it set} $X$ is called a retraction if $\rho\circ \rho = \rho$; note that this 
idempotent condition on $\rho$ means that $\rho$ acts as the identity map on its range $Z=\rho(X)$. So, 
for any Hausdorff topological space $X$, its retracts,  which are the images of {\it continuous} 
retractions, are closed subsets of $X$; henceforth Hausdorffness on $X$ will be 
in force. We regard $\rho$ as being non-trivial, if
it is neither a constant map nor the identity. 
In addition to the contrasts between proper
maps and retractions mentioned in 
the introduction, we note that non-trivial retraction mappings may very well fail to be 
closed maps notwithstanding their ranges being closed subsets always; indeed, as soon as $X$ is a 
non-compact complex manifold -- the general setting of
interest in this article -- they {\it always} fail to be closed maps, despite the 
general fact that 
$\rho(X)$
is always closed as a subset of $X$, for every
retraction map $\rho$ on $X$ as just mentioned (proofs detailed later below). However, 
the feature of 
being a quotient map is common to both retraction maps as well as proper maps. 
Getting back to the general case of a Hausdorff topological space $X$,
if we add the assumption that $X$ is connected, then so are every of its 
retracts $Z$. When $X$ is furthermore contractible i.e., the 
space $X$ deformation retracts to a 
point then by composing any such deformation retraction with the retraction map $\rho$, gives a 
homotopy {\it within $Z$} that reduces $Z=\rho(X)$ to a point, meaning that if $X$ is contractible 
then so are all of its retracts. If $X$ even happens to be a convex domain in some real 
topological vector space, then the retracts of $X$ are also strong deformation retracts; indeed, the 
homotopy $H(t,x)=(1-t)x + t \rho(x)$ substantiates this. We 
are ofcourse interested in {\it holomorphic retracts}: if $X$ is a complex manifold and 
$\rho$ is a holomorphic self-map of $X$ which is idempotent then, we say 
that $\rho$ is a {\it holomorphic retraction} and its image $Z=\rho(X)$ a {\it holomorphic retract}; as 
we shall only be talking about holomorphic retracts, we shall most often drop the 
adjective `holomorphic'. Moreover as mentioned in the introduction, our focus will 
be to determine the form of the 
retracts $Z=\rho(X)$ rather 
than the retractions $\rho$. A first basic and general fact about holomorphic 
retracts  that is helpful to have in the background for the discussion in the sequel is 
laid down in the following well-known theorem which says that they are not just analytic 
varieties but actually (connected) manifolds i.e., are smooth and have the same 
dimension throughout.

\begin{thm} \label{Firstthm}
Let $X$ be any complex manifold. Then any of its retracts is a (connected) complex 
submanifold of $X$ which is closed in $X$. 
\end{thm} 

\noindent An exposition by Abate of the proof due to H. Cartan of this theorem, can be 
found in the monograph \cite{Metrc_dynmcl_aspct} (specifically, lemma 2.1.3 therein). Abate calls 
this proof `clever' and 
interestingly, this rather short and
slick proof was given by Cartan only in 1986 in \cite{Cartan-on-retracts}; Abate also 
notes that the theorem was proven for the first time
by Rossi in \cite{Rossi}. We sketch here, a flawed proof -- the only place in this 
article where we deliberately do so -- based 
on standard ideas in this context, for the sole purpose that it will serve in introducing some of 
the ideas (and our notations) to be used later while outlining the reasons 
underlying the necessity of a more careful approach as in the aforementioned proofs by Rossi and Cartan, to appreciate
them better. We shall of-course clearly point out the flaw at the appropriate juncture,
so as to be able to utilize the remaining (correct!) facts detailed in what follows.
Indeed then, let 
$\rho: X \to Z$ be a retraction of $X$ onto $Z$. Rewrite the defining equation of $\rho$ being a retraction 
map namely, $\rho(\rho(x))= \rho(x)$ for all $x \in X$, in its infinitesimal form by 
differentiating using the chain-rule, as
\[
D\rho(\rho(x)) D \rho(x) = D \rho(x),
\]
which holds for all $x \in X$. For $x=z \in Z$, $\rho(x)=x$ and so this simplifies to say 
that for every $z \in Z$, $D\rho(z)$ is an idempotent linear operator
on $T_zX$ (the tangent space to $X$ at $z$); although fairly trivial, this will be of 
much use in the sequel, so we record this as a lemma for later reference.

\begin{lem} \label{Firstlem}
Let $X$ be a complex manifold of dimension $n$, $Z$ a retract of $X$ and 
$\rho: X \to Z$ a retraction map. Then, for any $p \in Z$, the derivative $D \rho(p)$ is 
a linear projection (thereby a retraction) of $T_pX$ (the tangent space to $X$ at $p$) onto some linear subspace
$L$ of $T_p X$.
\end{lem}

\noindent As the assertion is essentially `local', we may fix an arbitrary point $p \in Z$ and work in a chart about $p$. 
In particular, consider the mapping on a chart neighbourhood $U_p$ of $p$, given by
\[
F_p(x) := \left(D\rho(p) \circ \rho\right)(x) \; : \; U_p \to T_p X = L \oplus E \;\text{ (say)}.
\]
Denoting $D\rho(p)$ by $\pi_p$ we have
\[
DF_p(x) = \pi_p \circ D \rho_p = \pi_p^{\circ 2} = \pi_p
\]
implying that $F_p$ is of full rank. Note that $F_p$ maps $U_p$ into $L$.

\medskip

\noindent Continuing with the proof of theorem \ref{Firstthm}, we shall now show that in a 
neighbourhood of any given $p \in Z$, $Z$ is the level set of a holomorphic map whose 
derivative has constant rank near $p$; the 
implicit mapping theorem will then render the 
claim of the theorem. To this end then, we fix an arbitrary $p \in Z$ and passing to a 
holomorphic chart of the complex manifold $X$ near $p$, we may as well 
assume (as already mentioned), that we 
are in the complex Euclidean space $\mathbb{C}^N$, which allows us to 
define $\beta(x) = x - \rho(x)$ for $x$ near $p$, whose zero set is 
precisely $Z$ (all in the charts near $p$, to be precise). To show that the constancy of the 
rank of the derivative $D\beta(x)$ as mentioned, we first consider the derivative 
of $\rho$. The idempotence of the linear operator $D \rho(p)$ given by the above 
lemma, guarantees its diagonalizability as well. So, after a linear change of 
coordinates, $D \rho(p)$ is a diagonal matrix with diagonal entries either $1$ 
or $0$. Let $r$ be the number of $1$'s i.e., 
$r={\rm rank} (D \rho(p))$. We may as well assume by a permutation of coordinates if 
needed, that the first principal $r \times r$ submatrix of $D\rho(p)$ is the identity 
matrix of order $r$, which by continuity of the determinant function implies that the 
determinant of the same principal minor of the nearby derivative matrices $D \rho(x)$ 
for all $x$ in some neighbourhood (contained in the chart) about $p$, is non-zero as 
well. This just means the well-known lower semi-continuity of the rank function, which 
when applied to $x \mapsto {\rm rank} (D \rho(x))$ implies that the rank of $D \rho(x)$ is 
atleast $r$ for all $x$ in some neighbourhood of $p$. Now, observe that the same arguments may 
be applied to $D \beta(x)$ as well. Firstly, note that $D \beta (p) = I - D\rho(p)$ where $I$ 
is the identity matrix of size $n$, is also an idempotent matrix, indeed diagonal with diagonal 
entries only $1$ or $0$ and with the number of zeros being exactly $n-r$. In particular, 
${\rm rank}(D \beta(p))=n-r$ and the lower semicontinuity of the rank applied 
to $x   \mapsto {\rm rank} (D \beta(x))$ implies that the 
rank of $D \beta(x)$ is atleast $n-r$ in a neighbourhood of $p$. 
Now, we have a 
similar lower bound on the rank of $D \rho(x)$ as already noted, 
which is tempting to be
translated this into an bound in-terms of $\beta$ instead of $\rho$; specifically, the upper 
bound ${\rm rank} (D \beta(x)) \leq n-r$ in the some neighbourhood of $p$. While this is 
correct for $x=p$ and even {\it along} $Z$, it is not at all apriori clear as to why this bound 
continued to hold in an open neighbourhood of $p$. The reason for detailing this flawed argument
is that it is then easy to finish the proof, within {\it standard} elementary ideas of calculus on
differential manifolds. Indeed firstly, the foregoing 
pair of bounds when put
together, only means that the rank of $D \beta(x)$ is constant in some neighbourhood 
of $p$, leading to the conclusion by the constant rank theorem that $Z$ is a complex 
manifold of dimension $r$ throughout, as we already know $Z=\rho(X)$ is connected but alas, this 
part of the argument is 
flawed, as pinpointed. This finishes the discussion of the
proof of \ref{Firstthm}. Let us round this by noting a couple of very easy but noteworthy facts.
Retracts of compact complex manifolds are compact complex submanifolds. Retracts of Stein manifolds 
are Stein and likewise, the property of a complex manifold being an Oka manifold, is preserved
under passage to a retract thereof, as well. In addition to the contrasts between retractions and proper mappings mentioned in the introduction, we have the following. \\
\begin{prop}
Let $X$ be a non-compact complex manifold with 
$Z$ being any non-trivial retract of $X$. Then none of the
retraction maps 
$\rho: X \to Z$ is a closed map.
\end{prop}
\begin{proof}
Choose any convergent sequence $(z_n)$ in $Z$ of distinct points; as we already know $Z$ is closed, thereby the limit $z_0$
of the chosen sequence lies in  $Z$. Our strategy
will be to pick points $p_n$ in the fibers over $z_n$ such
that the $p_n$'s form a discrete, thereby a closed subset $C$ \
of $X$. If we let $K$ denote the (infinite) 
compact set $K = \{z_n\} \cup \{z_0\}$ and note that $\rho(C)$ 
equals $K$ without its limit point, we would have proven the
proposition.
The possibility of choosing a sequence $p_n$ as desired
above really requires the non-compactness in hypothesis.
Indeed, consider the sequence of fibers of $\rho$ given 
by $F_n := \rho^{-1}(z_n)$. 
Note firstly that being 
non-trivial analytic sets, each of the $F_n$'s is
non-compact and $\rho^{-1}(K\setminus \{z_0\}) = \cup_{n=1}^{\infty} F_n$. 
We then just use the fact from general topology 
that: being second countable, locally compact and Hausdorff,
$X$ admits an exhaustion 
by compact subsets $K_n$. In particular, since $X$ is not compact,
the complements $K_n^c = X \setminus K_n$ is non-empty for 
each $n \in \mathbb{N}$,
ensuring that we can choose a point $p_n \in F_n \cap K_n^c$.
To prove that $p_n$'s form a discrete subset of $X$ as 
needed by our strategy, assume to get a contradiction
 that there exists a 
subsequence $(p_{n_k})$ converging to $p$. This implies that there exists $N \in \mathbb{N}$ 
such that the $N$-th compact captures our sequence eventually: $p_{n_k} \in K_N$ for all 
$n_k$; but that contradicts the choice of $p_n$, finishing the proof
that $p_n$ has no 
limit point in $X$. Thus $C := \{p_n: n \in \mathbb{N}\}$
is a closed subset of $X$ whose image $\rho(C) = \{z_n: n\in \mathbb{N}\}$ fails to be closed
as its limit point $z_0$ lies outside it, establishing that $\rho$ is not a closed map.
\end{proof}

\noindent We turn to questions which arise when we have more than one retraction map
and some initial remarks about the basic question about how do we generate new
retraction maps out of old. Algebraic or other operations would require special 
assumptions about the ambient space; but whatever be the ambient space $X$,
viewing retraction maps as self-maps of $X$, we may always compose any two
such maps and ask when does it result in a retraction?
As the first answer, we have: if two retraction maps $\rho_1,\rho_2$ on 
$X$ commute, then $\rho:=\rho_1 \circ \rho_2 = \rho_2 \circ \rho_1$ 
is also a retraction map on $X$ whose image is the retract given by the intersection of the retracts $\rho_1(X)$ 
and $\rho_2(X)$. If we are allowed to change the space, then there are even 
more trivial ways such as the following. Given any retraction with 
notations as above, we may pick an arbitrary subset $A$ of the retract $Z$
and consider the restriction of 
$\rho$ to $X \setminus \rho^{-1}(A)$; this renders $Z \setminus A$ as a retract
of $X \setminus \rho^{-1}(A)$. In other words, we may get retracts of $X$ as
intersections of retracts of a space containing $X$ with $X$.
The question of whether all retracts of $X \setminus A$
can be realized as being obtained this way is 
interesting only when $X$ and $A$ satisfy certain conditions
i.e., the converse of the last statement indeed fails
and there are nice counterexamples towards this. To 
put down the simplest of
examples in the complex analytic category, consider $X=\mathbb{C}^2\setminus \{\rm w-axis\}$
where the standard coordinates are herein and henceforth in this article, denoted by 
$(z,w)$. We then observe that the retract $Z$ described by the equation 
$zw=1$ which is the image of the retraction $(z,w) \mapsto (z,1/z)$, cannot be 
realized as mentioned just earlier -- the proof of this 
is recorded in remark \ref{Analytic-compl-Unbdd dom}
following a lemma to answer the aforementioned question in the 
positive direction.\\    

\noindent In passing, we put down a rudimentary fact to keep in the background about comparing 
retracts of a pair of related spaces;
we include its proof alongwith statement despite its obviousness (as with certain other facts), owing to its use for the rest of this article.

\begin{prop} \label{Bih}
Suppose $X_1$ and $X_2$ are biholomorphically equivalent complex manifolds, with 
$F: X_1 \to X_2$ being a biholomorphism. Then 
there is a one-to-one correspondence between retracts of $X_1$ and $X_2$, rendered via conjugation by $F$.
In-particular, the image of a retract of a complex manifold $X$ under any member of 
${\rm Aut}(X)$ is also a retract of $X$.
\end{prop}

\begin{proof}
Let $R_1,~ R_2$ be retracts of $X_1,~ X_2$ respectively; so, there exist holomorphic 
retraction maps
$\rho_1: X_1 \to X_1$ and $\rho_2 : X_2 \to X_2$ with $\rho_1(X_1) = R_1$ and $\rho_2(X_2) = R_2$.
Let $Z_1 := F(R_1),~ Z_2 := F^{-1}(R_2)$.
Consider the map $G : X_2 \to X_2$ defined by $G(z) = F \circ \rho_1 \circ F^{-1}(z)$.
The idempotence of $\rho_1$ implies the same about $G$; so, $G$ is a holomorphic 
retraction on $X_2$ and ${\rm image}(G) = Z_1$. Thereby, $Z_1$ is a retract of $X_1$.\\
In other words, if $R_1$ is a retract of $X_1$ then $F(R_1)$ is a retract of $X_2$.
Similar arguments show that if $R_2$ is a retract of $X_2$ then $F^{-1}(R_2)$ is a retract of $X_1$.
\end{proof}

\noindent As mentioned in the introduction, one may deduce out of the above elementary fact, the well-known
characterization of retracts of $\Delta^N$, yet another way of 
demonstrating the  dramatic failure of the Riemann mapping theorem in
multidimensional complex analysis, as soon as the dimension $N$ gets any bigger than one.
\begin{cor} \label{Bih_1}
Let $D$ be any bounded balanced convex domain with $\mathbb{C}$-extremal boundary then $\Delta^N$ and 
$D$ are biholomophically inequivalent.
\end{cor}
     
\begin{proof}
If possible assume that $\Delta^N$ and $D$ are biholomorphically equivalent. Since 
automorphism group of $\Delta^N$ acts transitively, there exist a biholomorphism $F$ 
from $\Delta^N$ to $D$  with $F(0) = 0$. By Cartan's 
theorem (theorem 2.1.3 in \cite{Rud_book_Fn_theory_unitball}), $F$ is a 
linear transformation. By theorem \ref{Vesntn}, every retract of $D$  
passing through origin is a linear subspace of $D$. If $D_L := D \cap L$ is 
a linear subspace of $D$, then $F^{-1}(D_L)$ is also linear. By proposition
\ref{Bih}, this leads to the
conclusion that {\it all} retracts of $\Delta^N$ passing through origin are linear subspaces,
which contradicts the fact that we have highlighted 
several times namely, the existence in plenty, of non-linear retracts passing through the origin 
for the polydisc. 
\end{proof}

\noindent In passing, we may ofcourse note that
by using the well-known theorem of Kaup--Upmeier \cite{Kaup-Upm}, we may replace the polydisc in the 
above by various other product domains which unlike polydiscs are neither convex nor homogeneous. 
Specifically, we may take a product $B_1\times \ldots \times B_k$ (with $k \geq 2$) of bounded balanced pseudoconvex
domains with the property that one of the factors admits a non-linear holomorphic map into another, to
ensure the existence of non-linear retracts, thereby the proof of the above corollary to go through. 
There are of-course better proofs and our reason for discussing this proof is to
simultaneously make a few general remarks about {\it non-linear} retracts and about 
some simple ways of generating ample examples of retracts.
We note that there are plenty of product domains of the kind just mentioned (satisfying the property 
about the existence of non-linear holomorphic maps, as mentioned above). To indicate
a way of generating examples, let 
$E_{2m} := \{ (z,w) \; : \; \vert z \vert^{2m} + \vert w \vert^2  <1 \}$. For an
arbitrary point $(z,w) \in E_{2m}$ note that $\vert z \vert <1$, thereby that,
$\vert z^{n} \vert^{2l} +  \vert w \vert^2 < \vert z \vert^{2m} + \vert w \vert^2$ for
every $l \geq m$. Consequently,
$(z^{n},w)$ lies in $E_{2l}$ for every $l \geq m$, implying that
$(z,w) \mapsto (z^{n},w)$, gives an example of a non-linear holomorphic mapping as soon as $n \geq 2$ as desired and
more importantly thereby, a non-linear retract passing through the origin in the
polyball $E_{2m} \times E_{2l}$ for every $l \geq m$. Needless to say, similar
constructions can be done with $l_p$-balls and more; we shall not 
digress further along this line.\\

Since for a good part of the article, we deal with circular domains, we state one of the most famous results about mappings between such domains which will be used in numerous occasions in this article; we shall simply refer to this theorem due to H.Cartan as Cartan's theorem.
\begin{thm}(H. Cartan)\label{Cartan}
Suppose $\Omega_1,\Omega_2$ are circular domains in $\mathbb{C}^N$ with $0 \in \Omega_1 \cap \Omega_2$ and $\Omega_1$ is bounded. If $F$ is a biholomorphic map of $\Omega_1$ onto $\Omega_2$ with $F(0) = 0$, then $F$ is a linear transformation.
\end{thm}
\noindent There shall also be occasions where another result of H.Cartan will be used, as stated next. We shall refer to it as the classical uniqueness theorem of Cartan.
\begin{thm}
Suppose $\Omega$ is a bounded domain in $\mathbb{C}^N$ and $F: \Omega \to \Omega$ is a holomorphic map. If for some $p \in \Omega, ~ F(p) = p$ and $F'(p) = I$, then $F(z) = z$ for all $z \in \Omega$.
\end{thm}
\subsection{Balanced Domains in $\mathbb{C}^N$}
In this subsection, we recall the basic properties of balanced domain and its associated Minkowski function, following \cite{Jrncki_invrnt_dst}.
\begin{defn}
Let $D \subset \mathbb{C}^N$ be a balanced domain. Define its Minkowski function $h_D: \mathbb{C}^N \longrightarrow \mathbb{R}^{+}$,
$$
h_D(z):=\inf \{t>0: z / t \in D\}, \quad z \in \mathbb{C}^N .
$$
\end{defn}
\begin{prop}\label{Prop_Minkowski}
We have with the notations as in the above definition, the following properties of the Minkowski functional.
\begin{enumerate}\rm
\item [(a)] $h_D(\lambda z)=|\lambda| h_D(z)$, for all $z \in \mathbb{C}^N, \lambda \in \mathbb{C}$.
\item [(b)] $D=\left\{z \in \mathbb{C}^N: h_D(z)<1\right\}$.
\item [(c)] If $h: \mathbb{C}^N \longrightarrow \mathbb{R}^{+}$is an upper semicontinuous function such that $h(\lambda z)= |\lambda| h(z), z \in \mathbb{C}^N, \lambda \in \mathbb{C}$, then the set $D:=\left\{z \in \mathbb{C}^N: h(z)<1\right\}$ is a balanced domain and $h \equiv h_D$.
\item [(d)] If $h_D$ is continuous, then $\partial D=\left\{z \in \mathbb{C}^N: h_D(z)=1\right\}$.
\item [(e)] $D$ is convex if and only if $h_D$ is a seminorm, i.e., $h_D(z+w) \leq h_D(z)+h_D(w), z, w \in \mathbb{C}^N$.
\item [(f)] If $D$ is bounded, then $h_D^{-1}(0)=\{0\}$.
\end{enumerate}
\end{prop}

\noindent An important little observation that will be particularly useful to us is the following. Fix any norm $\|\cdot\|$ on $\mathbb{C}^N$. Then, for any $z \in \mathbb{C}^N$, the Minkowski functional of $D$ can be expressed as: 
\[
h_D(z) \;= \; \frac{\|z\|}{\|\pi(z)\|},
\] 
where 
$\pi(z) = \alpha_0 z \in \partial D$ with $\alpha_0 := \sup \{\alpha > 0: \alpha z \in D\}$.

\subsection{Linear Retracts}
\begin{defn} \label{Linear-retract-defn}
Let $D$ be any bounded balanced domain in $\mathbb{C}^N$ and $L$ a 
linear subspace of $\mathbb{C}^N$. We say that $D_L:=D \cap L$ is 
a {\it linear retract} of $D$, provided there exists an idempotent 
linear operator on $\mathbb{C}^N$ which maps $D$ onto $D_L$ (and 
thereby $\mathbb{C}^N$ onto $L$).
\end{defn}

\begin{rem} \label{Linear-remark-1}
A natural question arising for this definition, answering which will also 
settle some ambiguity that may otherwise arise in the sequel is the 
following. If it 
so happens that for a linear subspace $L$ of $\mathbb{C}^N$, $D_L:=D \cap L$  turns 
out to be a holomorphic retract i.e., the image of a possibly non-linear 
retraction $\rho$ of $D$, could it be said that $D_L$ is a linear retract 
in the sense of the above definition? We shall first answer this question under 
the assumption of convexity on $D$ as our domain of interest $D$ is convex; indeed 
then, the answer is in the affirmative and the reason is as follows. The image 
of $\rho$ being a linear subspace of $D$, namely $D_L$, it is in particular 
itself a bounded balanced convex domain in the Banach space $(L, \mu)$ where $\mu$ 
is the Minkowski functional of $D$ which is a norm on $\mathbb{C}^N$ by virtue of 
the convexity of $D$. This enables us to apply proposition C of Simha's 
article \cite{Simh} (exposited in theorem 8.1.2 of \cite{Rud_book_Fn_theory_unitball} as well) which 
when applied to the pair of Banach spaces $(\mathbb{C}^N, \mu)$, $(L, \mu)$ and 
the map $\rho$ between their unit balls,  shows that the norm of the linear operator 
$D\rho(0)$ is atmost one. As there are points of norm one (with respect to the 
norm $\mu$) in $\overline{D}$ which are mapped to points of norm one in $L$ 
by $D\rho(0)$ namely, points of $\partial D_L$ itself, we immediately see that 
the norm of $D\rho(0)$ is exactly one. This leads to the desired conclusion 
that $D \rho(0)$ is a linear retraction of $D$ retracting $D$ onto $D_L$, as 
is brought out by the next general remark.
\end{rem}

\begin{rem}\label{Linear-remark-2}
Let $D$ be any bounded balanced convex domain in $\mathbb{C}^N$ and $\mu$ its 
Minkowski functional; as $D$ is convex, $\mu$ is a norm and $D$ is the unit ball 
with respect to $\mu$. Let $L$ be any complex linear subspace of $\mathbb{C}^N$ 
considered now as a Banach space with the norm $\mu$. Suppose there exists a 
linear projection 
$\pi_L : \mathbb{C}^{N} \to L$ of norm one, where
we equip $\mathbb{C}^{N}$ and $L$ both with the norm $\mu$. This means that $\pi_L$ 
maps $D$ onto $D_L:= L \cap D$, the unit ball in $L$ with respect to the same
norm $\mu$ restricted to $L$, which is fixed pointwise by $\pi_L$ (after 
all, $\pi_L$ itself is a retraction of $\mathbb{C}^N$ onto $L$); thus, $\pi_L$ is 
a linear retraction of $D$ onto $D_L$. Thus, a linear subspace of $D$ which must 
be of the form $D_L= D \cap L$ for some linear subspace $L$ of $\mathbb{C}^N$, 
is a linear retract of $D$ if and only if there exists a linear projection of 
the Banach space $(\mathbb{C}^N, \mu)$ of norm one, mapping $D$ onto $D_L$. 
Finally, let us just remark here once for clarity that while this 
projection $\pi_L$ is linear, it may well not be (and most often not) an orthogonal projection. 
\end{rem}

\noindent Let us now address the basic question of generating (examples of) new retracts out of old. 
To begin with a first rudimentary step towards this, we put down the following transitivity result.
\begin{lem}\label{retr-transit}
Let $X$ be any complex manifold. If $Z$ is a retract of $X$ and $W$ is a retract of $Z$ then $W$ is a 
retract of $X$.
\end{lem}
\begin{proof}
Since $Z$ is a retract of $X$ and $W$ is a retract of $Z$, there exist a 
pair of retraction maps $\rho_1: X \to Z$ and $\rho_2: Z \to W$. Note that 
if $x \in Z$ then $\rho_1 \circ \rho_2(x) = \rho_2(x)$. Note that 
$\rho_2 \circ \rho_1$ is a map from $X$ to $W$, and also that
\[
\rho_2\circ \rho_1 \circ \rho_2 \circ \rho_1(x) = \rho_2 \circ \left(\rho_1 
\circ \rho_2(\rho_1(x))\right) = \rho_2 \circ \rho_2 (\rho_1(x)) = \rho_2 \circ \rho_1(x).
\]
So, $\rho_2 \circ \rho_1$ is a retraction map on $X$ verifying that $W$ is a retract of $X$.
\end{proof}
\noindent While the above elementary lemma provides a manner of obtaining retracts of lower dimension, we may ask 
for a result with the converse effect. This will be addressed in the forthcoming sections at 
its appropriate place. Let us only mention here the answer to a special case of 
this question in the realm of 
balanced domains and linear retracts, namely the following. 
Suppose $D_{L_1}$ and $D_{L_2}$ are linear retracts of a bounded balanced 
convex domain $D$; then, is it true that $D_{L_1 \oplus L_2}$ a linear retract of $D$? The following
proposition shows that except essentially for the Euclidean ball,
it can never happen that for {\it every} pair of linear retracts such a sum of (linear) retracts is 
again a retract.
\begin{prop}
Let $D$ be a bounded balanced convex domain in $\mathbb{C}^N$. If $D$ is not linearly isometric 
to Euclidean ball $\mathbb{B}$, then there exist one-dimensional linear subspaces $L_1, L_2$ 
such that $D_{L_1 \oplus L_2}$ 
is not a linear retract of $D$.
\end{prop}
\begin{proof}
It suffices to assume that $D_{L_1 \oplus L_2}$ is a 
linear retract of $D$ for every pair of (complex) lines $L_1, L_2$ through the origin $\mathbb{C}^N$
to derive a contradiction.
Indeed with such an assumption, we would have that
every $2$ 
dimensional subspace of $D$ is a linear retract of $D$. But then, this would imply by theorem-A in \cite{Bhnblst} that
$(\mathbb{C}^N, \|\cdot\|)$ is a Hilbert space, where $\|\cdot\|$ denotes the Minkowski functional 
on $D$ which in turn implies the existence of a unitary map $T$ from 
$(\mathbb{C}^N, \|\cdot\|)$ to $(\mathbb{C}^N, \|\cdot\|_{\ell^2})$. 
Therefore $T(D)=\mathbb{B}$, contradicting the hypothesis that $D$ is not equivalent to $\mathbb{B}$.
\end{proof}

\noindent For $D$ as in the above theorem, we have in background the observation that the convexity of $D$, implies by 
the Hahn-Banach theorem that every one dimensional 
subspace of $D$ is a retract of $D$. One may want to ascertain the converse, which is addressed in the following.

\begin{prop}\label{cnvx_1dim_retrct}
Suppose every one-dimensional linear subspace of a bounded balanced domain $D$  is indeed a retract, then 
domain $D$ must be convex.
\end{prop}

\begin{proof}
Assume to get a contradiction that there exists a pair of distinct points $p,q \in \partial D$, with 
the property that their mid-point, fails to lie within the closure of the domain i.e., $m:= (p+q)/2 \; \not\in \overline{D}$.
Let $T$ a linear retraction map (i.e., linear projection on $\mathbb{C}^N$) which maps $D$ onto $L_m$, the one-dimensional 
linear subspace spanned by this (mid-)point $m$ -- the existence of such a map is guaranteed by hypothesis.
Then, note first of all that: $T$ fixes $L_m$ pointwise; so $T(m)=m$, thereby $T(m)$ lies outside $\overline{D}$. On the other hand, 
$T(p)$ and $T(q)$ are a pair of points 
on $L_m \cap \overline{D}$, as $T$ maps $\overline{D}$ into $\overline{D} \cap L$.
And, the mid-point of these image-points must lie in  
$L_m \cap \overline{D}$: by balancedness of $\overline{D}$, the line-segment joining $T(p)$ and $T(q)$
has got to be contained within $\overline{D}$ -- note that this does not require convexity of $\overline{D}$! But then we
have
\[
\frac{T(p) + T(q)}{2} \; = \; T\left( \frac{p+q}{2} \right) \;=\; T(m) 
\]
It follows that $T(m)$ lies within $\overline{D}$. However, by earlier observations leading to
$T(m)=m$, we have that the same point $T(m)$ lies outside $\overline{D}$! This absurdity only means that 
the assumption that $D$ is convex cannot hold, finishing the proof.
\end{proof}

\subsection{Notions of extremality and strict convexity}\label{extremality and strict convexity}
\begin{defn}
Let $K$ be the field $\mathbb{R}$ or $\mathbb{C}$. Let $A \subset \mathbb{C}^N$ be an arbitrary set and
$a \in \partial{A}$. We introduce a notion that captures the set of all directions along which 
one can move from $a$, at-least a little bit while staying within the boundary $\partial A$.
Towards this, define the face of $\partial A$ at $a$ by
\begin{equation*}
F_{K}(a,\partial{A}) = \{u \in \mathbb{C}^N: \exists \epsilon > 0, 
a +tu \in \partial{A}, \text{ for all } t \in K ~with ~ |t| < \epsilon\}
\end{equation*}
reduces to $\{0\}$. It is immediate that $F_K$ is closed under scalar multiplication. But a stronger 
claim that $F_K$ is a linear subspace of $\mathbb{C}^N$ requires more assumptions on $A$. Indeed, 
if $A$ is convex set, then $F_{K}(a,\partial{A})$ is $K$-linear subspace.
We shall mostly deal with the case when $A = \Omega$ is a domain.
In such a case, by abuse of notation we shall write $F(a,\partial \Omega)$ or $F(a,\overline{\Omega})$
in place of $F(a,\Omega)$. 
A boundary point $a \in \partial{\Omega}$ is said to be a \textit {weakly~ $K$-extremal} boundary point
for $\Omega$, if $F_{K}(a,\partial{\Omega})=0$;
if the face reduces to $\{0\}$ for every $a \in \partial{\Omega}$, then we say that the domain 
$\Omega$ is \textit{weakly} $K$-extremal. \\
A boundary point $a \in \partial{\Omega}$ is said to be a \textit{strongly~ $K$-extremal} boundary point
for $\Omega$, if the face
\[ F_{K}(a,\overline{\Omega}) = \{u \in \mathbb{C}^N: \exists \epsilon > 0, 
a +tu \in \overline{\Omega}, \text{ for all } t \in K ~with ~ |t| < \epsilon\}
\]
reduces to $\{0\}$. As is usual, if all boundary points are strongly $K$-extremal, then 
we say that the domain itself is strongly $K$-extremal. In general, if $\Omega$ is strongly $K$-extremal, then we shall drop the adjective `strongly' 
and call $\Omega$ simply as a $K$-extremal domain.
\end{defn}

\noindent We now prove that in the case of bounded balanced pseudoconvex
domains, strongly $\mathbb{C}$-extremal and weakly $\mathbb{C}$-extremal are equivalent, thereby we
can forget about the qualifications `weak' and `strong' essentially for the
rest of this article. We leave the easier implication about strong extremality 
implying weak extremality aside and detail only the relatively non-trivial reverse implication.

\begin{prop}\label{strong-weak-extrm}
Suppose $D$ is a bounded balanced pseudoconvex domain and $p \in \partial D$ is a weakly 
$\mathbb{C}$-extremal boundary point. Then it is also a strongly $\mathbb{C}$-extremal boundary point.
\end{prop}

\begin{proof}
Assume that $p$ is not strongly $\mathbb{C}$-extremal. Then there exists  $\epsilon_0 > 0$ 
and non-zero $u \in \mathbb{C}^N$ such that 
$p+tu \in \overline{D}$ for all $t \in \mathbb{C}$ with $|t| < \epsilon_0$.
After a rescaling, we may assume that $\epsilon_0 = 1$. Then the affine (holomorphic) map $f$ 
defined by $f(t) = p+tu$ maps $\Delta$ into $\overline{D}$; note that $f(0) = p$. 
Let $h$ denote the Minkowski functional of $D$. 
Consider the map $\varphi: \Delta \to \mathbb{R}$ defined by $\varphi(z) = h \circ f(z)$.
Note that $\varphi$ is a plurisubharmonic function whose modulus attains maximum at $z = 0$.  
By maximum principle of plurisubharmonic functions, $\varphi$ is identically constant, 
which is equal to $1$. This means that $h \circ f \equiv 1$. So ${\rm image}(f) \subset \partial D$. 
Hence $p$ is not weakly $\mathbb{C}$-extreme boundary point. This finishes the proof
by contradiction. 
\end{proof}

\begin{rem}
If $K = \mathbb{R}$, then in the case of bounded balanced convex domains, 
the notion of weak $\mathbb{R}$-extremality and strong $\mathbb{R}$-extremality are equivalent. 
\end{rem}

\noindent We would now like to make a couple of 
clarifications about the case $K = \mathbb{C}$ versus the case $K=\mathbb{R}$. To begin 
with, note that we obviously have 
\[
    \vec{F}_{\mathbb{C}}(a,\overline{\Omega}) = \vec{F}_{\mathbb{R}}(a,\overline{\Omega}) \cap i\vec{F}_{\mathbb{R}}(a,\overline{\Omega}).
\]
It follows that if $\Omega$ is weakly (resp. strongly) $\mathbb{R}$-extremal, then $\Omega$ is weakly (resp. strongly) $\mathbb{C}$-extremal. 
But the converse is not true. To see this, consider 
$B_{\ell^1}: = \{(z,w)\in \mathbb{C}^2: |z| + |w| < 1\}$;
we shall presently proceed to showing that this is $\mathbb{C}$-extremal 
but not $\mathbb{R}$-extremal. Perhaps even more basically, we must remark 
that (vacuously) every domain in $\mathbb{C}$ is $\mathbb{C}$-extremal
whereas only the strictly convex ones are $\mathbb{R}$-extremal. To discuss a `non-vacuous' example,
we discuss the well-known $\ell^1$-ball as mentioned above whose proof is included here owing to its
basic importance for us.

\begin{lem}
Let $\rho : \Delta \to \overline{B_{\ell^1}}$  be a holomorphic map. If image of $\rho$ 
is contained in $\partial B_{\ell^1}$ then $\rho$ is a constant map. 
\end{lem}

\begin{proof}
Write $\rho(z) = (\rho_1(z), \rho_2(z))$. Note that $\rho_1, \rho_2$ 
are holomorphic maps from  $\Delta$ to $\overline{\Delta}$. Write in polar-coordinates:
$\rho_1(0) = r_1 e^{i\theta_1}$ and $\rho_2(0) = r_2e^{i\theta_2}$. Define $g: \Delta \to \overline{\Delta}$ by 
\[
g(z) = \rho_1(z)e^{-i\theta_1} + \rho_2(z) e^{-i\theta_2}.
\]
Note that $g$ is holomorphic and $g(0) = 1$. By maximum modulus 
principle, $g$ is constant, indeed $g \equiv 1$. This when 
written out in terms of the $\rho_j's$ reads
\[
\rho_1(z)e^{-i\theta_1} + \rho_2(z) e^{-i\theta_2} = |\rho_1(z)e^{-i\theta_1}| + |\rho_2(z)e^{-i\theta_2}|, 
\text{ both being equal to } 1.
\]
This in-particular, means that equality holds in triangle inequality. This implies 
that $\rho_1(z)e^{-i\theta_1} = c\rho_2(z) e^{-i\theta_2}$, for some $c \in \mathbb{C}$.
Hence, $|\rho_1(z)| = k|\rho_2(z)|$, where $k = |c| \geq 0$. As 
$|\rho_1(z)| + |\rho_2(z)| = 1 $, this leads to
\[
k|\rho_2(z)| + |\rho_2(z)| = 1 \]
This implies that $|\rho_2(z)| = 1/(1+k).$ So $\rho_2$ is constant by maximum modulus principle, thereby $\rho_1$ is constant as well, finishing the proof.
\end{proof}
\begin{prop}
    $B_{\ell^1}: = \{(z,w)\in \mathbb{C}^2: |z| + |w| < 1\}$ is $\mathbb{C}$-extremal but not $\mathbb{R}$-extremal.
\end{prop}

\begin{proof}
We first show that $B_{\ell^1}$ is $\mathbb{C}$-extremal. To do this, let $a \in \partial B_{\ell^1}$. Suppose there
exists $ \epsilon_0 > 0$  and  $u_0 \in \mathbb{C}^N$ such 
that $a + tu_0 \in \partial B_{\ell^1}$, for all $t \in \mathbb{C}$ with $|t| < \epsilon_0$.\\
Without loss of generality, we may assume after an appropriate scaling that $\epsilon_0 = 1$; this ensures that
the affine-linear (holomorphic) map defined by $\rho(t) = a + tu_0$ maps $\Delta$ into $\partial B_{\ell^1}$. 
By the above lemma, $\rho$ is constant, which is equal to $a$. This implies $u_0 = 0$, finishing the proof
of $\mathbb{C}$-extremality. \\
\smallskip
We next show that $B_{\ell^1}$ is {\it not} $\mathbb{R}$-extremal. To 
do this, first let $a = (1/2,1/2) \in \partial B_{\ell^1}$
and consider the face
\[
F_{\mathbb{R}}(a,\partial B_{\ell^1}) = 
\{u \in \mathbb{C}^2: \exists \epsilon > 0,~ a + tu \in \partial B_{\ell^1}~, 
\forall t \in \mathbb{R} ~\text{with} ~|t| < \epsilon \}.
\]
Note then that $(1,-1) \in F_{\mathbb{R}}(a,\partial B_{\ell^1})$. Hence $B_{\ell^1}$ is not $\mathbb{R}$-extremal.
\end{proof}
We shall now discuss the relationship between the above extremality notions and that of strict convexity.
\begin{defn}
Let $\Omega$ be a  balanced bounded convex domain 
in $\mathbb{C}^N$; so, $\mathbb{C}^N$ may be regarded as a Banach space with 
respect to the norm $\|\cdot\|$ whose unit ball is $\Omega$.  We say that $\Omega$ is \textit{strictly $K$-convex} if for every  $K$-linear functional $T$ with $\|T\| = 1$, corresponds just one $z \in \overline{\Omega}$ such that $Tz = 1$.
\end{defn}
In order to establish equivalence of $\mathbb{R}$-extremality with strict convexity, 
it is convenient to do this, the intermediate concept namely of rotundity, which is defined as follows:
\begin{defn}
Let $\Omega$ be a  balanced bounded convex domain 
in $\mathbb{C}^N$; so, $\mathbb{C}^N$ may be regarded as a Banach space with 
respect to the norm $\|\cdot\|$ whose unit ball is $\Omega$. 
We say that $\Omega$ is \textit{rotund} if $tx_1 + (1-t)x_2 \in \Omega$ whenever $x_1$ and $x_2$ are different points of $\partial \Omega$ and $0 < t < 1$.
\end{defn}

\noindent First we prove that the notion of rotundity and strict convexity are equivalent in 
the realm of bounded balanced convex domains.

\begin{thm}\label{strct_cnvx}
Let $\Omega$ be a bounded balanced convex domain in $\mathbb{C}^N$. Then $\Omega$ is rotund if and only if $\Omega$ is $\mathbb{R}$-extremal.
\end{thm}    
 
\begin{proof}
Assume $\Omega$ is rotund. Let $a \in \partial \Omega$ and 
assume that there exists $u \in \mathbb{C}^N$ such that $a+tu \in \partial \Omega$ for all $t \in \mathbb{R}$ with $|t| < \epsilon$. 
Note that $\frac{1}{2}(a + (a+tu)) = a + \frac{t}{2} u \in \partial\Omega$ for 
all $t \in \mathbb{R}$ with $|t| < \epsilon$. By proposition 5.1.2 of \cite{Meggn}, 
$a = a+tu$ for all $t \in \mathbb{R}$ with $|t| < \epsilon$ and this implies 
that $u = 0$. Therefore, $\Omega$ is $\mathbb{R}$-extremal.\\

\noindent Conversely assume that $\Omega$ is $\mathbb{R}$-extremal. Let $z_1,z_2 \in \partial\Omega$.
\[ 
(1-t)z_1 + tz_2  = z_1 + t(z_2 - z_1)\in \partial\Omega \quad \text{for ~all}\ t \in [0,1].
\]
It follows that when $\vert t \vert \leq 1/2$, we have
\[ 
\frac{z_1+z_2}{2} + t(z_2 - z_1) = (1/2 -t)z_1 + (1/2 + t)z_2 \in \partial{\Omega}.
\]
Therefore $z_2-z_1 \in F_{\mathbb{R}}(\frac{z_1 + z_2}{2}, \partial{\Omega})$. Since $\Omega$ is 
$\mathbb{R}$- extremal, $z_2 = z_1$ and therefore $\Omega$ is rotund.
\end{proof}

\noindent As the relation between the pair of notions gotten by taking $K=\mathbb{C}$ versus $K=\mathbb{R}$ will
be basic importance in the sequel, we discuss this below. But one would perhaps like to 
ascertain first the 
relation between strict convexity and extremality introduced above. This follows from known facts 
(see for instance \cite{Meggn}), but owing to its central importance for us, we include the 
details.  
\begin{thm}    
Let $\Omega$ be a bounded balanced convex domain in $\mathbb{C}^N$. Then $\Omega$ is 
strictly $\mathbb{R}$-convex if and only if $\Omega$ is $\mathbb{R}$-extremal.
\end{thm}     
	
\begin{proof}
Suppose that $\Omega$ is strictly $\mathbb{R}$-convex. If possible,
assume that $\Omega$ is not $\mathbb{R}$- extremal. Then by theorem \ref{strct_cnvx}, $\Omega$ is not rotund. 
By corollary 5.1.16 of \cite{Meggn}, there exists $\mathbb{R}$-linear functional $f$ supports more than one
point. This implies that there exists $z_0,w_0 \in \overline{\Omega}$ such that $f(z_0) = f(w_0) = \|f\|$. Take $\alpha := \|f\|$. Consider the $\mathbb{R}$-linear functional 
$g$ defined by $g(z) = \frac{1}{\alpha}f(z)$. Note that
$\|g|\ = 1$ and $g(z_0) = g(w_0) = 1$. This contradicts 
the strict $\mathbb{R}$-convexity of $\Omega$. Therefore, $\Omega$ is $\mathbb{R}$- extremal.\\

\noindent For the converse, assume that $\Omega$ is $\mathbb{R}$-extremal. Let $T$ be 
a $\mathbb{R}-$ linear map with $\|T\| = 1$. Suppose that there exist
$z,w \in \overline{\Omega}$ such that $Tz = 1 = Tw$. Then 
$z,w \in \partial{\Omega}$, otherwise $T\left( z/\|z\|\right) = 1/\|z\| > 1 $.
Note that $T((1-t)z + t w) = 1$.
Then $(1-t)z + tw \in \partial{\Omega}$ for all $t \in [0,1]$. 
In particular, if $t = 1/2$, we get $\frac{1}{2}(z +  w) \in \partial{\Omega}$. Note then that
\[ 
\frac{z+w}{2} + t(w - z) = (1/2 -t)z + (1/2 + t)w 
\in \partial{\Omega}~ \ \text{for all t} \in \mathbb{R}  \ \text{with}\ |t| < 1/2.
\]
Then $w-z \in F_{\mathbb{R}}(\frac{z + w}{2}, \partial{\Omega})$.
Since $\Omega$ is $\mathbb{R}$-extremal, it follows that $w = z$. 
Therefore $\Omega$ is strictly $\mathbb{R}$-convex.
\end{proof}

\begin{thm}
Let $\Omega$ be a bounded balanced convex domain in $\mathbb{C}^N$. Then $\Omega$ 
is strictly $\mathbb{R}$-convex if and only if $\Omega$ is strictly $\mathbb{C}$-convex.
\end{thm}
\begin{proof}
Let $\Omega$ be a bounded balanced convex domain 
in $\mathbb{C}^N$. Note that $\Omega$ is the unit ball 
with respect to some norm $\|\cdot\|$ in $\mathbb{C}^N$. Assume $\Omega$ is 
strictly $\mathbb{R}$-convex. We want to prove that $\Omega$ is 
strictly $\mathbb{C}$-convex. Let $T$ be $\mathbb{C}$ -linear functional with $\|T\| = 1$ such 
that $T(z) = T(w) = 1$ for some $z,w \in \overline{\Omega}$. Every $\mathbb{C}$ -linear 
functional must be an $\mathbb{R}$-linear functional. Then $T$ is 
a $\mathbb{R}$-linear functional with $\|T\| = 1$ and $T(z) = T(w) = 1$. 
Since $\Omega$ is strictly $\mathbb{R}$-convex, we get $z = w$.\\
Conversely, assume that $\Omega$ is strictly $\mathbb{C}$-convex. We have to prove 
that $\Omega$ is strictly $\mathbb{R}$-convex. By above theorem, it suffices 
to show that $\Omega$ is $\mathbb{R}$-extremal. If possible assume that there
exists $u \in \mathbb{C}^N,~ a \in \partial\Omega$ and an $\epsilon > 0$ 
such that  $a + tu \in \partial\Omega$ for all $t \in \mathbb{R}$ with $|t| < \epsilon$. 
Let $M = span_{\mathbb{C}}\{u\}$. We split the proof now into two steps, formulated as `claims'.

\underline{Claim-1:} $a \notin M$\\
\textit{Proof of Claim-1:} If possible assume that $a \in M$. Since $a \in \partial \Omega$,
there exists non-zero $\lambda \in \mathbb{C}$ such that $a = \lambda u$. 
Choose $t_0$ be a non-zero real number such that $|t_0| < \epsilon$ and $|\lambda +t_0| \neq |\lambda|$.
Taking norm on both sides of the equation $a = \lambda u$, we get $|\lambda| = 1/\|u\|$.
From the assumption $a+tu \in \partial \Omega$, we obtain
\[
    1 = \|a+t_0u\| = \|\lambda u +t_0 u\| = \|u\|\cdot|\lambda+t_0| = \frac{|\lambda +t_0|}{|\lambda|}.
\]
The above equations tells that $|\lambda +t_0| = |\lambda|$. This 
contradicts our choice of $t_0$ and  
finishes the proof of the claim.\\

\underline{Claim-2:} ${\rm dist}(a,M) = 1$, where ${\rm dist}$ denotes the 
distance induced by the norm $\|\cdot\|$ with respect to which $\Omega$ is the 
unit ball.\\
\textit{Proof of Claim-2:} Let us begin by recalling that 
${\rm dist}(a,M) = \inf\{\|a - su\|: s \in \mathbb{C}\}$ and
note firstly (by considering $s=0$) that ${\rm dist}(a,M) \leq 1$.
To prove this claim then, it suffices to show that $\rm dist(a,M) \geq 1$. To this 
end, it is convenient to rewrite the definition of the aforementioned distance as
${\rm dist}(a,M) = \inf\{\|a + su\|: s \in \mathbb{C}\}$ -- this is not a mere sign change,
it is a changing over to the `dual' viewpoint that the aforementioned distance is also the 
infimum of distances of the origin to points on the affine line through $a$ in 
the direction of $u$. 
Since $D$ is convex, there exists a supporting hyperplane for every $p \in \partial D$. 
In particular, there exists a supporting hyperplane 
$\widetilde{H} = \{z \in \mathbb{C}^N: g(z) = g(a)\}$ (for some $\mathbb{R}$-linear 
functional $g$) at the point $a$. Let $l := \{a+tu: t \in \mathbb{R}\}$. Let us 
first verify that 
$l \subset \tilde{H}$. 
To prove this by contradiction, assume the contrary that $l \not\subset \tilde{H}$. 
Consider the linear subspace $H := \{z \in \mathbb{C}^N \; : \; g(z) = 0\}$. 
By the definition of supporting hyperplane, $a \notin H$. 
So, $span \{a\} \cap H = \{0\}$.

\medskip

\noindent We now show that $span \{u\} \cap H = \{0\}$ as well.
Indeed, if $\lambda u \in H$ for some $\lambda \in \mathbb{C}$, then $g(\lambda u) = \lambda g(u) = 0$.
Since $l \not\subset \widetilde{H}$, there exists a non-zero real number $t_0$ such 
that $a+t_0u \notin \widetilde{H}$. 
This means that $u \notin \ker(g)$. From $\lambda g(u) = 0$, we obtain 
that $\lambda = 0$. Hence $span \{u\} \cap H = \{0\}$.

\medskip

\noindent Now by Claim-1,  $a$ and $u$ are linearly independent. 
Since $span\{a\} \cap H = \{0\}$ and $span\{u\} \cap H = \{0\}$, it follows that $dim(H) < N-1$.
This contradicts the fact that $H$ is a hyperplane. Hence $a+su \in \widetilde{H}$ for all $s \in \mathbb{R}$.
Therefore $\|a+su\| \geq 1$. This implies that $\rm dist(a,M) \geq 1$. This completes the proof of the claim.\\

\noindent Completing the proof of the theorem consists in now observing that: 
by the Hahn--Banach theorem, there exists a $\mathbb{C}$-linear functional 
$f$ with $\|f\| = 1$ such that $f(z) = 0$ for all $z \in M$ and $f(a) = \rm dist(a,M)$.
Note then that $f(a+tu) = f(a) = 1$. Since $\Omega$ is strictly $\mathbb{C}$-convex, this means $a = a+tu$,
which in turn implies $u = 0$. So $\Omega$ is $\mathbb{R}$-convex, finishing the proof in all.
\end{proof}
\begin{rem}
We shall henceforth not explicitly mention the field and simply refer to a strictly $K$-convex domain as strictly convex domain. 
\end{rem}
\begin{rem} 
$\mathbb{C}$-extremal domains need not be strictly $\mathbb{C}$-convex. 
To give the simplest example, we consider again the $\ell^1$ ball $B_{\ell^1}$. 
Consider $T: \mathbb{C}^2 \longrightarrow \mathbb{C}$ defined by $T(z,w) = z + w$. We note that
$T$ is linear and $\|T\|_1 = 1$, but $T(1,0) = T(0,1) = 1$.
Thus $B_{\ell^1}$ is not strictly $\mathbb{C}$-convex, but is $\mathbb{C}$-extremal.
\end{rem}

\noindent We conclude with one other notion of extremality, whose 
definition is more restrictive than that of $\mathbb{C}$-extremality. 
\begin{defn}
Let $D \subset \mathbb{C}^N$ be a bounded domain. We say that a point $a \in \partial D$ is {\it holomorphically extreme}
if there is no non-constant holomorphic mapping $\phi: \Delta \rightarrow \overline{D}$ such that $\phi(0)=a$.
\end{defn}

\noindent The purpose of recalling this notion in this article about retracts is to record the details 
for proposition \ref{JJ-improved} in the introductory section; which will be done in the next section. 

\subsection{Geodesics and totally geodesic subsets}
\noindent We recall the notion of totally $K$-geodesic subsets of a general domain 
(or a complex manifold) $D$ 
from \cite{Kosnski}. Let $p_1,p_2$ be (distinct) points in $D$; the 
pair $\delta = (p_1,p_2)$ shall
be referred to as
a `datum'. A Kobayashi extremal for $\delta$ is a
holomorphic map $k : \Delta \to D$ such that both $p_1$ and $p_2$ are 
in the range of $k$, and
the pseudo-hyperbolic distance (also known as the M\"{o}bius distance on $\Delta$) between the preimages of $p_1$ and $p_2$ is 
minimized
at $k$, over all holomorphic maps from $\Delta$ to $D$. We shall call the 
range of a Kobayashi
extremal with datum $\delta$, a $K$-geodesic through $\delta$. If a set $V \subset D$ 
has the property that
for any $p_1,p_2 \in V$, a $K$-geodesic through $(p_1,p_2)$ is contained in $V$, we shall 
say that $V$
is totally $K$-geodesic. A similar definition can be made with invariant metrics or distances, in particular the Caratheodory distance, but we do not pursue such matter here.\\

\noindent The simple fact that the Kobayashi distance $k_D(p,q)$ 
between a pair of points $p,q$ in a domain 
$D$ remains invariant (i.e., $k_D(p,q)=k_Z(p,q)$) under 
passage to a retract $Z \subset D$, was recorded
with proof explicitly as lemma $1$ in \cite{HanPeters-Zeager}. As the proof is very simple
and uses only the contractive feature of the Kobayashi distance under holomorphic maps,
we therefore only note here that the same holds much more generally for any contractive 
metric or distance on any complex manifold.

\begin{lem}\label{contr}
Let $X$ be any complex manifold and $Z\subset X$, a retract. For a 
contractive pseudo-distance $d$ and contractive (infinitesimal) pseudo-metric $\delta$, 
we have
\[
d_Z(p,q)=d_X(p,q),\;\; \text{ and }\;\; \delta_Z(p,v) = \delta_X(p,v)
\]
for every pair of points $p,q \in Z$ and $v \in T_pX$.
\end{lem}

\noindent It follows atleast for bounded convex domains $D \subset \mathbb{C}^N$ that
retracts of $D$ are totally geodesic; the more standard 
definition of `totally geodesic' shall also be made use of, while 
recalling the work \cite{Mok} of Mok. Getting back to the above lemma, it also shows that if $X$ is hyperbolic then 
every retract of $X$ is hyperbolic.

\subsection{One dimensional Retracts}\label{Car}

\noindent There are plenty of one dimensional retracts which can be obtained 
from Caratheodory complex geodesics (referred to henceforth as $C$-geodesic) whose definition 
we now recall.
\begin{defn} \label{C-geod}
We say that a holomorphic map $\varphi: \Delta \rightarrow D$ is a 
complex $C$-geodesic of $D$ if for all $\zeta_1, \zeta_2 \in \Delta$, we have 
$c_D\left(\varphi\left(\zeta_1\right), \varphi\left(\zeta_2\right)\right)=p\left(\zeta_1, \zeta_2\right)$. Occasionally, by abuse of language,
we shall also refer to the image of 
$\varphi$ as a complex $C$-geodesic.
\end{defn}

\noindent Here $c_D$ denotes the Caratheodory distance; as we shall be working with a variant of this
as well, we recall both of them and their notations:
\[
    c_D(z,w) = \sup \{p(f(z),f(w)): f \in \mathcal{O}(G,\Delta)\}
\]
\noindent The variant of $c_D$ denoted $c_D^*$ which will also be useful later for us, is defined as follows.
\[
c_D^*(z,w) = \sup \{m(f(z),f(w)): f \in \mathcal{O}(G,\Delta)\}
\]
`Dualizing' this definition i.e., by swapping the roles of $G$ and $\Delta$, 
considering $\mathcal{O}(\Delta,G)$ in place of $\mathcal{O}(G,\Delta)$, we have
what is known as the Lempert function of $D$ denoted $l_D$. For the 
precise definition of this invariant function as well as other invariant metrics and distances 
alongwith various of their fundamental properties particularly of the Kobayashi metric, 
we refer to the best known book on the subject \cite{Jrncki_invrnt_dst}; here we shall only recall to the extent 
of relevance to retracts and convenience of later use in this article. 
Thus, we only mention that we can similarly define complex $K$-geodesic (i.e., complex Kobayashi geodesic) if we replace 
$c_D$ by $K_D$ in the above definition. By Lemma 2.6.1 (ii) in \cite{Abate_Itrtn}, every complex $C$-geodesic is 
a complex $K$-geodesic. But the converse is not true because there are domains in $\mathbb{C}^N$ which are Kobayashi hyperbolic but not 
Caratheodory hyperbolic. Owing to the direct 
relationship between $C$-geodesics and retracts, 
we shall mostly work with $C$-geodesics rather than $K$-geodesics.
Vesentini \cite{Vesentini_1}, gives the following characterization of the complex geodesics of $D$.

\begin{prop}\label{Car_Inftsml}
Let $\varphi: \Delta \rightarrow D$ be a holomorphic map. 
For $\varphi$ to be a complex geodesic of $D$, it is necessary 
and sufficient that $\varphi$ satisfies one of the following conditions:
\begin{enumerate}
\item [1.] there are two distinct points $\zeta_1$ and $\zeta_2$ in $\Delta$ such that $c_D\left(\varphi\left(\zeta_1\right), \varphi\left(\zeta_2\right)\right)=c_{\Delta}\left(\zeta_1, \zeta_2\right)$.
\item [2.] there exists $\zeta_1 \in \Delta$ and a nonzero vector $v \in \mathbb{C}$ such that $\gamma_D\left(\varphi\left(\zeta_1\right), \varphi^{\prime}\left(\zeta_1\right) \cdot v\right)=\gamma_{\Delta}\left(\zeta_1, v\right)$.
\end{enumerate}
\end{prop}
\noindent The following proposition relates complex $C$-geodesics with retracts. This seems well-known; it can 
be found stated in \cite{Dineen}. However, the proof details seem not to have been recorded anywhere
which we therefore include here; particularly
because it is a basic result about one-dimensional retracts.
\begin{prop}\label{C-geodesc_Retrct}
Let $D$ be a domain in $\mathbb{C}^N$. If $\varphi: \Delta \to D$ is a complex 
$C$-geodesic, then $\varphi(\Delta)$ is a one dimensional retract of $D$. 
\end{prop}
\begin{proof}
First we want to prove that $\varphi$ has a holomorphic left inverse. For all  $\zeta_1, \zeta_2 \in \Delta$, we have
\[ 
c_D\left(\varphi\left(\zeta_1\right), \varphi\left(\zeta_2\right)\right) = p(\zeta_1, \zeta_2)
 \]
By definition of $c_D$, for all $\zeta_1,\zeta_2 \in \Delta$ we have
\[ 
\sup_{f \in \mathcal{O}(D,\Delta)}p\left((f \circ \varphi)(\zeta_1),
(f \circ \varphi)(\zeta_2)\right) = p(\zeta_1,\zeta_2).
\] 

\noindent There exists $f_n \in \mathcal{O}(D,\Delta)$ such that
\[ 
p\left((f_n \circ \varphi)(\zeta_1),(f_n \circ \varphi)(\zeta_2)\right) \to p(\zeta_1,\zeta_2) \quad \text{as}~ n \to \infty.\] 
By Montel's theorem, there exists a subsequence $f_{n_k}$ converging to $g$ uniformly on compact subsets
of $D$. So $p\left((g \circ \varphi)(\zeta_1),(g \circ \varphi)(\zeta_2)\right) = p(\zeta_1,\zeta_2)$ for all $\zeta_1,\zeta_2 \in \Delta$. By Schwarz Pick lemma, we have 
$g \circ \varphi \in {\rm Aut}(\Delta)$. This implies that $\varphi$ has a holomorphic left inverse. \\
 Let $\psi \in \mathcal{O}(D,\Delta)$ such that $\psi \circ \varphi = id_{\Delta}$. If we set $\rho = \varphi \circ \psi$, then $\rho$ is a holomorphic self map on $D$ and
\[ \rho \circ \rho = (\varphi \circ \psi) \circ (\varphi \circ \psi) = \varphi \circ (\psi \circ \varphi) \circ \psi = \varphi \circ \psi = \rho .\]
This proves that $\rho$ is a holomorphic retraction map on $D$ and $\varphi(\Delta)$ is a one-dimensional retract of $D$.
\end{proof}

\noindent It may be surprising that the hypothesis has no condition about $D$. Indeed, it may happen that $c_D \equiv 0$, there are no complex $C-$ geodesic and the result is vacuously true. The converse is true for 
the class of domains of much interest in this
article namely, bounded balanced pseudoconvex domains. More generally, we have the following.

\begin{thm}\label{Blncd_Rtrct_Cmplx_Geodsc}
If $D$ is a simply connected Kobayashi hyperbolic domain in $\mathbb{C}^N$, then the one-dimensional retracts of $D$ are precisely its complex $C$-geodesics.
\end{thm}
\begin{proof}
By proposition \ref{C-geodesc_Retrct}, we get that every complex $C$-geodesic is a retract. We have 
only to prove that 
every one-dimensional retract is a complex $C$-geodesic. Let $Z$ be a $1$-dimensional retract 
of $D$. Note that $Z$ is a non-compact simply connected Riemann surface.  By the uniformization theorem, $Z$ is 
biholomorphic to one of the Riemann surfaces: unit disk $\Delta, ~\hat{\mathbb{C}},~ \mathbb{C}$. Since $Z$ is non-compact and hyperbolic, $\hat{\mathbb{C}}$ and $\mathbb{C}$ are ruled out. Hence $Z$ is biholomorphic to $\Delta$. Let $\phi$ be a biholomorphic map from $\Delta$ to $Z$. By proposition 1.2 (the equivalence of parts a) and d) therein ) from \cite{Dineen} (proposition 11.1.7 in \cite{Jrncki_invrnt_dst}), we obtain that $\phi$ is a complex $C$-geodesic. So $Z$ is an image of a complex $C$-geodesic. 
\end{proof}
\noindent Recall that by proposition \ref{cnvx_1dim_retrct}, not all one-dimensional linear subspaces of 
bounded balanced domain $D$ are retracts of $D$, unless $D$ is convex. This
naturally leads to the question about which are those one-dimensional linear subspaces of bounded balanced non-convex domain $D$
which fail to be a retract of $D$? We next lay down a corollary that answers this; its statement requires the notion 
of convexity {\it at a boundary point}, whose definition we recall from \cite{Jrncki_invrnt_dst}.
\begin{defn}\label{C-convex}
Let $D =\{z \in \mathbb{C}^N:h(z) < 1\}$ be a balanced domain in 
$\mathbb{C}^N$ where as usual $h$ denotes the Minkowski functional of $D$. Let $a$ be any point of $D$, other than the 
origin such that $D$ is bounded in the direction 
of $a$; note that this forces $h(a)$ to be non-zero. The domain $D$ is said to be \textit{convex at} the boundary point $a/h(a)$, if there 
exists a $\mathbb{C}$-linear functional $L : \mathbb{C}^N \to \mathbb{C}$ with $|L| \leq h$ and $|L(a)| = h(a)$. Note that 
the latter (equality) condition may be written as 
$\vert L(p) \vert=1$ where $p=a/h(a) \in \partial D$.
When the aforementioned pair of 
conditions hold, we also say that $\partial D$ is \textit{convex at} the boundary point $p$ (or that the balanced domain $D$
is convex near the boundary point closest in the direction of $a$).
 The boundary $\partial D$ (or the domain $D$ itself)
 is said to be \textit{convex near} $a/h(a)$ if there exists 
 an open ball $B$ around $a/h(a)$ such that $B \cap \partial D$ 
is convex at every $q \in B \cap \partial D$.\\

\noindent One may want to relate the above
notion of convexity with the more standard geometric one
whose requirement does not involve its boundary points explicitly.
Namely, we recall that $D$ is said to be \textit{geometrically convex near} the boundary point $p$ if there exists an open set $U$ about $p$
such that $U \cap D$ is convex in the usual sense that 
for every pair of points $p_1,p_2\in U \cap D$, we 
have $(1-t)p_1 +tp_2 \in U \cap D$ for all $t \in [0,1]$.
\end{defn}
\noindent The relation between the above pair of notions of convexity will be discussed in section \ref{Vesentini}. Here
we shall clarify their connection with another standard one namely, the analogue 
of the above definition \ref{C-convex} obtained by restricting 
to real scalars instead of complex ones.
\begin{defn}\label{R-convex}
Let $D =\{z \in \mathbb{C}^N:h(z) < 1\}$ be a balanced domain in 
$\mathbb{C}^N$ where as before $h$ denotes the Minkowski functional of $D$. 
Let $a$ be any point of $D$, other than the 
origin such that $D$ is bounded in the direction of $a$
i.e., $h(a) \neq 0$. The domain $D$ is said to be \textit{convex at} the boundary point $a/h(a)$, if there 
exists a $\mathbb{R}$-linear functional $f : \mathbb{C}^N \to \mathbb{R}$ with $|f| \leq h$ and $|f(a)| = h(a)$. We note that this definition can be made for any real vector space, of-course $\mathbb{R}^N$ being the case of our interest.
\end{defn} 
\begin{prop}
Let $D$ be a balanced domain in $\mathbb{C}^N$. 
The above two definitions \ref{C-convex} and \ref{R-convex} are equivalent.
\end{prop}
\begin{proof}
Assume that $D$ is convex at $a/h(a)$ in the
sense of
definition \ref{C-convex}. Then there exists a 
 $\mathbb{C}$-linear functional $L: \mathbb{C}^N \to \mathbb{C}$ 
 such that $|L(z)| \leq h(z)$ for all $z \in \mathbb{C}^N$ and 
 $|L(a)| = h(a)$. Consider the real linear functional 
 $: \mathbb{C}^N \to \mathbb{R}$ defined by $T(z) = Re(\alpha L(z))$, 
 where $\alpha = |L(a)|/L(a)$. 
 Note that $|\alpha| = 1$ and we obtain that
 \[
|T(z)| = |Re(\alpha L(z))| \leq |\alpha||L(z)| = |L(z)| \leq h(z).
 \]
 Hence we obtain that $|T(z)| \leq h(z)$ for all $z \in \mathbb{C}^N$ and 
 $|T(a)| = Re(\alpha L(a)) = |L(a)| = h(a)$. Therefore $D$ is 
 convex at $a/h(a)$ according to definition \ref{R-convex}.\\
 
Conversely, assume that $D$ is convex at $a/h(a)$ as per
 definition \ref{R-convex}. Then there exists a $\mathbb{R}$-linear
 functional $f: \mathbb{C}^N \to \mathbb{R}$ with $|f(z)| \leq 
 h(z)$ and $|f(a)| = h(a)$. Define a $\mathbb{C}$-linear functional
  $L: \mathbb{C}^N \to \mathbb{C}$ by $L(z) = f(z)-if(iz)$. Observe
   that for each $z \in \mathbb{C}^N$, there exist $\lambda_z \in 
\mathbb{C}$ with $|\lambda_z| = 1$ such that $\lambda_zL(z) = 
|L(z)|$. Note that
\[
|L(z)| = |Re(\lambda_zL(z))| = |Re(L(\lambda_zz))| = 
|f(\lambda_zz)| \leq h(\lambda_zz) = |\lambda_z|h(z) = h(z)
\]
Hence, we get that $|L(z)| \leq h(z)$ for all $z \in \mathbb{C}^N$.
In-particular, $\vert L(a) \vert \leq h(a)$. But then 
the very definition of $L$ shows that $\vert L(a) \vert \geq h(a)$
and we thereby get that $\vert L(a) \vert = h(a)$ (and $ia$ lies in
the kernel of $f$ though this observation is not necessary for 
the proof per se). 
Hence, we have proved that $D$ is convex at $a/h(a)$ according to definition \ref{C-convex}.
\end{proof}

We also note that the notion of holomorphic extremality is equivalent to 
the notion of $\mathbb{C}$-extremality of the 
boundary in the 
presence of convexity, as clarified in the 
following proposition.
\begin{prop}\label{hol_C_extrm}
Let $D$ be a bounded balanced pseudoconvex domain in 
$\mathbb{C}^N$. Assume that $D$ is convex near $p \in 
\partial D$. Then $D$ is $\mathbb{C}$-extremal at $p$ if and only 
if $D$ is holomorphic extreme at $p$.
\end{prop}
\begin{proof}
    Since every holomorphic extreme boundary point of $D$ is $\mathbb{C}$-
extremal, it suffices to show that if 
$D$ is $\mathbb{C}$-extremal at $p$
then $D$ is holomorphic extreme at $p$. Assume that $D$ 
is $\mathbb{C}$-extremal at $p$.
Let $\text{co}(D)$ denote the convex hull of $D$. Firstly note that 
that $\text{co}(D)$ is a balanced convex domain. 
Next, we divide the proof into the following steps:\\

\textit{Step-1:} If $p \in \partial D$ and $D$ is convex at $p$ then 
$p \in \partial \text{co}(D)$.\\
By proposition 2.3.1 (b) in \cite{Jrncki_invrnt_dst}, we get that 
$\gamma_D(0;p) = h(p) = 1$. Applying proposition 2.3.1 (d), we obtain that 
$\hat{h}(p) = 1$, where $\hat{h}$ is the Minkowski function of $\text{co}(D)$. 
This shows that $p \in \partial \text{co}(D)$.\\

\textit{Step-2:} $p$ is a $\mathbb{C}$-extremal boundary 
point of $\text{co}(D)$.\\
To prove this by contradiction, assume that 
$p$ is not $\mathbb{C}$-extremal boundary point of $\text{co}(D)$. Then 
there exists non-zero $u \in \mathbb{C}^N$ such that $p+tu \in \text{co}(D)$ for all $t \in \Delta$. After applying proposition 2.3.1(a) in \cite{Jrncki_invrnt_dst} we get that $\gamma_D(0,
p+tu) = 1$. Since $D$ is convex near $p$, we may assume that $D$ is convex at 
$(p+tu)/h(p+tu)$ for all $t \in \Delta$; the equivalence 
of (iv) and (v) in proposition 2.3.1 in \cite{Jrncki_invrnt_dst} shows that $h(p+tu) = 1$ for all $t \in \Delta$. This implies that $p+tu \in \partial D$ for all $t \in \Delta$, which contradicts the fact that $p$ is 
$\mathbb{C}$-extremal point of $D$.  This finishes the proof of Step 2.\\ 

Thereafter, applying 
lemma 11.3.7 in \cite{Jrncki_invrnt_dst} to $\text{co}(D)$, we obtain 
that $p$ is a holomorphic extreme point of $\text{co}(D)$.
Using arguments as in the proof of Step 2, we obtain that $p$ is a holomorphic extreme point of $D$. This finishes the proof of the proposition.
\end{proof}

\noindent The next pair of results show that the existence (or non-existence)
of retracts of a {\it balanced} pseudoconvex domain through the 
origin in any given direction is in essence, 
completely determined by the convexity of the nearest boundary point in that 
direction, as clarified by the foregoing definitions.
\begin{cor}\label{No_1dim_Retrct}
Let $D$ be a balanced pseudoconvex domain in $\mathbb{C}^N$ and $a \in D$ 
such that $D$ is not convex at $a/h(a)$. Then there does not exist any
one-dimensional retract of $D$ passing through $0$ and $a$. 
\end{cor}
\noindent A proof of this will be provided later in the course of other things in
the next section. We next get to a positive 
result.
\begin{prop}
Let $D$ be a balanced pseudoconvex domain in $\mathbb{C}^N$ and $a \in D$ 
such that $D$ is convex at $a/h(a)$. Then  there always exists a
one dimensional linear retract of $D$ passing through $0$ and $a$. 
\end{prop}
\begin{proof}
Since $D$ is convex at $p$, there exists a $\mathbb{C}$-linear functional $L : \mathbb{C}^N \to \mathbb{C}$ such that $|L(z)| \leq h(z)$ for all $z \in \mathbb{C}^N$ and $|L(p)| = 1$. 
By multiplying $L$ by a suitable unimodular constant, we may assume that $L(p) = 1$.  Consider the linear transformation $T : \mathbb{C}^N \to \mathbb{C}^N$ defined by $T(z) = L(z)p$. 
Note that $T$ maps $D$ into $D$ because 
\[
h(T(z)) = h(L(z)p) = |L(z)| h(p) \leq h(z) < 1.
\]
Moreover, $T$ is a linear projection map:
\[
T(T(z)) = T(L(z)p) = L(L(z)p)p = L(z)L(p)p = L(z) p = T(z).
\]
This shows that $\text{span}(p) \cap D$ is a linear retract of $D$.
Since $D$ is convex at $p$, there exists a $\mathbb{C}$-linear functional $L : \mathbb{C}^N \to \mathbb{C}$ such that $|L(z)| \leq h(z)$ for all $z \in \mathbb{C}^N$ and $|L(p)| = 1$. 
By multiplying $L$ by a suitable unimodular constant, we may assume that $L(p) = 1$.  Consider the linear transformation $T : \mathbb{C}^N \to \mathbb{C}^N$ defined by $T(z) = L(z)p$. 
Note that $T$ maps $D$ into $D$ because 
\[
h(T(z)) = h(L(z)p) = |L(z)| h(p) \leq h(z) < 1.
\]
Moreover, $T$ is a linear projection map:
\[
T(T(z)) = T(L(z)p) = L(L(z)p)p = L(z)L(p)p = L(z) p = T(z).
\]
This shows that $\text{span}(p) \cap D$ is a linear retract of $D$.
\end{proof}
\begin{rem}
\begin{enumerate}
\item In the case of bounded balanced convex domains, the geometric Hahn-Banach theorem ensures that every one-dimensional subspace of $D$ is a retract of $D$.
\item Every one dimensional retract $Z$ of $D$ passing through the origin and $a$ can be parametrized by $\lambda a/h(a)$ for $\lambda \in \Delta$.
\end{enumerate}
\end{rem}
\noindent We spell out the following immediate consequence as it settles a first 
basic question about the very existence of any non-trivial retract in bounded balanced domains
of holomorphy of central interest of this article.

\begin{prop} \label{Exist-result}
If $D$ is an arbitrary bounded balanced pseudoconvex domain in $\mathbb{C}^N$, then
there exists atleast one (non-trivial) retract passing through origin.
\end{prop}
\begin{proof}
Let $a \in \partial D$ be the farthest point from the origin. Note that $D$ is convex at $a$.
By the above proposition, there always exists a one-dimensional linear retract of $D$ passing through $a$.
Alternatively, we may observe that the orthogonal projection onto 
${\rm span}(a)$ maps $D$ onto it and thereby a retraction map on $D$. The aforementioned retract is 
the image of this retraction map.
\end{proof}
\noindent Combining theorem \ref{Blncd_Rtrct_Cmplx_Geodsc} and theorem 3.2 in \cite{Lemp_geodesc}, 
the following result is obtained. 
\begin{thm} Let $D$ be a bounded convex domain of $\mathbb{C}^N$,
\begin{enumerate}
\item[(1)] Given two points $x$ and $y$ of $D$, there exists one dimensional 
retract $R$ of $D$ passing through $x$ and $y$.    
\item [(2)] Given $x \in D$ and $v \in T_x(D)$, there exists one dimensional 
retract $R$ of $D$ passing through $x$ with tangent vector $v$.  
\end{enumerate}
\end{thm}

\begin{rem}
Assume for this remark that $D$ is convex. The existence in 
abundance of one dimensional retracts of strictly convex domains (in 
$\mathbb{C}^N$) is well-known due to the work of Lempert. Another well-known work which 
obtains retracts of higher dimensions under certain conditions is \cite{Ian_Grhm} of Graham; as 
noted therein, the condition in the above theorem that $D_L= D \cap L$ be a linear retract 
of $D$ which is equivalent to the existence of a norm one projection 
onto $L$ (denoted by $\pi_L$ in the proof), is always met if $L$ is one 
dimensional -- this follows by the Hahn--Banach theorem which may be used to first get a 
linear functional $f$ on $(\mathbb{C}^N, \mu)$ of norm one which takes the value $1$ 
at a boundary point of $D_L$ (one may view $f$ as mapping $D$ 
onto $D_L \subset L \simeq \mathbb{C}$); the desired projection is then obtained by taking 
the linear projection of $\mathbb{C}^N$ onto $L$ along
$\ker(f)$. On the other hand (also noted therein), if there exists a norm one projection 
of the Banach space $(\mathbb{C}^N, \mu)$ onto a every linear subspace, then $(\mathbb{C}^N, \mu)$ 
is actually a Hilbert space i.e., $\mu$ is actually induced by an inner product owing to classical results
mentioned in the introduction namely that of Kakutani in \cite{Kakutni} for the real case and Bohnenblust in \cite{Bhnblst} 
for the complex case. Thus, not every linear subspace intersected 
with $D$ need be a retract of $D$. 
\end{rem}

\subsection{Fixed point sets and Retracts}
As we already saw in the introduction, fixed point sets need not be realizable as retracts
in bounded non-convex domains. However, there is still a connection between fixed point sets and retracts under 
mild hypotheses, which we formulate in the following proposition:   
\begin{prop}
Let $D$ be a bounded taut domain in $\mathbb{C}^N$. Assume that 
$f: D \to D$ be a holomorphic map which is not an automorphism 
and $ \text{Fix}(f)$ is a non-empty subset of $D$ which consists 
of more than one point. Then there exists a non-trivial 
retract $Z$ of $D$ 
such that ${\rm Fix}(f) \subset Z$.
\end{prop}

\begin{proof}
Consider the sequence of holomorphic maps on $D$ given by $f_n := f^{\circ n}$ which 
denotes as usual the $n$-fold composition of $f$ with itself. Since 
$ \text{Fix}(f) \neq \phi$, $(f_n)$ is not compactly divergent. By theorem 
2.1.29 in \cite{Abate_Itrtn}, there exists a retract $Z$ of $D$ and a 
holomorphic retraction $\rho : D \to Z$ such that $\rho$ is a limit point 
of the sequence $\{f_n\}$.  Note that $\rm Fix(f) \subset \rm Z$. Note also that $\rm dim(Z) \neq 0$, 
because $Z$ is a connected complex manifold and $\#(Fix(f)) > 1$. By Corollary 2.1.30 in \cite{Abate_Itrtn}, 
$\rm dim(Z)$ is equal to the number of eigenvalues of $df_{z_0}$ contained in 
$\partial \Delta$. If all eigenvalues of $df_{z_0}$ are in $\partial \Delta$,
 then by theorem 2.1.21 (v) in \cite{Abate_Itrtn}, we get $f$ is an 
 automorphism. This contradicts the hypothesis. Hence, we can conclude 
 that $\rm dim(Z) < N$. This proves that $Z$ is a non-trivial retract of 
 $D$ and $\rm Fix(f) \subset Z$.
\end{proof}
 
\subsection{Hyperconvex Domains}\label{Hypercnvx}
In this subsection, we briefly recall a standard class of pseudoconvex domains containing convex domains, namely hyperconvex domains. Hyperconvexity is an intermediate notion between convexity and pseudoconvexity that is a generalization of a key 
concept (going by the name of barriers) characterizing 
solvability of the Dirichlet problem in dimension one. Owing to this, 
we recall the connection between hyperconvexity and psh (plurisubharmonic) barrier boundary points. To begin with, let us recall the definition of a psh barrier point from \cite{barrier}.
\begin{defn}
    A boundary point $a$ of an domain $D \subset \mathbb{C}^N$ 
    is said to be a \textit{(global) psh barrier point} for $D$ if 
there is a negative 
plurisubharmonic function $\varphi$ on $D$ with 
$\displaystyle\lim_{D \ni z \to a}\varphi(z) = 0$.
We call $a$ a \textit{local psh barrier point} for $D$ if there is a (connected) open neighborhood $U$ of $a$ such that $a$ is a psh barrier point for $D \cap U$.
 We shall say that $D$ has 
psh barrier boundary (resp. local psh barrier boundary) if all of its boundary points are psh (resp. local psh) barrier points.
\end{defn}
\begin{rem}
Examples of domains whose boundary satisfying the local psh barrier condition, as in the definition, abound. All bounded strictly pseudoconvex domains, more generally, all bounded (weakly) pseudoconvex domains of finite type, serve as examples. More interestingly, non-examples are obtained by looking at domains whose boundaries contain one or more piece of analytic varieties. For instance, while all analytic polyhedra $D$ satisfy this condition; as soon as we remove any non-trivial analytic set $A$, $D\setminus A$ no longer satisfies this condition at any of the (boundary) points $a \in A$.
\end{rem}
Let us recall the definition of a hyperconvex manifold from \cite{Hyp-cnvx}.
\begin{defn}
A complex manifold $M$ is said to be \textit{hyperconvex} if there exists a bounded strictly plurisubharmonic (psh) exhaustion function $\varphi$ on $M$.
\end{defn}
\begin{rem}
\begin{itemize}
\item [(i)]  If $F: M_1 \to M_2$ is a proper holomorphic map between complex manifolds $M_1$ and $M_2$ with $M_2$ is hyperconvex, then any bounded strictly plurisubharmonic exhaustion function $\varphi$ on $M_2$ induces a bounded strictly plurisubharmonic exhaustion function on $M_1$ via the pullback $\psi := \varphi \circ F$. Hence, $M_1$ is hyperconvex. In particular, this shows that hyperconvexity is a biholomorphic invariant.
\item [(ii)] Every hyperconvex domain in $\mathbb{C}^N$ is pseudoconvex (see theorem 5.2.1 and proposition 5.2.2 in \cite{Kobyshi}).
\item [(iii)] In the case of bounded balanced pseudoconvex domains, hyperconvexity is equivalent to the continuity of their Minkowski functionals (see remark 3.2.3 a), c) in \cite{Jrncki_invrnt_dst}).
\end{itemize}
\end{rem}
\begin{prop}\label{retrct_hypcnvx}
Every retract of a hyperconvex complex manifold is hyperconvex.
\end{prop}
\begin{proof}
Let $Z$ be a retract of a hyperconvex complex manifold $X$ and let $\varphi : X \to [-\infty,c)$ be a bounded plurisubharmonic exhaustion on $X$. Note that $\varphi_{|_Z}$ serves as a bounded plurisubharmonic exhaustion function on $Z$. Hence, $Z$ is a hyperconvex complex manifold.
\end{proof}
The following theorem (see theorem 1.17.5 in \cite{Jrnck_Frst_stps}) describes the relationship between hyperconvexity and the existence of psh barriers.
\begin{thm}\label{hypcnvx_local}
Let $D \subset \mathbb{C}^N$ be a domain. Then $D$ is hyperconvex if and only if for any boundary point $a \in \partial D$ (including $a=\infty$ when $D$ is unbounded), $a$ is a psh barrier point for $D$.
\end{thm}
\begin{rem}\label{hyp_cnvx_rem}
If $D$ as in the above theorem is bounded, then by theorem 12.4.7 in \cite{Jrncki_invrnt_dst}, $D$ is hyperconvex iff every boundary point of $D$ admits a local psh barrier boundary. 
\end{rem}

\subsection{Growth Vectors and Tangent Cone} In this subsection, we recall the definitions and related results concerning growth vectors and tangent cones from \cite{Whitney}. 
\begin{defn}
A set $C \subset \mathbb{C}^m$ is an \textit{open complex cone} if $z \in C$ implies $\lambda z \in C$ for all $\lambda \in \mathbb{C}$.
\end{defn}
We may construct open complex cones by taking an arbitrary set $S \subset \mathbb{C}^N \setminus \{0\}$ and considering $C = C(S) = \{\lambda z | z \in S,~ \lambda \in \Delta\}$. We then refer to $C$ as the cone generated by $S$. The associated \textit{punctured (complex) cone} is given by $C^* = \{\lambda z | z \in S,~ \lambda \in \Delta^*\}$. Note that if $S$ is convex, then $C(S)$ is a convex cone. 
\begin{defn} Suppose $f$ is holomorphic near $a$, and ord ${ }_a f \geqq v$. Then we may write
\begin{equation}\label{homo_growth}
f(x)=  \sum_{|\lambda|=v} A_\lambda\left(x_1-a_1\right)^{\lambda_1} \cdots\left(x_n-a_n\right)^{\lambda_n}  +\sum_{|\lambda|=v+1} F_\lambda(x)\left(x_1-a_1\right)^{\lambda_1} \cdots\left(x_n-a_n\right)^{\lambda_n}
\end{equation}
where the $F_\lambda$ is holomorphic near $a$. The first sum on the right of (\ref{homo_growth})  is the \textit{initial polynomial} of $f$ at $a$; call it $f_a^{[*]}(x-a)$.
\end{defn}
As a consequence of (\ref{homo_growth}),
 consider any vector $v \in \mathbb{C}^n$, use $x=a+z v, z \in \mathbb{C}$ for sufficiently small. Then we have the expansion
\[
f(a+z v)=z^v\left[f_a^{[*]}(v)+z R(a, z, v)\right] .
\]
where $R$ is holomorphic for small $z$.

\begin{defn}\label{growth} Let $f$ be holomorphic and not identically zero near $a$, with $f(a)=0$. A non-zero vector $v$ is a \textit{growth vector} for $f$ at $a$ if for the function $g(z) :=f(a+z v)$ defined on  sufficiently small disk $\Delta_r \subset \mathbb{C}$, we have $\operatorname{ord}_0 g=\operatorname{ord}_a f$.
\end{defn}
\begin{prop}\label{non-growth} With $f$ as in the above definition \ref{growth}, let $V_f$ be the set of non-growth vectors for $f$ at $a$. Then $V_f$ is 
the set of solutions of $f_a^{[*]}(v)=0$; hence it is an algebraic cone, and is
closed and nowhere dense in $\mathbb{C}^N$.
\end{prop}
\begin{rem}
Note that $dim(V_f) = 2N-2$, because by above proposition, $V_f$ is the zero set of homogeneous polynomial $f_a^{[*]}$. 
\end{rem}
\begin{defn}
Let $V$ be an analytic variety in $\mathbb{C}^N$. The tangent cone $C(V, p)$ is the set of all vectors $v$ with the property that for each $\varepsilon>0$ there is a point $q \in V$ and there is a complex number $\alpha \in \mathbb{C}$ such that
\[
|q-p|<\varepsilon, \quad|\alpha(q-p)-v|<\varepsilon
\]
Equivalently, $v \in C(V,p)$ if and only if there exist sequences $(q_i)$ and $(\alpha_i)$ such that
\[
q_i \in V, \quad q_i \rightarrow p, \quad \alpha_i\left(q_i-p\right) \rightarrow v
\]
\end{defn}
Note that the tangent cone is a complex cone. Thus the direction of (non-zero) $v$ is the limit of directions of secants from $p$ to points of $V$ converging to $p$.\\
The following theorem establishes the relationship between non-growth vectors of $f$ at $p$ and the tangent cone to the zero set of $f$ at $p$. 
\begin{thm}\label{tangent_cone} Let $f$ be holomorphic and not identically zero near $p$ with $Z$ as its zero set and $\operatorname{ord}_p f=m$. For a non-zero vector $v \in \mathbb{C}^N$, define the function
\[
g_v(t)=f(p+t v),~ t \in \Delta_r,
\]
where $\Delta_r$ is a disc of sufficiently small radius $r>0$. Then, $\operatorname{ord}_0 g_v>m$ if and only if
$v \in C(Z, p)$, the tangent cone to $Z$ at $p$.
\end{thm}

\subsection{Analytic and Semi-analytic sets}\label{Semi-anal}

In this subsection, we shall recall the definitions and results regarding \textit{algebraic} and \textit{analytic} sets (varieties) over one of the fields $\mathbb{R}$ or $\mathbb{C}$), from \cite{Singularity} and \cite{Strat}. Throughout this subsection, the term \textit{analytic} refers to real analytic when $K=\mathbb{R}$, and complex analytic when $K=\mathbb{C}$. Let $M$ denote a connected $n$-dimensional (real or complex) analytic manifold and let $\mathcal{O}(M)$ denote the ring of analytic functions from $M$ to $K$.

\begin{defn}\label{Def1} Let $A$ be an analytic subset of $M$. A point $p \in A$ is called \textit{smooth}, of dimension $d$, if and only if there exists an open neighbourhood $W$ of $p$ in $M$ such that $W \cap A$ is an analytic submanifold of $W$ of dimension $d$ (over the field $K$ ). Thus, $p \in A$ is smooth of dimension $d$ if and only if there exists an open neighborhood $W$ of $p$ in $M$ and $f_{d+1}, \ldots f_n \in \mathcal{O}(W)$ such that $W \cap A=Z\left(f_{d+1}, \ldots, f_n\right)$ and, for all $x \in W, \nabla f_{d+1}(x), \ldots, \nabla f_n(x)$ are linearly independent.\\
We denote the set of smooth points of $A$ by $\stackrel{\circ}{A}$, and the set of smooth points of dimension $d$ by $\stackrel{\circ}{A}^{(d)}$. A smooth point of $A$ of highest dimension is called a \textit{regular point} (this definition of \textit{regular} is not followed by all authors); we denote the set of regular points of $A$ by $A^{\text {reg }}$.
A point $p \in A$ which is not smooth is called a \textit{singular point} (or, a singularity). We denote the set of singular points, the singular locus, of $A$ by $\Sigma A$. 
\end{defn}
\begin{thm} Let $A$ be an analytic subset of $M$. Suppose that $0 \leq d \leq n$. Then
\begin{enumerate}
\item $\stackrel{\circ}{A}^{(d)}$ is a d-dimensional analytic submanifold of $M$, and is an open subset of $A$;
\item the $\stackrel{\circ}{A}^{(d)}$ are disjoint for different $d, \stackrel{\circ}{A}=\stackrel{\circ}{A}^{(0)} \cup \ldots \stackrel{\circ}{A}^{(n)}, \stackrel{\circ}{A}$ is an analytic submanifold of $M$, and $\stackrel{\circ}{A}$ is open in $X$;
\item $\stackrel{\circ}{A}$ is dense in $A$; and
\item $\Sigma A$ is a closed, nowhere dense subset of $A$.
\end{enumerate}
\end{thm}
\begin{defn}\label{Def2} The \textit{dimension} (over $K$), denoted $\operatorname{dim} A$, \textit{of an analytic set} $A \subseteq M$ is the largest $d$ such that $\stackrel{\circ}{A}^{(d)}$ is non-empty. The codimension of $A$ is the number
$\dim M - \dim A $. The dimension of $A$ \textit{at a point} $p \in A$, denoted $\operatorname{dim}_p A$, is the largest $d$ such that $p$ is in the closure of $\stackrel{\circ}{A}^{(d)}$.
\end{defn}

\begin{defn} A subset $V$ of $M$ is called a \textit{semivariety} if locally at each point $x \in V$  it is a finite union of subsets defined by equations and inequalities

$$
f_1=\cdots=f_k=0 \quad \begin{cases}g_1 \neq 0, \ldots, g_l \neq 0 & \text { (complex case) }, \\ g_1>0, \ldots, g_l>0 & \text { (real case) },\end{cases}
$$

where $f_j$ 's and $g_k$ 's are real (or complex) analytic (or algebraic) depending on the case under consideration.
\end{defn}
In the real algebraic case, semivarieties are commonly referred to as semialgebraic sets; in the complex algebraic case, they are known as constructible, and in either analytic case, they are called semianalytic sets.
\begin{defn}Let $V$ be a semianalytic subset of $M$.
We then define smooth points, regular points and singular points exactly as in the above Definition \ref{Def1} for analytic subsets, using the same notations.

In the case of semivarieties, it is known that the smooth points of a semianalytic set are dense. Therefore, we define the dimension of $X$, and the dimension of $X$ at a point exactly as in the above Definition \ref{Def2}.
\end{defn}
\section{Proofs of theorems \ref{JJ-improved}, \ref{Polydisk}, \ref{Graph}, \ref{MainThm} } \label{Polyballs&mainThms-sectn}

\noindent One of our main goals of this section is to characterize the holomorphic retracts of the Cartesian product 
$B_1 \times \ldots \times B_N$ where the $B_j$'s are bounded 
balanced domains of holomorphy. The result seems to be new at-least to the
extent of not being noted elsewhere; in-fact
new even in the case when all the $B_j$'s are convex
in which case the above product will be referred to as a polyball for short. However we hasten to admit at the outset here, that the proof of theorem \ref{MainThm} about 
the aforementioned goal, closely follows the recent work \cite{BBMV-union}; with the new ingredient being theorem 
\ref{extr-Schw-lem-holextrbdy}, an analogue of Mazet's 
Schwarz lemma for the non-convex case, all discussed below.
We give here a
thorough treatment despite repetition of some of the essential ideas
as it allows for 
a better exposition of the proof of both this theorem, as well as  proofs 
of theorems prior to it as in the 
title of this section, which are fundamental to this article.\\

\noindent Recall that theorem \ref{MainThm} has $2$ parts,  
wherein its part (A) requires the boundaries of all the
factors of the product domain to satisfy an extremality condition.
We discuss some examples for this now. Since the 
result is new even for the strictly convex case, we shall dwell upon
this special case of interest first.
Thus until further notice,
suppose the $B_j$'s are bounded balanced strictly convex domains or equivalently, the unit ball with respect to 
some norm in $\mathbb{C}^{N_j}$ for some $N_j \in \mathbb{N}$, for each $j=1,\ldots, N$; as noted in the 
previous section, this is equivalent to their  
boundaries being $\mathbb{R}$-extremal and implies (owing to the
temporary assumption of convexity) the holomorphic extremality condition required by \ref{MainThm} (A). As it turns out that the essential arguments for the case of a product 
of a pair of such balls extend to the more general case, we shall for simplicity of exposition, 
focus attention for the most part on the case of a product of $2$ such balls.
So, let
$B_1 \subset \mathbb{C}^{N_1}, B_2\subset \mathbb{C}^{N_2} $ (where $N_1,N_2 \in \mathbb{N}$ with say, $\min\{N_1,N_2\}=N_1$) 
denote a pair of bounded balanced convex domains with $\mathbb{R}$-extremal boundaries;
in-particular, $B_1, B_2$ are balls with respect to some 
norms (on $\mathbb{C}^{N_1}, \mathbb{C}^{N_2}$ respectively) and needless to say are domains 
of holomorphy. The simplest example of such a domain is the standard unit disc 
in $\mathbb{C}$ which we denote by 
$\Delta$. However, many nice properties of the disc may well not hold for such domains 
(i.e., bounded balanced convex domains). To begin with, balls in the $l^p$-norm 
for $1<p<\infty$, while satisfying the
condition that their boundary points are $\mathbb{R}$-extremal, need not always be 
infinitely smooth -- they are atmost $[p]$-smooth, where $[\cdot]$ denotes the greatest 
integer function. They are thus atleast $C^1$-smooth and to get an example where 
this $C^1$-smoothness fails, we may take an intersection of a pair of suitable dilates of the
ball in the $l^{p_1}$ norm 
with that of the $l^{p_2}$ norm, for a pair of positive distinct reals $p_1,p_2$ in the 
interval $(1,\infty)$ to get examples of domains whose boundaries are not $C^1$-smooth 
but nevertheless satisfy the $\mathbb{R}$-extremal condition; as this is essentially the well-known
fact that the intersection of a pair of strictly convex domains is again strictly convex, we shall 
provide details with $\mathbb{R}$-extremality replaced by $\mathbb{C}$-extremality, in the following.

\begin{prop} \label{R-extrml}
If $D_1,D_2$ are any pair of (not 
necessarily convex) domains in $\mathbb{C}^N$ with weakly (resp. strongly) $\mathbb{C}$-extremal boundary, 
then so is every connected component of $D:= D_1 \cap D_2$.
\end{prop}

\begin{proof}
    Let $D_1,D_2$ be pair of domains with weakly $\mathbb{C}$-extremal boundary. To prove this result  by contradiction, assume $a \in \partial D$ is not weakly $\mathbb{C}$-extremal. That is, there exists non-zero $u \in \mathbb{C}^N$ and there exists $\epsilon > 0$ such that $a + tu \in \partial D~ \text{for~ all } ~ t \in \mathbb{C} ~ \text{with}~ |t| < \epsilon$. We shall split the proof into cases depending upon which 
`piece' of the boundary $a$ lies; herein, piece refers to one of the components
in the decomposition of the boundary of $D$ as
\[
\partial{(D_1 \cap D_2)}  = (\partial D_1 \cap  D_2) \cup (D_1 \cap \partial D_2) \cup (\partial D_1 \cap \partial D_2)
\]
Without loss of generality assume that $\epsilon = 1$. Consider a map $\alpha: \Delta \to \mathbb{C}^N$ defined by $\alpha(t) = a+tu$. Since $u \neq 0$, $\alpha$ is a non-constant holomorphic map. Note that $\alpha$ is a homeomorphism from $\Delta$ onto $\alpha(\Delta)$, which is in $\partial D$. Consider the sets $S_1 := \{t \in \Delta: \alpha(t) \in \partial{D_1}\}$ and $S_2 := \{t \in \Delta: \alpha(t) \in \partial{D_2}\}$. By above decomposition, $S_1 \cup S_2 = \Delta$ and note that $S_1 = \alpha^{-1}(\partial D_1 )$  and $S_2 = \alpha^{-1}(\partial D_2 )$. Note that $\partial D_i \cap \partial D$ is a closed subset of $\partial D$ for $i = 1,2$. By continuity of $\alpha$, $S_1$ and $S_2$ are closed sets in $\Delta$.\\
We want to prove that $S_1$ or $S_2$ has atleast one interior point. Suppose $S_1$ has no interior points in $\Delta$. Let $\rm int(S_i)$ denotes the set of all interior points of $S_i$ for $i = 1,2$. Note that
\[
S_2 = \overline{S_2} = \overline{\Delta \setminus S_1} = \Delta \setminus int(S_1) = \Delta
\]
So, we get $S_2 = \Delta$. Hence we can conclude that $S_1$ or $S_2$ has atleast one interior point.\\
Without loss of generality assume that $S_1$ has an interior point. Let $t_0 \in \Delta$ be an interior point of $S_1$. This means that there exists $\delta > 0$ such that $A := \{t \in \mathbb{C}: |t-t_0| < \delta \} \subset S_1$.  Choose any $s \in \mathbb{C}$ with $|s| < \delta$. Let $b := a+t_0u$. Note that $b \in \partial D_1$, because $t_0 \in S_1$. We can write $b +su = a + (t_0+s)u$ and observe that $t_0 + s \in A$. Hence $t_0 +s \in S_1$ for all $s \in \mathbb{C}$ with $|s| < \delta$. By definition of $S_1$, we obtain that $b+su \in \partial D_1$ for all $s \in \mathbb{C}$ with $|s| < \delta$. Since $D_1$ is weakly $\mathbb{C}$-extremal, we have $u = 0$. This contradicts our assumption that $a$ is not a weakly $\mathbb{C}$-extremal point of $\partial D$. Hence we conclude that $a$ is weakly $\mathbb{C}$-extremal.\\

\noindent Next, assume that $D_1,D_2$ are domains in $\mathbb{C}^N$ with strongly
$\mathbb{C}$-extremal boundary. Let $a \in \partial (D_1 \cap D_2)$.
If possible assume that there exists $\epsilon > 0$ and $u \in
\mathbb{C}^N$ such that $a+tu \in \overline{D_1 \cap D_2}$ for all
$t \in \mathbb{C}$ with $|t| < \epsilon$. Note that 
$\overline{D_1 \cap D_2} = \overline{D_1} \cap \overline{D_2}$. 
If $a+tu \in \overline{D_1 \cap D_2} $ for all $t \in \mathbb{C}$ 
with $|t| < \epsilon$, then $a+tu \in \overline{D_1} \cap 
\overline{D_2}$ for all $t \in \mathbb{C}$ with $|t| < \epsilon$.
Since $D_1$ and $D_2$ are domains with strongly $\mathbb{C}$-extremal 
boundary, we have $u = 0$. Therefore, $D_1 \cap D_2$ is a domain
with strongly $\mathbb{C}$-extremal boundary.
\end{proof} 

\noindent Thus an intersection of finitely many balls with $K$-extremal 
boundaries (for either case $K=\mathbb{R}$ or $\mathbb{C}$)
is again a convex domain with $K$-extremal boundary. In-particular, for instance, the intersection
of the $l^1$-ball with a suitable dilate of another $l^p$-ball for $p \in (1, \infty)$ can be used to
get hold of a concrete instance of a domain 
with $\mathbb{C}$-extremal (but not smooth) boundary. As an example which is 
not derived out of balls with respect to the $l^p$-norms (which are symmetric under a
permutation of coordinates, despite not being strictly
isotropic), we may for example, consider 
the `more' anisotropic ball
$\mathcal{E} := \{(z,w) \in \mathbb{C}^2\; : \; \vert z \vert^2 + \vert w \vert^4<1\}$
referred to as a `complex ellipsoid' or 
an `egg' domain . So, 
there is a good variety of bounded balanced convex domains with $\mathbb{R}$-extremal and $\mathbb{C}$-extremal
boundaries. The form of all holomorphic retracts of such domains passing through the 
point with respect to which they are balanced namely the origin, follows
from the theorem due to Vesentini, already stated in the introductory section; though 
not in as straightforward a manner, for their products. The product $B_1 \times B_2$, though 
bounded convex and balanced, certainly fails to satisfy the condition that its boundary 
points be $\mathbb{R}$-extremal, even when both factors $B_1, B_2$ satisfy this condition 
individually. Indeed even in the simplest case when $B_1=B_2 = \Delta$, the only points 
on the boundary of $B_1 \times B_2$ which are $\mathbb{R}$-extremal are the points of its 
distinguished boundary, which is a rather thin subset of the boundary. 
As mentioned at the start, one of our goals is to first describe the retracts of such a product  
$B_1 \times B_2 \subset \mathbb{C}^{N_1 + N_2}$, the unit ball with respect to the norm 
$\parallel (w,z) \parallel = \max\{\vert w \vert_{\mu_1}, \vert z \vert_{{\mu_2}} \}$ where 
$\mu_1, \mu_2$ denote the norms in which $B_1, B_2$ are the standard (open) unit balls. \\

\noindent Thus in particular, by theorem \ref{Firstthm} all retracts of 
the domains as in theorem \ref{MainThm} namely $D:= B_1 \times B_2$, are submanifolds thereof 
and moreover are all contractible since
 $B_1 \times B_2$ is; as already indicated, the retracts
 in the case $B_1,B_2$ are convex, are also 
 strong holomorphic deformation retracts -- indeed, the linear homotopy between the identity 
 map of $D$ and any retraction gives a homotopy in which the retract remains pointwise 
 fixed \footnote{In general, while every (continuous) retract of a contractible topological 
 space $X$ is a deformation retract much as any self-map of $X$ is homotopic to the identity, not 
 every retract is a {\it strong} deformation retract. Infact, there are contractible 
 spaces which 
 do not admit a strong deformation retraction to any of its points, though obviously 
 they admit a deformation retraction to each of its points.} during the homotopy and for 
 each point of time in the homotopy, the mapping is holomorphic. While this means that 
 all retracts of $D$ are topologically trivial thereby placing strong restrictions on 
 which complex submanifolds of $D$ can happen to be a retract, we shall see on the other 
 hand that there is some variety as well: for instance if the factors $B_1, B_2$ are 
 either a polydisc or a Euclidean ball, there are retracts for every possible dimension 
 which thereby have got to be holomorphically distinct. Furthermore again, not every 
 contractible complex submanifold of $D$ passing through the origin can qualify to be a 
 retract, for as we shall show, retracts through the origin of 
$B_1 \times B_2$ are submanifolds in normal form in the sense that they are all graphs of 
holomorphic maps; as is very well-known, graphs of holomorphic functions are submanifolds 
which are simultaneously in parametric form as well as in level set form and the implicit 
function theorem allows us to put every submanifold locally in such a normal 
form (as graphs). In view of the foregoing theorem (and the implicit function theorem) 
then, retracts are obviously graphs of holomorphic maps (over some set of axes) locally; however 
in the specific case of retracts of $D=B_1 \times B_2$ passing through the origin, we shall 
show that they are actually so globally i.e., each retract of the product domain $D$ is the 
graph of a single holomorphic map over some linear submanifold of $D$. As we shall encounter 
retracts in `graphical form' in several cases in the sequel, we make a general definition for clarity.
\begin{defn}
Let $X,Y$ be a pair of complex manifolds, $S$ a submanifold of $X$ and $f$ a $Y$-valued holomorphic 
map on $S$ i.e., $f$ maps $S$ into $Y$. Then, the {\it graph of $f$ over $S$} is the subset 
of $X \times Y$ given by $\tilde{S} = \{ (s,f(s)) \; :\; s \in S\}$ (i.e., our definition views 
the graph as a subset of $X \times Y$ rather than $S \times Y$). Needless to say, $\tilde{S}$ is 
then also a complex manifold, indeed a complex submanifold of $X \times Y$.
\end{defn}
\noindent Another general fact to be noted at the outset is that for any product space 
$B_1 \times B_2$ where $B_1 \subset \mathbb{C}^{N_1},B_2 \subset \mathbb{C}^{N_2}$ are 
some pair of domains (or for that matter, complex manifolds), the product of a 
retract of 
$B_1$ with one of $B_2$, gives a retract of $B_1 \times B_2$ and the product of retractions is a 
retraction. In particular, each of the factors $B_j$ (for $j=1,2$) are retracts of $B_1 \times B_2$ as 
the canonical projections 
$\pi_j: B_1 \times B_2 \to B_j$ for $j=1,2$ are also retractions. A further 
general observation that must not be missed is that while the product of an inclusion of one factor 
with a retract of the other is a retract of the product domain as we just noted, arbitrary 
holomorphic maps between the factor domains $B_j$ give rise to a retract of the product: if 
for instance, $f:B_1 \to B_2$ is any holomorphic map, then 
the self-map of $D=B_1 \times B_2$ given by
\[
(z,w) \to (z,f(z))
\]
is actually a retraction of $D$ whose image (which is nothing but the graph of $f$) is 
thereby a retract thereof. More generally, graphs over retracts in $B_1$ or in $B_2$ 
also give rise to retracts of $D$; indeed for instance, let 
$\rho:B_1 \to S \subset B_1$ be a retraction, $f:S \to B_2$ be any holomorphic map, then the 
self-map of $D$ given by
\begin{equation} \label{graphoverreract}
(z,w) \to (\rho(z),f(\rho(z))
\end{equation}
is a retraction of $D$ onto the graph of $f$ over $S$ in $D$. These preliminary observations lead 
us to another important one for the sequel, namely a converse: projections of those 
retracts \footnote{Projections of retracts need not be retracts as is demonstrated for example, by 
any complex line segment $l$ in the bidisc $\Delta^2$ whose boundary $\partial l$ is contained 
in $\partial \Delta \times \Delta$; note then that
in such a case, the projection $\pi_2$ onto the second factor, when restricted to $l$ is not 
proper and more importantly, $\pi_2(l)$ is compactly contained in $\Delta$ thereby showing 
that it is not a retract of $\Delta$ and providing the required counterexample. Note in passing that this 
does not contradict lemma \ref{retr-transit}.} 
of $D=B_1 \times B_2$ which are graphs in $D$ over some submanifold of $B_1$ (resp. $B_2$), are 
retracts of $B_1$ (resp. $B_2$) as well. In 
light of all this, our already mentioned goal 
thus amounts to showing that these exhaust all possibilities for the retracts of our special case 
of interest $D=B_1 \times B_2$, passing through the origin, any of whose associated
retractions we shall denote as usual  by 
$\rho$; its range $\rho(D)$ will be denoted by $Z$ and our goal is to realize every 
retract $Z$ as a graph over some linear submanifold of $D$ which we shall simply refer to 
as a linear subspace of $D$, a terminology that we recall here a bit more than in the 
first section, to introduce further associated terms. 

\begin{defn}
By a {\it linear 
subspace of a balanced domain} $D$ in $\mathbb{C}^N$ 
(such as our $B_j$'s  or their product $D$ as above), we mean the 
intercept $D_L=L \cap D$ where $L$ is 
any complex linear subspace of $\mathbb{C}^N$. By the boundary of such a linear 
subspace which we shall denote by $\partial D_L$, we mean the closed set $L \cap \partial D$.
\end{defn} 

\noindent Before delving into a detailed analysis of retracts 
through the origin of 
balanced domains $D$, one may like to see at this stage what can be 
said if at all, about such retracts of bounded balanced {\it convex} 
domains $D$ in general. Let us begin with the trivia that retractions $\rho$ 
associated to such retracts of $D$ fix the origin $\rho(0)=0$, and are self-maps 
of $D$; note also that $D$ being a bounded balanced convex domain, is the unit ball 
with respect to some norm $\left \Vert \cdot \right \Vert$ on $\mathbb{C}^N$, indeed 
its Minkowski functional. This already prompts us to apply an appropriate version of 
the Schwarz lemma on $\rho$. A befitting version for this context alongwith an 
elementary proof is provided by proposition C of the article
\cite{Simh} by Simha, by which we get $\left \Vert D \rho (0) \right \Vert \leq 1$. However, 
unlike the classical Schwarz lemma, equality herein need not always imply linearity of the 
map, to attain which, additional conditions need to be imposed. This brings us to the version
of the Schwarz lemma due to Mazet mentioned at the outset, which we shall generalize below so as to make recurrent use of it in the sequel. 

\begin{thm} \label{Mazet-Schwarz} \rm(Theorem 3.6 of \cite{Mazet_Rextrm})
Let $(F, \left\Vert \cdot \right \Vert)$ be a complex Banach space, $B_F$ its standard open unit ball centered at the origin 
$B_F=\{ v \in F \; : \; \left\Vert v \right \Vert <1\}$ and $f$ a holomorphic curve in $B_F$ i.e., a holomorphic map
$f:\Delta \to B_F$. Suppose $\left\Vert df(0) \right \Vert=1$ and that 
$df(0)$ is $\mathbb{R}$-extremal in $\overline{B}_F$. Then, $f$ must be a linear map:
$f(z) = df(0)z$.
\end{thm}

\noindent It must be noted that the condition $\parallel df(0) \parallel=1$ essentially 
means that $df(0)$ is an isometry from $\mathbb{C}$ into $F$, a condition which underlies 
and entails linearity of certain maps to be considered later. If we drop the 
$\mathbb{R}$-extremal condition in the above theorem, then the conclusion about the 
linearity of $f$ may well fail, as shown by the simplest non-linear holomorphic map $f(z)=(z,z^2)$ 
of the disc into the bidisc $\Delta^2$; indeed, note that $df(z)=(1,2z)$, so $df(0)=(1,0)$ 
thereby $\left\Vert df(0)  \right \Vert_{l^\infty}=1$ verifying all other conditions 
except for the $\mathbb{R}$-extremal condition which we 
dropped. The following extension also due to Mazet, which completely drops 
restrictions on the dimension of the source of the
map is no less relevant in the sequel. We shall generalize this to some extent below for our needs.

\begin{thm}\label{Mazet-Schwarz_1} \rm(Theorem 3.7 of \cite{Mazet_Rextrm})
    Let $E,F$ be Banach spaces with their unit balls denoted $B_E, B_F$ respectively. Let $f: B_E \to B_F$ be any holomorphic map. Then we have the following:
    \begin{enumerate}
        \item[i)] If there exists $u$ of norm 1 in $E$ such that $f^{\prime}(0) \cdot u$ is $\mathbb{R}$-extremal in $\overline{B_F}$ then $f(0)=0$.
        \item[ii)] If, for every $u$ of norm $1$ in $E$, $f^{\prime}(0) \cdot u$ is $\mathbb{C}$-extremal in $B_F$ then $f$ is a polynomial of degree at most 2.
        \item[iii)] If the two previous hypotheses are simultaneously satisfied, then $f$ is a linear map, $f(u)=$ $f^{\prime}(0) \cdot u$.
\end{enumerate}
\end{thm}

\noindent An analogue of the above for the non-convex case
can be derived from some standard ideas from chapter $11$ of \cite{Jrncki_invrnt_dst}, which we do now. This will quickly lead to a generalization of Vesentini's theorem for the non-convex case, which will however be derived a bit later after theorem \ref{alpha(L)=Z} where it is more relevant. \\

\begin{thm} \label{extr-Schw-lem-holextrbdy}
Let $D$ be a balanced pseudoconvex (not necessarily bounded) domain in $\mathbb{C}^N$ 
and  $\varphi : \Delta \to D$ be a holomorphic map. If 
$\varphi(0) = 0$ and $\varphi'(0)$ happens to be a holomorphic extreme 
boundary point in $D$, then $\varphi$ is a linear map: 
$\varphi(t) = \varphi'(0)t$. More generally, if $G$ is a bounded balanced 
pseudoconvex domain in $\mathbb{C}^M$ and $F: G \to D$ be a holomorphic 
map with the property that $F(0) = 0$ and for all $u \in \partial G$, $DF_{|_0}(u)$ 
is a holomorphic extreme boundary point of $D$, then $F$ is linear.
\end{thm}
\begin{proof}
We start with the holomorphic map $\varphi: \Delta \to D$  
as in the first part of the statement i.e., satisfying $\varphi(0) = 0$ with $\varphi'(0) = p$ being a 
holomorphic extreme boundary point of $D$.  As $\varphi(0) = 0$, 
we can write $\varphi(t) =t\tilde{\varphi}(t)$, where 
$\tilde{\varphi} \in \mathcal{O}(\Delta,\mathbb{C}^N)$. 
Since $\varphi$ maps $\Delta$ into $D$, we have 
\[
h(\varphi(t)) = |t|h \circ \tilde{\varphi}(t) < 1.
\]
Because the above inequality holds for all $t \in \Delta$, 
we obtain a similar
inequality (but with the possibility of 
equality): $h \circ \tilde{\varphi}(t) \leq 1$ for all 
$t \in \Delta$.
Note that 
\[
\tilde{\varphi}(0) = \varphi'(0) = p.
\]
This implies that  $h\circ \tilde{\varphi}(0) = h(p) = 1$.
By the maximum principle applied to the subharmonic function 
$h \circ 
\tilde{\varphi}$, we get that $h \circ \tilde{\varphi} \equiv 1$. 
This shows that ${\rm image}(\tilde{\varphi}) \subset \partial D$. By the hypothesis that $p \in \partial D$ is a holomorphic extreme 
boundary point of $D$, we obtain that $\tilde{\varphi} \equiv p$ 
and hence we conclude that $\varphi(t) = tp$ for all 
$t \in \Delta$.\\

We shall now use the above
to prove the general case expressed in-terms of the map $F$ as in the 
statement of the theorem. Let $p$ be an arbitrary  
point in $\partial G$. Consider the 
holomorphic map $\varphi: \Delta \to D$ defined by $\varphi(t) = F(tp)$. 
Note that $\varphi(0) = 0$ and $\varphi'(0) = DF_{|_0}(p)$ is a 
holomorphic extreme boundary point of $D$. By the 
special case $G=\Delta$ already dealt with, we obtain that 
$\varphi(t) = tDF_{|_0}(p)$. As $p$ was arbitrarily chosen, this means that $F(tp)=DF_{|_0}(tp)$
for all $p$ in the boundary of the {\it bounded} balanced domain $G$
(and all $t \in \Delta$). It follows thereby that 
$F(z) = DF_{|_0}(z)$, finishing the proof that $F$ is a linear map.
\end{proof}

\noindent We now get to the retract $Z$ of 
our product domain $D=B_1 \times B_2$ and denote
its tangent space at the origin $T_0Z$  by $L$, which ofcourse is a complex linear 
subspace of $\mathbb{C}^{N_1+N_2}$; as we have already mentioned, we shall not be talking 
of trivial retracts and hence the possible dimensions of $L$ (over $\mathbb{C}$) 
are $1,2,\ldots,N_1+N_2-1$ which are also the possible dimensions of $Z$ as a complex 
submanifold. Which of these dimensions are attained and what would be their forms then, are 
matters more subtle than the general features mentioned at the outset and as we shall see, 
actually depend on specific features of the factor domains $B_1,B_2$. However, just the 
assumption of $\mathbb{R}$-extremality on the boundaries of $B_1,B_2$ will enable us to say 
far more about $Z$ than just the already mentioned fact at the outset that it is a contractible 
submanifold of $D$. A first reason for this (as to why we shall be able to say more at all) is 
that we are in the setting of a complex Euclidean 
space which renders all the tangent spaces of $D$ to be canonically identified and for our case, 
parts of tangent spaces are contained within the domain $D$ itself; more precisely and in 
particular, there is a part of $L$ that is contained within our domain namely, the 
intercept $D_L=L \cap D$. This means that $\rho$ can be applied to this intercept (of the tangent space to $Z$) itself, which naturally gives rise 
to the question of the relationship between $\rho(D_L)$ and $Z$. As to this, the version 
of the identity principle for complex analytic sets (as in section 5.6 of the first chapter of the book \cite{Chirka} by Chirka, for instance) applies to ascertain that 
$\rho(D_L)=Z$ where $D_L=L \cap D$, provided we know apriori that $\rho(D_L)$ is a complex submanifold -- obviously, 
$\rho(D_L) \subset Z$ always; the role of the identity principle is to ensure the equality $\rho(D_L)=Z$ which will be used several times in the sequel. Actually we can do better than doing this apriori verification that $\rho(D_L)$ is a 
submanifold of $Z$ in each of the multiple cases that we shall be considering for the analysis 
of $Z$; while splitting up into multiple cases here
would be seen to naturally arise owing to the difference in the forms of $Z$ that cases result 
in, the result $\rho(D_L)=Z$ does not depend on any such case we may be in and can be shown to 
hold rather generally. As this also fits in well with our plan of first noting general facts 
following from already known results, we now proceed to establish a general result about 
this as in theorem \ref{alpha(L)=Z} below, which holds for any bounded balanced  pseudoconvex 
domain $D$ -- no $\mathbb{R}$-extremal hypotheses is required. 
While the proof of this will be done by borrowing an idea from \cite{Abate_Isometry}, implementation 
of this idea requires preparatory work which itself yields another fundamental fact namely, 
that $D_L=D \cap L$ where $L=T_0Z$, itself is a retract of $D$, though not every linear subspace 
of $D$ can be its retract, unless $D$ is biholomorphic to the $l^2$-ball. Indeed, $D_L$ is 
a {\it linear retract} of $D$ in the sense of definition \ref{Linear-retract-defn}.\\

\noindent We can now progress towards our 
aforementioned general result on fully attaining $Z$ as the image 
of $L$ under $\rho$, the proof of which requires the subsequent lemmas 
which in turn follow fairly directly from the ideas in the 
pair of remarks \ref{Linear-remark-1} and \ref{Linear-remark-2},
allowing us to skip details here.

\begin{lem} \label{Linearlyretracting}
Let $D$ be any bounded balanced convex domain in $\mathbb{C}^N$, $Z$ a
retract passing through the origin with $\rho: D \to Z$, a retraction
map. Let $L$ denote the tangent space to $Z$ at the origin and $\mu$ the Minkowski functional of $D$.  Then, $D\rho(0)$ is a norm one projection of $(\mathbb{C}^N, \mu)$ onto $(L,\mu)$, thus mapping $D$ onto $D_L := D \cap L$ and showing in particular that $D_L$ is a linear retract of $D$.  
\end{lem}

\noindent If we relax the convexity condition on $D$ to 
pseudoconvexity, it is still true that the Minkowski functional of $D$ 
and the infinitesimal Kobayashi metric of $D$ agree at the 
origin (proposition 3.5.3 of \cite{Jrncki_invrnt_dst} wherein counterexamples 
to show the failure of the equality when pseudoconvexity is dropped are also given), 
thereby implying that the Kobayashi indicatrix of $D$ at the origin is $D$ 
itself; likewise the Kobayashi indicatrix of $D_L$ now also a  balanced 
pseudoconvex domain in $L$ is also (a copy of) $D_L$ itself. The general 
property of contractivity of the Kobayashi metric under holomorphic mappings 
when applied to $\rho$ viewed as a map from $D$ to
$D_L$, then yields that $D\rho(0)(D) \subset D_L$. But then as $D\rho(0)$ 
fixes $L \supset D_L$ pointwise, it follows that 
$D\rho(0)(D) = D_L$, showing that the above lemma holds in this more general 
case which we record separately for some convenience.

\begin{lem}\label{Dalpha}
Let $D$ be any balanced pseudoconvex domain in $\mathbb{C}^N$, $Z$ a
retract passing through the origin with $\rho: D \to Z$, a retraction
map. Let $L$ denote the tangent space to $Z$ at the origin and $D_L := D \cap L$. Then, the linear map $D\rho(0)$ is a projection which retracts
$D$ onto $D_L$.
\end{lem}
\noindent As a corollary to this lemma, we 
give a proof of corollary \ref{No_1dim_Retrct} not
because this proof avoids geodesics but
because it brings out the reason directly. 
\begin{cor} \label{notcvx-no-retract}
Let $D$ be a balanced pseudoconvex domain in $\mathbb{C}^N$. If there exists a point $p \in \partial D$ such that $D$ is not convex at $p$, then there does not exist any one-dimensional retract of $D$ passing through the origin in the direction of $p$. 
\end{cor}
\begin{proof}
Assume to obtain a contradiction that there exists a one-dimensional retract $Z$ of $D$ passing through the origin in the direction of $p$. 
Let $\rho$ be any one of the retraction maps on $D$ such that ${\rm image}(\rho) = Z$. By the previous lemma \ref{Dalpha}, the linear projection map $\pi := D\rho_{|_0} \;:\; \mathbb{C}^N \to \mathbb{C}^N$ maps $D$ 
onto $D_L$, where $L = T_0Z$. Consider the affine subspace 
$K := p + \ker(\pi)$ of $\mathbb{C}^N$. Note that $K$ is a complex hyperplane at $p$. Since $D$ is not convex at $p$, there exists a point $q \in K$ such that $q \in D \cap K$. 
But $\pi(q) = p$, which contradicts the fact that $\pi$ maps $D$ into $D$. 
\end{proof}
\noindent We return to the preceding lemma for a little clarification that saves some confusion once and for all. In case $Z$ is already a linear subspace of $D$ to begin with, the above lemma \ref{Dalpha},
still applies to show that a linear subspace of pseudoconvex bounded balanced domain $D$, 
which is apriori only known to be attainable as the image of a holomorphic (i.e., possibly 
non-linear) retraction map on $D$, can indeed be obtained as the image of a linear retraction 
map as well i.e., it is valid to say that $D_L$ as in the above lemma is a linear retract 
of $D$ without any further qualifications. We are now fully prepared to prove the 
highlighted general result about the equality $\rho(D_L)=Z$ whose statement is 
formulated with slightly different notations than the one that we had adopted for our 
specific domain $D$. This is indeed the content of theorem \ref{Graph}, whose statement we lay down
again here for convenience, and then prove it based on an idea in \cite{Abate_Isometry}.

\begin{thm}\label{alpha(L)=Z}
Let $D$ be any bounded balanced pseudoconvex domain in $\mathbb{C}^N$, $Z$ a
retract passing through the origin with $\rho: D \to Z$, a retraction
map. Let $L$ denote the tangent space to $Z$ at the origin. Then,
$\rho_{\vert D_L}$ is a biholomorphism mapping $D_L=D \cap L$ onto $Z$ and upto a linear change of coordinates, $Z$ is the graph of a holomorphic map over $D_L$ taking values in $M:=\ker \left( D \rho(0) \right)$.
\end{thm}

\begin{proof}
We first apply the foregoing lemma which realizes $D_L=D \cap L$ as a 
linear retract of $D$ by the linear projection
$\pi_L:=D \rho(0)$.
We then consider the map $F:=\pi_L \circ \rho_{\vert_{D_L}}$, which is a map 
from $D_L$ into itself. Note that $F$ fixes the origin and that
$DF(0)$ is just the identity map as $D\rho(0)$ is the identity map 
on $L$ (since $L$ is a tangent space to $Z$ which is fixed pointwise by $\rho$) and $\pi_L$ 
is actually a linear retraction onto $L$ and so in particular is the identity on $L$.
As $D_L$ is a bounded subdomain of the complex linear space $L$ containing the origin which 
is fixed by $F$, we may very well apply the
classical uniqueness theorem of H. Cartan to deduce that $F$ itself is 
just the identity map. It follows 
that $\rho$ restricted to $D_L$ and $\pi_L$ restricted to $Z$ are inverses of each other 
and being holomorphic, finishes the proof of the first statement. It 
only remains to observe that as the inverse of $\rho_{\vert_{D_L}}$ actually 
happens to be the linear projection $\pi_L$ onto the linear subspace $D_L$ along its kernel $M=\ker \left( D \rho(0) \right)$, 
this means that $Z=\rho(D_L)$ is realized essentially as the graph of a 
(single) holomorphic map over $D_L$ taking values in $M=\ker \left( D \rho(0) \right)$.
\end{proof}

\noindent While the above theorem allows retracts  to be non-linear and 
we shall indeed furnish several examples of non-product (balanced, pseudoconvex) domains possessing non-linear
retracts through origin in the form of graphs, we recall that such 
a phenomenon is precluded by Vesentini's theorem in the convex case when the boundary is $\mathbb{C}$-extremal.
That such extremal boundary conditions are indeed necessary will be shown in course of 
establishing a converse to Vesentini's theorem later. However, the hypothesis that the entire boundary 
be holomorphically extreme as in the extension of Vesentini's theorem given by proposition  1.4 in \cite{JJ}
is inconvenient for discussing interesting intermediate examples in the sequel. 
This is addressed by proposition \ref{JJ-improved} whose
proof we now lay down. Though 
this can be done by an imitation of the
proof of the aforementioned proposition in \cite{JJ}, we prefer to include some details as this is a
result about retracts and further, an occassion
will arise later in this article wherein the 
aforementioned relaxation will actually be used. 
It also provides an answer towards the natural query as to whether an extension of Vesentini's theorem 
holds if we add only an extremality hypothesis just as much as is needed about the boundary,
but none about convexity
in the setting of the above theorem \ref{alpha(L)=Z}. \\

\noindent \textit{Proof of proposition \ref{JJ-improved}}: This is one of the few occasions in this article wherein we discuss $2$ proofs, with a one sentence proof following this, to show a possible advantage of circumventing geodesics at times, as mentioned in the introduction.
Choose any holomorphic retraction map $\rho : D \to D$ 
with $\rho(D) = Z$. We may observe by the above theorem \ref{alpha(L)=Z},
that the restriction $\rho_{|_{D_L}}$ is a holomorphic map on $D_L$ whose image 
is precisely equal to $Z$; but we shall not need this for the rest of this proof here.
Recall that $Z$ is a connected complex submanifold and 
the tangent space to $Z$ at the origin, is denoted by
$L$; so, $L=T_0Z = \{v \in \mathbb{C}^N: \rho'_{|_{D_L}}(0)v = v\}$. 
Let $v \in D_L, v \neq 0$. By our assumption on $\partial D_L$, it follows that the mapping 
\[
    \phi: \Delta \ni \lambda \longmapsto \lambda v/h_{D_L}(v) \in D_L .
\]
is the unique (modulo $Aut(\Delta)$) $K_{D_L}$-geodesic for $(0,v/h_{D_L}(v))$;
here of-course $h_{D_L}$ denotes the Minkowski functional of the balanced
domain $D_L \subset L$. Since $D_L$ is a linear subspace 
of $D$, $h_{D_L}(v) = h_{D}(v)$. So $\phi$ 
is the unique (modulo $Aut(\Delta)$) $K_{D}$-geodesic for the datum $(0,v/h_{D_L}(v))$. Note that 
\[ 
K_D \left( \rho_{|_{D_L}} \circ \phi(0); (\rho_{|_{D_L}} \circ \phi)'(0) \right) 
= K_D \left(0, v/h_{D_L}(v) \right) = K_D \left( \phi(0);\phi'(0) \right).
\]
Hence $\rho_{|_{D_L}} \circ \phi$ is a $K_D$-geodesic for $(0,v/h_{D_L}(v))$.
Because of our hypothesis, we have that $v/h_{D_L}(v)$ is a holomorphic extreme point of $D$.
By uniqueness of complex geodesics in $D$ (exposited as proposition 11.3.5 in \cite{Jrncki_invrnt_dst}), 
it follows that $\rho_{|_{D_L}} \circ \phi = \phi \circ \psi$, where $\psi \in {\rm Aut}(\Delta)$. 
Substituting $\psi^{-1}(h_{D_L}(v))$ which is a point in $ \Delta$, 
into the above equation, we see that $v$ lies in the image of $\rho_{|_{D_L}}$. Hence $D_L \subset Z$. 
Now $T_0Z$ and $Z$ have the same dimension. Since $Z$ is connected, it 
follows from the principle of analytic continuation that $Z = D_L$. 
Hence $Z$ is linear. \qed\\

An alternative short proof of this proposition is obtained simply by noting that
it is a direct corollary of theorem \ref{extr-Schw-lem-holextrbdy}; specifically, by
applying the second part of that theorem to 
any one of the holomorphic retraction maps
$\rho$ restricted to $D_L$.
\qed\\

\noindent Some remarks about the foregoing pair of results are in order. \noindent Firstly, a remark about the last proof: note that   
the arguments up-till the conclusion that ${\rm image}(\tilde{\varphi}) \subset \partial D$, goes through in the 
setting of above theorem (same as theorem \ref{Graph}), whose statement 
we may therefore refine as follows. If we call the largest complex
analytic variety within $\partial D$ containing a point $p \in \partial D$ as the `holomorphic face' of $\partial D$ at $p$,
then we may say that the mapping (denote it by $F$ say), whose graph is the retract $Z$,
as in theorem \ref{Graph}, takes values within the union 
of such holomorphic faces at boundary points $p$, for 
$p$ varying in $\partial D_L = \partial D \cap L$. As we already know
that the range of the map $F$ is contained in $\ker(D \rho(0))$,
the above helps refine this information about ${\rm range}(F)$. We thereby get another proof of proposition \ref{JJ-improved} as an immediate corollary of the foregoing theorem. This owes to the assumption that the holomorphic face on $\partial D_L$ becoming a singleton in the situation of that proposition.
\begin{rem}
To {\it indicate} just a few concrete examples to illustrate
where proposition \ref{JJ-improved} may be applied (but not the other results in this context {\it directly} at-least),
we first deal with  domains
where the holomorphic extremality condition is 
neither violated maximally nor minimally; examples for `maximal' violation is also considered but in 
section \ref{Irreducibity}. Among the simplest of domains 
satisfying the hypotheses of proposition \ref{JJ-improved} in this intermediate fashion, we
%
consider the domain $\Omega_1 := \Delta^2 \cap D_q$, where 
$D_q := \{(z,w) \in \mathbb{C}^2: |z|^q + |w|^q < 2^q\}$ for $0 < q < 1$ is a dilated copy of the
$\ell_q$-ball. Note that $\Omega_1$ is a bounded balanced pseudoconvex domain 
and $\partial \Omega_1$ fails to be holomorphically extreme; though of-course, there exists 
an open piece of $\partial \Omega_1$, which is holomorphically extreme. A combined application of theorem 
\ref{converse-of-Vesentini}, theorem \ref{Graph} and corollary \ref{No_1dim_Retrct}, proves that every retract passing through origin is either a 
linear subspace of $\Omega_1$ or a
graph of a non-linear holomorphic function over some linear 
subspace of $\Omega_1$.\\

\noindent The above example was non-convex; so, next we consider a convex domain just by taking
its limiting case i.e., letting $q \to 1$ in the above namely, 
$\Omega_2 := \Delta^2 \cap D_1$, where $D_1 := \{(z,w) \in \mathbb{C}^2: |z| + |w| < 3/2\}$ 
is some concrete dilated copy of the $\ell_1$-ball. Note that $\partial \Omega_2$ is not strictly convex, indeed not 
holomorphically extreme. Though no little piece of the boundary is $\mathbb{R}$-extremal, there exists 
an open piece of the boundary which satisfies the 
$\mathbb{C}$-extremality condition. By proposition \ref{JJ-improved} whose applications we are after, we obtain that 
every retract through the origin is either a linear subspace 
of $\Omega_2$ or a graph of a non-linear holomorphic function 
over some linear subspace of $\Omega_2$.\\
\end{rem}

\begin{rem}
While it is clear that without any extremality assumptions, linearity of retracts is unreasonable to be expected
(polydisc being the simplest `counterexample') and, holomorphic extremality 
seems the best possible condition for obtaining such a strong conclusion as in the proposition 
\ref{JJ-improved} just discussed,
it is restrictive as well. For, there are plenty of balanced pseudoconvex domains whose boundaries fail to
satisfy such a strong extremality condition. To give just one example whose boundary does not contain any flat 
pieces, which indeed satisfies the weaker $\mathbb{C}$-extremality but not the holomorphic extermality condition,
consider the bounded balanced pseudoconvex domain obtained by intersecting the domain
$H=\{(z,w) \in \mathbb{C}^2 \; : \; \vert zw \vert <1/2 \}$ with the standard Euclidean ball $\mathbb{B}$
in $\mathbb{C}^2$. This indicates how we may construct plenty of such examples
which also we shall deal later in the sequel at appropriate places; in-particular, subsequent to the 
next remark. 
\end{rem}

\begin{rem}\label{diffeo-holo}
In the setting of the foregoing proposition \ref{JJ-improved},
being a graph, $Z$ is easily seen to be biholomorphic to $D_L$ as complex 
manifolds, in the 
setting of the above theorem. Thereby, it immediately follows that any pair of retracts 
through the origin
in $D$, of the same dimension  are diffeomorphic. Indeed, being star-shaped domains 
in $L$, they 
are diffeomorphic 
to $L$ -- no pseudoconvexity is required here; the general fact being used here is 
that any star-shaped 
domain in $\mathbb{R}^N$ is $C^\infty$-diffeomorphic to $\mathbb{R}^N$ (exercise $7$ of chapter $8$ of
\cite{Brocker-Janich-difftop-bk}). It is therefore of interest then (as is 
in Several Complex Variables), to note that 
such a pair of retracts (through the origin of the
same balanced domain $D$ and) of the same dimension despite being mutually diffeomorphic need not be biholomorphic
(indeed, they are rarely so). To see this first by what is perhaps the 
simplest example, consider
$D=\Delta^2 \times \mathbb{B}^2$ as a subdomain of $\mathbb{C}^4$; let $Z_1$ denote the projection of $D$ onto the span 
of the first pair of standard 
basis vectors of $\mathbb{C}^4$ and $Z_2$, the projection of $D$ onto the span of the 
remaining last pair of standard 
basis vectors of $\mathbb{C}^4$. By the well-known biholomorphic inequivalence of the ball and polydisc, 
$Z_1 \not \simeq Z_2$ as
$Z_1,Z_2$ are just copies of the
bidisc $\Delta^2$ and the ball $\mathbb{B}^2$ respectively.
To give another example, again as simple as possible but this time of a ball $D$ which is not 
biholomorphic to a product domain, in-fact with boundary both smooth as well as holomorphically extreme and,
which contains pairs of {\it linear}
retracts (thereby through the origin), of the same dimension that are biholomorphically inequivalent, consider
\[
D = \{ z \in \mathbb{C}^3 \; : \; \vert z_1\vert^2 + \vert z_2 \vert^2 + \vert z_3 \vert^4 <1 \}.
\]
Note that the projection $Z_1$ of $D$ onto the $(z_1,z_2)$-plane is 
$\mathbb{B}^2 \subset \mathbb{C}^2 \hookrightarrow \mathbb{C}^3$; whereas the projection $Z_2$ of
$D$ onto the $(z_2,z_3)$-plane is a copy of the `egg domain'. Specifically, 
\[
Z_2 \simeq \{(z_2,z_3) \in \mathbb{C}^2 \; : \; \vert z_2 \vert^2 + \vert z_3 \vert^4 <1 \}
\]
which is not biholomorphically equivalent to $\mathbb{B}^2$; thus, $Z_1 \not \simeq Z_2$. 
A noteworthy little observation that shall be discussed later is that while all linear projections
are retraction maps on $\mathbb{C}^N$, their restrictions to balls $D$ 
(centered at the origin, with respect to some norm on $\mathbb{C}^N$)
need not induce retraction mappings on $D$; in-fact, most often they do 
not. However, orthogonal projections 
onto spans of coordinate vectors, in the foregoing as well as forthcoming examples 
do (so, $Z_1,Z_2$ are indeed retracts of $D$ in the 
above examples).
Further such examples of irreducible smoothly bounded  balanced convex domains
in any dimension $N \geq 3$, can be constructed by considering the domain 
\[
\mathcal{E} := \{ z \in \mathbb{C}^N \; : \; \vert z_1 \vert^{p_1} + \ldots + \vert z_N \vert^{p_N} <1 \},
\]
where the exponents $p_j$'s are taken to be bigger than one and distinct (we disallow $N=2$ for that case 
is exceptional owing to the fact that 
all non-trivial retracts of convex domains in this case are at-most of dimension one and
any pair of one-dimensional retracts are biholomorphic to $\Delta$). Then, for every $1 <k < N$, 
there exists a pair of linear retracts of $\mathcal{E}$ of the same dimension $k$ which are diffeomorphic but
biholomorphically inequivalent. This is easily seen for instance, by projecting $\mathcal{E}$ onto the span of 
different subsets of $k$-many of the standard basis vectors of $\mathbb{C}^N$ following the foregoing 
concrete examples.
\end{rem}

\noindent The converse of above theorem \ref{alpha(L)=Z} is also true.
\begin{prop}\label{Graph_retract}
Let $D$ be a bounded balanced pseudoconvex domain and $D_L: = D \cap L$
be a linear retract of $D$. Let $\varphi$ be any holomorphic map from $D_L$ to $M$, where
$M$ is a linear subspace 
complementary to $L$ i.e.,
$\mathbb{C}^N= L \oplus M$, is a holomorphic map with the property that $(w,\varphi(w)) \in D$, whenever $w \in D_L$.
Then ${\rm Graph}(\varphi)$, the graph of $\varphi$ over $D_L$ in $D$, is a retract of $D$.
\end{prop}
\begin{proof} Since $D_L$ is a linear retract of $D$, there exists
a linear projection map $P$ such that $P(D) = D_L$. By 
idempotency of $P$, we can write $\mathbb{C}^N = L \oplus M$, where
$M = \ker(P)$ and $L = {\rm image}(P)$. Consider the map $\psi: D \to D$
defined by $\psi(z) = (P(z),\varphi(P(z)))$; this 
can be verified to be idempotent by a direct computation, thereby showing that $\psi$ 
is a holomorphic retraction from 
$D$ onto ${\rm Graph}(\varphi)$ over $D_L$, finishing the proof.
\end{proof}
\noindent 

\noindent It must be noted that the above does not mean that non-linear retracts 
automatically come into existence
by taking the graph of any holomorphic 
function over any given linear retract of $D$ for, the graph can 
fail to be contained within $D$; that this is the only obstruction to generating
non-linear retracts by graphs is the content of the above proposition.
However, the obstruction is substantial: for instance, even in the 
simple case of $\mathbb{B}$, the standard unit ball with respect to the usual $l^2$-norm 
on $\mathbb{C}^N$, every retract through the origin is just a linear subspace and in 
fact, all retracts of 
$\mathbb{B}$ are affine-linear, as was first established by Rudin who then exposited it 
in chapter $8$ of his book \cite{Rud_book_Fn_theory_unitball} as well. The proof of this, at least as 
in \cite{Rud_book_Fn_theory_unitball}, uses the special feature of the $l^2$-norm that it comes from 
an inner product; however, this fact about the linearity of retracts through the 
origin is actually true in much greater generality as stated in proposition \ref{JJ-improved} .\\

\noindent To not be amiss to 
discuss concrete examples and the simplest ones at that, 
recall that the first standard choices for $D$, namely the $l^p$-balls
was already mentioned in the introduction, specifically theorem \ref{Lindenstrauss -- Tzafiri}.
Recall also that we had promised an example to show the failure of this theorem for $p=\infty$ i.e.,
when $D$ is the polydisc.\\

 \textit{Example}: The linear map 
$T: \mathbb{C}^3 \longrightarrow \mathbb{C}^3$ defined 
by $T(z_1,z_2,z_3) = \left(z_1,z_2,(z_1 + z_2)/2 \right)$ is a projection of 
norm $1$ in $(\mathbb{C}^3, \|\cdot\|_{\infty})$. But its image is not 
spanned by mutually disjoint 
supported vectors. To see this, assume to obtain a contradiction that 
there exists a basis $\{v_1,v_2\}$ whose supports 
are disjoint. Write $v_1 = (\zeta_1,\zeta_2,\zeta_3), ~ v_2 = (\eta_1,\eta_2,\eta_3)$. Consider 
the case $Supp(v_1) = \{1,2\}$ and $Supp(v_2) = \{3\}$. Then $\zeta_1,\zeta_2 \neq 0, \zeta_3 = 0$ 
and $\eta_1,\eta_2 = 0, \eta_3 \neq 0$. Note that $image(T)$ is the 
subspace $w_1+w_2 -2w_3 = 0$. Since $v_1,v_2 \in image(T)$, we 
have $\zeta_1 + \zeta_2 = 2\zeta_3$ and $\eta_1 + \eta_2 = 2\eta_3$ which 
leads to $\zeta_1 + \zeta_2 = 0$ 
and $\eta_3 =0$. It then follows that $v_2 = 0$, a contradiction. Similar arguments, show 
that the other cases here also do not arise.\qed \\

\noindent We now finish off the discussion of this $p=\infty$
by proving the characterization of its linear retracts as stated in proposition \ref{Polydisk}.\\

{\textit {Proof of proposition \ref{Polydisk}}}:  Assume $L$ is a linear retract
of $\Delta^N$. We can write $L = Y \cap \Delta^N$, for some linear 
subspace $Y \subset \mathbb{C}^N$. By theorem-3 in \cite{Heath_Suff}, there 
exists $ J = \{j_1,j_2,\ldots,j_r\}$ and holomorphic 
functions $F_i : M \to \Delta,~ \ 1\leq i \leq N$ 
where $M = \{z \in \Delta^N: z_j = 0~ \text{if} ~j \notin J\}$ such that   $F_i(z) = z_i$ if $i \in J$ and  
\[
L = \Delta^N \cap Y = \{(F_1(z),\ldots,F_N(z)): z \in M\}.
\]
Observe that the $F_j$'s are holomorphic functions 
from $\Delta^N$ to $\Delta$ and $F_j(z) = G_j(z_{j_1},\ldots,z_{j_r})$ for some holomorphic function $G_j$, which is bounded by $1$. Without loss of generality assume that each $F_i$'s are linear, since $F(z) = (F_1(z),\ldots,F_N(z))$ restricted to $L$ is linear. So, we get
\[ F_i(z) = \begin{cases}  z_i,  \quad \qquad \qquad \qquad \qquad \text{if}~ i \in J \\ c_{i,j_1}z_{j_{1}} + \ldots + c_{i,j_r}z_{j_r}, \quad \text{if}~ i \notin J, \\ 
 \end{cases} \] 
for some scalars $c_{i,j},~ j \in J$.\\

\noindent   Fix $m \notin J$ and for $1 \leq k \leq N$, define 
\[ 
w_{k} := 
\begin{cases}
\frac{|c_{m,j_i}|}{c_{m,j_i}}, \ \text{if}\  k = j_i  ~\text{for some}~i \ \text{and}\  c_{m,j_i} \neq 0 \\
0, \quad \quad \text{otherwise} 
\end{cases}.
\]
We know that $|F_m(z)| \leq 1$ for all $z \in \Delta^N$. If $w := (w_1,\ldots,w_N)$ then  
\[
|F_m(w)|  = \sum_{j \in J}|c_{m,j}| \leq 1 .
\]
Hence we get $\sum_{j \in J} |c_{m,j}| \leq 1$.
Recall that $M$ is spanned by $\{e_{j_1},\ldots,e_{j_r}\}$. Since $F$
is determined by the variables $z_{j_1},\ldots,z_{j_r}$, we get $Y$ is spanned by $\{F(e_{j_1}),\ldots,F(e_{j_r})\}$.  
By definition of $F_i$, we get 
\[ F_i(e_{j_k}) = \begin{cases}  \delta_{i,j_k},  \quad \qquad  \text{if}~ i \in J \\ c_{i,j_k}, \qquad  \text{if}~ i \notin J, \\ 
 \end{cases} \] 
 where $\delta_{i,j}$ denotes the Kronecker delta function.\\
\noindent By definition of $F$, we  get
\[ 
F(e_{j_k}) = (F_1(e_{j_k}), \ldots,F_N(e_{j_k})) = e_{j_k} + \sum_{i \notin J} c_{i,j_k} e_i.
\]
This shows that $Y = span\{v_k\}_{k=1}^r$ where $ v_k = e_{j_k} + \sum_{i \notin J} c_{i,j_k} e_i  $.\\

\noindent To attain the converse implication, assume that $L = Y \cap \Delta^N$ for some subspace $Y = span\{v_1,\ldots,v_r\}$ and also satisfies the conditions (i) and (ii). Define $J := \{j_1,\ldots,j_r\}$ and  
\[ F_i(z) = \begin{cases}  z_i,  \quad \qquad \qquad \qquad \qquad \text{if}~ i \in J \\ c_{i,j_1}z_{j_{1}} + \ldots + c_{i,j_r}z_{j_r}, \quad \text{if}~ i \notin J, \\ 
 \end{cases} \] 
Note that $F_i(z)$ is a holomorphic function from $\Delta^N$ to $\Delta$, because by condition (ii) of the theorem.
Consider the map $F$ defined in terms of these $F_i$'s by $F(z) := (F_1(z),\ldots,F_N(z))$. 
It can be readily checked that $F$ is a holomorphic retraction map on $\Delta^N$ whose image is
$span\{v_k\}_{k=1}^r \cap \Delta^N$, where $v_k = e_{j_k} + \sum_{i \notin J} c_{i,j_k}e_i$, which completes the proof.\\

\noindent As noted above, there exists a linear retract of $\Delta^N$ which is not spanned by mutually disjoint supported vectors. The converse implication however, is true and can now be derived as a corollary.\\

\begin{cor}
    If $Y$ is a linear subspace of $\mathbb{C}^N$ spanned by disjointly supported vectors $\{v_1,\ldots,v_r\}$, then $Z = Y \cap \Delta^N$ is a linear retract of $\Delta^N$.
\end{cor}

\begin{proof}
Without loss of generality assume that $\|v_i\|_{\infty} = 1$ for $i = 1,\dots,r$. 
There exists $J = \{j_1,\ldots,j_r\}$ such that $|\pi_{j_i}(v_i)| = 1$, where $\pi_i$ 
denotes the projection to the $i^{th}$ component. After multiplying an unimodular constant by $v_i$, we may assume that $\pi_{j_i}(v_i) = 1$. So, we can write  $v_k = e_{j_k} + w_k$. Note that $w_k$'s are disjointly supported vectors.
If $m \in \{1,2,\ldots,N\} \setminus J$, then we define  
$c_{m,{j_k}} := \pi_m(w_k)$. Since $\text{supp}(w_i)$'s are disjoint, we 
get that $\pi_m(w_k)\pi_m(w_l) = 0$ for $k \neq l$. Hence $\sum_{j \in J}|c_{m,j}| \leq 1$. By proposition \ref{Polydisk}, $Y = span\{v_1,\ldots,v_k\}$ is 
a linear retract of $\Delta^N$.
\end{proof}

\begin{prop}\label{lower_l_infty}
    Let $W$ be a linear subspace of $\mathbb{C}^N$. Then $W \cap 
\Delta^N$ is a $k$-dimensional linear retract of $\Delta^N$ if 
and only if $W$ is isometric to $(\mathbb{C}^k,\|\cdot\|_{\infty})$.
\end{prop}
\begin{proof}
One part of the implication will need tools that will be discussed only later below. Namely: if $W$ is isometric to $(\mathbb{C}^k,\|\cdot\|_{\infty})$, then 
by proposition \ref{lower_l_p} below, $W \cap \Delta^N$ is a $k$-dimensional linear 
retract of $\Delta^N$. We shall therefore now deal with the converse here; to do this, assume that $W$ is a $k$-dimensional linear 
retract of $\Delta^N$. By proposition  \ref{Polydisk}, there exists $J =\{j_1,
\ldots,j_k\} \subset \{1,2,\ldots,N\}$ such that $L$ is spanned 
by $v_1,\ldots,v_k$, where each $v_r$ is of the form
\[
    v_r = \left(\sum_{i \notin J}c_{i,j_r}e_i\right) +e_{j_r}
\]
and for each $m \in \{1,2,\ldots,N\} \setminus J$ satisfies 
$\sum_{j \in J}|c_{m,j}| \leq 1$. Without loss of generality 
assume that $J = \{j_1,\ldots,j_k\} = \{1,2,\ldots,k\}$. Consider the 
linear map $S: W \to \mathbb{C}^k$ defined by $S(v_i) = e_i$ for 
$i = 1,2,\ldots,k$ and extended to other vectors of $W$ by 
linearity. Now we claim that $S$ is an isometry. To prove this 
claim, write each vector $v \in W$ as $v = \lambda_1v_1 + \ldots +\lambda_kv_k
$ for some $\lambda_i \in \mathbb{C}$. Note that $S(v) = 
(\lambda_1,\ldots,\lambda_k)$. Choose $i_0 \in \{1,2,\ldots,k\}$ 
such that $\|(\lambda_1,\ldots,\lambda_k)\|_{\infty} = |\lambda_{i_0}|$. 
Hence $\|S(v)\| = |\lambda_{i_0}|$. Using the expression of $v_r$,
we can expand $v$ as
\[
    v = \sum_{l=1}^k\sum_{i=k+1}^N\lambda_l(c_{i,l}e_i + e_l)
\]
Observe that the first $k$ components of $v$ are $\lambda_1,\ldots,
\lambda_k$ and for each $i = k+1~ \text{to}~ N$ , $i^{\text{th}}$ component of 
$v$ is $\sum_{l=1}^k\lambda_lc_{i,l}$. Note that for each $i = k+1$ to $N$, we have  
\[
    \left|\sum_{l=1}^k\lambda_lc_{i,l}\right| \leq |\lambda_{i_0}|\left( 
    \sum_{l=1}^k|c_{i,l}|\right) \leq |\lambda_{i_0}|
\]
Hence we obtain that $\|S(v)\|_{\infty} = \|v\|_{\infty}$. This 
finishes the proof that $S$ is an isometry.
\end{proof}

\noindent We turn now to the proof of theorem
\ref{MainThm}. First, a few words about notations about retractions $\rho$ specific to this theorem.
When written out explicitly, $\rho$ has $N_1+N_2$ components:
$(\rho_1,\rho_2, \ldots, \rho_{N_1}, \rho_{N_1+1}, \ldots, \rho_{N_1+N_2})$ where each 
of the $\rho_j's$ for $j=1,\ldots,N_1+N_2$ can apriori depend on all of the 
variables $w_1, \ldots, w_{N_1},z_{1}, \ldots, z_{N_2}$; so, we 
use $w=(w_1, \ldots, w_{N_1})$ as our notation for the coordinates on $\mathbb{C}^{N_1} \supset B_1$ 
and $z=(z_{1}, \ldots, z_{N_2})$ for that on $\mathbb{C}^{N_2} \supset B_2$. We denote the Minkowski functionals of $B_1,B_2$ by
$\mu_1, \mu_2$ respectively. We spend some words for the 
relevant convex
case i.e.,
the special case when $\mu_1,\mu_2$ norms such that their 
unit balls are strictly convex -- we refer to this as `the convex case' for brevity, till the proof of theorem \ref{MainThm} is completed. This explicit spelling out is for the  
convenience of
future use of this result in conjunction with results
on norm one projections
in Banach spaces from the functional analysis literature.\\

\noindent Next, note that as our non-trivial retract $Z$ is necessarily a non-compact subset of 
$B_1 \times B_2$, it will therefore meet the boundary of $D$ 
i.e., $\overline{Z} \cap  \partial (B_1 \times B_2) \neq \emptyset$ where 
it may noted that
\[
\partial (B_1 \times B_2)
= (\partial B_1 \times B_2) \cup (B_1 \times \partial B_2)
\cup (\partial B_1 \times \partial B_2).
\]
However, rather than this disjoint decomposition of the boundary, we shall for the 
most part work with a slightly different decomposition in which the pieces are not 
disjoint but are closed (in 
$\mathbb{C}^{N_1 + N_2}$); essentially,
\begin{equation} \label{nondisjtdecomp}
\partial ( B_1 \times B_2) = (\partial B_1 \times \overline{B}_2) \cup (\overline{B}_1 \times \partial B_2).
\end{equation} 
Moreover, we shall  break up our discussion depending on where $L$ (rather than $Z$) meets the boundary of 
$D:=B_1 \times B_2$ i.e.,
let $\partial D_L:= L \cap \partial D$ and consider the various cases:
\begin{itemize}
\item [(I)] $\pa D_L \subset \pa B_1 \times \pa B_2$.
\item [(II)] $\pa D_L \subset (\pa B_1 \times B_2) \cup (\pa B_1 \times \pa B_2) = \partial B_1 \times \overline{B}_2$.
\item [(III)] $\pa D_L \subset (B_1 \times \pa B_2) \cup (\pa B_1 \times \pa B_2) = \overline{B}_1 \times \partial B_2$.
\item [(IV)] $\pa D_L$ intersects both $\pa B_1 \times B_2$
 and $B_1 \times \pa B_2$.
\end{itemize}
Observe that while these cases are not mutually exclusive they are indeed exhaustive; to 
see this clearly, first note the two obviously mutually exclusive cases:
\begin{itemize}
\item[(i)] $\partial D_L$ is  contained in one of
$\partial B_1 \times \overline{B}_2$ or $\overline{B}_1 \times \partial B_2$, or,
\item[(ii)] case-(i) does not hold i.e., $\partial D_L$ is not fully contained within one 
of the pieces mentioned in (i).
\end{itemize} 
As $\partial D_L$ is contained in $\partial ( B_1 \times B_2)$ by definition, case-(ii) 
means by equation (\ref{nondisjtdecomp}), that $\partial D_L$ must intersect both 
the pieces of case-(i) and conversely. Now, if $\partial D_L$ intersects both 
pieces mentioned in case-(i), then noting the subtle differences in 
the closure signs, we observe that either case-(IV) or case-(I) must 
happen in such an event.
This verifies that the cases (I) through (IV) suffice to exhaust all
possibilities; we note in passing that since the last case (IV), does 
not preclude the possibility that $L$ intersects 
$\partial B_1 \times \partial B_2$ (and hence subsumes case-(I)), the possibility that $L$ intersects each of 
$\partial B_1 \times B_2$, $B_1 \times \partial B_2$ and 
$\partial B_1 \times \partial B_2$ is taken care of, as well. We now take up each case in turn. 
We rigorously argue out these cases trying to minimize repetition 
at the same time, for we will provide no details for the extended product $B_1 \times B_2 \times \ldots \times B_N$. 

\subsection*{Case-(I)}
In this case, both the projections $\pi_1: D_L \to B_1$ and
$\pi_2: D_L \to B_2$ are proper. Indeed, if $\{(w^k,z^k): k \in \mathbb{N}\}$ is a sequence in $D_L$ accumulating only on its boundary, then by definition of the case-(I) that we are in now, $w^k \to \partial B_1$ and $z^k \to \partial B_2$
i.e., $\pi_1(w^k,z^k) \to \partial B_1$ and  
$\pi_2(w^k,z^k) \to \partial B_2$.\\

\noindent Next, note that ${\rm dim}(D_L) = {\rm dim}(L)\leq\min\{N_1,N_2\}=N_1$ since proper holomorphic maps cannot decrease dimension and $\pi_1$ is a proper holomorphic map carrying 
$D_L$ into the $N_1$-dimensional $B_1$. Infact, as $D_L$ is just a linear subvariety of $D$, the properness of $\pi_1$ means also that the fibres of $\pi_1$ are actually singletons; so, $\pi_1$ maps $D_L$ biholomorphically onto its range in $B_1$ and the inverse of this map can be used to parametrize $D_L$ as the graph of a linear map over the linear subspace $\pi_1(D_L) \subset B_1$ as:
\begin{equation}\label{L-as-graph}
D_L = \{ \big( w, \beta_1(w), \ldots, \beta_{N_2}(w)\big) \; : \; w \in \pi_1(D_L)\}.
\end{equation}
for some $\mathbb{C}$-linear functionals $\beta_1, \ldots, \beta_{N_2}$ on the complex linear subspace $\pi_1(L) \subset \mathbb{C}^{N_1}$. To tie this up with our goal of describing $Z$, we shall now show that $Z=D_L$! Indeed to this end, pick any point
 $(w^0,z^0) \in \partial D_L \subset \partial B_1 \times \partial B_2$ and 
consider the complex line segment $l$ through it i.e., $l$ is the image of the map 
$\varphi(t)$ defined for $t \in \Delta$, by
\[
\varphi(t) = (w^0 t, z^0t) = (w^0,z^0)t \;= t(w^0,z^0),
\]
where we have written out $\varphi(t)$ in various notations only to indicate once
of the alternative usages (that may be more natural in certain contexts to come later). 
Next compose this parametrization of $l$ with our 
retraction $\rho$:
\[
t \longmapsto \big( \rho_1(tw^0,tz^0), \ldots, \rho_{N_1+N_2}(tw^0,tz^0)\big).
\]
This is ofcourse the restriction of the retraction $\rho$ to $l$ and we shall now analyze it componentwise. More precisely, let 
\[
g_1(t) = (\pi_1 \circ \rho \circ \varphi) (t) = 
\big(\rho_1(tw^0,tz^0), \ldots, \rho_{N_1}(tw^0,tz^0) \big)
\]
and 
\[
g_2(t) = (\pi_2 \circ \rho \circ \varphi)(t) = \Big( \rho_{N_1+1}(tw^0,tz^0), \ldots, \rho_{N_1+N_2}(tw^0,tz^0) \Big).
\]
So $g_1,g_2$ are holomorphic maps from the unit disc into the balls $B_1, B_2$ respectively, mapping $0 \in \Delta$ to the origin in the respective balls. We are thus essentially in the setting of 
theorem \ref{extr-Schw-lem-holextrbdy} above
(in-case $B_1,B_2$ are strictly convex, the Schwarz lemma of Mazet recalled above as theorem \ref{Mazet-Schwarz} may also be applied), provided only that we verify that the derivatives of these maps at the origin satisfy the extremality condition as required. To do this, let us begin with the map $g_1$ whose derivative we compute using the chain-rule as:
\[
    dg_1(0)= d\pi_1 \circ d \rho_{\vert_{\varphi(0)}} \circ d \varphi(0).
\]
Then, using the facts that $d\pi_1 \equiv \pi_1$ (as the projection $\pi_1$ is just linear),
$\varphi$ maps $0 \in \Delta$ to the origin in $B_1$ which inturn is fixed by $\rho$, and that 
$d\rho\vert_0$ is just the identity map on $T_0Z$ and that 
$(w^0,z^0) = d \varphi(0) \in T_0Z$, we get
$dg_1(0)=w^0 \in \partial B_1$ which is a holomorphically extreme
boundary point by hypothesis.
The aforementioned theorem applied to $g_1$ shows that 
that $\pi_1 \circ \rho(tw_0,tz^0) = tw_0$, which means that $\rho^1:= \pi_1 \circ \rho$ acts like the projection onto the 
$w$-axes on $D_L$:
\[
{\rho^1}_{\vert_{D_L}}(w,z) = w.
\] 
The same considerations then applied to the other component $g_2$ gives likewise for all $t \in \Delta$ that
$\rho^2(tw_0, tz^0) = tz^0$ 
where $\rho^2= \pi_2 \circ \rho$. In other words, $\rho^2$ also acts as a projection, this time onto the $z$-space:
\[
{\rho^2}_{\vert_{D_L}}(w,z) = z.
\]
Putting this together with the similar conclusion made about $\rho^1$, this just means that $\rho$ acts as the identity mapping on $l$. But this argument may be carried out for {\it every} one-dimensional linear subspace of $D_L$, as we are in the case that $\partial D_L$ is completely contained within 
 $ \partial B_1 \times \partial B_2$. It therefore follows that $\rho$ must act as the identity mapping on $D_L$ leading to the conclusion that $D_L=Z$; thus, $Z$ in the present case (Case-(I)) must be of the form 
\[
Z=\{ \big( w, \beta_1(w), \ldots, \beta_{N_2}(w)\big) \; : \; w \in \pi_1(D_L) \}.
\]
Finally, what the fact that the boundary of $Z=D_L$ 
(what we have denoted by $\partial D_L$), lies in $\partial B_1 \times \partial B_2$ means for the functionals $\beta_j$'s appearing in the above description of $Z$, is that the linear map represented by their collective (whose graph is $Z$), namely
\begin{align}\label{beta-map}
 w \longmapsto \left( \beta_1(w), \ldots, \beta_{N_2}(w) \right)
\end{align}
maps $B_1^L:=\pi_1(D_L)$ into $B_2$.\\

\noindent In this paragraph, we remark about the special case when
$B_1,B_2$ are both strictly convex. We may phrase the last observation
as saying that
the map (\ref{beta-map})
is an isometry from the Banach space ($\pi_1(L)$, $\mu_1$) into the Banach space ($\mathbb{C}^{N_2}$, $\mu_2$) and therefore maps $B_1^L$ which is the unit ball in $\pi_1(L)$ with respect to the norm
$\mu_1$ isometrically into $B_2$; the mapping is surjective if $N_1=N_2$ and $\pi_1(L) = \mathbb{C}^N$ where $N=N_1=N_2$. Hence, we see that this places strong restrictions about when this case-(I) can happen. For instance, this case cannot occur
if $N_1=N_2=N>1$ and $B_1 = \Delta^N$ is the standard unit polydisc and $B_2 = \mathbb{B}$, the standard ball with respect to the $l^2$-norm in $\mathbb{C}^N$ and ${\rm dim}(D_L)=N$; or more generally, case-(I) cannot arise as soon as 
$\pi_1(D_L)$ is not biholomorphic to any of the linear sections of $B_2$. But then again, one must not be amiss to note that case-(I) does happen at least once regardless of what particular balls $B_1,B_2$ are, namely when ${\rm dim}(D_L)=1$; indeed, the complex line segment spanned by every point in $\partial B_1 \times \partial B_2$ is a retract (as is more generally true of the $\mathbb{C}$-line segment spanned by any boundary point of $D$, as we are dealing
only with the convex case in this para).

\subsection*{Case-(II):} $\partial D_L \subset \partial B_1 \times \overline{B}_2$. In this case, note that the
projection $\pi_1: D_L \to B_1$ is proper and again the general fact that proper holomorphic maps cannot decrease dimension gives ${\rm dim} D_L \leq {N_1}$. Noting that the arguments in case-(I) leading to (\ref{L-as-graph}), relied only on the properness of $\pi_1$ and not that of $\pi_2$, we see that the same arguments apply in this case as well to show again that $D_L$ must be of the form
\[
D_L = \{ \big( w, \beta_1(w), \ldots, \beta_{N_2}(w)\big) \; : \; w \in \pi_1(D_L)\},
\]
for some $\mathbb{C}$-linear functionals $\beta_1, \ldots, \beta_{N_2}$ on the complex linear subspace $\pi_1(L) \subset \mathbb{C}^{N_1}$; again, the fact that the boundary of $D_L$ 
 lies in $\partial B_1 \times  B_2$ implies that
 the linear transformation given by
\[
T_\beta(w) = \big( \beta_1(w), \ldots, \beta_{N_2}(w) \big)
\]
maps $B_1^L:= \pi_1(D_L)$ into $B_2$; this may be phrased in 
the convex case as: $T$
is an isometry from the Banach space ($\pi_1(L)$, $\mu_1$) into  ($\mathbb{C}^{N_2}$, $\mu_2$). The mapping is surjective if $N_1=N_2$ and $\pi_1(L) = \mathbb{C}^N$ where $N=N_1=N_2$.
Picking a point $(w^0,z^0) \in \partial D_L$, considering the complex line segment joining the origin to this point and the restriction of the retraction of $\rho$ to
this segment, we get by the arguments as in Case-(I) by applying 
theorem \ref{extr-Schw-lem-holextrbdy}
(or in the convex case, the Schwarz lemma of Mazet), that ${\rho^1}_{\vert_{D_L}}(w,z) = w$ where $\rho^1=\pi_1 \circ \rho$. As points of $D_L$ are of the form $(w, T_\beta(w))$ for $w \in \pi_1(D_L)$, we have
for all $w \in \pi_1(D_L)$ that 
\begin{equation} \label{maineqnII}
\rho(w, T_\beta(w)) = \big(w, \rho^2(w, T_\beta(w)) \big)
\end{equation}
where $\rho^2= \pi_2 \circ \rho$; thereby,
it follows that $\rho(D_L)$ is of the form
\[
\rho(D_L)=\{ \big(w,\rho^2(w, T_\beta(w))\big) \; : \; w \in \pi_1(D_L)\}.
\]
This means that $\rho(D_L)$ is the graph of a $\mathbb{C}^{N_2}$-valued holomorphic function over 
$B_1^L:=\pi_1(D_L)$, more precisely the graph of the $B_2$-valued holomorphic map given by
\[
w \to \big(\rho_{N_1+1}(w, T_\beta(w)), \ldots, \rho_{N_1+N_2}(w, T_\beta(w)) \big).
\]
for $w$ varying through $B_1^L$, the unit `ball' of $\pi_1(L)$ with respect to the functional $\mu_1$. Recalling that theorem \ref{alpha(L)=Z}
 gives $Z=\rho(D_L)$ whose form we just determined, finishes this case, yielding that $Z$ is a graph of a $B_2$-valued holomorphic map on the linear subspace $B_1^L$, which we also denote by $(B_1)_{L_1}$.\\

\subsection*{Case-(III)} $\partial D_L \subset \overline{B}_1 \times \partial B_2$. In this case, the projection onto the second-factor namely,
$\pi_2 : D_L \to B_2$ is proper and we deduce as usual at the outset, that: the general fact that proper holomorphic maps do not decrease dimension implies that ${\rm dim} D_L \leq N_2$.\\

\noindent   We then start as in Case-(I) by picking a boundary point 
$(w^0, z^0) \in \partial D_L$, which we may assume to not be from the part of $D_L$ that lies (if it does at all) in $\partial B_1 \times \partial B_2$ for this has already been 
dealt with in case-(I); the arguments that follow do not depend on such an assumption and we therefore do not make any assumption about which part of the boundary of $D_L$ the point $(w^0, z^0)$ comes from. Denote by $l'$, the complex line spanned by this point and consider the restriction of our retraction map $\rho$ to the intercept $l' \cap D$ of this line with our domain
$D=B_1 \times B_2$, which we call $l$. Indeed, compose the standard parametrization of $l$ namely $\varphi (t) = t(w^0,z^0)$ with $\rho$ followed by $\pi_2$ to get the  holomorphic curve in $B_2$ given by the parametrization
 $g_2(t) :=   (\pi_2 \circ \rho \circ \varphi)(t)$ for $t \in \Delta$. Noting that $g_2$ sends $0 \in \Delta$ to the origin in $B_2$, computing by the chain-rule that 
 $dg_2(0)=z^0 \in \partial B_2$ and recalling that all boundary points of $B_2$ are holomorphically 
 extreme, we see that all conditions required by theorem \ref{extr-Schw-lem-holextrbdy} (or if one so wishes, the Mazet Schwarz lemma
 for the convex case) are met; which when applied to  $g_2$ yields that this map is a linear map:
 $g_2(t)=z^0t = tz^0$. To see the implications of this, let us write this more elaborately using the original definition of $g_2$ as $\pi_2( \rho\circ \varphi (t) ) = t z^0$, which unravels
 explicitly as
\[
 \big( \rho_{N_1+1}( \varphi(t) ), \ldots, \rho_{N_1+N_2}(\varphi(t)) \big) = (tz^0_1, \ldots, t z^0_{N_2}).
\]
This means that for all $j=1,2, \ldots, N_2$, the restriction of $\rho_{N_1+j}$ to $l$ just returns the $j$-th variable:
\begin{equation*} 
{\rho_{N_1+j}}_{\vert_{l}} (w,z_1,\ldots,z_{N_2}) = z_j.
\end{equation*}
This just means that ${\rho^2}_{\vert_l}(w,z)=z$. Now, we may extend this independence 
of $\rho^2$ on $w$ throughout $l$, by recalling that $l$ was actually an arbitrarily chosen complex line segment from $D_L$. Indeed, recall that $l$ is the line spanned by $(w^0,z^0)$ which was just an arbitrary point of $\partial D_L$. 
Observing that every one dimensional complex linear subspace of $L$ can be obtained 
as the $\mathbb{C}$-span of some point of $\partial D_L$ (and that the foregoing arguments 
to reach the last equation works equally well for every complex line segment through the origin  in $D_L$), we deduce the persistence of the foregoing equation throughout $D_L$. That is to say, $\rho^2$ restricted to $D_L$ does not depend on $w$ and infact, 
\begin{equation}\label{alphonL}
    {\rho^2}_{\vert_{D_L}}(w,z) = z 
\end{equation}
Though we cannot claim this form for $\rho^1$ as well, observe that the fact that 
the linear projection $\pi_2: D_L \to B_2$ is proper, means that $\pi_2$ maps $D_L$ 
biholomorphically (indeed, as a linear isomorphism) onto its image $\pi_2(D_L)$, a linear subspace of 
$B_2$. In particular therefore, $\pi_2$ is invertible over $ \pi_2(D_L)$. Denoting the 
$w$-component of the inverse of this map 
by $T^\beta$, we see (analogously to Case-(II)) that this means that points of $D_L$ 
are of the form $(T^\beta(z),z)$ for some (holomorphic) linear map $T^\beta$ of 
$z=(z_1, \ldots,z_{N_2})$ varying in $\pi_2(D_L)$.
If we the compose this map with our retraction $\rho$, we deduce as before that 
points of 
$\rho(D_L)$ are precisely of the form
\begin{equation} \label{maineqnIII}
\rho\big( \beta(z),z \big) = \big( \rho^1(T^\beta(z),z),z \big) \;\; \text{for all} \; z \in \pi_2(D_L),
\end{equation}
meaning that $\rho(D_L)$ is a $k$-dimensional holomorphic graph over 
the linear subspace $\pi_2(D_L)$ of $B_2$, where 
$k={\rm dim}(D_L) = {\rm dim}(Z)$. As this verifies in particular that $\rho(D_L)$ is a submanifold 
of $Z$, we may now apply the identity principle for analytic sets or just recall theorem 
\ref{alpha(L)=Z}, which allows us to say that the whole of $Z$ is obtained by 
applying $\rho$ just once to $D_L$ i.e., $\rho(D_L)=Z$; as we just deciphered the 
form of $\rho(D_L)$, this finishes determining $Z$. 

\begin{rem}
The reason we are able to apply the reasoning of the first case in this case, `line-by-line' to each complex line in $D_L$ but not be confronted with the possibility that the projection $\pi_2(Z)$ is a `ruled-surface/ruled-submanifold' but actually a linear subspace is because of the first made observation by the argument based on theorem \ref{Mazet-Schwarz} that 
$\pi_2(\rho(tw_0, tz^0))= tz^0$ which is stronger than just saying that the image of $D_L$ under $\rho$ when projected onto the second factor is a line -- $\rho$ actually preserves the parametrization in the $z$-variables.
\end{rem}

\subsection*{Case-(IV)} This is the final case when $\partial D_L$ intersects 
both $\Gamma_1:=\partial B_1 \times B_2$ as well as 
$\Gamma_2:=B_1 \times \partial B_2$. While this does not preclude the possibility that
$L$ does not intersect $\Gamma_0:=\partial B_1  \times \partial B_2$, certainly any one dimensional complex linear subspace $l'$ which intersects $\Gamma_j$ for $j=1$ or $2$, intersects the boundary of $D$ only in $\Gamma_j$ (note firstly that the $\Gamma_j$'s are disjoint). That is to say, if $l=l \cap D$ then $\partial D_L$ is fully contained in $\Gamma_j$ (for $j=1$ or $2$), as soon as it contains a point from $\Gamma_j$. Indeed, if 
$(w^0,z^0) \in \Gamma_j$ then the complex line $l$ is the two dimensional real linear subspace spanned by $(w^0,z^0)$ and $i \cdot (w^0,z^0)= (iw_0,iz^0)$. Noting that $w^0$ (resp. $z^0$) lies in the interior of $\overline{B}_1$ 
(resp. $\overline{B}_2$) if and only if $i w^0$ (resp. $iz^0$) does, verifies the claim that $\partial D_L$ is fully contained within $\Gamma_j$ (for $j=1$ or $2$) as soon as it intersects $\Gamma_j$. As $L$ intersects both $\Gamma_j$ for $j=1,2$ we have in 
particular that $L=T_0Z$ cannot equal $l'$ and ${\rm dim}(D_L) \geq 2$. This also means that for each of $\Gamma_1, \Gamma_2$ there are complex lines in $L$ which intersect it. Moreover, if $l$ is a complex line in $D_L$ which intersects $\Gamma_j$ 
(for $j=1$ or $2$) then lines obtained by perturbing the coefficients of linear equations which define $L$, slightly, also intersect the open piece $\Gamma_j$ (for $j=1$ or $2$) of the boundary of $D$ and by our forgoing remarks, the boundary of the intercept of each such line (near $D_L$) with $D$ is fully contained in $\Gamma_j$. Observe that the union of these lines contains an open subset $U_L$ of $D_L$. Next, we apply the Schwarz lemma arguments in Case-(II) (resp. Case-(III)) to each line that intersects $\Gamma_1$ (resp. $\Gamma_2$). As there are lines in $D_L$ which intersect $\Gamma_1$ as well as (other) lines in $D_L$ which intersect $\Gamma_2$, we have both the conclusions of these cases about the major pair of components of 
$\rho$ namely, $\rho^j:=\pi_j \circ \rho$ for $j=1,2$, which are that $\rho^1$ just returns the first set of variables on $D_L$: 
${\rho^1}_{\vert_{D_L}}(w,z)=w$ and likewise for the second component 
${\rho^2}_{\vert_{D_L}}(w,z)=z$. 
Putting together, we therefore see that 
$\rho$ acts just as the identity map, first atleast on $U_L$, and thereby on $L$ as well, by the identity principle applied to both the holomorphic maps $\rho^1, \rho^2$. As ${\rm dim}(D_L)={\rm dim}(Z)$ obviously, this means that $D_L=Z$. Therefore, we conclude 
in this case that $Z$ is a linear subspace of $D$ (of complex dimension $\geq 2$), which is neither contained in $B_1$ 
nor $B_2$. \\

\subsection*{Conclusion} The analysis of the foregoing cases can all be now summed up and 
it leads precisely to theorem \ref{MainThm}. Certain apparent redundancies
are better clarified by the following remarks.

\begin{rem}
Though case-(i) is a subcase of both case-(ii) and case-(iii) in the statement of theorem \ref{MainThm}, 
we have stated it separately, for we can provide 
finer information in this case
where the boundary $\partial D_L$ is contained in the distinguished boundary of $D$, namely
$\partial B_1 \times \partial B_2$. 
Indeed in case-(i), the linear 
retract $Z$ can be viewed as a graph of a linear map over a linear subspace in $B_2$ as well; this cannot be done 
in case-(iv) in general: consider for instance, the product of a one-dimensional retract of $B_1$ with one 
of $B_2$ which obviously cannot be expressed as a graph of a function neither over a linear subspace 
in $B_1$ nor over one in $B_2$ (unless the dimension $N=2$). Nevertheless, one may broadly say as in the opening sentence of the 
theorem, that retracts in all cases are indeed graphs of holomorphic maps over some linear subspace 
in the (full) domain $D$ just by theorem \ref{Graph} itself. However, what cannot 
be derived from that theorem but is readily implied by the above theorem is that 
certain graphs are ruled out from being retracts -- for instance, graphs of {\it non-linear} holomorphic 
maps which intersect $\partial B_1 \times B_2$ and
$B_1 \times \partial B_2$. \\
If we drop case-(i), viewing it as a subcase of 
case (ii) or (iii), then it may be noted that cases (ii) through (iv) are both mutually exclusive 
and exhaustive as well. Next, it must be noted that each of the possibilities expressed by cases (i) 
through (iv) are actually realized and can be used to construct the retracts. However, we note that 
though this theorem seems to provide a complete characterization of the retracts of $D$, we cannot 
pinpoint which of the linear subspaces of $D$ can happen to be its retract much as we cannot say 
which linear subspaces of $B_1$ or $B_2$ are their (linear) retracts as well -- this depends on the 
specific nature of the balls $B_j$ i.e., the geometric features of their boundaries, for it is 
known for certain that not every linear subspace of $D$ 
can be its retract in general, for instance even in the simplest case 
$B_1=B_2=\Delta$, in which case however we can indeed so pinpoint, as was 
precisely the purpose of our theorem \ref{Polydisk} above.   
\end{rem}

\noindent More importantly perhaps is the following observation.
\begin{rem}
    Although $Z$ is biholomorphic to $D_{L_1}$, it is unreasonable to expect at-least in general, that
    there is an automorphism of $D$ which maps $Z$ biholomorphic $D_{L_1}$. This fails already for the 
simplest polyball namely, the polydisc, wherein such a failure can be seen as follows. All
the automorphisms of the polydisc are linear fractional transformations, thereby each one of 
them maps a {\it complex} linear subspaces to at-most a quadratic complex surface. Hence, as 
soon as $Z$ is `sufficiently' non-linear (there are plenty of retracts in the polydisc which are
graphs of monomial maps of degree $>2$ for instance, which are herein regarded as 
`sufficiently' non-linear), it is 
certainly clear that there does not exist any automorphism of the polydisc mapping $Z$ onto $D_{L_1}$.
Theorem \ref{MainThm} may therefore be viewed as providing an answer to the problem of recasting 
retracts in a domain into a simple `normal' form (which cannot be reduced to a further simple one in 
general), by applying some appropriate automorphism of the domain.
\end{rem}
\begin{rem}\label{pi_1(Z)}
It may very well happen that neither $\pi_1(Z), \pi_2(Z)$ need be linear
subspaces of the domain, let alone retracts (neither
of the respective factors $B_1,B_2$ nor of $D$). 
For a simple example, consider the
polydisc in $\mathbb{C}^4$ which we view
as a product of $2$ copies of the bidisc i.e.,
let
$D := \Delta^2 \times 
\Delta^2$; and, consider the linear subspace $L$ of $\mathbb{C}^4$ 
defined by 
$L := span\{v_1,v_2\}$, where $v_1 = (1,1/2,1/4,0) \in \partial \Delta^2 \times \Delta^2,~ 
v_2 = (0,1/2,3/4,1) \in \Delta^2 \times \partial \Delta^2$.
We can identify first and second factor of $D$ as $\Delta^2 \times \{0\}$ and $\{0\} \times \Delta^2$. Note that the holomorphic map $\rho$ on $D$ defined by $\rho(z_1,z_2,z_3,z_4) = (z_1,\frac{z_1+z_4}{2},
\frac{z_1+3z_4}{4},z_4)$ gives a retraction map from $D$ onto $D_L$.
Note that $D_L$ is a retract of $D$ and $\partial D_L 
$ intersects $\partial \Delta^2 \times \Delta^2$ and $\Delta^2 
\times \partial \Delta^2$.  
Now we consider the 
projections restricted to $D_L$, which is given as follows:\\ 

$\pi_1 : L \to \pi_1(L)$ defined by 
\[
    \pi_1(z_1,\frac{z_1+z_4}{2},\frac{z_1+3z_4}{4},z_4) = (z_1,\frac{z_1+z_4}{2},0,0).
\]
and $\pi_2 : L \to \pi_2(L)$ defined by 
\[
\pi_2(z_1,\frac{z_1+z_4}{2},\frac{z_1+3z_4}{4},z_4) = (0,0,\frac{z_1+3z_4}{4},z_4).
\]
 Note that $\pi_1(D_L)$ and $\pi_2(D_L)$ are not linear subspaces 
of neither the respective factors
($\Delta^2 \times \{0\}$, $\{0\} \times \Delta^2$)
nor of $D$;
indeed, this can be 
confirmed for instance by noting that that $(0,1/2) = \pi_1(v_2) \in 
\pi_1(D_L)$ and $(1/4,0) = \pi_2(v_1) \in \pi_2(D_L)$, but their scalings: 
$(0,3/4) \notin \pi_1(D_L)$ and $(1/2,0) \notin \pi_2(D_L)$ despite $(0,3/4,0,0)$ and $(1/2,0,0,0)$
are within $D$.
\end{rem}

\noindent As mentioned in the introduction, we shall skip the details of the extension of the arguments above 
to the case of an extended product of an arbitrary number of `balls' and only lay down the end result;
in-fact we only do this for the case of convex
factors with $\mathbb{C}$-extremal boundaries
which is of special importance as it 
leads to an associated result about
norm-one projections with respect to those norms
whose unit balls are such products. The
corresponding statement for bounded balanced
pseudoconvex domains with
holomorphically extreme boundaries is also true 
and can be written down as in theorem \ref{MainThm}.\\

\begin{thm} \label{MainThmExtd}
Let $B_1,B_2, \ldots, B_N$ be a collection of bounded balanced convex domains 
with $\mathbb{C}$-extremal boundaries, possibly in
different complex Euclidean spaces (say, $B_j \subset \mathbb{C}^{n_j}$ for $j=1,\ldots, N$). Then, 
every non-trivial retract through the origin of 
$D=B_1 \times B_2 \times \ldots \times B_N$ is a holomorphic graph over a linear subspace 
which itself is a retract of $D$; if the retract is not linear, this linear subspace is 
infact a subspace of the product of a subcollection of these balls $B_j$'s with values 
in the product of the remaining balls in this collection.\\
More specifically and precisely, let $L=T_0Z$, $D_L=L \cap D$ and 
$\partial D_L = L \cap \partial D$. For each $p \in \{1,2, \ldots, N\}$, denote by
 $\Sigma_p$, the set of all $p$-tuples $(j_1,j_2, \ldots, j_p)$ where 
 $\{ j_1,j_2,\ldots,j_p\} \subset \{1,2, \ldots,N\}$ with 
 $j_1<j_2<\ldots<j_N$ (the cardinality of $\Sigma_p$ is $N \choose p$) and note that there exists any given $Z,D_L$ as above, $p \in \{1,2,\ldots,N\}$ such that $\partial D_L$ is contained in that part of the boundary of the convex body
$\overline{D}= 
\overline{B}_1 \times \overline{B}_2 \times \ldots \times \overline{B}_N$, obtained by replacing the factor $\overline{B}_l$ herein by $\partial B_l$ for each $l=j_1, j_2, \ldots, j_p$
where $\sigma=(j_1,j_2, \ldots, j_p) \in \Sigma_p$. Then, $Z$ is the graph of a 
$B^{\sigma_c}$-valued holomorphic map over a $q$-dimensional linear retract of 
$B_{j_1} \times B_{j_2} \times \ldots \times B_{j_p}$ (which is included in the 
canonical manner into our product domain $D=B_1 \times B_2 \times \ldots \times B_N$),
where:
\begin{itemize}
\item[(i)] $q \leq n_{j_1} + \ldots + n_{j_p}$ and 
 $(j_1,j_2, \ldots, j_p) \in \Sigma_p$ and,\\
\item[(ii)] $B^{\sigma_c}$ is the product of balls given by
$B_{k_1} \times B_{k_2} \times \ldots \times B_{k_{N-p}}$ where $\sigma_c \in \Sigma_{N-p}$ 
is the $(N-p)$-tuple complementary to $\sigma$ in the sense that
$\sigma^c=(k_1,\ldots,k_{N-p}) \subset \{1,\ldots,N\} \setminus \{j_1,\ldots,j_p\}$
and $k_1<k_2<\ldots<k_{N-p}$. 
\end{itemize}   
Let $m=\min\{n_1,n_2, \ldots,n_N\}$, $\hat{\mu}_m = \max\{\mu_1, \mu_2, \ldots,\mu_{m-1}, \mu_{m+1}, \ldots, \mu_N\}$, where for all $j=1,2,\ldots, N$ we denote by $\mu_j$, the Minkowski functional of $B_j$. Then, in the subcase when $\partial L$ 
is contained within the distinguished boundary of $D$ namely, $\partial B_e:=\partial B_1 \times \partial B_2 \times \ldots \partial B_N$, the retract $Z$ is actually a linear subspace of $D$ occurring as the graph of a (holomorphic) linear 
isometry of the Banach space $(\pi_m(L), \mu_m)$  into the Banach space
$(\mathbb{C}^{n_1} \times \ldots \times \mathbb{C}^{n_{m-1}} \times
 \mathbb{C}^{n_{m+1}} \times \mathbb{C}^{n_N}, \hat{\mu}_m) $, and therefore maps $B_m^L:=\pi_m(D_L)$ which is the unit ball in $\pi_m(D_L)$ with respect to $\mu_m$, 
 isometrically into $B_1 \times \ldots B_{m-1} \times B_{m+1} \times \ldots \times B_N$.\\
In the pending case when $\partial D_L$ intersects more than one connected component of 
$\partial D \setminus \partial B_e$, the retract $Z$ is 
again a linear subspace of $D$ of complex dimension equal to the number of connected components that it intersects.
\end{thm}

\noindent Needless to say, when the product domain happens
to be homogeneous such as 
a bounded balanced symmetric domain whose automorphisms 
are known, the above theorem
suffices to give all its retracts.

\section{Proof of theorem \ref{union-prob-polyballs}}

\noindent The proof follows 
ideas from the article \cite{Fost} of Forn\ae ss--Sibony.
Let us begin by noting that it follows from theorem \ref{Graph} proved in the foregoing section that
every retract of a bounded balanced homogeneous domain is biholomorphic to the graph of some holomorphic
function over a linear subspace of $D$.\\

\noindent  Since $\Omega$ is homogeneous, $\Omega / {\rm Aut}(\Omega)$ is singleton and hence compact. By the Main theorem in  \cite{Fost}, there exists a closed complex submanifold $Z$ of $\Omega$ such that M is biholomorphic to a locally trivial holomorphic fibre bundle with fibre $\mathbb{C}$ over $Z$. As mentioned in the beginning, by theorem \ref{Graph} we get that Z is biholomorphic to a linear retract $\Omega_L := \Omega \cap L$ for some linear subspace $L$ of $\mathbb{C}^N$. We can cover $M$ by coordinate charts biholomorphic to $\{U_{\alpha} \times \mathbb{C}\}$ where $U_{\alpha}$ is an open cover of $Z$. The coordinate transformations from the $\alpha$ to the $\beta$- coordinates are of the form
\[ \varphi_{\alpha \beta}(z,w) = (z,h_{\alpha \beta}(z)w + k_{\alpha \beta}(z))\]
for some holomorphic functions $k_{\alpha \beta}, h_{\alpha \beta}$ on $U_{\alpha} \cap U_{\beta}, h_{\alpha \beta} \neq 0$. The data $h_{\alpha \beta}$ determine a multiplicative Cousin problem on $Z$. Since $Z$ is homeomorphic to a polydisc, it follows by Oka's theorem (theorem 3.9 of \cite{oka_0} as well) that the multiplicative Cousin problem for $Z$ can be solved. Hence we might have assumed already above that $h_{\alpha \beta} \equiv 1$ for all $\alpha, \beta$. However in that case, the data $k_{\alpha \beta}$ determine an additive Cousin problem. By Oka's theorem in  (theorem 3.6 of \cite{oka_0}), the additive Cousin problem on $Z$ can also be solved. Hence $M$ is biholomorphic to $Z \times \mathbb{C}$.

\begin{rem}
We recall for one last time, that the class of homogeneous bounded {it balanced} domains 
is precisely the class of bounded balanced symmetric domains and that any domain of this
class, is biholomorphic to a Cartesian product 
of Cartan domains of types I -- IV 
(the second chapter, specifically remark 2.3.8 of the book \cite{Jrncki_invrnt_dst} and references therein), a polyball.
\end{rem}

\section{Proof of Theorems \ref{ell_q_ball_C3} and \ref{cnvx_hull}}\label{ell_q_ball}

Before discussing our result as in the 
title of this section on the retracts of  
$\ell_q$ balls for $0<q<1$ 
we first review the case $q\geq1$. This is {\it not} to 
repeat the details available in the classical
treatise of Lindenstrauss -- Tzafiri \cite{Lindstraus}, 
we note that this is 
still perhaps the only book where the 
the main theorem about this case namely, the one stated
as theorem \ref{Lindenstrauss -- Tzafiri} is 
dealt with. To indicate our purpose here in reviewing this case,
notice that the 
formulation of the determining criteria for the linear retracts
of $\ell^p$-balls as stated in the just-mentioned theorem
is highly coordinate dependent, leaving a
a more geometric understanding to be desired. While 
geometric reformulations of the same have been recorded
for instance in the survey \cite{Beata-survey}, unfortunately
this article neither outlines a proof nor a reference for where it 
could be found; the book of Lindenstrauss -- Tzafiri referred
therein does not give one side of the result that we need, which to 
the best of our knowledge is not recorded elsewhere as well. We therefore proceed to unraveling in some
detail, the geometric meaning of
theorem \ref{Lindenstrauss -- Tzafiri} needed for a better
understanding of the linear retracts of $l^p$-balls.\\

\noindent To begin with, we observe that a linear slice of an $\ell^p$-ball need not be isometric to any lower 
dimensional $\ell^p$-ball. In fact, this happens only in very 
special circumstances, indeed {\it precisely} when the 
slice is a retract. 
As already mentioned, this result is not really 
new, being contained in essence in theorem 6.3 of \cite{Beata-survey}; our purpose here is to provide the missing details, 
indeed a self-contained proof of this result, not to be found elsewhere. To attain this geometric result, we now recall a generalization of the notion of orthogonality in Euclidean spaces to normed
spaces wherein the norm fails to be induced by any inner product; one that is most relevant for our aforementioned purposes first.
\begin{defn}
Let $(X,\|\cdot\|)$ be a complex Banach space. For $x, y \in X$, we 
say that $x$ is  $p$-orthogonal 
 to $y$ and denoted $x \perp_p y$ if
 \[
 \begin{cases}
 \|x+ \lambda y\|^p = \|x\|^p + |\lambda|^p\|y\|^p,~\forall \lambda \in \mathbb{C},~\text{when} ~p < +\infty; \\
  \|x+\lambda y\| = \max\{\|x\|,|\lambda|\|y\|\},~\forall \lambda \in \mathbb{C},~\text{when} ~p = +\infty.
\end{cases}
\] 
\end{defn}
\noindent There are many notions of orthogonality and we shall not spend space in comparing the above with others except perhaps for the most popular  known as Birkhoff-James orthogonality, which can be found in \cite{Giles} and 
\cite{Faulk}. Denoted suggestively by `BJ', this is defined as
\[
x \perp_{BJ} y \Longleftrightarrow \|x +\lambda y\| \geq \|x\|~~\text{for every}~\lambda \in \mathbb{C}.
\]
Note that if $x \perp_p y$ then $x \perp_{BJ} y$ and $y \perp_{BJ} x$. Therefore, $p$-orthogonality is a stronger condition than Birkhoff-James orthogonality. This generalized orthogonality relation is, in general not symmetric. If it is indeed the case that $x \perp_{BJ} y$ and $y \perp_{BJ} x$, then we say that $x$ and $y$ are bi-orthogonal (or symmetrically orthogonal).
\\

\noindent Note that if $u,v \in \mathbb{C}^N$ are mutually disjointly supported vectors then 
$u$ and $v$ are bi-orthogonal vectors in $(\mathbb{C}^N,\|\cdot\|_p)$ 
for $1 \leq p \leq \infty$. But for each $p \neq 2$, there exists a linear subspace spanned 
by bi-orthogonal vectors which is not be spanned by mutually disjoint supported 
vectors, as illustrated in the following.\\

\textit{Example:} Consider the linear subspace $L = f^{-1}(0)$, where the 
linear functional $f: \mathbb{C}^3 \to \mathbb{C}$ defined by 
$f(z_1,z_2,z_3) = 2z_1+z_2-z_3$. Next we claim that $L$ is spanned by 
the bi-orthogonal vectors $(1,-1,1)$ and $(0,1,1)$. To prove this claim, 
let $v_1 = (1,-1,1)$ and $v_2 = (0,1,1)$. Note that
\[
\|v_1+\lambda v_2\|^p = \|(1,-1+\lambda,1+\lambda)\|^p = 
1 + |1-\lambda|^p+|1+\lambda|^p 
\]
Since $|a+b|^p \leq 2^{p-1}(|a|^p + |b|^p)$, we obtain that 
$\|v_1 + \lambda v_2\|^p \geq 1 + \frac{2^p}{2^{p-1}} = 3 = \|v_1\|^p$. This shows that $v_1$ is orthogonal to $v_2$. Using similar arguments, we obtain that $v_2$ is orthogonal to $v_1$.
This finishes the proof of the claim. If $L$ is spanned by 
mutually disjoint supported vectors then by lemma 3.2 in  \cite{Baronti-Pappini}, $f$ has 
atmost two non-null components. Hence, we 
obtain that $L$ cannot be realized as a subspace spanned by mutually disjoint supported 
vectors.
\begin{lem}\label{ortho}
Let $P$ be a linear projection map on $\mathbb{C}^N$. Then 
$\|P\| =1$ if and only if ${\rm image}(P)$ is BJ-orthogonal to 
$\ker (P)$.
\end{lem}
\begin{proof}
Assume that $P$ is a projection of norm one. We want to show that 
${\rm image}(P)$ is BJ-orthogonal to $\text{ker}(P)$.
Towards this, let $z \in \text{image}(P)$ and 
$w \in \ker (P)$; so, $P(z +\lambda w) = P(z) $ by linearity of $P$. Being a projection of norm-one $P$ is contractive,
thereby
for all $\lambda 
\in \mathbb{C}$, we have $\|z\| = \|P(z+\lambda w)\| \leq \|z+\lambda w\|$. 
By the notion of BJ-orthogonality recalled above, $z$ is BJ-orthogonal to $w$.\\
Conversely assume that $\text{image}(P)$ is BJ-orthogonal to $\ker(P)$. 
Since $P$ is a linear projection map, we have $\mathbb{C}^N = \text{image}(P) 
\oplus \ker(P)$. So, we can write each $z \in \mathbb{C}^N$ as 
$z = w_1 + w_2$, where $w_1 \in \text{image}(P)$ and $w_2 \in \ker(P)$. 
Note that $\|P(z)\| = \|P(w_1)\| = \|w_1\|$. As $w_1 \perp_{BJ} w_2$, we get: 
$\|w_1\| \leq \|w_1 + w_2\| = \|z\|$
i.e., $\|P(z) \| \leq \|z\| $, finishing the proof.
\end{proof}

Next, we recall the following result due to Lamperti (Corollary 2.1 of \cite{Lamp}), which is an analogue of the parallelogram law for $\ell^p$-norms, determining precisely for which pair of vectors such a law holds.
\begin{lem}\label{disjnt}
Let $1 \leq p < \infty$ with $p \neq 2$. Two vectors $u,v \in \mathbb{C}^N$ are mutually disjoint supported if and only if  $u$ and $v$ are $p$-orthogonal.
\end{lem}
\noindent Though the following result is not new. For instance, refer to theorem 3 of chapter 6 in \cite{Lacey}. We therefore give a self-contained and direct exposition. For clarity, let us introduce a notation $\|\cdot\|_{p,k}$ to denote the $\ell_p$ norm on $\mathbb{C}^k$.
\begin{prop}\label{lower_l_p}
Let $W$ be a linear subspace of $(\mathbb{C}^N,\|\cdot\|_{p,N})$. If $W$ is isometric to 
$(\mathbb{C}^k,\|\cdot\|_{p,k})$ for $1 \leq p \leq \infty$ with $p \neq 2$, then 
there exists a norm-one projection from $\mathbb{C}^N$ onto $W$. 
\end{prop}
\begin{proof}
First, we consider the case $p \neq \infty$. Assume that $W$ is isometric to 
$(\mathbb{C}^k,\|\cdot\|_{p,k})$. 
Then there exist a basis $\{u_1,\ldots,u_k\}$ of $W$ and a linear isometry $S$ 
such that $\|u_j\|_{p,N} = 1$ and $S(u_j) = e_j$, where $e_j$ denotes the standard basis of $\mathbb{C}^k$ for $j = 1,2,\ldots,k$. 
Note that $T := S^{-1}$ is an isometry from $\mathbb{C}^k$ onto $W$ and 
$T(e_j) = u_j$ for $j = 1,2,\ldots,k$. First, we prove the result for 
$1 \leq p < \infty$ with $p \neq 2$. Since $e_j$'s are disjointly supported, 
we can apply the above lemma \ref{disjnt} to $e_j$. Hence for each $j \neq l$, we have
\[
    \|e_j+e_l\|_{p,k}^p + \|e_j-e_l\|_{p,k}^p = 2(\|e_j\|_{p,k}^p + \|e_l\|_{p,k}^p)
\]
As $T$ is an isometry, we obtain that
\[
    \|T(e_j)+T(e_l)\|_{p,N}^p + \|T(e_j)-T(e_l)\|_{p,N}^p = 2(\|T(e_j)\|_{p,N}^p + \|T(e_l)\|_{p,N}^p)
\]
Again applying the above lemma \ref{disjnt}, we get that the $u_j$'s are mutually 
disjoint supported vectors in $\mathbb{C}^N$. Following the proof of  proposition 2.a.1 in 
\cite{Lindstraus}, the details regarding the construction of norm-one projection are given as follows. For each $j =1,2,\ldots,k$, let $\sigma_j := supp(u_j)$. Consider the linear functional 
$f_j: W \to \mathbb{C}$ defined by $f_j(u_l) = \delta_{jl}$ 
(where $\delta_{jl}$ denotes the Kronecker-delta function) 
and extended to other vectors of $W$ by linearity.\\

First we want to prove that $\|f_j\|_{p,N} =1$ for each $j = 1,2,\ldots,k$. 
To prove this, write each $z \in W$ as $z = \sum_{l=1}^kc_lu_l$ 
for some scalars $(c_l)_{l=1}^k$. Note that 
\begin{equation}\label{dual}
    |f_j(z)| = \left|f_j\left(\sum_{l=1}^kc_lu_l\right)\right|
    = \left|\sum_{l=1}^kc_lf_j(u_l)\right| = |c_j| \leq  \|(c_1,\ldots,c_k)\|_{p,k} 
    = \|S(z)\|_{p,k}. 
\end{equation}
Since $S$ is an isometry, we obtain that $|f_j(z)|\leq \|z\|_{p,N}$. 
This implies that $\|f_j\|_{p,N} \leq 1$ for all $j =1,2,\ldots,k$. Since 
$f_j(u_j) = 1$, we obtain that $\|f_j\|_{p,N} = 1$.
By the Hahn -- Banach extension theorem, $f_j$ can be extended  
to a linear functional on $\mathbb{C}^N$ of the same norm which we also
denote by $f_j$. So, $f_j: \mathbb{C}^N \to \mathbb{C}$ is a 
linear functional with 
$\|f_j\|_{p,N} = 1$. For each $j$ with $1 \leq j \leq k$, write $f_j(z) = \sum_{i=1}^N c_{j,i}z_i$ for some scalars $(c_{j,i})_{i=1}^N$. Define another linear functional $u_j^*$ on $\mathbb{C}^N$ by $u_j^*(z) = \sum_{i \in \sigma_j}c_{j,i}z_i$. Note that 
\[
\|u_j^*\|_{p,N}^q = \sum_{i \in \sigma_j} |c_{j,i}|^q \leq \sum_{i =1}^N |c_{j,i}|^q = \|f_j\|_{p,N}^q = 1
\]
where $q$ is the dual exponent of $p$ (i.e. $1/p + 1/q =1$.). This implies in particular that $|u_j^*(\tilde{e_i})| = |c_{j,i}| \leq 1$, where $\tilde{e_i}$ denote the standard basis of $\mathbb{C}^N$. As $u_j^*(u_j) = 1$, we obtain that $\|u_j^*\|_{p,N} = 1$ for every $j$. Now we define the 
projection $P : \mathbb{C}^N \to W$ by 
$P(z) = \sum_{j =1}^k u_j^*(z)u_j$. 
To prove this proposition, it suffices to show that $\|P\|_{p,N} = 1$. If $z= \sum_{i=1}^N a_i\tilde{e_i}$ then 
$u_j^*(z) = \sum_{i=1}^N a_i u_j^*(\tilde{e_i}) = \sum_{i \in \sigma_j}a_i c_{j,i}$.
 Since $\|u_j^*\|_{p,N} = 1$, $|u_j^*(z)|^p \leq \sum_{i \in \sigma_j}|a_i|^p$. Hence we obtain that $\|P(z)\|_{p,N}^p = \sum_{j=1}^k|u_j^*(z)|^p \leq \|z\|^p$. This proves that $P$ is a norm-one projection from $\mathbb{C}^N$ onto $W$. \\
 
Next, we prove the result for $p = \infty$. Write $S(z) = (S_1(z),\ldots,S_k(z))$. Note that $S_j(z)$ 
is a linear functional on $W$ for each $j$ with $1 \leq j \leq k$. 
By the Hahn -- Banach theorem, $S_j$ can be extended to a linear 
functional  on $\mathbb{C}^N$ of the same norm which we denote by 
$\tilde{S_j}$. So, $\tilde{S_j} : \mathbb{C}^N \to \mathbb{C}$ is 
a linear functional so that $\|\tilde{S_j}\| = \|S_j\| \leq 1$. 
Define the linear map $\tilde{S}: \mathbb{C}^N \to 
\mathbb{C}^k$ by $\tilde{S}(z) = (\tilde{S}_1(z),\ldots,\tilde{S}_k(z))$. 
Note that $\|\tilde{S}\| \leq 1$ because $\|\tilde{S}\| = \sup\limits_{\substack{1 \leq j \leq N, ~
 \|z\| = 1}}
|S_j(z)| \leq 1$. Since $\tilde{S}(u_j) = e_j$, we get that $\|\tilde{S}\| = 1$. 
Consider the linear map $P: \mathbb{C}^N \to W$ defined by $P(z) = T\circ 
\tilde{S}(z)$. Note that $P$ maps $\mathbb{C}^N$ onto $W$ and $P_{|_W} = {\rm id}_W$. 
Hence, we get that $P$ is a norm one projection from $\mathbb{C}^N$ onto $W$.
\end{proof}

\begin{prop}
Let $\mathbb{B}_p$ be the $\ell_p$-ball in $\mathbb{C}^N$ where 
$1 \leq p \leq \infty$. The linear subspaces of $\mathbb{C}^N$ such 
that $(\mathbb{B}_p)_L := \mathbb{B}_p \cap L$ is a retract of 
$\mathbb{B}_p$ are precisely those for which $(\mathbb{B}_p)_L$ 
is biholomorphic to a lower dimensional $\ell_p$-ball $\mathbb{B}_p^k$ for 
some $k$ with $1 \leq k \leq N$.
\end{prop}
\begin{proof}
Assume that $(\mathbb{B}_p)_L$ is biholomorphic to a lower 
dimensional $\ell_p$-ball $\mathbb{B}_p^k \subset \mathbb{C}^k$. 
By the theorem of W. Kaup and H.Upmeier in 
\cite{Kaup-Upm} (exposited as Remark 16.4.2 d) in \cite{Jrncki_invrnt_dst}), after passing to a different biholomorphic mapping we may assume that $\varphi(0) = 0$. By an application of 
Cartan's theorem, we get that $\varphi$ is linear. Note that for each $v \in (\mathbb{B}_p)_L$, we have
\[
    \|v\|_{p} = h_{(\mathbb{B}_p)_L}(v) = K_{(\mathbb{B}_p)_L}(0,v) = K_{\mathbb{B}_p^k}(0,\varphi(v)) = \|\varphi(v)\|_{p,k}
\]
This implies that $\varphi$ is a linear isometry from $L$ onto 
$(\mathbb{C}^k,\|\cdot\|_{p,k})$. By the above proposition \ref{lower_l_p}, we obtain that 
$(\mathbb{B}_p)_L$ is a linear retract of $\mathbb{B}_p$.\\
Conversely assume that $(\mathbb{B}_p)_L$ is a linear retract of 
$\mathbb{B}_p^N$. First we consider the 
case $1 \leq p < \infty$. By theorem 2.a.4 in \cite{Lindstraus}, there exists an isometry 
from $L$ to $(\mathbb{C}^k,\|\cdot\|_{p,k})$. 
This implies in particular that $(\mathbb{B}_p)_L$ is biholomorphic to 
$\mathbb{B}_p^k$ for some $k$ with $1 \leq k \leq N$. Next we 
consider the case for $p = \infty$. By proposition \ref{lower_l_infty}, there exists 
an isometry from $L$ to $(\mathbb{C}^k,\|\cdot\|_{\infty})$. This 
proves that $\Delta^N_L$ is biholomorphic to lower dimensional 
polydisk in $\mathbb{C}^k$.
\end{proof}
\noindent We conclude the characterizations of retracts of 
$l^p$-balls in the convex case $p \geq 1$.
\begin{thm}
A linear subspace $R$ of the $\ell_p$-ball $\mathbb{B}_p$ is a  $k$-dimensional retract of the 
$\mathbb{B}_p$ in $\mathbb{C}^N$ if and only if there exist a finite sequence of retracts 
$R_1 \subset R_2 \subset \ldots \subset R_k \subset \mathbb{B}_p$ such that for each $j=1,2,\ldots,k$, $R_j$ is a $j$-dimensional linear retract of $\mathbb{B}_p$ and $R_k = R$.
\end{thm}
\begin{proof}
As the sufficiency part is utterly trivial, we need
assume that $R$ is a linear retract of the $\ell_p$-ball $\mathbb{B}_p$ and verify the reachability of $R$ via a 
monotonically increasing (finite) sequence of retracts 
as asserted. By theorem 2.a.4 in \cite{Lindstraus}, $R$ is isometric to $k$-dimensional $\ell_p$-ball $\mathbb{B}_p^k$. Recalling there always exist retracts of all possible dimensions in $\mathbb{B}_p$, we have
in-particular, a $(k-1)$-dimensional retract $R_{k-1}$ of $\mathbb{B}_p^k$.  By lemma \ref{retr-transit}, $R_{k-1}$ is a linear retract of $\mathbb{B}_p$. Again applying theorem 2.a.4 in \cite{Lindstraus} and lemma \ref{retr-transit} to $\mathbb{B}_p^{k-1}$, we obtain that there exists a $(k-2)$-dimensional retract $R_{k-2}$ of $\mathbb{B}_p$. Continuing this process, until we get a one-dimensional retract $R_1$ of $\mathbb{B}_p$, renders
the required sequence as in the statement of the theorem.
\end{proof}

\noindent We now put the above 
result together with the foregoing
more geometric formulation
of the linear retracts of $\ell_p$-balls 
as being  
precisely those linear intercepts which are
isometric copies of lower 
dimensional $\ell_p$-balls. 
Doing so, leads
to the conclusion that for any given retract $R$, 
there is a decreasing chain of 
isometric copies of lower dimensional 
$\ell^p$-balls reaching $R$ as in the 
above theorem. This is a 
feature shared by the $\ell_p$-`balls'
even for positive values 
of the parameter $p<1$. Indeed, we
now move onto this case
dealt with by theorem \ref{ell_q_ball_C3}.
Before going into its proof, we need to make a couple of observations
about the geometry of the $\ell_q$-ball $D_q$ in $\mathbb{C}^N$ for $q<1$
-- for a vivid distinction between the convex 
case $p \geq 1$ with the non-convex case $p<1$,
we use the symbol $q$ rather than $p$, for the latter case.
To begin with then, although the  absolute image of $D_q$
is concave (easily and vividly seen in the two-dimensional case), $D_q$ is actually not 
concave (nor convex) domain in $\mathbb{C}^N$; not even locally, near any of the smooth boundary points.
We may establish this by computing the real Hessian at any smooth boundary point with positive real coordinates 
as $D_q$ is multicircular; we show this in the 
two-dimensional setting for simplicity, while it holds in higher dimensions as well by straight-forward extension of the 
arguments and computations.
Rewrite the defining function for $D_q$ as
\begin{eqnarray*}
    \varphi(z,w)=z^{q/2}\overline{z}^{q/2} + w^{q/2}\overline{w}^{q/2} - 1
\end{eqnarray*}
To compute the real Hessian, note the following.
\begin{eqnarray*}
    \renewcommand\arraystretch{1.5}
    H &:=& \begin{bmatrix}
        \frac{\partial^2 \varphi}{\partial z \partial \overline{z}} & \frac{\partial^2 \varphi}{\partial z \partial \overline{w}} \\

        \frac{\partial^2 \varphi}{\partial w \partial \overline{z}} & \frac{\partial^2 \varphi}{\partial w \partial \overline{w}}
    \end{bmatrix} 
    = 
    \begin{bmatrix}
        \frac{q^2}{4}|z|^{q-2} & 0\\
        0 & \frac{q^2}{4}|w|^{q-2},
    \end{bmatrix}
\\
         \tilde{H} &:=& 
     \renewcommand\arraystretch{1.5}
      \begin{bmatrix}
        \frac{\partial^2 \varphi}{\partial z^2} & \frac{\partial^2 \varphi}{\partial z \partial w} \\

        \frac{\partial^2 \varphi}{\partial w \partial z} & \frac{\partial^2 \varphi}{\partial w^2 }
    \end{bmatrix}
    =
    \begin{bmatrix}
        \frac{q}{2}(\frac{q}{2} - 1)|z|^qz^{-2} & 0\\
        0 & \frac{q}{2}(\frac{q}{2} - 1)|w|^qw^{-2}
    \end{bmatrix}
\end{eqnarray*}
Before proceeding, we remark that we shall utilize the above
to compute the real Hessian only at the (smooth) absolute boundary points of $D_q$, such points being precisely those
points $(x,y)$ of $\partial D_q$ both of whose coordinates are 
positive; while this suffices as $D_q$ is Reinhardt, it is important to note that the (tangent) vectors along which the
Hessian form is evaluated need to be allowed 
to have non-real (complex) coordinates. Let $v = [v_1~ v_2]$ and $\overline{v} = [\overline{v_1}~ \overline{v_2}]$. Then the real Hessian at the point $(x,y)$ 
evaluated along the direction $v$ is 
\begin{equation}\label{Hes}
    v H \overline{v}^T_{|_{(x,y)}} + Re(v H v^T)_{(x,y)}\\
    = \frac{q^2}{4}\left( x^{q-2}|v_1|^2 + y^{q-2}|v_2|^2\right)  
+ \frac{q}{2} (\frac{q}{2} -1) Re\left( x^{q-2}v_1^2 + y^{q-2}v_2^2\right).
\end{equation}
If $v_1, v_2$ are purely imaginary, then the above expression 
(\ref{Hes}) is strictly positive. If $v_1$ and $v_2$ are real, then the above expression (\ref{Hes}) becomes
\[
    \frac{q^2-q}{2}\left(x^{q-2}v_1^2 + y^{q-2}v_2^2\right)
\]
which is negative as $q < 1$. We conclude therefore that $D_q$ is neither convex nor concave 
(though of course pseudoconvex) at $(x,y) \in \partial D_q$. Owing 
to already made remarks this holds at each point
of $\partial D_q$ other than the negligible set of non-smooth 
boundary points of $D_q$.\\

\begin{prop}
    Let $D$ be a bounded balanced pseudoconvex domain in $\mathbb{C}^N$. 
Suppose $L$ is a linear subspace of $\mathbb{C}^N$ such that ${\rm co}(\overline{D_L}) \neq C_L$, 
where $C_L := C \cap L$ with $C = {\rm co}(\overline{D})$, denoting the convex hull of $\overline{D}$. Then $D_L := D\cap L$ is not a linear retract of $D$.
\end{prop}
\begin{proof}
To prove this proposition by contradiction, assume that $D_L$ is a linear 
retract of $D$. Then there exists a linear projection $P : \mathbb{C}^N \to L$ such that $P$ maps $D$ onto $D_L$. 
Since ${\rm co}(\overline{D_L}) \neq C_L$ by hypothesis, there exists a point $p$ lying within the convex set $C_L$ but not in ${\rm co}(\overline{D_L})$.
Write $p \in C_L \subset C$ as the convex combination of extreme points of $C$
i.e. $p = \sum_{j=1}^k \alpha_je_j$, (where $\alpha_j$'s are positive reals with $\sum_{j=1}^k\alpha_j = 1$), where $e_j$'s are extreme points of $C$ 
(here of course $e_j$'s need not be the standard basis vectors, though we use this notation only in this proof for convenience and the extreme points of $C$ need not be finite in number). 
Such an expression is possible due to the Krein-Milman theorem applied to the compact convex set $C$.
By the linearity of $P$, we have
\[
p = P(p) = \sum_{j=1}^k \alpha_jP(e_j).
\]
As $P(e_j) \in \overline{D_L}$, this shows that $p$ can be written as the convex combination of points in $\overline{D_L}$. 
This implies that $p \in {\rm co}(\overline{D_L})$, which contradicts the hypothesis that ${\rm co}(\overline{D_L}) \neq C_L$. 
\end{proof}
We now proceed to prove theorem \ref{cnvx_hull}. 

\begin{proof}[Proof of theorem \ref{cnvx_hull}:]
Assume that $L$ is not spanned by vectors only from $\partial D \cap \partial C$. We want to show that $D_L$ is not a linear retract of $D$. 
By the above proposition, it suffices to show that ${\rm co}(\overline{D_L}) \neq C_L$. 
 Now we claim that there exists a vector $v \notin \partial D \cap \partial C$ such that 
$v$ is an extreme point of ${\rm co}(\overline{D_L})$.
To prove this claim by contradiction, assume that every extreme point of ${\rm co}(\overline{D_L})$ lies in $\partial D \cap \partial C$. 
We may write each $w \in {\rm co}(\overline{D_L})$, as a convex combination of the extreme points in ${\rm co}(\overline{D_L})$.
Let $w'$ be an arbitrary point in $L$. Note that there exists $\lambda \in \mathbb{C}^*$ such that $\lambda w' \in D_L$. 
So, we may express $\lambda w'$ as a convex combination of the extreme points of ${\rm co}(\overline{D_L})$. 
Since these extreme points all lie within $\partial D \cap \partial C$, we get that $w'$ belongs to span of vectors in $\partial D \cap \partial C$. 
Since this holds for every $w' \in L$, $L$ is spanned by vectors in $\partial D \cap \partial C$, which contradicts the hypotheses of the theorem. This finishes the proof of the claim.\\

By theorem 2.10.15 in \cite{Meggn}, $v \in \partial D_L = L \cap \partial D$. As $v \notin \partial C \cap \partial D$, we have $v \notin \partial C_L$.
To prove ${\rm co}(\overline{D_L}) \neq C_L$, 
choose $t_0 > 1$ such that $t_0v \in \partial C_L$. We want to show that $t_0v \notin {\rm co}(\overline{D_L})$. 
To prove this by contradiction, assume that $t_0v \in {\rm co}(\overline{D_L})$. 
Write $t_0v = c_1v_1+ \ldots + c_kv_k$ (where $c_j$'s are positive reals with $\sum_{j=1}^kc_j = 1$), where $v_j$'s are extreme points of ${\rm co}(\overline{D_L})$. 
After dividing $t_0$ on both sides, we obtain that
\[
    v = c_1\left(\frac{v_1}{t_0}\right) + \ldots + c_k\left(\frac{v_k}{t_0}\right)
\]
Since $v_j/t_0 \in D_L$ for each $j$, $v$ can be expressed as a convex combination of points in $D_L$. This contradicts the fact that $v$ is an extreme point of ${\rm co}(\overline{D_L})$. 
This proves that ${\rm co}(\overline{D_L}) \neq C_L$.
\end{proof}
\begin{lem}\label{cnvx_p}
Let $D$ be a bounded pseudoconvex domain in $\mathbb{C}^N$ and $p$ a boundary point of $D$ near which $\partial D$ is smooth. Fix $p \in \partial D$ and suppose that the (real) affine line $l(p,v) := \{p+tv: t \in \mathbb{R}\}$ with $v \in T_p^{\mathbb{C}}(\partial D)$ does not intersect $D$. Then the real Hessian $\mathcal{H}r(p)$ along $v$ is non-negative, where $r$ is any smooth defining function near $p$.  
\end{lem}
\begin{proof}
Let us begin with the standard parametrization of the (real) line $l$ passing through $p$ in the direction of $v$ with $v \in T_p^{\mathbb{C}}(\partial D)$, namely $\varphi(t) = p+tv$. 
Let $r$ be a local defining function of $D$ defined on an open neighbourhood $U$ of $p$. 
So, $D \cap U = \{z \in U: r(z) < 0\}$. Define $V := \varphi^{-1}(U)$ and consider the smooth function $\psi: V \to \mathbb{R}$ defined by $\psi(t) = r \circ \varphi(t)$. 
Since ${\rm image}(\varphi) = l$ never intersects $D$, $\psi(t) \geq 0$ for all $t \in V$. 
This implies that $\psi$ attains a local minimum at $t = 0$. Hence $\psi''(0) \geq 0$. 
Next we want to relate $\psi''(0)$ and $\mathcal{H}r(p)$ evaluated along the direction $v$. By an application of the chain rule, we obtain that
\begin{eqnarray*}
    \psi'(t) &=& \frac{\partial r}{\partial z_1}(\varphi(t))~ v_1 + \frac{\partial r}{\partial z_2}(\varphi(t))~ v_2 +\ldots +\frac{\partial r}{\partial z_N}(\varphi(t))~ v_N \\
    \psi''(t) &=& \frac{d}{dt}\left(\frac{\partial r}{\partial z_1}(\varphi(t)) ~v_1\right) + \frac{d}{d t}\left(\frac{\partial r}{\partial z_2}(\varphi(t)) ~v_2\right) + \ldots +\frac{d}{d t}\left(\frac{\partial r}{\partial z_N}(\varphi(t)) ~v_N\right) 
\end{eqnarray*}
For each $j = 1,2,\ldots,N$, we have
\[
\frac{d}{d t}\left(\frac{\partial r}{\partial z_j}(\varphi(t)) ~v_j\right) = \left[\frac{\partial^2 r}{\partial z_1 \partial z_j}(\varphi(t))~v_1 + \ldots + \frac{\partial^2 r}{\partial z_j^2}(\varphi(t))~v_j + \ldots + \frac{\partial^2 r}{\partial z_N \partial z_j}(\varphi(t))~v_N\right]v_j
\]
Hence, we obtain that $\psi''(0) = \displaystyle \sum_{j,k =1}^N \frac{\partial^2 r}{\partial z_j \partial z_k}(p)~v_jv_k$. 
The real Hessian $\mathcal{H}r(p)$ evaluated along $v$ is 
\begin{equation}\label{real_Hess}
2\sum_{j,k =1}^N \frac{\partial^2 r}{\partial z_j \partial \overline{z}_k}(p)~v_j\overline{v}_k + 2 Re\left(\sum_{j,k =1}^N \frac{\partial^2 r}{\partial z_j \partial z_k}(p)~v_jv_k\right)
\end{equation}
By the (Levi) pseudoconvexity of $D$ at $p$, we then have for every vector $v \in T_p^{\mathbb{C}}(\partial D)$, that:
\begin{equation*}
\sum_{j,k =1}^N \frac{\partial^2 r}{\partial z_j \partial \overline{z}_k}(p)~v_j\overline{v}_k \geq 0.
\end{equation*}
The second sum in (\ref{real_Hess}) is non-negative, and it equals $2\psi''(0)$, which was noted to be non-negative. It therefore follows from (\ref{real_Hess}) that the real Hessian $\mathcal{H}r(p)$ along $v$ is non-negative. 
\end{proof}
\begin{defn} \label{open-piece-retracts}
Let $D$ be a domain in $\mathbb{C}^N$. We say that $D$ admits retracts along an (or, {\it for some}, or {\it corresponding to an}) open piece of directions at $p \in D$ if there exists a (non-empty) open subset $\Gamma$ of $\partial \mathbb{B}$ such that for each $v \in \Gamma$, there is a retract $Z$ of $D$ passing through $p$ with $v \in T_0Z$.
\end{defn}

\noindent We now restate theorem \ref{ell_q_ball_C3} as follows and prove it.

\begin{cor}\label{l_q1_q2}
 Consider the balanced pseudoconvex domain given by
\[ 
D := \{(z_1,\ldots,z_N) \in \mathbb{C}^N: |z_1|^{q_1} + \ldots + |z_N|^{q_N} < 1\}
\]
where the $q_j$'s are all positive numbers less than one. Let $L$ be a linear subspace of $\mathbb{C}^N$ which cannot be spanned precisely by vectors only from the standard basis vectors i.e. $L \neq \spn(S)$ for any $S \subset B_{std} = \{e_1,\ldots,e_N\}$.
Then $D_L$ is not a retract of $D$; (thereby, the retracts of $D$ are precisely the orthogonal projections of $D$ onto a span of a finite subset of $B_{std}$). Consequently, $D$ does not admit any retract along any open piece of directions at the origin.
\end{cor}
\begin{proof}
Let $C$ denote the convex hull of $\overline{D}$. First we claim that $C = \overline{\mathbb{B}}_1$, where $\mathbb{B}_1 := \{(z_1,\ldots,z_N) \in \mathbb{C}^N: |z_1| + \ldots + |z_N| < 1\}$ , the $\ell_1$-ball in $\mathbb{C}^N$. 
As $q_j < 1$ for each $j$, we have $|z_j| < |z_j|^{q_j}$. This implies that $\sum_{j=1}^N |z_j| < \sum_{j=1}^N |z_j|^{q_j}$ and thereby $D \subset \mathbb{B}_1$. 
As $\overline{\mathbb{B}}_1$ is convex, $C \subset \overline{\mathbb{B}}_1$.
Now, $\overline{\mathbb{B}}_1$ is the convex hull of the extreme points of $\overline{\mathbb{B}}_1$, but these extreme points belong to $\partial D$. 
Therefore $C = \overline{\mathbb{B}}_1$.\\

\noindent Now, we claim that $\partial C \cap \partial D = \{\lambda e_i: \lambda \in \mathbb{C} ,~ |\lambda| = 1\}$, where $e_i$'s are the standard basis vectors of $\mathbb{C}^N$. 
To prove this claim by contradiction, assume that there exists $p$ different from the scalar multiple of $e_i$ belonging to $\partial C \cap \partial D$. 
First, consider the case that every component of $p$ is non-zero. In that case, $p$ is a smooth boundary point of $D$. Since $C$ is convex at $p$, there exists a supporting hyperplane $H_p$ to $C$ at $p$. Note that the supporting hyperplane at $p$ coincides with the tangent plane at $p$ because $p$ is a smooth boundary point of $D$.
As $D \subset C$ and $p \in \partial C \cap \partial D$, $H_p$ serves as a supporting hyperplane to $\partial D$ at $p$. 
Let $l$ be a (real) affine line passing through $p$ in the direction of $v$, which is contained in $H_p$. By the above lemma \ref{cnvx_p}, the real Hessian $\mathcal{H}r(p)(v) \geq 0$.
Since this holds for all vectors $v \in T_p(\partial D)$, $\mathcal{H}r(p)$ is positive semi-definite on $T_p(\partial D)$.\\

\noindent Now, we compute the real Hessian of $D$ at $p$. Since $D$ is multicircular, after applying an automorphism on $D$, we may assume $p$ is of the form  $(x_1,\ldots,x_N)$, where $x_i$'s are positive.
Upto a non-negative scaling factor of $2$, the real Hessian at the point $(x_1,\ldots,x_N)
$ along the vector $(v_1,\ldots,v_N)$ is given by
\begin{multline*}
\frac{q_1^2}{4} x_1^{q_1-2}|v_1|^2 + \frac{q_2^2}{4}x_2^{q_2-2}|v_2|^2 + \ldots + \frac{q_N^2}{4} x_N^{q_N-2}|v_N|^2\\ 
+ Re\left(\frac{q_1}{2} (\frac{q_1}{2} -1) x^{q_1-2}v_1^2 + \frac{q_2}{2} (\frac{q_2}{2} -1) x_2^{q_2-2}v_2^2 + \ldots + \frac{q_N}{2} (\frac{q_N}{2} -1) x_N^{q_N-2}v_N^2\right)
\end{multline*}
When $v_1,\ldots,v_N$ are positive reals, this real Hessian may be written as
\[
\frac{q_1^2-q_1}{2} x_1^{q_1-2}|v_1|^2 + \frac{q_2^2-q_2}{2}x_2^{q_2-2}|v_2|^2 + \ldots + \frac{q_N^2-q_N}{2} x_N^{q_N-2}|v_N|^2
\]
As $q_i <1$, this is negative, contradicting the above conclusion that the real Hessian is positive semi-definite on $T_p(\partial D)$.
This finishes the proof of the claim.\\

\noindent Next, we consider the case that at least one of the components of $p$ is zero. After a permutation of co-ordinates, we may assume that $p$ is of the form $(p_1,\ldots,p_M,0,\ldots,0)$ for some $M < N$. Consider the domain $D' := \{(z_1,\ldots,z_M) \in \mathbb{C}^M: |z_1|^{q_1} + \ldots + |z_M|^{q_M} < 1\}$. Note that $p' := (p_1,\ldots,p_M)$ is a smooth boundary point of $D'$ and $D'$ is convex at $p'$. By using the same argument applied to $D$, we conclude that this case does not arise. 
By the above theorem \ref{cnvx_hull}, if the linear subspace $L$ is not spanned only by vectors $e_i$ then $D_L$ is not a retract of $D$.\\

\noindent To conclude, let $p = (p_1,\ldots,p_N) \in \partial D$ with $p_i \neq 0$ for all $i$. Note that $p$ does not belong to a non-trivial linear subspace spanned by standard basis vectors. Hence, there does not exist a retract of $D$ passing through the origin in the direction of $p$. Let $S := \{(z_1,\ldots,z_N) \in \partial D: z_i \neq 0 ~\text{for ~all}~ i\}$. Note that $S$ is an open dense subset of $\partial D$ and for each $p \in S$, there does not exist a retract of $D$ passing through the origin in the direction of $p$. Since the complement of $S$ in $\partial D$ forms a nowhere dense subset of $\partial D$, there does not exist any retract of $D$ passing through the origin along an open piece of directions.
\end{proof}
Theorem \ref{cnvx_hull} can be employed to determine retracts of some non-decoupled domains as well. We demonstrate this in the following. 

\begin{cor}\label{non-decoupled} Consider the non-decoupled (balanced pseudoconvex) domain
given by $\Omega := \{(z_1,z_2) \in \mathbb{C}^2: |z_1|^{2q_1} + |z_2|^{2q_2} + |z_1z_2|^{2q_3} < 1\}$
where the $q_j$'s are all positive numbers less than $1/2$. Let $L$ be a linear subspace of $\mathbb{C}^2$ not spanned by vectors only from the standard basis vectors.
Then $\Omega_L$ is not a retract of $\Omega$; (thereby, the retracts of $\Omega$ are precisely the orthogonal projections of $D$ onto a span of the standard basis vectors). Consequently, there are no retracts of $\Omega$ passing through the origin along an open piece of directions.
\end{cor}
\begin{proof}
First, we compute the real Hessian of the defining function of $\Omega$ at positive real coordinates $(x,y)$.
Consider the function
\[
    \varphi(z_1,z_2) = |z_1|^{2q_1} + |z_2|^{2q_2} + |z_1z_2|^{2q_3} -1 = z_1^{q_1}\overline{z_1}^{q_1} + z_2^{q_2}\overline{z_2}^{q_2} + (z_1z_2)^{q_3}(\overline{z_1z_2})^{q_3} - 1
\]
Note that
\begin{eqnarray*}
    \renewcommand\arraystretch{1.5}
   H &=& \begin{bmatrix}
        \frac{\partial^2 \varphi}{\partial z \partial \overline{z}} & \frac{\partial^2 \varphi}{\partial z \partial \overline{w}} \\

        \frac{\partial^2 \varphi}{\partial w \partial \overline{z}} & \frac{\partial^2 \varphi}{\partial w \partial \overline{w}}
    \end{bmatrix} 
    = 
   \begin{bmatrix}
        q_1^2|z_1|^{2(q_1-1)} + q_3^2|z_1z_2|^{2q_3}|z_1|^{-2} & q_3^2|z_1z_2|^{2q_3}(z_1\overline{z_2})^{-1}\\
        q_3^2|z_1z_2|^{2q_3}(\overline{z_1}z_2)^{-1} & q_2^2|z_1|^{2(q_2-1)} + q_3^2|z_1z_2|^{2q_3}|z_2|^{-2}
    \end{bmatrix}
\end{eqnarray*}
In convenience, introduce the notation
\begin{equation*}
\renewcommand\arraystretch{1.5}
   \tilde{H} =   \begin{bmatrix}
        \frac{\partial^2 \varphi}{\partial z_1^2} & \frac{\partial^2 \varphi}{\partial z_1 \partial z_2} \\

        \frac{\partial^2 \varphi}{\partial z_2 \partial z_1} & \frac{\partial^2 \varphi}{\partial z_2^2 }
    \end{bmatrix}
\end{equation*}
 where~ 
\begin{align*}
\frac{\partial^2 \varphi}{\partial z_1^2} = q_1(q_1 - 1)|z_1|^{2q_1}z_1^{-2} + q_3(q_3-1)|z_1|^{2q_3}z_1^{-2}|z_2|^{2q_3},\\ 
\frac{\partial^2 \varphi}{\partial z_1 \partial z_2} = \frac{\partial^2 \varphi}{\partial z_2 \partial z_1} = q_3^2|z_1z_2|^{2q_3}(z_1z_2)^{-1},\\
\frac{\partial^2 \varphi}{\partial z_2^2} = q_2(q_2 - 1)|z_2|^{2q_2}z_2^{-2} + q_3(q_3-1)|z_1z_2|^{2q_3}z_2^{-2}.
\end{align*}

Let $v = [v_1~ v_2]$ and $\overline{v} = [\overline{v_1}~ \overline{v_2}]$. 
Assume that $v_1,v_2$ are real. Upto a non-negative scaling factor of $2$, the real Hessian at the point $(x,y)$ along the vector $v$ is 
\begin{multline*}
    \overline{v} H v^T|_{(x,y)} + Re(v \tilde{H} v^T)_{(x,y)}
    = \left[q_1(2q_1-1) x^{2(q_1-1)} + q_3(2q_3-1)x^{2(q_3-1)}y^{2q_3}\right]|v_1|^2 +\\ \left[q_2(2q_2-1) y^{2(q_2-1)} + q_3(2q_3-1)x^{2q_3}y^{2(q_3-1)}\right]|v_2|^2 
+ 4q_3^2 (xy)^{2q_3-1} xy 
\end{multline*}
The complex tangent space at $(x,y)$ is 
\[
\{(v_1,v_2): \left(q_1|z_1|^{2(q_1-1)}\overline{z}_1 + q_3|z_1|^{2(q_3-1)}|z_2|^{q_3}\right) v_1 = -\left(q_2|z_2|^{2(q_2-1)}\overline{z}_2 + q_3|z_2|^{2(q_3-1)}|z_1|^{q_3}\right) v_2\}
\]
This shows that every real vector $v = (v_1,v_2)$ in $T_p(\partial \Omega)$ for $p \in \partial \Omega$ with positive real co-ordinates whose components $v_1,v_2$ are necessarily opposite in sign.
 Hence we obtain that the real Hessian at the point $(x,y)$ along $v$ is negative.\\
 
Next, we claim that $C = \overline{\mathbb{B}}_1$, the $\ell_1$-ball in $\mathbb{C}^2$. 
Note that for each $j$, $|z_j| < |z_j|^{q_j}$. This implies that $|z_1|+|z_2| < |z_1|^{q_1} + |z_2|^{q_2} < |z_1|^{q_1} + |z_2|^{q_2} + |z_1z_2|^{q_3}$ and thereby $\Omega \subset \mathbb{B}_1$. 
As $\overline{\mathbb{B}}_1$ is convex, $C \subset \overline{\mathbb{B}}_1$.
Note that $\overline{\mathbb{B}}_1$ is the convex hull of extreme points of $\overline{\mathbb{B}}_1$ which belongs to $\partial \Omega$. 
Therefore $C = \overline{\mathbb{B}}_1$.
By using the same arguments in the above corollary, we can prove that $\partial C \cap \partial \Omega = \{\lambda e_i: \lambda \in \mathbb{C},~ |\lambda| = 1\}$, where $e_i$'s are the standard basis vectors in $\mathbb{C}^N$. 
After applying the above theorem \ref{cnvx_hull}, we obtain that if the linear subspace $L$ is not spanned only by vectors $e_i$, then $\Omega_L$ is not a retract of $\Omega$. Hence, there does not exist retracts of $\Omega$ passing through the origin with an open piece of directions.
\end{proof}
\noindent We now give a direct and more elementary proof for the retracts of the standard $\ell^q$-ball $D_q$ (where we no longer assume $N=2$ for that would
be an oversimplification unlike the matter in the preceding 
paragraph), it must first be noted that by proposition \ref{JJ-improved}, 
every {\it holomorphic} retract of $\ell_q$-ball through 
the origin is actually linear even when $q<1$. So, if $Y_{q}$ is a holomorphic retract 
of $D_{q}$ through the origin, then there exists
a (linear) projection $P: \mathbb{C}^N \longrightarrow \mathbb{C}^N$ 
such that $P(D_{q}) = Y_{q}$. Of-course, we are only 
interested in non-trivial retracts, so $Y_{q}$ is of the
form $Y_{q} = D_{q} \cap Y$ for some $r$-dimensional $\mathbb{C}$-linear 
subspace $Y$ of $\mathbb{C}^N$. 
We only need to 
show that  $Y$ cannot be of the form 
\begin{equation}\label{Y}
    \{(z_1,\ldots,z_N) \in \mathbb{C}^N: z_i = a_{1,i}z_1 + \ldots + a_{r,i}z_{r} ~\text{for}~ r+1 \leq i \leq N\},
\end{equation}
where not all $(a_{i,j})$ are zero.\\

\noindent To prove this theorem \ref{ell_q_ball_C3} by contradiction as already indicated above, assume that $Y_{q}$ is a linear retract of $D_{q}$ 
and a (linear) projection $P: \mathbb{C}^N \longrightarrow \mathbb{C}^N$ such that $P(D_{q}) = Y_{q}$. Let $\{e_i: 1 \leq i \leq N\}$
be the standard basis of $\mathbb{C}^N$.
With notations as in (\ref{Y}), define another basis $\{v_i  : 1 \leq i \leq N\}$ of $\mathbb{C}^N$, as follows: 
\begin{align*} 
    v_j &= (\underbrace{0,\ldots,0,t_j,0,\ldots,0}_{r},a_{j,r+1}t_j,\ldots,a_{j,N}t_j) \quad \text{for}~ 1 \leq j \leq r,\\ 
    v_j &= e_j \qquad \text{for}~ r+1 \leq j \leq N 
\end{align*}
 where
\[ 
 t_j = \frac{1}{\left(1 + \sum_{\substack{k=r+1}}^{N}|a_{j,k}|^q\right)^{1/q}}.
\]

\noindent For $1 \leq j \leq r$, write
    \[
        e_j = \frac{1}{t_j}v_j - \left( \sum_{l=r+1}^{N}a_{j,l}v_l\right).
    \]
Applying $P$ on both sides, we get 
\[
    P(e_j) = \frac{1}{t_j}P(v_j) - \left( \sum_{l=r+1}^{N}a_{j,l}P(v_l)\right).
\]
Since $P(v_j) \in Y$, there exists scalars $u_{1,j},\ldots,u_{r,j}$ such that
\begin{equation}
    P(v_j) = \left(u_{1,j},\ldots,u_{r,j},\sum_{k=1}^ra_{k,r+1}u_{k,j},\ldots,\sum_{k=1}^ra_{k,N}u_{k,j}\right)
~\text{for}~ r+1 \leq j \leq N.
\end{equation}
Since $P$ is a projection linear map from $\mathbb{C}^N$ onto $Y$ and $v_j \in Y$ for $1 \leq j \leq r$,  $P(v_j) = v_j$ for $1 \leq j \leq r$. Hence for $1 \leq j \leq r$, we have 
\begin{multline*}
    P(e_j) = \frac{1}{t_j}v_j - \left ( \sum_{l=r+1}^{N}a_{j,l}\left (u_{1,l},\ldots,u_{r,l},
    \sum_{k=1}^ra_{k,r+1}u_{k,l},\ldots,\sum_{k=1}^ra_{k,N}u_{k,l}\right ) \right ) \\
    = \left(- \sum_{l=r+1}^Na_{j,l}u_{1,l},\ldots,- \sum_{l=r+1}^Na_{j,l}u_{j-1,l},
        1-\sum_{l=r+1}^Na_{j,l}u_{j,l},-\sum_{l=r+1}^Na_{j,l}u_{j+1,l},\ldots,
-\sum_{l=r+1}^Na_{j,l}u_{r,l} ,\right.\\ 
\left. a_{j,r+1} -\sum_{l=r+1}^Na_{j,l}\sum_{k=1}^ra_{k,r+1}u_{k,l},
\ldots,a_{j,N}-\sum_{l=r+1}^Na_{j,l}\sum_{k=1}^ra_{k,N}u_{k,l} \right). 
\end{multline*}
Note that $P(v_j) \in \overline{D_q}$ for all $j$ . So, we get the following inequality for $1 \leq j \leq r$, 
\begin{equation}\label{3'}
    \left|1-\sum_{l=r+1}^Na_{j,l}u_{j,l}\right|^q + \sum_{1\leq k \leq r ~~ k \neq j}\left|\sum_{l=r+1}^Na_{j,l}u_{k,l}\right|^q + \sum_{s=r+1}^N\left|a_{j,s}-\sum_{l=r+1}^Na_{j,l}\sum_{k=1}^ra_{k,s}u_{k,l}\right|^q \leq 1. 
\end{equation}
Therefore, following inequality holds for $r+1 \leq j \leq N$:
\begin{equation}\label{4'}
    \sum_{k=1}^r|u_{k,j}|^q + \sum_{m=r+1}^N\left|\sum_{k=1}^ra_{k,m}~u_{k,j}\right|^q \leq 1.
\end{equation}
By reverse triangle inequality for  $q^{\text{th}}$ power where $0 < q < 1$: $\vert z-w \vert^q \geq \left \vert \vert z \vert ^q - \vert w \vert ^q \right \vert$ for all $z,w \in \mathbb{C}$, inequality (\ref{3'}) becomes
\[
    -\sum_{l=r+1}^N|a_{j,l}u_{j,l}|^q + \sum_{1\leq k \leq r ~~ k \neq j}\left|\sum_{l=r+1}^Na_{j,l}u_{k,l}\right|^q + \sum_{s=r+1}^N|a_{j,s}|^q - \left|\sum_{s=r+1}^N\sum_{l=r+1}^N\sum_{k=1}^r a_{j,l}~a_{k,s}u_{k,l}\right|^q \leq 0. 
\]
Combining the above inequality with (\ref{4'}), we obtain the following inequality
\begin{multline*}
    -\sum_{l=r+1}^N|a_{j,l}|^q \left(1-\sum_{1 \leq k \leq r~k \neq j}|u_{k,l}|^q - \sum_{m=r+1}^N \left|\sum_{k=1}^ra_{k,m}u_{k,l}\right|^q\right) + \sum_{1\leq m \leq r ~~ m \neq j}\left|\sum_{l=r+1}^Na_{j,l}u_{m,l}\right|^q + \\ \sum_{s=r+1}^N|a_{j,s}|^q -  \sum_{s=r+1}^N\sum_{l=r+1}^N \left|\sum_{k=1}^ra_{j,l}~a_{k,s}u_{k,l}\right|^q \leq 0. 
\end{multline*}
Simplifying this inequality, we get
\[
\sum_{l=r+1}^N|a_{j,l}|^q \sum_{1 \leq k \leq r~k \neq j}|u_{k,l}|^q + \sum_{1\leq m \leq r ~~ m \neq j}\left|\sum_{l=r+1}^Na_{j,l}u_{m,l}\right|^q \leq 0. 
\]
As all quantities on the left are obviously non-negative, it follows that
\begin{equation}\label{5'}
    a_{j,l}u_{k,l} = 0~~ \quad r+1 \leq l \leq N~ , 1 \leq j ,k \leq r~ \text{with}~ j \neq k. 
\end{equation}
Combining this result with inequality (\ref{3'}), we get
\begin{equation}\label{6'}
    \left|1-\sum_{l=r+1}^Na_{j,l}u_{j,l}\right| \leq t_j, \quad \text{ for each }  j \in \{1,2,\ldots,r\}.
\end{equation}
By our hypothesis about the $a_{i,j}$'s, atleast one of $a_{i,j}$ is non-zero. If $a_{j_0,l_0} \neq 0$ for some $j_0,l_0$ then from (\ref{5'}) we obtain that $u_{j,l_0} = 0$ for all $j$ with $j \neq j_0$. From (\ref{4'}), we get $|u_{j_0,l_0}| \leq t_{j_0}$. Therefore, we conclude that for all $(j,l)$, we have $|u_{j,l}| \leq t_j$. Using (\ref{6'}), we get
\begin{align*}
    |1| & =  \left|1- \sum_{l=r+1}^Na_{j,l}u_{j,l} + \sum_{l=r+1}^Na_{j,l}u_{j,l} \right|\\
        & \leq  \left|1- \sum_{l = r+1}^Na_{j,l}u_{j,l}\right| + \sum_{l=r+1}^N|a_{j,l}u_{j,l}| \\
        & \leq t_j + t_j (\sum_{l = r+1}^N|a_{j,l}|) < 1. 
\end{align*}
This is obviously a contradiction, which completes the proof.
\qed \\

\noindent For certain reasons below, we explicitly mention about the special case in dimension $2$. Owing to remarks made in the introduction, it 
follows as an {\it immediate}
corollary that $D_q = \{(z,w) \in \mathbb{C}^2: |z|^q + |w|^q < 1\}$ is not biholomorphic to any product domain.
A more important relevance of this example is noted next.

\begin{rem}
We recall from \cite{Vigue} that fixed point sets of holomorphic self-maps of a any given 
bounded (not necessarily balanced) convex domain $D$
are precisely the same as its retracts. The example $D_q$ in $\mathbb{C}^2$ above suffices to show 
that this coincidence fails as soon as we drop convexity. Indeed, for the non-convex
(but pseudoconvex) balanced domain $D_q$ in $\mathbb{C}^2$ as above, the fixed point set of the holomorphic 
self-map of $D_q$ as simple as the flip $(z,w) \to (w,z)$, namely the linear 
subspace $\{ (z,w) \in \mathbb{C}^2 \; : \; w=z \}$, is not realizable
as a holomorphic retract of $D_q$.
\end{rem}

\noindent The $\ell_q$-balls dealt with above are bounded by holomorphically extreme 
real hypersurfaces. We next discuss examples for the `other extreme' as well as some 
intermediate cases. We begin with an example where we can show existence of higher dimensional retracts;
the purpose being: to record atleast one simple example to affirmatively 
answer the question as to whether there exists non-homogeneous pseudoconvex domains in $\mathbb{C}^N$ (i.e., when $N$ is not 
restricted to be $2$ or the domain being convex)
with the property that there are retracts through every point of the domain and all possible dimensions;
given Lempert's apt
remarks about the non-existence of
retracts of dimension any bigger than one, {\it in general}.
As our purpose is in recording the simplest cases first, the next couple of examples may well not be impressive.\\

\textit{Example:} Consider the unbounded balanced non-convex domain given by
$H^N := \{(z_1,\cdots,z_N) \in \mathbb{C}^N: |z_1\ldots z_N| < 1\}$. Since none of the boundary 
points of $H^N$ are convex, there are no one-dimensional retracts passing through the origin
except for the lines which are contained within one of the coordinate `planes' i.e., lines spanned
by vectors one of whose coordinates is zero. Such lines being completely contained within the domain, 
it follows that $H^N$ is a non-hyperbolic domain, in particular. Of-course $H^N$ is not biholomoprhic to 
$\mathbb{C}^N$ or any such domain which is `maximally non-hyperbolic' wherein the Kobayashi metric 
vanishes identically; thereby $H^N$ seems to serve as a nice first example of an intermediate case
with respect to hyperbolicity and in-fact, with respect to convexity as well.\\

\noindent   We now get to showing the existence of 
retracts of all possible dimensions through each point in this domain. 
To this end, pick any $c = (c_1, c_2, \ldots,c_N) \in H^N$ and any positive integer $k$ less than $N$;
we shall show that there exists a 
$k$-dimensional retract of $H^N$ passing through $c$. 
Consider the affine subspace 
$A = \{(z_1,\ldots,z_N)\in \mathbb{C}^N: z_j = c_j ~\text{for} ~k+1 \leq j \leq N\}$,
defined by setting an appropriate number of components equal to that of the fixed point $c$.
It suffices to prove that  
the $k$-dimensional affine subspace $H^N \cap A$ 
is a retract of $H$. First, we do this for $k=N-1$ and the
rest follows recursively as noted below. To this end, we let 
$L_{c_N} := \{(z_1,\ldots,z_{N-1},z_N) \in \mathbb{C}^N: z_N = c_N\} $
and show that the $(N-1)$-dimensional affine subspace $H^N \cap L_{c_N}$ of $H^N$ 
is a retract of $H^N$. We split into $2$ cases depending on $c_N$ being zero 
or not. If $c_N = 0$, then the linear projection map  
$\pi(z_1,\ldots,z_N) = (z_1,\ldots,z_{N-1},0)$ suffices to render 
${\rm image}(\pi) = H^N \cap L_{c_N}$ as a retract of $H^N$. If $c_N \neq 0$, then we consider the holomorphic 
map $\rho$ on $H^N$ defined by 
\[
    \rho(z_1,\ldots,z_N) = \left(\frac{1}{c_N}z_1z_N,z_2,\ldots,z_{N-1},c_N\right). 
\]
Let $\rho_1,\ldots,\rho_N$ be the components of $\rho$. Note that
\[
    |\rho_1 \cdots \rho_N| = \left|\frac{1}{c_N}z_1z_Nz_2\cdots z_{N-1}c_N\right| = |z_1\cdots z_N| < 1 
\]
So $\rho$ maps $H^N$ into $H^N$ and
\[
    \rho(\rho(z))  = \rho(\frac{1}{c_N}\rho_1\rho_N,\rho_2,\ldots,\rho_{N-1},c_N) 
                  = \left(\frac{1}{c_N}(\frac{1}{c_N}z_1z_Nc_N),z_2,\ldots,z_{N-1},c_N\right) = \rho(z)
\]
verifying that $\rho$ is a retraction map onto $H^N \cap L_{c_N}$ which is
biholomorphically -- indeed linearly -- equivalent to $H^{N-1}_{1/c_N}$ which in-turn is linearly equivalent 
to $H^{N-1}$ (in-case $c_N=0$, note that $H^N \cap L_{c_N}$ is just a copy of $\mathbb{C}^{N-1}$).\\

\noindent The same arguments as above, show that $H^{N-1} \cap L_{c_{N-1}}$ is either 
a holomorphic retract of $H^{N-1}$ or a holomorphic retract of $\mathbb{C}^{N-1}$ depending on whether $c_{N-1}$ is non-zero or not. 
By lemma \ref{retr-transit}, we obtain that  $\{(z_1,\ldots,z_{N}) \in \mathbb{C}^N: z_{N-1} = c_{N-1}, z_N = c_N\} \cap H^N$ 
is a retract of $H^N$. Continuing this process, we establish in $k$ steps, that $H^N \cap A$ 
is a $k$-dimensional retract of $H^N$.\qed \\

\noindent At this juncture, we can settle a question about whether {\it holomorphic} retraction mappings are 
automatically submersions so as (for instance) towards rendering hypotheses such as (ii) in theorem \ref{not-thro-origin} redundant in general.
We briefly address this by considering the two-dimensional version of the above domain (we note 
here that $D_L$ is not bounded 
as stipulated by
theorem \ref{not-thro-origin}).

\begin{thm}\label{H2}
Consider the retract $Z = \{(\zeta,1): \zeta \in \Delta\}$ of 
$H^2$ in $\mathbb{C}^2$. Then there does not exist any
retraction map from $H^2$ onto $Z$, which is submersive at origin.
\end{thm}

\begin{proof}
Let $\rho = (\rho_1,\rho_2)$ be any of the retractions which map $H^2$ onto $Z$. Since 
second component of any point in $Z$ is $1$, we have 
$\rho_2 \equiv 1$. No such drastic conclusion cannot be made for the first component of $\rho$ namely $\rho_1$,
as we only have to begin with that it is a holomorphic function 
which maps $H^2$ onto $\Delta$; nevertheless we may deduce about its form 
as a consequence of certain features of $H^2$.
To do so, consider first its restriction:
$f(w) := \rho_1(0,w)$.
Being an entire function which maps into $\Delta$, Liouville's theorem
implies that $f$ is a constant 
function. As $f(1) = 0$, that constant is zero, thereby  
$f(w) = \rho_1(0,w) = 0$ for all $w \in \mathbb{C}$.  
As $H^2$ is a (logarithmically convex) complete Reinhardt domain, 
$\rho_1$ can be expressed as a single power series expansion on $H^2$. This means that 
there exist scalars $c_{j,k}$ such that for all $(z,w) \in H^2$, 
\begin{equation*}
    \rho_1(z,w) = \sum_{j,k =0}^{\infty}c_{j,k} z^jw^k
\end{equation*}
As the series above converges absolutely on $H^2$, we may rewrite by grouping together the
terms as indicated next: 
 \begin{equation}\label{Rho}
     \rho_1(z,w) = P_1(z) + P_2(w) + P_3(z,w)
 \end{equation}
where 
\begin{equation*}
     P_1(z) = \sum_{j=1}^{\infty} c_{j,0}z^j,~ P_2(w) = \sum_{k=0}^{\infty} c_{0,k}w^k, ~
     P_3(z,w) = \sum_{j,k =1}^{\infty} c_{j,k}z^jw^k.
 \end{equation*}
 Since $\rho_1(0,w) = 0$ for all $w \in \mathbb{C}$, we obtain that $P_2 \equiv 0$. 
Similarly we can consider another entire function $g : \mathbb{C} \to \Delta$ defined by $g(z) = \rho_1(z,0)$. 
By the same arguments, we obtain that $P_1 \equiv 0$. Therefore from the above equation (\ref{Rho}), we arrive
at our aforementioned deduction about the form of the first component namely,
$\rho_1(z,w) = zwh(z,w)$ for some holomorphic 
function $h$ on $H^2$. We may now prove that $\rho$ is not submersive at origin
by directly writing down the expression for its Jacobian matrix, which at any point $(z,w)$
is if the form
\[
\begin{bmatrix}
    w(z\frac{\partial h_2}{\partial z} + h_2(z,w)) & z(w\frac{\partial h_2}{\partial w} + h_2(z,w)) \\ 
    0 &0
\end{bmatrix}
\]
It is immediate that for $(z,w)=(0,0)$ this is the zero matrix,
finishing the proof that none of the retraction maps onto $Z$ is not submersive at origin. 
\end{proof}
\begin{rem}
While the above theorem with a very similar proof holds for any of the hyperbolic retracts $Z^h$ of $H^2$
(to give that none of the retraction mappings onto $Z^h$ is a submersion at the origin), it 
is worth noting down a concrete example $Z$ as in the theorem for an amusing point. Namely, while 
linear subspaces (i.e., through origin) of a balanced domain which are obtained as images of 
non-linear retraction mappings are indeed obtainable as images of linear maps, this is not the
case with affine-linear retracts in general. Indeed, this can 
be seen as a corollary of the above theorem. While the retract $Z = \{(\zeta,1): \zeta \in \Delta\}$ of $H^2$ is 
an affine-linear subspace of $H^2$, there exists no affine linear projection mapping the domain onto $Z$
since such a map would be a retraction which is submersive everywhere (ruled out by the above theorem).
We also note that $Z$ gives a concrete example of a hyperbolic retract (being a copy of the disc) of the 
non-hyperbolic domain $H$ which therefore serves as a good intermediate between 
completely non-hyperbolic domains such 
as $\mathbb{C}^N$ which do not admit any hyperbolic retracts on the one hand, and, hyperbolic domains which admit only 
hyperbolic ones on the other. A more important point is noted
next.
\end{rem}

\noindent A direct consequence of the above theorem is that unless we are in very special settings
even within the class of balanced pseudoconvex domains, the hypothesis
(ii) along-with boundedness 
of $D_L$ as in theorem \ref{not-thro-origin} cannot be dropped. 
For a quick (counter-)example, observe that 
the retract $Z$ in the foregoing theorem \ref{H2} cannot 
be realized as the graph over any linear retract of $H^2$
as in the conclusion of theorem \ref{not-thro-origin}.
Indeed, the only (non-trivial) linear retracts of $H^2$ are the pair
of coordinate axes, on neither of which $Z$ can be
expressed as a graph, for then $Z$ would be biholomorphic 
to $\mathbb{C}$, manifestly false. It is 
also possible to give a non-affine example in
this simple domain itself
to show that there are retracts which are impossible to be realized as a graph over any affine linear subspace
of $H$, thereby showing (again) that the conclusion
of theorem \ref{not-thro-origin} can fail 
as soon as some of its hypotheses is dropped
as mentioned above, which is
the purpose of the following corollary.

\begin{cor}
Consider the retract $\tilde{Z} = \{(\zeta e^{-\zeta},e^{\zeta}): \zeta \in \Delta\}$ of 
$H^2$ in $\mathbb{C}^2$. Then there does not exist a 
retraction map from $H^2$ onto $\tilde{Z}$, which is submersive at the origin.
\end{cor}
\begin{proof}
Suppose to arrive at a contradiction that $\tilde{\rho} : H^2 \to H^2$ is a retraction map with 
${\rm image}(\tilde{\rho}) = \tilde{Z}$ and such that it is a submersion at the origin. 
Consider the automorphism $\psi : H^2 \to H^2$ defined by 
$\psi(z,w) = (ze^{-zw},we^{zw})$. Conjugating $\rho$ by $\psi$, 
we then have the holomorphic map 
$\rho : H^2 \to H^2$ given by $\rho(z,w) = \psi^{-1} \circ \tilde{\rho} \circ \psi(z,w)$. 
Note that $\rho$ is a retraction map on $H^2$ and 
${\rm image}(\rho) = Z = \{(\zeta,1): \zeta \in \Delta\}$ and we may apply theorem \ref{H2} to 
$\rho$. But then the submersivity of $\tilde{\rho}$ at the origin is equivalent to that 
of $\rho$ at the origin, as $\psi$ is not just a conjugating automorphism but also fixes at the origin; this
contradicts theorem \ref{H2}, finishing the proof.
\end{proof}

\noindent We next record a pair of bounded modifications of $H^2$ in the above examples wherein the 
some of the results hold as is, but some arguments cease to. More importantly, these give examples of 
bounded balanced non-convex (but pseudoconvex) domains wherein (i) for each point, there is
a (non-trivial) retract passing through it and (ii) they are equivalent to product domains but contain an open piece of directions (an open subset of the unit sphere $\partial\mathbb{B}$ representing directions) along which retract exist. \\

\textit{Example:} Consider the domain $\Omega := \Delta^2 \cap H^2_{1/4}$,
where $H^2_{1/4} =\{(z,w)\in \mathbb{C}^2: |zw| < 1/4\}$ (and $\Delta^2$ 
the standard bidisc, as usual).
We demonstrate below that for each point $(a,b) \in \Omega$, there exists a non-trivial retract 
of $\Omega$ passing through it.\\

\noindent Let $L$ denote the affine complex line $\{(z,w) \in \mathbb{C}^2: w=b\}$.
Observe that to attain our goal, it suffices to show that $L \cap \Omega$ is a retract of $\Omega$
for which purpose, we first decompose the boundary of our domain $\Omega$ as:
\[ 
    \partial \Omega = (\partial \Delta^2 \cap H^2_{1/4}) \cup 
    (\Delta^2 \cap \partial H^2_{1/4}) \cup (\partial \Delta^2 \cap \partial H^2_{1/4}).
\]
Note that $L \cap \Omega$ can be viewed as an open subset of $L$ and viewed thus, its boundary 
satisfies $\partial (L \cap \Omega) = L \cap \partial \Omega$. We split-up the rest of 
the proof into claims and cases. \\

\noindent \textit{Claim-1:} Either $\partial (L \cap \Omega) \subset 
 \partial \Delta^2$ or $\partial (L \cap \Omega) \subset \partial H^2_{1/4}$.\\
\textit{Proof:} Assume to get a contradiction that $\partial (L \cap \Omega) \nsubseteq 
\partial \Delta^2$ and $\partial (L \cap \Omega) \nsubseteq \partial H^2_{1/4}$. Then there are points $(a_1,b)$ and 
$(b_1,b)$ in $\partial(L \cap \Omega)$ such that 
$(a_1,b) \notin \partial \Delta^2$ and $(b_1,b) \notin \partial H^2_{1/4}$. 
From the above decomposition of $\partial \Omega$ and $(a_1,b), (b_1,b) \in \partial \Omega$,
we obtain that $(a_1,b) \in \Delta^2 \cap \partial H^2_{1/4}$ and $(b_1,b) \in \partial \Delta^2 \cap H^2_{1/4}$. 
This implies that $a_1,b_1,b$ satisfies the following inequalities:
\begin{eqnarray*}
    |a_1b| = 1/4, ~ \max\{|a_1|,|b|\} <1 \\
    |b_1b| < 1/4, ~ \max\{|b_1|,|b|\} = 1.
\end{eqnarray*}
Note that $|b|,|a_1| < 1$ because $\max\{|a_1|,|b|\} < 1$. Using the inequality 
$|b| < 1$ and the equality $\max\{|b_1|,|b|\} = 1$, we obtain $|b_1| = 1$. 
Since $|b_1b| < 1/4$, we have $|b| < 1/4$. Hence we get that
\[
    \frac{1}{4} = |a_1b| < \frac{1}{4}|a_1| 
\]
This implies that $|a_1| > 1$, which contradicts the fact that $|a_1| < 1$, finishing the proof of claim-$1$.\\

\noindent We are now in a position to bifurcate the rest of the arguments into two cases depending on whether 
$\partial (L \cap \Omega) \subset \partial \Delta^2$ or 
$\partial (L \cap \Omega) \subset \partial H^2_{1/4}$. The arguments in either
of these cases proceeds by fixing a point $(a',b)$ in $\partial (L \cap \Omega)$; it is also 
introduce the notations: $Z := L \cap \Delta^2$,
$W := L \cap H^2_{1/4} $. Note that $Z$ is bounded whereas $W$ could be unbounded, depending upon where $L$ 
intersects $\partial \Delta^2$. We first deal with the easier case.\\

\noindent \underline{Case-1:} $\partial(L \cap \Omega) \subset \partial \Delta^2$.\\

\noindent In this case, by the aforementioned decomposition of $\partial \Omega$, we first note that the 
boundary point $(a',b)$ lies in 
$\partial \Delta^2 \cap \overline{H^2_{1/4}}$. The map $\rho$ on $\Delta^2$ 
defined by $\rho(z,w) = (z,b)$ suffices to give a holomorphic 
retraction map on $\Delta^2$ with ${\rm image}(\rho) = Z$. However, we 
must check that $Z$ is a retract of $\Omega$, not just of $\Delta^2$. To do this of-course
we need only consider the restriction
$\rho_{|_{\Omega}}$, which maps $\Omega$ into $Z$ and that would finish the proof provided we verify the following.\\

\noindent \textit{Claim-2:} $Z = L \cap \Omega$.\\
\noindent \textit{Proof:} Firstly note that if $(z,b) \in Z$ 
then $|z| < |a'| = 1$ because $(a',b) \in \partial \Delta^2$ and $(a,b) \in \Delta^2$. 
If $(z,b) \in Z$, then $|zb| < |a'b|$. Since $(a',b) \in 
\overline{H^2_{1/4}}$, we get that $|zb| < 1/4$. Hence $Z \subset  L \cap \Omega$. 
Since $\Omega \subset \Delta^2$, we have $L \cap \Omega \subset Z$. This finishes the proof of the claim. \\

\noindent It now follows immediately that $\rho$ maps $\Omega$ into itself 
and ${\rm image}(\rho_{|_{\Omega}}) \subset L \cap \Omega = Z$; indeed, as we have for every $z \in Z = L \cap \Omega$ that $\rho(z) = z$, 
we have ${\rm image}(\rho_{|_{\Omega}}) = L \cap \Omega$.
This finishes the proof completely that $L \cap \Omega$ is a retract of $\Omega$ passing through $(a,b)$ in the above case.\\

\underline{Case-2:} $\partial(L \cap \Omega) \subset \partial H^2_{1/4}$.\\

\noindent In this case, again by the aforementioned decomposition of $\partial \Omega$, we start by noting that the 
boundary point $(a',b)$ lies in  
$\overline{\Delta^2}  \cap  \partial H^2_{1/4}$. Note that $b \neq 0$ as $(a',b) \in \partial H^2_{1/4}$.
Now in contrast to the previous case, no linear choices work (the proof of which is easy but we shall not digress into that 
here), thereby leading us to consider the non-linear map 
$\varphi: H^2_{1/4} \to H^2_{1/4}$ defined by $\varphi(z,w) = (\frac{1}{b}zw,b)$.
which is of-course a holomorphic retraction map on $H^2_{1/4}$ with ${\rm image}(\varphi) = W$. 
It remains to realize $W$ as a retraction of our smaller bounded 
$\Omega$, rather than the unbounded $H^2_{1/4}$. Again we need only consider the restriction $\varphi_{|_{\Omega}}$
(which maps $\Omega$ into $W$), provided we verify the following.\\

\noindent \textit{Claim-3:} $W = L \cap \Omega$.\\

\noindent \textit{Proof:} If $(z,b) \in W$ then $|z| < |a'|$, because if $|z| \geq |a'|$ then $|zb| \geq |a'b| = 1/4$, 
which contradicts the fact that $(z,b) \in W$. Since $(a',b) \in \overline{\Delta^2}$, we have $|z| < 1$. 
Hence $(z,b) \in L \cap \Delta^2$ and therefore $W \subset L \cap \Omega$. 
As $\Omega \subset H^2_{1/4}$, we have $L \cap \Omega \subset W$, finishing the proof of the claim.\\

\noindent So, $\varphi$ maps $\Omega$ into $\Omega$ and the proof that ${\rm image}(\varphi) = L \cap \Omega$ is a non-trivial retract of 
$\Omega$ passing through $(a,b)$ is now completed in all cases.

\begin{rem}
Consider the retract $Z = \{(\zeta,b): |\zeta| < \frac{1}{4|b|}\}$ of $\Omega$ as above. We observe that 
$Z$ can be expressed as a graph over a linear {\it retract}
(through the origin) unlike the situation
in the unbounded version, as follows.
Fix any $z_0$ with $\vert z_0 \vert \leq 1/4$ and let $L = \{\lambda(z_0,1): \lambda \in \mathbb{C}\}$,
the linear subspace of $\mathbb{C}^2$ spanned by the boundary point $(z_0,1)$. If we let $\phi$ be the
the holomorphic map given by $\phi(\lambda) = (\frac{1}{4b}\lambda^2,b)$, then
it is straightforward to check that $\phi$ maps $\Delta$ into $H^2$, indeed onto $Z$,
owing to the chosen constraint on $z_0$. 
If we view $\mathbb{C}^2$ as the direct sum of $L$ with the span of $(0,b)$ (indeed, the $z$-axis!),
and define $\varphi$ on $L$ by $\varphi\left( \lambda(z_0,1) \right) = \phi(\lambda)$,
we observe then that: $Z$ is in-fact the 
graph of   
$\varphi$ over the linear subspace $\Omega_L = \Omega \cap L$.
An interesting trivia to note here that when we choose $z_0=0$, 
we 
obtain the `horizontal' retract $Z$ as the graph over the `vertical' subspace obtained
by intersecting the $w$-axis with the domain i.e., the same $L$ when $z_0=0$. At the other extreme i.e.,
when $\vert z_0 \vert=1/4$, we observe that $\varphi$ is just the restriction of the
retraction map $f(z,w)=(zw/b, b)$. In any case, we can neither choose $f$ nor
$\varphi$ to be affine-linear owing to theorem \ref{H2}.
\end{rem}

\noindent As another difference from the unbounded $H^2$, thereby
a purpose of the above $\Omega$, is that the sets: $S_1$ of 
unit vectors representing directions along which there
are linear retracts in $\Omega$, as well as, $S_2$
of those unit vectors along which there are no non-trivial
retracts through the origin, are both non-trivial subsets
of the unit sphere
neither of which have measure zero (and with both 
having non-empty interior in the sphere).
Finally here, we also mention that there are non-linear 
retracts along the directions given by $S_1$ through
the origin. Thus there exist bounded balanced 
pseudoconvex domains which are not biholomorphic
to product domains but still have non-linear 
retracts through the origin. Indeed, we may 
establish $D:=\Delta^2 \cap H^{1/2}$ is irreducible as
follows. Assuming the contrary, we observe that it 
would follow by the Riemann mapping theorem that
$D$ is biholomorphic to the bidisc $\Delta^2$. 
The homogeneity of $\Delta^2$ together with 
an application of Cartan's theorem then renders
a linear equivalence between $\Delta^2$ and $D$
which is impossible as the former is convex
whereas the latter is not.\\

\noindent We may ask for variants of the above $H^2_{1/4} \cap \Delta^2$ and more; we shall
only record the case when the polydisc $\Delta^2$ is replaced by the Euclidean ball $\mathbb{B}$.\\ 

\textit{Example:} Consider the domain $\Omega := \mathbb{B} \cap H^2_{1/4}$,
where $\mathbb{B} := \{(z,w) \in\mathbb{C}^2: |z|^2 + |w|^2 < 1\}$ 
be the euclidean ball centred at $0$ of radius $1$ and $H^2_{1/4} =\{(z,w)\in \mathbb{C}^2: |zw| < 1/4\}$ .\\
We want to prove that for each $(a,b) \in \Omega$, there exist a  retract 
(of dimension 1) of $\Omega$ passing through $(a,b)$.\\

Let $L$ denotes the hyperplane $\{(z,w) \in \mathbb{C}^2: w=b\}$ . 
Firstly note that 
\[ 
    \partial \Omega = (\partial \mathbb{B} \cap H^2_{1/4}) \cup 
    (\mathbb{B} \cap \partial H^2_{1/4}) \cup (\partial \mathbb{B} \cap \partial H^2_{1/4})
\]
Note that $\partial (L \cap \Omega) = L \cap \partial \Omega$ 
.
 First we prove the following claim:\\

 \textit{Claim-1:} Either $\partial (L \cap \Omega) \subset 
 \partial \mathbb{B}$ or $\partial (L \cap \Omega) \subset \partial H^2_{1/4}$ for $i = 1,2$.\\

\textit{Proof of the claim-1:} If possible assume that $\partial (L \cap \Omega) \nsubseteq 
\partial \mathbb{B}$ and $\partial (L \cap \Omega) \nsubseteq \partial H^2_{1/4}$ . Then there exist points $(a_1,b)$ and 
$(b_1,b)$ in $\partial(L \cap \Omega)$ such that 
$(a_1,b) \notin \partial \mathbb{B}$ and $(b_1,b) \notin \partial H^2_{1/4}$. 
From the above decomposition of $\partial \Omega$ and $(a_1,b), (b_1,b) \in \partial \Omega$, we obtain that $(a_1,b) \in \mathbb{B} \cap \partial H^2_{1/4}$ and $(b_1,b) \in \partial \mathbb{B} \cap H^2_{1/4}$. 
This implies that $a_1,b_1,b$ satisfies the following equations:
\begin{eqnarray*}
    |a_1|^2 + |b|^2 < 1,~ |a_1b| = 1/4\\
    |b_1|^2 + |b|^2 =1,~ |b_1b| < 1/4
\end{eqnarray*}
 Dividing $|a_1b|$ (which is equal to $1/4$) on both sides of the equation $|b_1b| < 1/4$, we get $|b_1| < |a_1|$.
 Squaring and adding $|b|^2$ on both sides of the equation 
 $|b_1| < |a_1|$, we get that $1 = |b_1|^2 + |b|^2 < |a_1|^2 + |b|^2 < 1$, which is vividly a contradiction.
This finishes the proof of the claim.\\

So we can divide into two cases depending on whether 
$\partial (L \cap \Omega) \subset \partial \mathbb{B}$ or 
$\partial (L \cap \Omega) \subset \partial H^2_{1/4}$. 
\\

Fix a point $(a',b)$ in $\partial (L \cap \Omega)$ and let $Z := L \cap \mathbb{B}$,
$W := L \cap H^2_{1/4}$.\\

\underline{Case-1:} $\partial(L \cap \Omega) \subset \partial \mathbb{B}$.\\

 Since $(a',b) \in \partial \Omega$ and the above 
decomposition of $\partial \Omega$, we obtain that $(a',b) \in 
\partial\mathbb{B} \cap \overline{H^2_{1/4}}$. 
We know that $Z$ is an  retract of $\mathbb{B}$. Then there exist a holomorphic retraction map $\rho$ on $\mathbb{B}$ such that $\rho(\mathbb{B}) = Z$.
 Now we prove that $Z$ is  a retract of $\Omega$. To prove this, consider the map 
 $\rho_{|_{\Omega}}$, which maps $\Omega$ into $Z$. Next we prove the following claim:\\

 \textit{Claim -2:} $Z = L \cap \Omega$.\\

 \textit{Proof of claim-2:} Note that $|a'| \leq \frac{1}{4|b|}$, because if $|a'| > 
 \frac{1}{4|b|}$ then $|a'b| > 1/4$ and thereby implies that $(a',b) \notin \overline{H^2_{1/4}}$, which contradicts our assumption that $(a',b) \in \partial(L \cap \Omega) \subset \partial \mathbb{B}$. 
 If $(z,b) \in Z$ then $|z| < |a'|$ because if $|z| \geq |a'|$, then $|z^2| + |b|^2 \geq |a'|^2 + |b|^2 = 1$, which contradict the fact that $(z,b) \in \mathbb{B}$. 
 Note that $|zb| < |a'b| \leq 1/4$. 
 This implies that  $(z,b) \in L \cap \Omega$ and hence 
 $Z \subset L \cap \Omega$. 
 Since $\Omega 
 \subset \mathbb{B}$, we have $L \cap \Omega \subset Z$, . This finishes the proof of the claim. \\

 Using the above claim-2, $\rho$ maps $\Omega$ into itself and $image(\rho_{|_{\Omega}}) \subset L \cap \Omega = Z$. 
 If $z \in Z = L \cap \Omega$, then $\rho(z) = z$. So 
 $image(\rho_{|_{\Omega}}) = L \cap \Omega$.
  Hence $L \cap \Omega$ is a retract of $\Omega$ passing through $(a,b)$.\\

  \underline{Case-2:} $\partial(L \cap \Omega) \subset \partial H^2_{1/4}$.\\

Firstly note that $b \neq 0$ because $(a',b) \in \partial H^2_{1/4}$.
  Consider the holomorphic map $\varphi: H^2_{1/4} \to H^2_{1/4}$ defined by $\varphi(z,w) = (\frac{1}{b}zw,b)$.\\
  Note that $\varphi$ is a retraction map on $H^2_{1/4}$. Now we consider the map  $\varphi_{|_{\Omega}}$, which maps $\Omega$ into $W$. 
Next we prove the following claim:\\

\textit{Claim-3:} $W = L \cap \Omega $.\\

\textit{Proof of claim-3:} First we prove that $|a'| \geq \frac{1}{4|b|}$. If 
$|a'| < \frac{1}{4|b|}$ then $|a'b| < 1/4$ and thus implies that 
$(a',b) \in H^2_{1/4}$, which contradicts the fact that $(a',b) \in \partial H^2_{1/4}$ .  Hence we proved that $|a'| \geq \frac{1}{4|b|}$\\

  If $(z,b) \in W$, then $|z| < \frac{1}{4|b|}$. 
 Note that 
 \[
     |z|^2 + |b|^2 < \left(\frac{1}{4|b|}\right)^2 + |b|^2 \leq |a'|^2 + |b|^2.
 \]
 Since 
 $(a',b) \in (L \cap \partial \Omega) \subset \partial H^2_{1/4}$ and the above decomposition of $\partial \Omega$, we obtain that $(a',b) \in \overline{\mathbb{B}}$. Hence $|z|^2 + |b|^2 < 1$ and 
 therefore we get that $W \subset L \cap \Omega$. 
 Since $\Omega \subset D$, we have $L \cap \Omega \subset W$. This finishes the 
 proof of the claim.\\

Hence $\varphi$ maps $\Omega$ into $\Omega$.
Therefore $image(\varphi) = L \cap \Omega$ is a one-dimensional retract of $\Omega$ passing through $(a,b)$.
\\

\noindent We conclude this section by noting that the $\ell_q$
ball in $\mathbb{C}^N$ is not a Lempert domain, the 
class of domains considered in the very recent article
\cite{GZ}, since a balanced pseudoconvex domain is a Lempert domain if and only if it is convex, as
verified below. Indeed, recall that a taut domain $D \subset \mathbb{C}^N$ is a \textit{Lempert domain} if
the identity $c_D = \ell_D$ holds.

\begin{prop}\label{Lemp_cnvx}
Let $D$ be a balanced pseudoconvex domain in $\mathbb{C}^N$. Then $D$ is 
a Lempert domain if and only if $D$ is convex.
\end{prop}

\begin{proof}
Assuming first that $D$ is a Lempert domain, we have $c_D = \ell_D$. 
By proposition 3.1.11 in \cite{Jrncki_invrnt_dst}, we have $p(0,h(z)) = \ell_D(0,z) = \tanh^{-1}(h(z))$.
Thereby,
\[
    c_D(0,z) = \tanh^{-1}(c_D^*(0,z)) =  \tanh^{-1}(h(z)) = \ell_D(0,z).
\]
This implies that $c_D^*(0,z) = h(z)$ for all $z \in D$. By proposition 2.3.1 
in \cite{Jrncki_invrnt_dst}, $D$ is convex.\\
The converse is the famous theorem of Lempert (exposited as theorem 11.2.1 
in \cite{Jrncki_invrnt_dst}).
\end{proof}

\section{Proof of theorems \ref{converse-of-Vesentini} and \ref{Charac-retract-1diml}}\label{Vesentini}

\noindent We begin by connecting the notions of convexity as used in 
the statement of the theorem, usual geometric convexity and continuity of 
the Minkowski functional -- throughout this section $D$ will stand for a bounded 
balanced domain of holomorphy in $\mathbb{C}^N$ without further mention; we may only mention
additional assumptions if made. This will be done in the statement and in course
of the proof of the following lemma. Though this is not really a new 
result, we include the proof not only to keep it self-contained but also because
it is crucial for what follows.

\begin{lem} \label{cty-Minkowfnal}
If $\partial D$ is convex near $p$ then the Minkowski functional $h$ of $D$
is continuous near $p$. 
\end{lem}

\begin{proof}
Let us firstly quickly note that if $D$ is convex  near $p$ in the sense of definition \ref{C-convex}, then 
$D$ is geometrically convex near $p$ i.e., there exist an open ball $U$
in $\mathbb{C}^N$ centered at $p$ such that $U \cap D$ is convex in the usual sense of being 
closed under the formation of line segments. Indeed, this follows by considering the closed half-spaces
(containing the domain) given by the linear functionals as per definition \ref{C-convex} due to the points of
$\partial D \cap U$ being convex boundary points; supporting hyperplanes (and likewise half-spaces
containing $D \cap U$) at other points of $\partial(D \cap U)$ exist just by convexity of the ball $U$.
As the intersection of these closed half-spaces precisely equals $\overline{U \cap D}$, the convexity of
$U \cap D$ follows.\\

\noindent Let $C := \bigcup_{z \in \partial D \cap U}\{\lambda z: \lambda \geq 0\}$ and 
$C_D := C \cap D$. As $U \cap D$ is geometrically convex, so is the cone generated by it, thereby assuring that $C_D$ is convex. Next, we show the subadditivity of $h$ on $C$.  Let $z,w \in C_D$. By definition 
of the Minkowski functional, for each $\epsilon > 0$ there exist positive reals $t,s$ such that 
$z/t, w/s \in D$ and $t < h(z) + \epsilon$, $s < h(w) + \epsilon$.
Since $C_D$ is open, there exist $t',s'$ such that $t' < t, ~s' < s$ and $z/t', w/s' \in C_D$. Writing $(z+w)/(s'+t')$ as the convex combination:
\[
    \frac{z+w}{s'+t'} = \frac{t'}{s'+t'} \cdot \frac{z}{t'} + \frac{s'}{s'+t'} \cdot \frac{w}{s'},
\]
it follows that by the convexity of $C_D$,
that $(z+w)/(s+t) \in C_D$. Hence $h(z+w) \leq s'+t' < h(z) + h(w) + 2\epsilon$.
As $\epsilon$ is arbitrary, we get $h(z+w) \leq h(z) + h(w)$ for all $z,w \in C_D$. This shows that $h$ is sub-additive on $C_D$. To prove the subadditivity of $h$ on $C$, let $z,w$ be any two points in $C$. Choose $k$ so that $k > \max\{h(z),h(w)\}$. Note that $z/k,w/k \in C_D$. Applying the subadditivity of $h$ on $C_D$ and homogenity of $h$, we obtain that $h(z/k+w/k) \leq 1/k~ h(z) + 1/k~ h(w)$. This shows that $h$ is subadditive on $C$. The subadditivity of $h$ on $C$ and together with homogenity of $h$ yields that $h$ is a convex function on $C$. Hence $h$ is continuous on $C$ and in particular $h$ is continuous near $p$. 
\end{proof}

\noindent The converse of the above lemma is false. Indeed, there are plenty of (counter-) examples,
the simplest of which was given in the remark \ref{rem-cnvrse-Vesen} of the introduction;
another explicit example is obtained be examining the non-smooth boundary points of the
domain called $D_h$ in section \ref{Irreducibity} below; it is not difficult to construct 
families of further examples in plenty.\\

{\it Proof of theorem  \ref{converse-of-Vesentini}}: 
We shall work with the more general and weaker 
assumption that $\partial D$ is convex {\it at} the point $p$ together-with the
continuity of the Minkowski functional $h$ of $D$; the other assumption that 
$p$ is {\it not} a $\mathbb{C}$-extremal boundary point as in the statement 
of theorem \ref{converse-of-Vesentini} remains in force. 
Just by the hypotheses that $D$ is balanced, we may describe $D$ as 
$D = \{z \in \mathbb{C}^N: h(z) < 1\}$.
As $p \in \partial D$ is not $\mathbb{C}$-extremal, there exist $\epsilon_0 > 0$ 
and a non-zero $u \in \mathbb{C}^N$ such that 
$p+tu \in \partial D$ for all $t\in \mathbb{C}$ with $|t| < \epsilon_0$. 
After a rescaling of $u$, we may assume $\epsilon_0 = 1$ 
and also (by the hypotheses) that $h$ is continuous at $p+tu$ for all $t \in \Delta$.\\

\noindent Consider the holomorphic map $\varphi: \Delta \to \mathbb{C}^N$ defined by
$\varphi(t) = (p+tu)t$. As $p+tu \in \partial D$ for all $t \in \Delta$ and $h$ 
is continuous at $p+tu$ for all $t \in \Delta$, we have by a well-known 
fundamental fact (see proposition \ref{Prop_Minkowski} (d)) that $h(p+tu) = 1$ for all $t \in \Delta$. Hence  
\[
h(\varphi(t)) = h((p+tu)t) = |t|h(p+tu) = |t| < 1.
\]
So, $\varphi$ maps $\Delta$ into $D$. Next we want to prove that $\varphi$ is a complex $C$-geodesic in $D$. 
First we claim that there exist $t_0 \in \Delta^*:=\Delta \setminus \{0\}$ such that $D$ is convex at $p+t_0u$. 
By the definition of convexity of $\partial D$ at $p$, there exist a $\mathbb{C}$-linear functional $L$ such that
\[
    |L(p)| = h(p) = 1~~~\text{and}~~ |L(z)| \leq h(z) \;\; \text{ for all } z \in \mathbb{C}^N
\]
To prove the aforementioned claim, it suffices to show that there exist $t_0 \in \Delta^*$ such that
\begin{equation}\label{cnvx}
    |L(z)| \leq h(z)~~ \text{and}~~|L(p+t_0u)| = h(p+t_0u) = 1.
\end{equation}
Since $\mathbb{C}^N = \ker(L) \oplus span\{p\}$, we can write $u =u_1 +u_2$, where $u_1 \in \ker(L)$ and $u_2 \in span\{p\}$.
Write $u_2 = \lambda p$ for some $\lambda \in \mathbb{C}$. So
\[ 
    |L(p+tu)| = |L(p) +tL(u)| = |L(p) + t\left(L(u_1) + L(u_2)\right)| = |1+\lambda t||L(p)| = |1+\lambda t|.
\]
We need only prove that there exist $t_0 \in \Delta^*$ such that $|1+\lambda t_0| = 1$,
which we do now. If $\lambda = 0$, then $u \in \ker(L)$. Hence (\ref{cnvx}) holds 
for any $t_0 \in \Delta^*$. Assume that $\lambda \neq 0$.
Write $\lambda t = re^{i\theta}$ where $\theta \in [-\pi,\pi)$. From this we get $t = \frac{1}{\lambda} (re^{i\theta})$. The relations $|1+\lambda t| = 1$ and $|t| <1$ becomes $|1+re^{i\theta}| = 1$ and $|re^{i\theta}| < |\lambda| $. We want to prove that there exist $(r,\theta)$ such that 
\[
    |re^{i\theta} + 1| =1 ~~ \text{ and } ~~~r < |\lambda|.
\]
Geometrically above equation represents the intersection of circle
centred at $-1$ with radius $1$ and the disc centred at $0$ of radius $|\lambda|$.
So there exist $t_0 \in \Delta^*$ such that $|1+\lambda t_0| = 1$. This proves the claim.\\

\begin{rem}
The arguments above can actually be used to show that $\partial D$ is convex at all points 
near $p$ in the complex line segment through $p$ in the direction of $u$. 
\end{rem}

To proceed further with the proof that $\varphi$ is a complex $C$-geodesic, let
us now note that as $D$ is convex at the boundary point $p+t_0u$  by the foregoing; thereby also 
at $(t_0/\vert t_0 \vert) (p+t_0u)$.
We then have by proposition 2.3.1(b) in \cite{Jrncki_invrnt_dst}, with the 
same notations therein, that
\[
    c_D^*(0,(p+t_0u)t_0) = h((p+t_0u)t_0) = |t_0|h(p+t_0u) = \vert t_0 \vert;
\]
for the definition of $c_D^*$ refer to section \ref{Car}. Recall also the definition of the 
Caratheodory distance $c_D$ and its relationship to $c_D^*$: $c_D = \tanh^{-1}c_D^*$.
So 
\[
    c_D(0,(p+t_0u)t_0) = \tanh^{-1}\left(c_D^*(0,(p+t_0u)t_0)\right) = \tanh^{-1}|t_0|
\]
For the term at the left note that: $c_D(\varphi(0),\varphi(t_0)) = c_D(0,(p+t_0u)t_0)$;
the term on the right can just be realized as the Poincare distance between the $t_0$ and 
the origin in the disc $\Delta$. Indeed if $m(a,b)$ denotes the Mobius distance between two points $a,b \in \Delta$,
recall that $p(a,b) = \tanh^{-1}(m(a,b))$. 
Therefore, $c_D(\varphi(0),\varphi(t_0)) = p(0,t_0)$.
By proposition 11.1.4 in \cite{Jrncki_invrnt_dst}, we have 
$c_D(\varphi(\lambda),\varphi(\mu)) = p(\lambda,\mu)$ for all $\lambda,\mu \in \Delta$. 
Hence $\varphi$ is a complex $C$-geodesic on $D$. It follows that $Z = {\rm image}(\varphi)$ is a one-dimensional retract 
of $D$. To finish, let us check that $Z$ is indeed non-linear. Assume the contrary that $Z$ is linear. 
Let $b$ be a non-zero element in $Z$; it has got to be of the form: $b = (p+t_0u)t_0 \in Z$ for some $t_0 \in \Delta^*$. 
As $Z$ is linear, $\frac{1}{2}(p+t_0u)t_0 \in Z$.  This implies that
\begin{equation}\label{C-extrm}
    \frac{1}{2}(p+t_0u)t_0 = (p+s_0u)s_0,
\end{equation}
for some $s_0 \in \Delta^*$.
So, $h(\frac{1}{2}(p+t_0u)t_0) = h(s_0(p+s_0u))$ and this implies that $|t_0|h(p+t_0u) = 2|s_0|h(p+s_0u)$.
Since $p+tu \in \partial D$ for all $t \in \Delta$ and $D$ is convex, we have $|t_0| = 2|s_0|$.
We can write $s_0 = (t_0/2)e^{i\theta}$ for some $\theta \in [-\pi,\pi)$. From (\ref{C-extrm}), we get
\[
    (p+t_0u)t_0 = t_0 e^{i\theta}(p+\frac{t_0}{2}ue^{i\theta}).
\]
Since $t_0 \neq 0$, the above equation becomes
\begin{equation}\label{C-extrm1}
    p+t_0u = e^{i\theta}p + \frac{e^{i2\theta}t_0}{2}u.
\end{equation}
We claim $p,u$ are linearly independent. To prove this, assume the contrary: $p = \lambda u$ for some $\lambda \in \mathbb{C}$.
Applying  $h$ on both sides, we get
\[
h(p) = |\lambda|h(u).
\]
By the hypotheses that $D$ is bounded and from proposition \ref{Prop_Minkowski} (f), we obtain that $h(u) \neq 0$.
Hence we get $|\lambda| = 1/h(u)$. We now split into two cases depend upon whether $h(u) > 1$ or not.\\

\noindent If $h(u) \leq 1$, then $p+\frac{1}{2\overline{\lambda}} u  \in \partial D$. Since $p = \lambda u$, we have 
\begin{equation}\label{C-extrm2}
p+\frac{1}{2\overline{\lambda}} u = \lambda u + \frac{1}{2\overline{\lambda}} u = 
\left(\frac{2|\lambda|^2+1}{2\overline{\lambda}}\right)u.
\end{equation}
Applying $h$ on the above equation (\ref{C-extrm2}), we get 
\[
    1= h\left(p+\frac{1}{2\overline{\lambda}}u\right) = \left(\frac{2|\lambda|^2+1}{2|\lambda|}\right) h(u) = \frac{2+(h(u))^2}{2}
\] 
Comparing the left-most and the right-most terms in this equation, we obtain $h(u) = 0$
which is a contradiction.\\

\noindent If $h(u) >1$, then $p+\lambda u \in \partial D$. From $p = \lambda u$, we get that $p+\lambda u = 2\lambda u$. Applying $h$ on both sides, we get
\[
1 = h(p+\lambda u) = 2|\lambda|h(u) = 2,
\]
vividly a contradiction.\\

\noindent This establishes the linear independence of 
$p$ and $u$. Now we get back to finishing the proof that $Z$ is non-linear, to which
end, we get back to the equation (\ref{C-extrm1}). Since $p,u$ are linearly independent, we obtain
in-particular, $ e^{i2\theta}/2 = 1$. As such a $\theta$ does not exist, this
concludes our proof by contradiction to establish that $Z$ is a non-linear retract of $D$.\\

An alternative proof of this theorem is provided below without the use of complex geodesics. \\

First, we consider the case that $\partial D$ is convex near $p$. As $\partial D$ is not $\mathbb{C}$-extremal at $p$, there exists a non-zero holomorphic map $\psi : \Delta \to F_{\mathbb{C}}(p,\partial D)$. 
After rescaling $\Delta$, if needed, we may assume that $\psi(t)$ is a non-vanishing holomorphic map on $\Delta^*$. 
We want to show that there exists $\delta > 0$ such that $p+t\psi(t) \in \partial D$ for all $t \in \Delta_{\delta}$, where $\Delta_{\delta}$ denotes the disk centred at the origin of radius $\delta$.  \\

Now, we claim that there exists $\delta > 0$ such that, for every $u \in F_{\mathbb{C}}(p,\partial D) \cap S$, we have $p +su \in \partial D$ for all $s \in \Delta_{\delta}$, where $S = \partial \mathbb{B}_{\ell_1}$ denote the unit sphere in the $\ell_1$ norm.
To prove this, let $\text{co}(D)$ denote the convex hull of $D$. Note that $ F_{\mathbb{C}}(p,\partial D) = F_{\mathbb{C}}(p,\partial (\text{co}(D)))$ because $D$ is convex near $p$. For any $u \in F_{\mathbb{C}}(p,\partial (\text{co}(D))) \cap S$ can be written as a convex combination of vectors $u_j \in F_{\mathbb{C}}(p,\partial (\text{co}(D))) \cap S$. 
i.e. $u = c_1u_1 + \ldots + c_ku_k$, where $c_j$'s are non-negative reals with $c_1 + c_2 +\ldots + c_k = 1$.
Since each $u_j \in F_{\mathbb{C}}(p,\partial (\text{co}(D)))$, there exists $\epsilon_j > 0$ such that $p+su_j \in \partial \text{co}(D)$ for all $s \in \Delta_{\epsilon_j}$. 
Let $\delta := \min \{\epsilon_1,\ldots,\epsilon_k\}$. Then for each $s \in \Delta_{\delta}$, 
\[
p+su = p+s(c_1u_1 +\ldots + c_ku_k) = c_1(p+su_1) + \ldots + c_k(p+su_k)
\]
Each term $p +su_j \in \overline{\text{co}(D)}$, and since $\overline{\text{co}(D)}$ is convex, it follows that $p+su \in \overline{\text{co}(D)}$ for all $s \in \Delta_{\delta}$. 
Since $D$ is convex near $p$, we may assume that $D$ is convex at $p+su$ for all $s \in \Delta_{\delta}$. By proposition 2.3.1(b) in \cite{Jrncki_invrnt_dst}, we have $\gamma_D(0;p+su) = h(p+su)$. Applying proposition 2.3.1(d) in \cite{Jrncki_invrnt_dst}, we obtain that $h(p+su) = \widehat{h}(p+su) \leq  1$, where $\widehat{h}$ is the Minkowski function of $\text{co}(D)$. This implies that $p+su \in \overline{D}$ for all $s \in \Delta_{\delta}$.
From the proof of proposition \ref{strong-weak-extrm}, we obtain that  $p+su \in \partial D$ for all $s \in \Delta_{\delta}$, which finishes the proof of the claim.\\  

Since $F_{\mathbb{C}}(p,\partial D)$ is closed under scalar multiplication, for each $t \in \Delta^*$, we have 
        \[
            \frac{\psi(t)}{\|\psi(t)\|} \in F_{\mathbb{C}}(p,\partial D) \cap S.
\]
Since $\phi(t)= t \|\psi(t)\|$ is continuous on $\Delta^*$, there exists $\delta_0 > 0$ such that $\|t \psi(t)\| \leq \delta$ for all 
$t \in \Delta_{\delta_0}$. 
Hence by the above claim, for all $t \in \Delta_{\delta_0}^*$, we have
\[
    p + t\psi(t) = p + t \|\psi(t)\| \left(\frac{\psi(t)}{\|\psi(t)\|}\right) \in \partial D
\]
 Define the holomorphic map $\tilde{\varphi} : \Delta \to \mathbb{C}^N$ by $\tilde{\varphi}(t) = p+\delta t\psi(\delta t)$. 
Now, we consider the map $\varphi : \Delta \to D$ defined by $\varphi(t) = t \tilde{\varphi}(t)$. 
Hence $Z := \text{image}(\varphi)$ is the graph of the holomorphic map $\phi(t) = \delta t^2\psi(\delta t)$ over the linear subspace $\spn(p) \cap D$ of $D$. 
By proposition  \ref{Graph_retract}, it follows that $Z$ is a retract of $D$.\\

Next, we assume the weaker and general assumption that $\partial D$ is convex at
$p$, together with the Minkowski functional of $D$ being continuous near $p$.  Choose any holomorphic map $\tilde{\varphi} : \Delta \to (p + F_{\mathbb{C}}(p,\partial D)) \cap \partial D$ , since $\partial D$ is not $\mathbb{C}$-extremal at $p$, there exists atleast one such holomorphic map $\tilde{\varphi}$. 
So, we write $\tilde{\varphi}(t) = p + \tilde{\psi}(t)$, where $\tilde{\psi} \in \mathcal{O}(\Delta)$ with $\tilde{\psi}(0) = 0$ and $\tilde{\psi}(t) \in F_{\mathbb{C}}(p,\partial D)$ for all $t \in \Delta$.
Now, we consider the holomorphic map $\varphi: \Delta \to \mathbb{C}^N$ defined by
$\varphi(t) = t \tilde{\varphi}(t)$. In the above proof, we showed that $\varphi$ maps $\Delta$ into $D$. Note that $\text{span}(p) \cap \text{image}(\tilde{\psi}) = \{0\}$ because if $u$ is a non-zero vector in $F_{\mathbb{C}}(p,\partial D)$ then $p$ and $u$ are linearly independent (as shown in the above proof; look at the paragraph following equation(\ref{C-extrm1})). This implies that $Z = \text{image}(\varphi)$ is a graph of the holomorphic map $\phi(t) = t\tilde{\psi}(t)$ over the linear retract $\text{span}(p) \cap D$ of $D$. 
By theorem  \ref{Graph_retract}, it follows that $Z$ is a retract of $D$.

\qed

\begin{rem}
\begin{enumerate}
\item[(i)] If $N=2$, then we can replace convex near $p$ by convexity at $p$ because if $D$ is convex at
$p$ then $D$ is convex at $p+tu$ for all $t \in \Delta$.

\item[(ii)] The above proof gives an explicit (quadratic) retract. Of course, 
retracts of higher degree as well as other non-linear retracts, may be obtained by taking
$\varphi(t) = tp+t\tilde{\psi}(t)u$, where $\tilde{\psi}$ is an arbitrary holomorphic self-map of $\Delta$
fixing the origin. We note that after factoring out $t$, the resulting map 
$t + \tilde{\psi}u$ takes values in $\partial D$ i.e., gives an analytic disc in the boundary $\partial D$.
\end{enumerate}
\end{rem}
\noindent The next proposition establishes that this exhausts all possible examples of non-linear
(one-dimensional) retracts through the origin in the direction of a boundary point that fails to be
$\mathbb{C}$-extremal. Indeed, this proposition is an important special case of 
theorem \ref{Charac-retract-1diml} whose proof will be based on this and we therefore deal with it first. Moreover, this gives finer information about the one-dimensional retracts than theorem \ref{Charac-retract-1diml}, which is an assertion about retracts of all dimensions. Before we get to the details, it is very helpful to introduce a notation
for brevity:
if $a =(a_1,\ldots,a_k)$ is any point in $\mathbb{C}^k$ (for some $k \in \mathbb{N}$)
and each of $v_1,v_2,\ldots,v_k$ are 
vectors in $\mathbb{C}^N$, then
we define $a \cdot v = a_1v_1 + a_2 v_2 + \ldots + a_k v_k$; indeed this is the `$a$-linear'
combination of the vectors $v_1, \ldots, v_k$.
\begin{prop} \label{non-linear-for-convrs-Vesen}
Let $D$ be a bounded balanced pseudoconvex domain in $\mathbb{C}^N$.
Suppose that $\partial D$ is convex at $p \in \partial D$ but not 
$\mathbb{C}$-extremal at $p$. 
Let $Z$ be an arbitrary one-dimensional 
retract of $D$
passing through $0$ in the direction of $p$; let $\rho: D \to Z$ be a retraction map.
Then, $Z$ can be be realized as the image of an analytic disc of the form 
$\varphi(t)=tp + t \tilde{\psi}(t) \cdot v$ 
where $\tilde{\psi}: \Delta \to \mathbb{C}^{N-1}$ is an analytic disc with 
$\tilde{\psi}(0)=0$ and where $v$ stands for a list of basis vectors $(v_1, \ldots, v_{N-1})$
of $\ker \left( D \rho(0) \right)$. Also, after factoring out $t$, the resulting 
map $t \mapsto p +  \tilde{\psi}(t) \cdot v$ gives an analytic disc in the boundary $\partial D$; indeed,
this analytic disc is the radial projection of $Z$ onto the boundary of our balanced domain $D$. We may 
also view $Z$ as the graph of the map on ${\rm span}(p) \cap D$ given by $tp \to t \tilde{\psi}(t) \cdot v$ 
taking values in $\ker \left( D \rho(0) \right)$.\\

\noindent Suppose next that $\partial D$ is convex {\it near} $p$. In-case $p$ is a smooth 
boundary point, we have: $\ker \left( D \rho(0) \right)=H_{p,\mathbb{C}}^{0}$, 
where $H_{p,\mathbb{C}}^{0}$ denotes the translate 
of the complex supporting hyperplane to $\overline{D}$ at $p$, to the origin; if $p$
is not a smooth boundary point, $\ker \left( D \rho(0) \right)$ coincides with 
one of the translates of the supporting hyperplanes at $p$, to the origin. In either case, 
$Z$ may be viewed as the graph of a holomorphic map on the linear retract ${\rm span}(p) \cap D$ 
taking values in a bounded balanced pseudoconvex domain in the face $F_\mathbb{C}(p, \partial D)$,
indeed in a linear projection of $D$ into $F_\mathbb{C}(p, \partial D)$.
\end{prop}
\begin{proof}
Let $Z$ be a one dimensional retract of $D$ and $\rho: D \to Z$ be one of the retraction
maps on $D$ whose image is $Z$; it suffices to address and focus on the case when
$Z$ is non-linear, which we assume for the rest of the proof -- by theorem \ref{converse-of-Vesentini}
and its proof above, this case is non-vacuous and the thrust of the claim is about 
non-linear retracts. By theorem \ref{Graph},
$Z$ is a graph of a holomorphic map $\psi$ over $D_L := D \cap L$, where 
$L$ as usual is the tangent space $T_0Z$. 
Note that since $Z$ passes through 
origin in the direction of $p$, $L=T_0Z$ is the subspace spanned by $p$
and $Z$ may be expressed as $Z=\{ tp + \psi(t) \cdot v \}$, where
\[
\psi(t) \cdot v \; = \; \psi_1(t)v_1 + \ldots + \psi_{N-1}(t)v_{N-1}
\]
with $v_1, \ldots,v_{N-1}$ being a basis for the $(N-1)$-dimensional subspace
$M:=\ker\left( D\rho(0) \right)$.
Note that $\psi(0) = 0$ and $\psi'(0) = 0$ because $Z$ passes through $0$ in the direction of $p$ and 
$p,v_1,v_2,...,v_{N-1}$ are linearly independent.
Note that we can write $\psi(t) = t \widetilde{\psi}(t)$, for some holomorphic 
map $\widetilde{\psi}$ with $\widetilde{\psi}(0) = 0$.\\
   
\noindent Now, let $f(t) = h\circ \varphi(t)$, where $h$ as usual is the Minkowski functional 
and $\varphi(t) = tp +\psi(t) \cdot v$ for $t \in \Delta$; the last mentioned factorization 
of $\psi$ leads to a similar one about $\varphi$: we may write 
$\varphi(t)=t \widetilde{\varphi}(t)$ where $\widetilde{\varphi}(t)= p + \widetilde{\psi}(t) \cdot v$.
Note that pseudoconvexity of $D$ implies $h$ and thereby $f$ is  plurisubharmonic. By
maximum principle of plurisubharmonic functions, we get that for all $t \in \Delta$,
\[
    |f(t)| \leq  \lim\sup_{\substack{|\lambda| \to 1 \\ \lambda \in \Delta}}|f(\lambda)|
\]
As $|f(t)| < 1$ for all $t \in \Delta$, this limsup is also atmost $1$.
Therefore $h(\varphi(t)) \leq  1$ for all $t \in \Delta$. It follows that
$|t|h(p+\widetilde{\psi}(t) \cdot v) \leq 1$. This implies that
\[
    \lim\sup_{\substack{|t| \to 1 \\ t \in \Delta}}h(p+\widetilde{\psi}(t) \cdot v) \leq 1
\]
Hence $h(p+\widetilde{\psi}(t) \cdot v) \leq 1$ for all $t \in \Delta$.
Note that $h(p+\widetilde{\psi}(0) \cdot v) = h(p) = 1$. By another application of the maximum principle, $h(p+\widetilde{\psi}(t) \cdot v) = 1$ for all $t \in \Delta$.
As the set of points $z$ where $h(z) = 1$, is always part of $\partial D$, it follows that
$\widetilde{\varphi}(t) = p + \widetilde{\psi}(t) \cdot v \in \partial D$ for all $t \in \Delta$. 
As $\varphi$ parametrizes our retract $Z$, it follows that $\widetilde{\varphi}$ gives 
the radial projection of $Z$ onto the boundary $\partial D$ (equivalently, one may view $Z$ 
as lying within the `surface' spanned by points in the image of $\widetilde{\varphi}$ i.e., 
the `surface' obtained by joining points of this analytic disc to the origin).\\

\noindent We next proceed to relate $\ker \left( D\rho(0) \right)$ with the supporting hyperplane(s)
at the boundary point $p$ near which $\partial D$ is convex, thereby $D$ is geometrically convex
near $p$ as well. For ease of exposition, we first consider the simpler case when $p$ happens to be 
a smooth boundary point. In such a case, the complex supporting hyperplane at $p$ to $\overline{D}$
(near $p$ to be specific as $D$ need not be convex, our assumption is only that $D$ is convex
near $p$ and this will be implicitly understood for the rest of this proof) is
uniquely determined. Denote by $H_{p, \mathbb{C}}^0$,
the translate of this complex supporting hyperplane 
to the origin. We claim that the sought-for relationship in general first is that of the containment: 
$ \ker \left( D\rho(0) \right) \subset H_{p, \mathbb{C}}^0$ and in the setting of this
proposition being proved, these subspaces are actually the same. The proof of this claim 
is as follows.
First we want to show that
$\ker(D\rho(0)) \subset H_{p,\mathbb{R}}^0 := H_{p,\mathbb{R}} - p$, 
where $H_{p,\mathbb{R}}$ denotes the real supporting hyperplane to $\overline{D}$ in a neighbourhood of $p$.
Let $T := D\rho(0)$. 
Suppose there exist $v \in \ker(T)$ such that $v \notin H_{p,\mathbb{R}}^0$.
By definition of $H_{p,\mathbb{R}}$ and $v \notin H_{p,\mathbb{R}}^0$,
we get $(p+{\rm span}(v)) \cap D$ is non-empty. Let $w \in (p+{\rm span}(v)) \cap D$.
Writing $w = \lambda v + p$ for some $\lambda \in \mathbb{C}$, we see by the linearity of the map $T$ that
\[
T(w) = T(\lambda v + p) = \lambda T(v) + T(p) = p.
\]
This contradicts the fact that $T$ maps points of $D$ to (interior) points of $D$.
This establishes $\ker(T) \subset H_{p,\mathbb{R}}^0$. Now, remembering that $\ker(T)$ is a $\mathbb{C}$-linear 
subspace, we conclude that it is contained within the maximal complex subspace of
$H_{p,\mathbb{R}}^0$ given by the ortho-complement of ${\rm span}(ip)$ in $H_{p,\mathbb{R}}^0$.
As $\ker(T)$ is also of complex co-dimension one, the aforementioned containment is 
actually as equality: $\ker(T) = H_{p,\mathbb{C}}^0$. Next we consider the general case 
that $p$ is (possibly) a non-smooth boundary point.
Now we claim that $\ker(T)$ is a translate of one of 
the complex supporting hyperplanes at $p$. It suffices to prove that $(p+ \ker(T)) \cap D = \phi$.
Suppose there exist $v \in \ker(T)$ such that $p+v \in D$.
It follows by linearity that $T(p + v) = p$ which as before gives the contradiction that $T$ maps 
$D$ into $D$. This establishes $(p+\ker(T)) \cap D = \phi$,  thereby that $p+\ker(T)$ is a complex
supporting hyperplane at $p$.\\

\noindent Next, we prove $image(\widetilde{\varphi}) \subset p+F_{\mathbb{C}}(p,\partial D)$.
We divide the proof into the following steps labelled as claims.\\

\underline{Claim-1}: $image(\widetilde{\varphi}) \subset \partial D \cap H_{p,\mathbb{R}}$.\\
We already proved that $image(\widetilde{\varphi}) \subset \partial D$.
It remains to show that $image(\widetilde{\varphi}) \subset H_{p,\mathbb{R}}$.
Note that $H_{p,\mathbb{R}}$ can be described by a single linear equation 
with real coefficients, say
\[
    \sum_{j=1}^Na_j {\rm Re}(z_j) + \sum_{j=1}^N b_j {\rm Im}(z_j) = \alpha,
\]
for some constants $a_j,b_j,\alpha \in \mathbb{R}$. Using the rudimentary identity 
${\rm Re} (i\zeta) = -{\rm Im} (\zeta)$ for any complex number $\zeta$,
we rewrite the above equation 
\[
{\rm Re}\left(\sum_{j=1}^N ((a_j - ib_j)z_j-\alpha)\right) = 0.
\]
We shall use below the fact that the left hand side is a pluriharmonic function of $z=(z_1, \ldots,z_N)$,
which we call $\Phi$ i.e.,  $\Phi: \mathbb{C}^N \to \mathbb{R}$ by
\[
    \Phi(z) = {\rm Re}\left(\sum_{j=1}^N ((a_j - ib_j)z_j -\alpha)\right).
\]
Since $\partial D$ is convex near $p$ and $\widetilde{\varphi}(0) = p$, there exist
a smaller subdisc $\Delta'$ of $\Delta$ such that $\partial D$ is convex near $\widetilde{\varphi}(t)$
for all $t \in \Delta'$.
Observe that that $\Phi \circ \widetilde{\varphi}_{|_{\Delta'}}$ is a harmonic function on $\Delta'$
and $\Phi \circ \widetilde{\varphi}_{|_{\Delta'}} \leq 0$ on $\Delta'$. 
Since $\Phi \circ \widetilde{\varphi}_{|_{\Delta'}}(0) = 0$, by the maximum principle we obtain
$\Phi \circ \widetilde{\varphi}_{|_{\Delta'}} \equiv 0$ on $\Delta'$ .
Consider the holomorphic function $F : \Delta \to \mathbb{C}$ defined by 
\[
     F(t) = \sum_{j=1}^N \left((a_j -ib_j)\widetilde{\varphi}_j(t) - \alpha \right),
\]
where $\widetilde{\varphi}_j$ denotes the $j^{\text{th}}$ component of $\widetilde{\varphi}$.
Since $\Phi \circ \widetilde{\varphi}_{|_{\Delta'}}$ vanishes on $\Delta'$, ${\rm Re}(F)$ vanishes on $\Delta'$.
By the open mapping theorem applied to $F$, we deduce that $F$ vanishes on $\Delta'$. 
It follows then by the identity principle that $F \equiv 0$ on $\Delta$. 
In particular, ${\rm Re} F = \Phi \circ \widetilde{\varphi}$ vanishes on $\Delta$. Hence we conclude that $image(\widetilde{\varphi}) \subset H_{p,\mathbb{R}}$. This finishes the proof of claim-1.\\

\noindent To proceed to the next claim, let $V$ be the real affine hull 
of $\partial D \cap H_{p,\mathbb{R}}$. \\ 

\noindent \underline{Claim-2:} $V = p+F_{\mathbb{R}}(p,\partial D)$, where $F_{\mathbb{R}}(p,\partial D)$
denotes the real face of $\overline{D}$ at $p$. \\
First we prove the inclusion $V \subset p + F_{\mathbb{R}}(p,\partial D)$ If $u$ belongs to $V$, then there exist non-zero vectors $v_i \in H_{p,\mathbb{R}}^0$  and 
non-zero real $t_i$ for $1 \leq i \leq k$ such that $u = p+(t_1v_1+\ldots +t_kv_k)$ and $p+v_i \in \partial D$.
Since $H_{p,\mathbb{R}}^0$ is a subspace over $\mathbb{R}$, 
$p+sv_i \in H_{p,\mathbb{R}}$ for all $s \in \mathbb{R}$.
Recall from the proof of lemma \ref{cty-Minkowfnal}, we get that  $\partial D$ is convex near $p$ implies 
the sub-additivity of Minkowski functional $h$ and thereby for all $s \in [0,1]$ we have
\[
    h(p+(1/2-s)v_i) = h((1/2+s)p + (1/2-s)(p+v_i)) \leq (1/2+s)h(p) + (1/2-s)h(p+v_i) = 1
\]
This means that $p+tv_i \in \overline{D}$ for all $t \in \mathbb{R}$ with $|t| < 1/2$. 
Using the fact that $H_{p,\mathbb{R}} \cap D = \phi$ and $p+tv_i \in H_{p,\mathbb{R}}$, 
we get that $p+tv_i \in \partial D$ for all $t\in \mathbb{R}$ with $|t| < 1/2$. So $v_i \in F_{\mathbb{R}}(p,\partial D)$ for $1 \leq i \leq k$.
Hence $u = p+(t_1v_1 + \ldots +t_kv_k) \in p + F_{\mathbb{R}}(p,\partial D)$.\\
To establish the reverse inclusion, let $u \in  F_{\mathbb{R}}(p,\partial D)$. This 
means that there exist $\epsilon > 0$ such that 
$p+tu \in \partial D$ for all $t \in \mathbb{R}$ with $|t| < \epsilon$.
Recall from the proof of theorem \ref{converse-of-Vesentini}, we get that
each supporting hyperplane at $p$ acts as a supporting hyperplane at $p+tu$ for all $t \in \mathbb{R}$ with $|t| << 1$.
This implies that there exist $\epsilon' > 0$ such that 
$p+tu \in \partial D \cap H_{p,\mathbb{R}}$ for all $t\in \mathbb{R}$ with $|t| < \epsilon'$.
This finishes the proof of claim-2.\\
\begin{rem}
    As a consequence of this claim, 
    it follows that $F_{\mathbb{C}}(p,\partial D) +p$ is contained 
    in every complex supporting hyperplane to $\partial D$ at $p$ i.e., every complex supporting 
hyperplane to the boundary at the point $p$ has got to contain the copy of the face $F_{\mathbb{C}}(p,\partial D)$
translated to $p$.
\end{rem}
Continuing with the proof of the proposition \ref{non-linear-for-convrs-Vesen}, write $\widetilde{\varphi}(t) = p + \widetilde{\phi}(t)$, 
where $\widetilde{\phi}: \Delta \to \mathbb{C}^N$ is the holomorphic map given by
$\widetilde{\phi}(t) = \widetilde{\psi}(t) \cdot v$. Let $L_p$ denote the smallest 
$\mathbb{R}$-linear subspace of $\mathbb{C}^N$ containing $image({\widetilde{\phi}})$ ($L_p$ is 
of-course uniquely determined as the intersection of all such $\mathbb{R}$-linear subspaces). We now  
 prove that $L_p$ is a complex vector space. Note that $L_p$ is the solution of
some (finite) system of  linear equations with real coefficients,
where each linear equation can be written in the following form
\[
    \sum_{j=1}^Na_j {\rm Re}(z_j) + \sum_{j=1}^N b_j {\rm Im}(z_j) = 0. 
\]
Since $image(\widetilde{\phi})$ is contained in $L_p$, we have
\[
    \sum_{j=1}^Na_j {\rm Re}(\widetilde{\phi}_j) + \sum_{j=1}^N b_j {\rm Im}(\widetilde{\phi}_j) = 0,
\]
where $\widetilde{\phi} = (\widetilde{\phi}_1,\ldots,\widetilde{\phi}_N)$.
Recalling the identity ${\rm Re}(i\zeta) = -{\rm Im}(\zeta)$, the above equation can be rewritten as
${\rm Re}\left(\sum_{j=1}^N(a_j-ib_j) \widetilde{\phi}_j \right) = 0$.
Owing to the Cauchy--Riemann equations, $\sum_{j=1}^N(a_j-ib_j) \widetilde{\phi}_j  = 0$. In particular,
${\rm Im}\left(\sum_{j=1}^N(a_j-ib_j) \widetilde{\phi}_j \right) = 0$ where-from:
${\rm Re}\left(i \left(\sum_{j=1}^N(a_j-ib_j) \widetilde{\phi}_j \right)\right) = 0$.
Application of this observation to each of the linear equations which 
define $L_p$, we infer that $L_p$ is closed under multiplication by $i$.
Hence $L_p$ is a complex linear subspace. Recalling the results established in `Claim-1' and 
`Claim-2', we then get that
$image(\widetilde{\phi}) \subset  F_{\mathbb{R}}(p,\partial D)$.
Since $L_p$ is a complex vector space, $image(\widetilde{\phi}) \subset i F_{\mathbb{R}}(p,\partial D)$.
Remembering the fact that $F_{\mathbb{C}}(p,\partial D) = F_{\mathbb{R}}(p,\partial D)
\cap i F_{\mathbb{R}}(p,\partial D)$, we get that $image(\widetilde{\phi}) \subset F_{\mathbb{C}}(p,\partial D)$.
Hence $image(\widetilde{\varphi}) \subset p+F_{\mathbb{C}}(p,\partial D)$ and thereby, 
$image(\widetilde{\varphi}) \subset (p+F_{\mathbb{C}}) \cap \partial D$.\\

\noindent The last remaining part of the proposition comprises only in recasting $Z$ -- realized as the image of 
$\widetilde{\varphi}$ above -- as a graph on ${\rm span}(p) \cap D$ which as we know is a linear retract owing 
to $p$ being a convex boundary point of $D$. To be brief, let us only mention that this is facilitated by the
decomposition:
\[
\mathbb{C}^N \; = \; {\rm span}(p) \oplus F_\mathbb{C}(p, \partial D) \oplus \tilde{V},
\]
where $\tilde{V}$ is a complementary subspace to ${\rm span}(p) \oplus F_\mathbb{C}(p, \partial D)$.
If $\pi$ denotes the projection of $\mathbb{C}^N$ onto ${\rm span}(p) \oplus F_\mathbb{C}(p, \partial D)$, 
then note that $\pi(Z)=Z$. Indeed, this is just a re-statement of the fact that 
$Z$ is contained within ${\rm span}(p) \oplus F_\mathbb{C}(p, \partial D)$ i.e., has no $\tilde{V}$-component.
From the fact that $Z=image(\widetilde{\varphi})$ and the expression 
of $\widetilde{\varphi}(t)$, it is now easy to see that $Z$ is the graph of 
the holomorphic map on ${\rm span}(p)$ given 
by $tp \to t \tilde{\psi} \cdot v$; where $\tilde{\psi}(t) \cdot v$ is a holomorphic map from $\Delta$
into the face $F_\mathbb{C}(p, \partial D)$ with the property that $p+\tilde{\psi}(t) \cdot v$ is
contained within the connected component of
\[
\left(p + F_\mathbb{C}(p, \partial D) \right) \cap \partial D;
\]
indeed in its interior, as soon as $\tilde{\psi} \not \equiv 0$. Noting that this connected component 
is contained within the affine linear projection of $D$ into $p+F_\mathbb{C}(p, \partial D)$, we 
are led to the last statement of the proposition; translating this projection back to the origin, gives a 
bounded balanced pseudoconvex domain which is a projection of $D$ into the lower 
dimensional complex linear subspace given by the face $F_\mathbb{C}(p, \partial D)$.
This finishes the proof of the proposition (which was for one-dimensional retracts).\\
\end{proof}

\noindent Using the foregoing proposition and the arguments in its proof, it is possible to retracts of arbitrary dimension but we do this at the
end of this section, after the following remarks and examples. \\

\begin{rem}
When $N =2$, we may rephrase the conclusion above more precisely as follows:
every non-linear retract $Z$ passing through the origin in the direction of $p$
can be realized as the image of an
analytic disc of the form $\varphi(t) = tp + t\widetilde{\psi}(t)u$, 
where $u$ belongs to the face $F_{\mathbb{C}}(p,\partial D)$ and  
$\widetilde{\psi}$ maps $\Delta$ holomorphically onto bounded domain 
in $\mathbb{C}$ containing the origin, with $\psi(0) = 0$ and $\psi'(0) = 0$. The
proof of this is as follows.
First observe that from the proof of above proposition, we obtain there exist non-zero $v \in \mathbb{C}^N$ 
such that $p+\widetilde{\psi}(t)v \in \partial D$ for all $t \in \Delta$.
Since $Z$ is non-linear, $\widetilde{\psi}$ is not identically zero. By continuity of 
$\widetilde{\psi}$ and $\widetilde{\psi}(0) = 0$, we get $v \in F_{\mathbb{C}}(p,\partial D)$.
It follows from the connectedness of image of $\widetilde{\psi}$, that $\widetilde{\psi}$
maps $\Delta$ into $\widetilde{\Delta}$.
By the open mapping theorem, $\widetilde{\psi}$ maps $\Delta$ onto an open subset 
$\widetilde{\psi}(\Delta)$ within ${\rm int}(\widetilde{\Delta})$. 
\end{rem}

\begin{rem}
Though there may well be plenty of retraction maps $\rho$ on $D$ 
such that $image(\rho) = Z$, we observe that for every retraction map $\rho$, 
we have $\ker D\rho(0) \supset F_{\mathbb{C}}(p,\partial \Omega)$.
This can be seen by an argument by contradiction as follows.
If possible assume that there exist $u \in F_\mathbb{C}(p,\partial \Omega)$ such that $D\rho(0)(u) \neq 0$.
Let $u' := D\rho(0)(u)$. Since $D\rho(0)$ maps $\mathbb{C}^N$ onto ${\rm span}(p)$,
$u' = \lambda p$ for some $\lambda \in \mathbb{C}$. Since $D$ is convex near $p$, we get by 
recalling the proof of theorem \ref{converse-of-Vesentini}, 
that 
$p$ and $u$ are linearly independent. It follows that $u'$ and $u$ are linearly independent. 
Hence $\dim(\ker \left(D\rho(0\right)) < N-1$, contradicting the fact that $\ker \left( D\rho(0)\right)$ 
is of complex co-dimension one.
\end{rem}
\noindent Now, a few words about examples of domains satisfying the hypotheses of the 
foregoing results just discussed. The simplest examples of domains as in 
proposition \ref{non-linear-for-convrs-Vesen} and
theorem \ref{converse-of-Vesentini} are obtained by taking products of bounded balanced domains
of holomorphy with some convex boundary points. Other examples (which are not products) can be
obtained by taking an intersection of sublevel sets of moduli of linear functionals with 
suitable balanced domains of holomorphy; to just mention one concrete instance of this 
kind, consider
\[
D := \{ (z,w) \in \mathbb{C}^2 \; : \; \vert z + w \vert <c \} \cap B.
\]
where $B$ is a strictly convex ball such as $l^p$ ball with $p\neq 1, \infty$ and $c=1$. 
In-fact, we can also let $B$ be a non-convex balanced domain of holomorphy such as 
the $\ell_q$-`ball' for $0<q<1$ (in which case we must choose $c$ appropriately
for instance $c=1/2$) to obtain more concrete examples.\\

\noindent Let us now conclude this chapter with a couple of remarks and more importantly, proof of theorem \ref{Charac-retract-1diml}. We formulate the first
one as the following corollary to highlight that: as a consequence of lemma \ref{cty-Minkowfnal}
and our proof of theorem \ref{converse-of-Vesentini} immediately verifies the truth
of that theorem in an important special case.
\begin{cor}
Let $D$ be a bounded balanced pseudoconvex domain in $\mathbb{C}^N$ and $p \in \partial D$.  
If $\partial D$ is convex near $p$ and the
only one-dimensional retract through the origin in the direction of $p$ is the linear one, then $p$
is a $\mathbb{C}$-extremal boundary point.
\end{cor}

\begin{rem}
If $D$ fails to be convex near $p$, then the following two cases arise:
\begin{enumerate}
\item[(i)] $D$ is convex at $p$ but not near $p$. In this case, $span \{p\} \cap D$ is a linear retract of $D$. 
Non-linear retracts of $D$ may or may not exist i.e., no definite statement can be 
made in general as it depends on finer features of the geometry of $\partial D$ near $p$.
\item[(ii)] $D$ is neither convex at $p$ nor near $p$. In this case, there are no retracts of $D$ passing through origin in the direction of $p$.
\end{enumerate}
\end{rem}

\noindent As mentioned above, we shall now use the foregoing proposition \ref{non-linear-for-convrs-Vesen} and the arguments in its proof, to extend it to retracts of arbitrary dimension as stated in theorem \ref{Charac-retract-1diml}, which we now prove. \\

\textit{Proof of theorem \ref{Charac-retract-1diml}}:
Observe 
that we may apply (the already proven)
theorem \ref{Graph} which applies to 
the setting of the theorem being proved.
We can therefore express $Z$ as a graph of a holomorphic map 
$F$ from $D_L$ into $D_M$, where  $M = \ker(D\rho_{|_0})$; notice that $\mathbb{C}^N = L \oplus M$. We then pick any point $p$ in $L \cap \partial D$ and
consider the retract over the segment spanned by $p$ i.e., let $Z_p$ denote 
the graph of $F$ over $D_p := D \cap {\rm span}(p)$ which is a linear 
retract precisely by virtue of the 
convexity hypothesis at $p$. This makes it clear that
 $Z_p$ can be expressed as $\{tp+F(tp): t \in \Delta\}$
and also that $Z_p$ is a retract. 
The foregoing proposition and the arguments in its proof are
therefore applicable to this one-dimensional retract $Z_p$; we thereby
get that $F_{|_{D_p}}$ 
takes values within the complex face $F_{\mathbb{C}}(p,\partial D)$. 
Viewing $L$ as the union of one-dimensional subspaces, we obtain 
that $F$ takes 
values in $\bigcup_{p \in \partial D}F_{\mathbb{C}}(p,\partial D)$.
\qed

\section{Proof of theorem \ref{not-thro-origin}}
\noindent  Consider the holomorphic map $G: D_L \to \mathbb{C}^N$ defined by the composition of the restricted mappings:
\[ 
    G(z) = D\rho_{|_0} \circ \rho_{|_{D_L}}(z).
\]
Note that $G$ actually maps into $D_L$: as $\rho$ maps $D$ into itself
and $D\rho_{|_0}$ is a linear 
projection mapping $D$ onto $D_L$, it follows that $G$ maps $D_L$ into $D_L$. 
Note that for all $z \in D_L$, we have
\[
DG_{|_0}(z) = D\rho_{|_0} \circ D\rho_{|_0}(z) = D\rho_{|_0}(z) = z.
\]
By the hypotheses that $D_L$ is convex, we can choose $v \in D_L$ 
such that $v$ is an $\mathbb{R}$-extremal boundary point in $D_L$. 
Note that $DG_{|_0}(v) = v$. 
By Mazet Schwarz lemma \ref{Mazet-Schwarz_1} applied to $G : D_L \to D_L$, we obtain that $G(z) = z$ for all $z \in 
D_L$. This means 
that $D\rho_{|_0} \circ \rho_{|_{D_L}}$ is the identity map on $D_L$. 
Hence we conclude that $Z$ is a graph of a holomorphic function 
over $D_L$, finishing the proof of theorem \ref{not-thro-origin}.\\

\noindent We ought to make a few remarks to settle any first questions that may arise 
as to whether the result is vacuous(!) i.e., whether there
exist retracts not passing through the origin at all. For this let us only recall Lempert's results
which guarantee the existence of one-dimensional retracts $Z$
through any given pair of points $p,q$ which we take to be linearly independent.
Such a $Z$
cannot pass through origin, for otherwise, the strict convexity 
hypotheses here leads by Vesentini's
theorem to the conclusion that $Z$ is linear. But then the linear independence of $p,q$
means that ${\rm dim}(Z) >1$, a contradiction; this finishes the
verification about the existence of many retracts
not passing through the origin, the matter of the above theorem. This discussion 
also answers another basic question: are there examples of a bounded balanced domain
of holomorphy $D$, wherein there exist pairs of one-dimensional retracts which are
not ${\rm Aut}(D)$-equivalent i.e., there is no automorphism carrying one onto the other?
In light of the above discussion which actually only requires
$\mathbb{C}$-extremality, we may readily cite $D$ to be the $\ell^1$-ball in 
which the origin is fixed by every of its automorphisms; a retract through any pair 
of linearly independent points in $D$ is therefore not ${\rm Aut}(D)$-equivalent 
to any of the retracts passing through the origin. Finally next here, we get to the question of at-least one 
example wherein all hypotheses of the theorem are satisfied.\\

\noindent \textit{Example:}
Consider the $\ell_1$-ball 
\[
    B := \{(z,w)\in \mathbb{C}^2: |z| + |w| < 1\}.
\]
Consider the holomorphic map $\varphi : \Delta \to B$ defined by 
$\varphi(\zeta) = \left((\frac{1+\zeta}{2})^2,(\frac{1-\zeta}{2})^2\right)$. 
By proposition 16.2.2 in \cite{Jrncki_invrnt_dst}, we obtain that 
$\varphi$ is a complex geodesic in $B$. Hence 
\[
Z :={\rm image}(\varphi) = 
\Big \{\left((\frac{1+\zeta}{2})^2,(\frac{1-\zeta}{2})^2\right) 
\; : \; \zeta \in \Delta
\Big\} 
\]
is a retract of $B$ which we note: does not pass through the origin. 
Next we want 
to find a retraction map $\rho$ on $B$ such that $\text{image}(\rho) = Z$. 
Towards this, it suffices to determine a 
holomorphic
left-inverse for the map $\varphi$ as
discussed in proposition \ref{C-geodesc_Retrct} of the preliminaries section. 
It is easy to see the map $\phi: B \to \Delta$ 
defined by $\phi(z,w) = z-w$ serves as 
the required left-inverse: $\phi \circ \varphi(\zeta) = \zeta$ 
for all $\zeta \in \Delta$. It follows that the  map 
$\rho(z,w) = \varphi \circ \phi(z,w)$ gives the holomorphic retraction of $B$ onto $Z$. As 
$B$ is convex with $\mathbb{C}$-extremal boundary, the only hypotheses that needs to be checked is 
that about the derivative of $\rho$.
For this we need its explicit expression 
which can be computed to be:
\[
\rho(z,w) = \left(\left(\frac{1+z-w}{2}\right)^2,\left(\frac{1-z+w}{2}\right)^2\right).
\]
In-particular, $\rho$ maps the origin 
to $p=(1/4,1/4)$ at which point $L=T_pZ$ is
given by $L = \{(z,w)\in \mathbb{C}^2: z+w = 0\}$.
Next, we compute the matrix of $D\rho_{|_{(z,w)}}$.
\[
    D\rho_{|_{(z,w)}} = \frac{1}{2}
    \begin{bmatrix}
        1+z-w & -(1+z-w) \\
        -(1-z+w) & 1-z+w
    \end{bmatrix} 
\]
In particular,
\[
    D\rho_{|_{(0,0)}} = \frac{1}{2}
    \begin{bmatrix}
        1 & -1 \\
        -1 & 1
    \end{bmatrix} 
\]
So, $D\rho_{|_{(0,0)}}(z,w) = (\frac{z-w}{2},\frac{-(z-w)}{2})$ which is apparently
a projection operator and of norm-one with 
respect to the $\ell_1$-norm.
This concludes the verification of all hypotheses 
of theorem \ref{not-thro-origin}.\\

We discuss briefly the form of the conclusion as well.
Now we want to show that $Z$ is a graph of a holomorphic function over $B_L=B \cap L$,
 where $L = \{(z,w)\in \mathbb{C}^2: z+w = 0\}$ which we note is the
 tangent space to $Z$ at $p=(1/4,1/4)$. 
 Firstly note that 
 $B_L = \{(z,-z): |z| < 1/2\}$. Consider the holomorphic map $\psi: B_L \to B$ defined by 
 \[
 \psi(z,-z) = \varphi(2z) = \left(\left(\frac{1+2z}{2}\right)^2,\left(\frac{1-2z}{2}\right)^2\right).
 \]
Denote the projection operator $D\rho_{\vert_0}$ by $\pi$ for short. 
 Note that
 \[
 \pi\circ \psi_{|_{B_L}}(z,-z) = \pi\left(\left(\frac{1+2z}{2}\right)^2,\left(\frac{1-2z}{2}\right)^2\right) = \frac{1}{2}(2z,-2z) = (z,-z)
 \]
The above computation shows that $\pi \circ \psi_{|_{B_L}}$ is the identity map on $B_L$. This implies that $Z$ is the graph of holomorphic function $\psi$ over $B_L$. The reason for the 
existence of non-linear retractions $\rho$ 
with $D \rho_{\vert_0}$ being a norm-one 
projection as in the above example, owes to the 
boundary failing to be $\mathbb{R}$-extremal
(though it is $\mathbb{C}$-extremal). This is 
brought out by the next result.\\

\begin{prop}
    Let $D$ be a bounded balanced strictly convex domain in $\mathbb{C}^N$. 
Let $\rho: D \to D$ be a retraction map on $D$ with $p = \rho(0)$ and $Z := 
{\rm image}(\rho)$. If $D\rho_{|_0}$ is a norm-one projection i.e., a projection which maps 
$D$ onto $D_L$ (where $L = T_pZ$ as usual), then $\rho$ is a linear map and hence 
$Z$ is a linear retract of $D$.
\end{prop}
\begin{proof}
    Consider the holomorphic map $F: D_L \to D$ defined by $F(z) = \rho_{|_{D_L}}(z)$. 
    Note that $DF_{|_0}(z) = D\rho_{|_0}(z)$, which is a linear projection 
onto $D_L$. In particular, $D\rho_{|_0}$ is the identity map on $D_L$. Applying 
Mazet's Schwarz lemma to the map $F: D_L \to D$, we obtain that $F(z) = DF_{|_0}(z)$. 
Hence we get that ${\rm image}(F) = D_L \subset Z$. Note that $D_L$ and $Z$ have the same 
dimension. Since $Z$ is connected, it follows from the principle of analytic 
continuation that $Z = D_L$. This finishes the proof of the proposition.
\end{proof}

\section{Proof of theorems \ref{prod_anal_poly},~\ref{max_noncnvx} and Corollary \ref{gen-anal-poly}}\label{Irreducibity}

\noindent  Before we begin to work out the proof of the pair of theorems as in the title of this section, we need to make some general remarks. 
Theorems \ref{prod_anal_poly} is about the irreducibility of balanced
$\mathbb{C}$-extremal analytic polyhedra. It will be derived as a consequence of theorem \ref{max_noncnvx} about retracts through the origin in such domains; this will therefore be established before theorem \ref{prod_anal_poly}. Along the way, we shall attain useful propositions about biholomorphic inequivalences between co-analytic domains and analytic polyhedra (or more generally, bounded hyperconvex domains). 
Also included in this section is the proof of corollary \ref{gen-anal-poly} about the irreducibility
of the anisotropic non-convex egg domains, whose boundaries are holomorphically extreme
in good contrast to analytic polyhedra; nevertheless, this will be simultaneously attained alongwith the proof of theorem \ref{prod_anal_poly}. \\

\noindent The above exercise in showing irreducibility may seem to need justification, owing to the existence of well-known methods which yield better results for certain classes of domains even if they are different.
To put this matter in perspective we recall that the 
famous inequivalence between the Euclidean ball $\mathbb{B}$ and the polydisc 
$\Delta^N$ for all $N \geq 2$, can be shown {\it using retracts} as well; while
we believe that this proof will atleast be of mild interest to atleast note in passing, we have better justifications below.
Now, we have already put
this proof for a record in the `preliminaries' section wherein we have formulated it more generally 
for a ball $D$ with respect to 
an arbitrary norm with $\mathbb{C}$-extremal boundary and product domains much more 
general than polydiscs (so that proof is not repeated here). Here, we would indeed 
like to hasten to mention that this does 
not obtain for {\it such} $D$: results as good as can be obtained from  already well-known techniques in this realm
can! -- so, this leads to the question as to why one may want to even consider this proof even if 
the problem of biholomorphic inequivalence and its special case of establishing irreducibility of a domain $D$ i.e.,
showing that $D$ is not biholomorphic to a product of lower dimensional domains, is undoubtedly fundamental. 
To address this briefly here,
let us recall that while the well-known techniques alluded to above, more explicitly such as that of Remmert -- Stein 
(exposited well in \cite{Encyclo-Vol6}; more specifically, see theorem $4.1$ therein alongwith the
subsequent theorem $4.2$ due to Rischel)
for showing irreducibility of domains, 
work for far more general
domains $D$ beyond balanced domains of holomorphy, they impose the condition that $D$ not have any
non-trivial analytic varieties in the boundary. We shall presently show that the route via retracts
can very well dispense with such conditions, if $D$ is a balanced domain of holomorphy. 
To exposit this better by comparing  with 
the Remmert -- Stein techniques, let us mention that 
the aforementioned proof by retracts for $\mathbb{B} \not \simeq \Delta^N$,
shares (unlike a proof based on comparison of automorphism groups for instance), with the best possible techniques 
in this context, the essential idea about 
the difference in the lack or presence of analytic discs in the 
boundary of $D$, as being responsible for such inequivalences; this difference has a bearing 
on the geometry of retracts through the origin in $D$. Moreover, this provides a shorter way to establish
inequivalences without having to worry about the boundary extendibility (or boundary behaviour)
of biholomorphic mappings since the boundary geometry is `internalized' (atleast partly) as features about the retracts
through points in the interior; for instance, holomorphic extremality of the boundary of a bounded balanced domain of holomorphy entails linearity of the retracts through its center and these linear slices have holomorphically extreme boundaries again.  We believe this will 
be useful for establishing new inequivalences between pairs of bounded balanced domains in {\it both}
of whose boundaries there reside non-trivial analytic discs. We now illustrate these points with a couple of progressively complex examples and theorems. We start
with a concrete {\it and} simple example of a special analytic polyhedron; indeed, this is the simplest example of a bounded balanced analytic polyhedron with $\mathbb{C}$-extremal boundary and will serve as a running example for us. Specifically, consider the question of whether
 \begin{equation}\label{D_h}
D_h := \{(z,w) \in \mathbb{C}^2: |w^2 - z^2| < 1\} \cap \{(z,w) \in \mathbb{C}^2: |zw| < 2\}
\end{equation}
is biholomorphic to a product domain \footnote{A motivation for the specific choices here, is 
that the simplest examples of {\it bounded} product domains namely
polydiscs are defined as sublevel sets of a pair of (homogeneous) holomorphic linear polynomials 
and we are interested in {\it bounded} domains whose definition bears some semblance to that of polydiscs,
but are nevertheless irreducible. Indeed as will be seen, this example is the simplest special analytic polyhedron which is balanced and non-linear.}. To begin with analysing this example first, we note that
while an open set cut-out by such polynomial inequalities in general may fail to be connected,
such issues do not arise here: the homogeneity of the defining pair of quadratic polynomials leads to $D_h$ being balanced, 
thereby (immediately seen to be connected and) 
contractible. Let us note in passing that $D_h$ is not Reinhardt; more importantly, that $D_h$ is indeed 
bounded: to verify this, note that the defining pair of inequalities together leads to
\begin{equation} \label{D_H-bddness}
\vert z+ iw \vert ^2 = \left\vert (z^2-w^2) + 2i zw \right\vert 
\leq \vert z^2 - w^2 \vert + 2 \vert zw \vert < 5.
\end{equation}
The reverse triangle inequality
\[
\left\vert \vert z \vert - \vert w \vert \right\vert \leq \vert z+ iw \vert
\]
when squared and combined with \ref{D_H-bddness}, leads to
\[
\vert z \vert^2 + \vert w \vert^2 < 5 + 2 \vert zw \vert.
\]
Looking at the second inequality in the definition of $D_h$, we arrive at the conclusion that
$\vert z \vert^2 + \vert w \vert^2<9$, meaning $D_h \subset 3 \cdot \mathbb{B}$ finishing the
desired verification. Now, the simply connectedness of 
$D_h$ (unfortunately tempts) an application of the Riemann mapping theorem to reduce
the problem of establishing irreducibility of $D_h$ directly to that 
of proving the inequivalence of the hyperbolic domain $D_h$ with $\Delta^2$. The balancedness 
of the special model domain here namely $\Delta^2$, together with its homogeneity, enables a reduction of the 
problem by an application of Cartan's theorem to linear equivalence; this, in-turn is 
immediately settled as linear maps preserve convexity whereas $D_h$ is non-convex.\\

\noindent However, we would like to show this and other modified versions now by circumventing use of 
the Riemann mapping theorem which may not work in dealing with more complex cases.
Indeed, to spoil the above proof which is based on the convenient availability of a 
model domain in the equivalence class of balanced domains which are biholomorphic to products domains as nice
as the bidisc, let us remove an analytic hypersurface $A$ not passing 
through the origin\footnote{Not removing the origin thus, helps in rendering a short illustration of our use 
of retracts, owing to scarcity of results that we can readily cite and use. It may be noted to begin with here that 
this modified domain only has more analytic varieties in 
the boundary, that is to say: the modification to spoil the foregoing proof, has not been done at
the cost of the very purpose of the previous example!}, and consider 
$D_h^A = D_h \setminus \{A\}$. While the removal of $A$ from $D_h$
never disconnects $D_h$, it always destroys the simply connectedness of $D_h$, as is
always the case when proper analytic sets are subtracted from any domain.\\

\noindent To establish the irreducibility of 
such $D_h^A$, let us start with the most natural 
attempt of assuming to obtain a contradiction, the
existence of a biholomorphism $F$ from $D_h^A$ onto a product domain $D$; no less standard then, 
is to analyze the boundary behaviour of $F$. However, this is rather cumbersome owing to potential 
complexity of the boundary of the factors $B_1,B_2$. It is here where consideration of retracts 
provides a far more convenient approach, as already alluded to in the above, particularly for theorem 
\ref{prod_anal_poly}, one of the main theorems
of this section.\\

\noindent However, before we discuss the proof of 
theorem \ref{prod_anal_poly},
it is convenient to formulate part of its proof as the following more general lemma. 

\begin{lem} \label{Retr-of-anal-complements}
Let $D$ be any bounded hyperconvex domain in $\mathbb{C}^N$. Let $A$ be any non-trivial 
    analytic subset of $D$. Then any holomorphic retraction 
map on $D \setminus A$, 
be extended to a holomorphic retraction map on $D$.
Consequently, {\it every} retract $R^A$ of $D \setminus A$ is of the form $R^A=R \cap (D \setminus A)$
for some retract $R$ of $D$.
\end{lem}
\begin{proof}
Firstly, recall that removing an analytic set from a domain does not disconnect it; by definition
$A$ is closed in $D$, thereby $D \setminus A$ is a domain in $\mathbb{C}^N$. 
Let $\rho : D \setminus A \to D \setminus A$ be a holomorphic retraction map.
By applying the Riemann removable singularities theorem (for holomorphic functions of several complex variables)
to each of the component functions of $\rho$ we get a (uniquely determined) holomorphic extension of $\rho$
which we denote by $\tilde{\rho}$. The fact that $\rho$ maps $D \setminus A$ into itself implies firstly that
$\tilde{\rho}$ maps $D$ into its closure $\overline{D}$. However, it is easy to see that $\tilde{\rho}$ 
actually maps $D$ into itself. Indeed, pick any $a \in A$ and enquire if it is possible for $b=\tilde{\rho}(a)$
to lie in $\partial D$. Recalling the local structure of analytic sets, we know
that: there exist a complex line segment $\ell \subset D$ through $a$ which intersects $A$ only at $a$.
If $\varphi: \Delta \to D$ is a parametrization of $\ell$ with $\varphi(0)=a$, then
$\tilde{\rho} \circ \varphi$ gives an analytic disc (with its image) intersecting $\partial D$ only at $a$.
Since $D$ is hyperconvex, it follows from 
 the theorem \ref{hypcnvx_local} and its remark that $\tilde{\rho}(a)$ is a local psh
barrier point of $D$. So, there exists neighborhood $U$ 
of $b := \tilde{\rho}(a)$  and plurisubharmonic function $\psi$ 
on $U \cap D$ such 
that $\psi < 0$ on $U \cap D$ and 
$\displaystyle\lim_{ U \cap D \ni z \to b} \psi(z) = 0$. After shrinking the radius of the punctured disc
$\Delta^*$ (if needed), we may assume that 
$image(\tilde{\rho} \circ \varphi) \subset U$. Consider the 
subharmonic function $\phi: \Delta^* \to \mathbb{R}$ defined by $\phi(z) = 
\psi \circ \tilde{\rho} \circ \varphi(z)$. By applying the removable 
singularity theorem for subharmonic functions, we obtain that 
$0$ is a removable singularity for 
$\phi$. Hence $\phi$ extends to a subharmonic function on $\Delta$ which we 
continue to denote by $\phi$. By the limiting behaviour of $\psi$ at $b$, we obtain that $\phi(0) = 0$.
But then, the maximum principle for 
this subharmonic function is violated, under our assumption that $b \in \partial D$. We deduce 
therefore that $b$ cannot lie in $\partial D$ and it follows that 
$\tilde{\rho}(D) \subset D$.\\

\noindent We finish off by verifying that $\tilde{\rho}$ is a retraction on $D$. Indeed, we need only verify the
required idempotence equation along $A$ (as outside of $A$:
$\tilde{\rho}$ is the given retraction map $\rho$ on 
$D \setminus A$).  To do this, pick an arbitrary $a \in A$ and a sequence $(z_n)\subset D \setminus A$ with 
$z_n$ converging to $a$. Then
just by the idempotence of $\rho$ on $D \setminus A$, we have: $\rho\left(\rho(z_n)\right)=\rho(z_n)$,
wherein by passing to limit using the continuity of $ \tilde{\rho}$ twice, we get 
$\tilde{\rho} \left(\tilde{\rho}(a)\right) = \tilde{\rho}(a)$. As noted earlier, we may now
conclude that the same equation holds not just on $A$ but throughout $D$ i.e., 
$\tilde{\rho} \circ \tilde{\rho} \equiv \tilde{\rho}$; denoting
$\tilde{\rho}(D), \rho(D \setminus A)$ by $R^A$ and $R$ respectively, we arrive at the statement of the 
lemma.
\end{proof}

\begin{rem}\label{Stein_nbhd}
We first recall that for the class of much interest in this article namely bounded balanced
pseudoconvex domains, 
the hypothesis of the above
is not satisfied automatically as hyperconvexity of such domains is equivalent to 
continuity of the Minkowski functional (see Remark 2.45 (iii) in the preliminaries section); nevertheless, it turns out that the 
conclusion of the above lemma is true regardless 
of such regularity, owing to certain results that we 
briefly touch upon now. For our applications however,
we shall not require such more 
general observations.
To remark on the extension then, the above lemma can be derived for bounded 
pseudoconvex domains 
$D$ 
even 
if the existence of such a $\psi$ as in its 
proof above is not known, provided $D$ has a Stein neighbourhood basis. The only point to 
be checked is $\tilde{\rho}(D) \subset D$. It is indeed true that $\tilde{\rho}(D)$ lies within $D$ 
and can be derived rather quickly from known results (the arguments above were given to show a simple route under the aforementioned assumption 
about bounded plurisubharmonic defining functions). This follows from 
the last result in \cite{Abe}, owing to our hypothesis that $D$ has a Stein neighbourhood basis; as the work
\cite{Abe} has some non-major errors, we also cite the Masters thesis of Persson \cite{Persson}, wherein the
errors are set right with a better exposition.
Thus as $\tilde{\rho}(a)$ lies within $D$ and as $a$ was arbitrary, we 
conclude that $\tilde{\rho}$ maps $A$ (and thereby $D$ as well) into $D$. 
\end{rem}

\begin{rem} \label{Analytic-compl-Unbdd dom}
    One may want to know to what extent we may weaken the hypotheses
    of the above lemma. We do not pursue this here but draw attention
    to the fact that rather simple counterexamples exist to show that
    boundedness cannot be dropped easily. Indeed, consider 
    the unbounded domain $D=\mathbb{C}^2 \setminus A$, 
    where $A=\{z=0\}$, the $w$-axis. Notice
    then that $R^A=\{(z,1/z): z \in \mathbb{C}^*\}$ is a retract
    of $D$. We observe that there exists no retract $R$ 
    of $\mathbb{C}^2$
    such that $R^A=R \cap D$. Indeed, 
    this follows from the fact that the (smooth) analytic
    variety $R^A$ described by the equation $zw-1=0$, 
    is irreducible, together with the principle of
    analytic continuation for complex analytic sets.
\end{rem}

\noindent Notwithstanding the example in the above remark, we may
certainly remove analytic varieties of higher codimension to 
attain the analogue of the foregoing
lemma for the case when the domain is the entire complex Euclidean
space itself: to adapt the essential ideas of its proof to deal 
with the unboundedness now, we only need to employ 
(theorem 5B in \cite{Whitney}) in place of the standard Riemann removable
singularities theorem. We shall not repeat the details and
be content with stating it as follows.

\begin{cor} (Corollary to the proof of lemma 
\ref{Retr-of-anal-complements})
    Let $A$ be any analytic variety of codimension at-least two 
    in $\mathbb{C}^N$. Then every retract
    $R^A$ of $\mathbb{C}^N \setminus A$ is of the form
    $R^A= R \cap (\mathbb{C}^N \setminus A)$ for some 
    retract $R$ of $\mathbb{C}^N$.\\
\end{cor}

\noindent Without further ado now, we proceed to detailing the proof of proposition \ref{retr-of-prod-in-C2} and lay
down some of its generalizations as well. This will be used in the sequel.\\

\noindent \textit{Proof of proposition} \ref{retr-of-prod-in-C2} .
Let $p = (p_1,p_2) \in D_1 \times D_2$ . Since $D_1$ is 
bounded, there exist $r_1 > 0$ such that $D_1 \subset \Delta(p_1,r_1)$. Since 
$D_2$ is open, there exist $r_2 > 0$ such that $\Delta(p_2,r_2) \subset D_2$.  
Consider the domain $D := D_1 \times D_2$. Choose $v_2 \in T_{p_2}D_2$ 
with $|v_2| < \frac{r_2}{r_1}$. Consider a holomorphic map $F : D_1 \to \mathbb{C}$ 
defined by $F(z) = p_2 + (z-p_1)v_2$. Note that $F(p_1) = p_2$ and $F'(p_1) = v_2$. 
Take $\epsilon = \frac{r_2}{r_1}$. If $|v_2| < \epsilon$, then $F$ maps $D_1$ 
into $D_2$. Consider a holomorphic retraction map $G: D_1 \times D_2 \to D_1 \times D_2$ 
defined by $G(z,w) = (z,F(z))$. Hence $Z : = \{(z,F(z)): z \in D_1\}$ is a one-dimensional 
retract of $D_1 \times D_2$ passing through $(p_1,p_2)$ and $(1,v_2) \in T_p Z$. 
\qed 

\begin{prop}\label{Retrct_Dirctn}
Suppose $D_1, D_2$ be domains in $\mathbb{C}^N, \mathbb{C}^M$ respectively and at least 
one of domains is Caratheodory hyperbolic, say $D_1$; let $p = (p_1,p_2)$ be any 
given point in the product domain $D = D_1 \times D_2$. Let $v_1 \in T_{p_1}D_1$ be 
any non-zero vector. Then there exist $\epsilon > 0$, such that for all vectors 
$v_2 \in T_{p_2}D_2 \cong \mathbb{C}^{M}$ with $\|v_2\|_{\infty} < \epsilon$, 
there is a retract $Z$ of $D_1 \times D_2$ passing through $p$ with 
$T_pZ \ni (w_1,v_2) $, for some non-zero vector $w_1$ with $w_1 \in \spn_{\mathbb{C}} \{v_1\}$.
\end{prop}

\begin{proof}
Since $D_2$ is open in $\mathbb{C}^M$, there exist a 
polydisk $\Delta^M(p_2,r_2) \subset D_2$. Since $D_1$ is Caratheodory 
hyperbolic, we have $\gamma_{D_1}(p_1,v_1) \neq 0$. This means that there 
exist a non-constant holomorphic map $g : D_1 \to \Delta$ such that 
$g(p_1) =0$ and $g'(p_1)(v_1) \neq 0 $. So $g'(p_1)$ is a non-zero 
linear map from $\mathbb{C}^M$ onto $\mathbb{C}$. Hence there exist 
$w_1 \in T_{p_1}D_1$ with $w_1 = \lambda v_1$ for some $\lambda \in \mathbb{C}$ 
and $g'(p_1)(w_1) = 1$. For $v_2 \in T_{p_2}D_2$,  
consider the holomorphic map $\phi :D_1 \to \mathbb{C}^M$ defined 
by $\phi(z) = p_2 + v_2 g(z)$. Note that
\[ 
\|\phi(z) - p_2\|_{\infty} = \|v_2 g(z)\| = |g(z)|\|v_2\|_{\infty} \leq 1 \cdot \|v_2\|_{\infty} 
\] 
Write $ r_2 = (r_2^{(1)},\ldots,r_2^{(M)})$. Let $\epsilon := \min\{r_2^{(1)},\ldots,r_2^{(M)}\}$. 
Choose $v_2 \in T_{p_2}D_2$ with $\|v_2\|_{\infty} < \epsilon$. Hence $\phi$ maps $D_1$ into $D_2$. Note 
that $\phi(p_1) = p_2$ and $\phi'(p_1)(w_1) = v_2g'(p_1)(w_1) = v_2$. Consider the 
holomorphic map $F : D_1 \times D_2 \to D_1 \times D_2$ defined by $F(z,w) = (z,\phi(z)) $. 
Again note that $F(p_1,w) = (p_1,p_2)$ and $F'(p)(w_1,w_2) = (w_1,\phi'(p_1)(w_1)) = (w_1,v_2)$.  
Hence $Z : = \{(z,\phi(z)): z \in D_1\}$ is a retract of $D_1 \times D_2$ passing 
through $(p_1,p_2)$ and $(w_1,v_2) \in T_p Z$, where $w_1 \in \spn_{\mathbb{C}}\{v_1\}$. 
\end{proof}

\noindent We shall now employ this proposition in 
applying retracts to establish various biholomorphic inequivalences amongst members of $3$ classes of domains namely, analytic polyhedra, product domains and co-analytic domains i.e., domains of the form $D \setminus A$ wherein $D$ is a domain of holomorphy, an assumption that we generally make unless explicitly stated otherwise.  all such domains have the common feature that their boundaries are not entirely smooth and contain non-trivial analytic varieties. On the other hand (as also mentioned in the introductory section), there are differences in the location of the analytic boundary pieces: for instance, in a co-analytic domain there are interior analytic boundary points whereas analytic polyhedra are devoid of such `interior' boundary points, leading them to be biholomorphically inequivalent. To transform this intuition into a rigorous proof as well as to attain a more result a bit more general, we note that co-analytic domains are not hyperconvex whereas analytic polyhedra are. 
The biholomorphic invariance of hyperconvexity immediately renders the desired inequivalence but 
take a moment to record this in the following proposition among other related statements in order to be able to refer and make use of it in the sequel. We first formalize the
aforementioned notion of boundary points lying in the `interior' of the domain, which we will be used later.

\begin{defn}
Let $\Omega$ be a domain in $\mathbb{C}^N$ and $p \in \partial \Omega$. We say that $p$ is an \textit{interior boundary point} of $\Omega$ if there exists a neighbourhood $U_p$ of $p$ such that $U_p \cap \overline{\Omega} = U_p$.
\end{defn}

Next, we prove a lemma regarding the product of hyperconvex domains in $\mathbb{C}^N$. 
\begin{lem}\label{prod_hypcnvx}
    Let $D_1,D_2$ be domains in $\mathbb{C}^N, \mathbb{C}^M$ respectively. Then $D := D_1 \times D_2$ is hyperconvex if and only if $D_1$ and $D_2$ are hyperconvex.
\end{lem}
\begin{proof}
Assume that $D := D_1 \times D_2$ is hyperconvex and fix a point $(z_0,w_0) \in D$. Note that $D_1 \times \{w_0\}$ and $\{z_0\} \times D_2$ are retracts of $D$. 
By proposition \ref{retrct_hypcnvx}, $D_1$ and $D_2$ are hyperconvex.\\

Conversely, assume that $D_1, D_2$ are hyperconvex. 
Let $\varphi_1,\varphi_2$ be psh exhaustion functions of $D_1,D_2$ respectively. Consider the function $\varphi$ on $D_1 \times D_2$ defined by $\varphi(z,w) = \max \{\varphi_1(z),\varphi_2(w)\}$. Then $\varphi$ is plurisubharmonic and negative on $D_1 \times D_2$. For each $c < 0$, let $S^c := \{(z,w) \in D: \varphi(z,w) < c\}$, $S_1^c := \{z \in D_1: \varphi_1(z) < c\}$ and  $S_2^c := \{w \in D_2: \varphi_2(w) < c\}$ .  Note that
\[
    (z,w) \in S^c \Longleftrightarrow \varphi_1(z) < c~~\text{and}~\varphi_2(w) < c \Longleftrightarrow (z,w) \in S_1^c \times S_2^c.
\]
This shows that $S^c = S_1^c \times S_2^c$. Since $S_1^c$ and $S_2^c$ are relatively compact subsets of $D_1,D_2$ respectively, $S$ is a relatively compact subset of $D$. Hence, $D_1 \times D_2$ is hyperconvex.
\end{proof}
\begin{prop}\label{barrier}
   \begin{enumerate} \rm 
\item[(i)] Let $D_1$ be a complex manifold, $A_1$ be a non-empty analytic subset 
of $D_1$ and let $D_2$ be a bounded hyperconvex domain in 
$\mathbb{C}^N$, $A_2$ be an (possibly empty) analytic subset of $D_2$.  
If $F: D_1 \setminus A_1 \to D_2 \setminus A_2$ is a proper holomorphic map then $F$ can be extended as a holomorphic map $\tilde{F}$ from $D_1$ to $D_2$ with $\tilde{F}(A_1) \subset A_2$. Consequently, there does not exist any proper holomorphic map from 
$D_1 \setminus A_1$ to $D_2$ (as follows by considering the case $A_2 = \phi$).
In particular, there does not exist a biholomorphic map from 
$D_1 \setminus A_1$ onto $D_2$.
\item[(ii)] Let $D_1,D_2$ be bounded hyperconvex domains in 
$\mathbb{C}^N$  and 
$A_1,A_2$ be non-empty analytic subsets of $D_1,D_2$ respectively. 
If $D_1 \setminus A_1$ is biholomorphic 
to $D_2 \setminus A_2$ then $D_1$ is biholomorphic to $D_2$.
\end{enumerate}
\end{prop}

\begin{proof}
Assume that there exists a proper holomorphic map $F$ mapping $D_1 
\setminus A_1$ to $D_2 \setminus A_2$. By applying the Riemann removable 
singularity theorem to each of 
the component functions of $F$, we get a holomorphic extension of $F$ which we 
denote by $\tilde{F}$. The fact that $F$ maps $D_1 \setminus A_1$ into $D_2 \setminus A_2$  
implies that $\tilde{F}$ maps $D_1$ into $\overline{D_2}$.
 Using the same arguments in the proof of the lemma \ref{Retr-of-anal-complements}, we obtain that $\tilde{F}$ maps $D_1$ into $D_2$ and $\tilde{F}(A_1) \subset A_2$. Considering the special case 
 $A_2 = \phi$, we get that there does 
not exist any proper map from $D_1 \setminus A_1$ to $D_2$. This finishes the proof of (i).\\

To prove (ii), let $F$ be a biholomorphic map from $D_1 \setminus A_1$ onto 
$D_2 \setminus A_2$ and $G$ be the inverse of $F$. By the arguments similar to those used in (i), we 
obtain that $\tilde{F}$
maps $D_1$ into $D_2$ and  $\tilde{G}$ maps $D_2$ into $D_1$.
It is easy to see that $\tilde{F}$ and $\tilde{G}$ are 
biholomorphic maps. To prove this, pick an arbitrary point $z_0 \in A_1$ and a sequence $(z_n)$
in $D_1 \setminus A_1$ converges to $z_0 \in A_1$. 
Since $\tilde{F} \equiv F$ on $D_1 \setminus A_1$ and 
$\tilde{G} \equiv G$ on $D_2 \setminus A_2$, $z_n = \tilde{G} 
\circ \tilde{F}(z_n) \to 
\tilde{G} \circ \tilde{F}(z_0)$. By the uniqueness of limits, 
$\tilde{G} \circ \tilde{F}(z_0) = z_0$. This shows that 
$\tilde{G} \circ \tilde{F}$ is the identity map on $D_1$. 
Similarly, $\tilde{F} \circ \tilde{G}$ is the identity map 
on $D_2$. This proves that $\tilde{F}$ is a biholomorphism on $D_1$ and its inverse is $\tilde{G}$. Hence $D_1$ 
is biholomorphic to $D_2$.
\end{proof}

As a corollary of part (i) of the above proposition \ref{barrier} and by above lemma \ref{prod_hypcnvx}, we obtain the following.
\begin{cor}\label{local_psh}
Let $D$ be a pseudoconvex domain in $\mathbb{C}^N$ and 
$A$ be a non-empty analytic subset of $D$. Then $D \setminus A$ is not 
biholomorphic to a product of bounded hyperconvex domains.
\end{cor}

\noindent In particular, it follows immediately from this corollary that an 
analytic polyhedron is never biholomorphically equivalent to a co-analytic domain.

\begin{rem}
In general, it can very well happen that both $D$ and $D \setminus A$ are product (pseudoconvex) domains. Indeed the simplest example showing 
this is $D = \Delta^2$ and $D \setminus A = \Delta \times \Delta^*$ obtained by taking $A := \{(z,w) \in \Delta^2: w=0\}$; while there are plenty of other examples, the next proposition shows that within the class of bounded balanced domains of holomorphy in dimension two, this is the only one. This proposition also serves as one of the simplest applications of retracts to establishing biholomorphic inequivalences. 
\end{rem}
\begin{prop}\label{bih_Delta^2}
Let $D$ be any bounded balanced domain of holomorphy in $\mathbb{C}^2$ 
(whose Minkowski functional we denote by $h=h_D$), which is not biholomorphic to $\Delta^2$.
Let $A$ be any proper analytic subset of $D$.
Then neither $D$ nor $D \setminus A$ is biholomorphic to a product domain. 
\end{prop}
\begin{proof}
To prove this proposition by 
contradiction, assume that $D \setminus A$ is biholomorphic to a product domain 
$\Omega_1 \times \Omega_2$, where $\Omega_1,\Omega_2$ are domains in $\mathbb{C}$. 
 Let $F$ be a 
biholomorphism from $D \setminus A$ onto $\Omega_1 \times \Omega_2$ and 
$G$ be the inverse of $F$. Let $(p_1,p_2) \in \Omega_1 \times \Omega_2$. Note that 
$\Omega_1 \times \{p_2\}$ and  $\{p_1\} 
\times \Omega_2$ are retracts of $\Omega_1 \times \Omega_2$. Let 
$G$ be a biholomorphism from $\Omega_1 
\times \Omega_2$ onto $D \setminus A$. By proposition 
\ref{Bih}, $\tilde{R_1} := G(\Omega_1 \times \{p_2\})$ and 
$\tilde{R_2} := G(\{p_1\} \times \Omega_2)$ are retracts 
of $D \setminus A$. This implies that $\tilde{R_1}, \tilde{R_2}$ are biholomorphic to  $\Omega_1, \Omega_2$ respectively.  
Since $D \setminus A$ is biholomorphic to $\Omega_1 \times \Omega_2$, we obtain 
that $D \setminus A$ is biholomorphic to $\tilde{R_1} \times \tilde{R_2}$. 
By lemma \ref{Retr-of-anal-complements}, there exist one-dimensional retracts $R_1,R_2$ of $D$ 
such that $\tilde{R_1} = R_1 \cap (D \setminus A)$ and $\tilde{R_2} 
= R_2 \cap (D \setminus A)$.\\

If $R_i \subset D \setminus A$ for $i=1,2$ then $\tilde{R_i} = R_i$. 
Since $R_i$ is a simply connected Riemann surface, its biholomorphic image $\Omega_i$ is a simply 
connected domain in $\mathbb{C}$. By Riemann mapping theorem, $R_i$ is 
biholomorphic to $\Delta$. Hence we obtain that $D \setminus A$ is 
biholomorphic to $\Delta^2$, which contradicts part (i) of the above 
proposition \ref{barrier}. Next, we consider the case that $R_i \cap A \neq \phi$ for at least one $i$. Let $S$ be the set of all points
of intersection of $R_i$ with $A$. As $R_i \cap A$ does not have any interior point and they are of dimension one, $S$ is a discrete set 
and thereby can be enumerated as $(p_j)_{j=1}^{\infty}$. Fix one of the $p_j$'s say $p_1$ for simplicity. 
There exists an open set $U_1 \subset R_i$, which is contained within at least one chart about $p_1$ for $R_i$. 
Let $\phi$ be the biholomorphism from $\tilde{R_i}$ onto $\Omega_i$. By 
Casorati-Weiestrass theorem, $p_1$ is not an essential singularity for $\phi_{|_{U_1^*}}$, where $U_1^* = U_1 \setminus \{p_1\}$. 
So, the isolated singularity for $\phi$ at $p_1$ is removable when $\phi$ is viewed as a map into $\mathbb{C}_{\infty}$. Let $V_1 := \phi(U_1)$  and $V_1^* = V_1 \setminus \{q_1\}$, where $q_1 = \phi(p_1)$. 
By local-normal form theorem, after passing to suitable charts, we obtain that $\partial(\phi(U_1^*)) = \phi(\partial U_1) \cup \{q_1\}$ with $q_1 \notin \phi(\partial U_1)$. 
Hence $q_1$ is an interior boundary point of $V_1$. 
This argument holds for all $p_j$. Since $\phi$ is a biholomorphism on $\tilde{R_i}$, $\phi_{|_S}$ is one-one. Hence $R_i$ is 
biholomorphic to $\Omega_i' := \Omega_i \cup \phi(S)$. Since $R_i$ is simply connected, $\Omega_i'$ is simply connected and thereby $R_i$ is biholomorphic to $\Delta$ by Riemann mapping theorem.
Hence there exist 
discrete subsets $L_1,L_2 \subset \Delta$ (at least one of $L_i$ is non-empty) such that $\tilde{R_1}$ and $\tilde{R_2}$ 
are biholomorphic to $\Delta \setminus L_1$ and $\Delta \setminus L_2$ 
respectively. This implies that $D\setminus A$ is biholomorphic to $(\Delta \setminus L_1) 
\times (\Delta \setminus L_2)$. Write $(\Delta \setminus L_1) 
\times (\Delta \setminus L_2) = \Delta^2 \setminus S'$, where 
$S' = (\Delta \times L_2) \cup (L_1 \times \Delta)$.  
Note that $S'$ is a non-empty analytic subset of $\Delta^2$. By applying the Riemann removable 
singularity theorem to each of 
the component functions of $F$, we get a holomorphic extension of $F$ which we 
denote by $\tilde{F}$. The fact that $F$ maps $D \setminus A$ to $\Delta^2 \setminus S'$  
implies that $\tilde{F}$ maps $D$ into $\overline{\Delta}^2$.
By theorem 1.17.16 in \cite{Jrnck_Frst_stps}, both $D$ and $\Delta^2$ admit Stein neighborhood bases. Applying the same argument in the proof of the lemma \ref{Retr-of-anal-complements} and its remark, we conclude that  $\tilde{F}$ maps $D$ into $\Delta^2$ and $\tilde{G}$ maps $\Delta^2$ into $D$.
Using the methods similar to the proof of part (ii) of the above proposition \ref{barrier}, we obtain that $\tilde{F}$ and $\tilde{G}$ are 
biholomorphic maps. This proves that $D$ 
is biholomorphic to the polydisc $\Delta^2$.
\end{proof}

Next,we prove the corollary of above proposition \ref{Retrct_Dirctn} and theorem \ref{ell_q_ball_C3}. \\

Proof of Corollary \ref{gen-anal-poly}: 
If $F$ is a biholomorphic mapping from $D_q \setminus A$ onto a 
product domain $G = G_1 \times G_2$, then through $F(0)$ we have a retract along those directions 
given by proposition \ref{Retrct_Dirctn} -- note that the hypothesis about Caratheodory 
hyperbolicity therein (for $G$ here) is verified owing to the fact that $D_q \setminus A$ is bounded;
indeed thereby, every biholomorphic image of $D_q \setminus A$ must be Caratheodory hyperbolic 
as well (refer to theorem 18.2.1 of \cite{Jrncki_invrnt_dst}). In particular, 
there always exists a non-trivial retract (necessarily of dimension
one here) of $D_q$ for an open piece $S \subset \pi(S^h)$ of directions at $F(0)$; here,
 $\pi$ denotes the projection onto the Euclidean sphere $\partial \mathbb{B}$ given by 
$\pi(z)=z/\vert z \vert_{l^2}$. 
By the biholomorphicity of $F$, it follows that for an open piece $S'$ of 
directions,
there always exist a (non-trivial) retract of $D_q$ through the origin
with its tangent direction at $0$ belonging to $S'$.
But this contradicts theorem \ref{ell_q_ball_C3}, which implies that $D_q$ does not admit any retract along an open piece of directions at the origin. Hence, $D_q \setminus A$ cannot be biholomorphic to a product domain. By the same reasoning, it follows that $D_q$ itself is not biholomorphic to a product domain. \qed
\\

\noindent We conclude this section with a complete determination of 
retracts of $D_h \setminus A$ through the origin,
where $D_h$ is the polynomial polyhedron
(\ref{D_h}). As above, 
$A$ is an analytic subset of $D_h$ not passing through the origin, which is allowed 
to be empty but not $D_h$ itself. This 
of-course is tantamount to determining retracts of $D_h$ by 
lemma \ref{Retr-of-anal-complements} and
we may therefore forget about $A$.
\begin{ex}[] \label{Main-Example}    
We are yet to determine the linear retracts of $D_h$ (the motivation for this has already been mentioned in the prequel). First, note that almost every 
point of $\partial D_h$ fails to be holomorphically extreme; specifically, 
every point of $\partial D_h$
which is not in the intersection of the real analytic varieties $H_1$ given 
by $\vert z^2- w^2 \vert=1$
and $H_2$ given by $\vert z w \vert=2$. However, every point of $\partial D_h$ is 
a $\mathbb{C}$-extremal boundary point; this 
is substantiated by realizing $D_h$ as the intersection of the pair of balanced domains $D_1,D_2$ given 
respectively as the sub-level sets
$\vert z^2- w^2 \vert<1$ and $\vert zw \vert<2$ both of which have (weakly) $\mathbb{C}$-extremal boundaries and then using 
proposition \ref{R-extrml}.\\

\noindent Now, we contend that for certain special points $p$ (which we shall precisely determine in due course)
of $\partial D_h$ in the intersection
$H_1 \cap H_2$, there exists precisely one retract through the origin in the direction of $p$ which is
given by the linear subspace spanned by $p$; along directions given by any other boundary point there is 
no retract, completing the characterization of retracts through the center of $D_h$. 
We shall presently put forth the proof of this claim. Let us begin by settling at the outset, the 
non-existence of retracts along directions given by $D_h \setminus (H_1 \cap H_2)$: this  
follows by noting that every point of this open piece of the boundary is a non-convex boundary point of 
$D_h$ and invoking corollary \ref{No_1dim_Retrct}. With this we may restrict ourselves to points 
of $H_1 \cap H_2 \subset \partial D_h$ and proceed to discussing retracts through 
such points alone; it also being borne in mind that by retracts we mean the linear ones for, by proposition \ref{JJ-improved}
there are no non-linear holomorphic retracts through the origin in $D_h$.
To proceed forth indeed then, pick any point $p = (p_1,p_2)$ in the intersection $H_1 \cap H_2 \subset \partial D_h$;
so, $| p_2^2 - p_1^2 | = 1$ and $ \vert p_1p_2 \vert = 2$. 
Consider the linear projection $L$ from $\mathbb{C}^2$ onto the one-dimensional 
complex subspace spanned by $p$ denoted $\langle p \rangle$, defined 
by $L(z,w) = \lambda(z,w)p$, where 
\[
\lambda(z,w) = \frac{\overline{p_1} z + \overline{p_2} w}{|p_1|^2+|p_2|^2}.
\] 
To prove our claim, it suffices to check that $L$ maps $D_h$ into itself. 
Towards this, write $L$ in standard component notations $L=(L_1,L_2)$;
it needs to be verified that the pair of inequalities
\begin{align} \label{pt-in-D_h}
\vert \left(L_2(z,w)\right)^2 - \left(L_1(z,w)\right)^2 \vert <1 \;\;\; \text {and} \;\;\;
\vert L_1(z,w) L_2(z,w) \vert <2,
\end{align}
hold for all $(z,w) \in D_h$. We shall now deal with the quantities on the left of each of these inequalities;
we begin with the first one.
Note that
\begin{eqnarray*}
\vert \left(L_2(z,w)\right)^2 - \left(L_1(z,w)\right)^2 \vert
&=& |\lambda(z,w)|^2|p_2^2- p_1^2|
= |\lambda(z,w)|^2
\end{eqnarray*}
It therefore follows that for our task of verifying the first inequality at
(\ref{pt-in-D_h}), it suffices to find the maximum modulus of 
$\varphi(z,w)= \overline{p_1}z + \overline{p_2}w $ on $D_h$ or equivalently, on its closure
$\overline{D_h}$. While the compactness of $\overline{D_h}$ ensures this maximum being attained, 
holomorphicity of $\varphi$ allows us to restrict attention on $\partial D_h$ for determining
the maximum. In-fact as $\varphi$ is a {\it linear} functional, we may further restrict attention 
to the extreme points on the boundary of the convex hull of $\overline{D_h}$. However, rather than 
rigorously determining this convex hull and its extreme points, it is convenient here to directly proceed 
to solving our optimization problem on $\overline{D_h}$. Indeed, let just first use the holomorphicity 
of $\varphi$ to deduce that the aforementioned maximum -- call it $m$ -- cannot be attained in the open
piece $X_1 \cup X_2$ within the boundary $\partial D_h$, where
\begin{align*}
X_1 &= \{ \vert z^2-w^2 \vert =1 \} \cap \{ \vert zw \vert <2\},\\
X_2 &= \{ \vert z^2-w^2 \vert <1 \} \cap \{\vert zw \vert =2\}.
\end{align*}
To see this, suppose $q=(q_1,q_2)$ is a point in $\partial D_h$ where $m$ is attained by $\vert \varphi \vert$. Let
us only argue as to why $q$ cannot lie in $X_1$; the arguments for $q \not \in X_2$ are similar, therefore 
omitted. To argue by contradiction supposing $q \in X_1$, we would firstly have $q_1^2 - q_2^2 = e^{i \theta}$
for some $\theta \in (-\pi,\pi]$. We collect together all points in $X_1$ with this feature i.e., consider $\tilde{X}_1$ given by
\[
\tilde{X}_1 = \{ (z,w) \in X_1 \; : \; z^2 - w^2 = e^{i \theta} \}.
\]
An immediate application of the holomorphic implicit function 
theorem ensures that this is smooth i.e., $\tilde{X}_1$ is a Riemann surface (noting that 
the gradient of $z^2 - w^2$ is nowhere vanishing on this level 
set). Thus, we are assured of an analytic disc $\psi: \Delta \to \partial D_h$ whose image is actually contained in 
$\tilde{X}_1 \ni q$; we may assume $\psi(0)=q$. Recalling what we have supposed about $q$, we note
that the holomorphic function $h = \varphi \circ \psi$ on $\Delta$ assumes its maximum modulus 
at the origin, thereby reduces to a constant. This means that $\varphi$ is constant on 
$\psi(\Delta)$, thereby on $\tilde{X}_1$. But then it is not difficult to check that the 
projection of $\tilde{X}_1$ onto $\langle p \rangle$ is a non-trivial open subset of $\langle p \rangle$. This
means, owing to its linearity, that $\varphi$ must be constant on all of $\mathbb{C}^2$ which is
manifestly false. This concludes the proof of the observation that the maximum of $\vert \varphi \vert$
on $\partial D_h$ 
is attained within $H_1 \cap H_2$, which as we shall presently see is a real analytic submanifold of $\mathbb{C}^2$.
In-particular therefore, one may apply the method of Lagrange multipliers to determine this maximum (of 
$\vert \varphi \vert$ on this submanifold). However, 
for our purposes the exact maximum is not necessary and moreover, we shall adopt a more elementary approach, not
because it is elementary but owing rather to its {\it relative} quickness in directly leading to 
what is required for our goal here. 
To go about further solving this optimization problem, we rephrase it for convenience by `squaring it', as follows: maximize
\[ 
f(z,w) = \vert \varphi \vert^2= |\overline{p_1}z + \overline{p_2}w|^2
\]
subject to the simultaneous constraints:
\begin{equation} \label{constr-1}
\begin{cases}
g_1(z,w) := \vert zw \vert^2 \; = \; 4 \\
g_2(z,w):=\vert z^2 - w^2 \vert^2 \; = \; 1.  
\end{cases}
\end{equation}
Despite the squares, this is of-course {\it equivalent} to the previous maximization problem.
A first convenience of the squares here is that the (real) gradients $\nabla g_1$ and $\nabla g_2$
are easily checked to be nowhere vanishing along the level sets of $g_1, g_2$ respectively as mentioned above.
Consequently, it follows by the implicit function theorem that $H_1 \cap H_2$ is a smooth (real analytic) submanifold.
Now, getting to the aforementioned optimization problem, we shall work in polar coordinates and 
so write the variables $z = re^{i\alpha}, w = se^{i\beta}$ where we may let 
$\alpha, \beta$ vary in $\mathbb{R}$ rather than on a $2 \pi$-length interval 
such as $(-\pi,\pi]$, for convenience in the 
sequel; let us also note in passing that $r,s$ are also variables, taking values in non-negative reals, but 
each of the constraints at (\ref{constr-1})
curtail the possible values of all these variables, as we shall detail below. Also write 
$\overline{p_1} = r_1e^{i\theta_1} ,\overline{p_2} = r_2e^{i\theta_2}$; here ofcourse,
$\theta_1,\theta_2 \in (-\pi,\pi]$ and $r_1,r_2$ are fixed numbers which also satisfy 
both the above constraints at (\ref{constr-1}).
The first advantage of polar coordinates is of-course the dropping of absolute values in the 
constraint $\vert zw \vert=2$ to write it as $rs = 2$, which we use below in the form $r^2s^2=4$. 
The other constraint (\ref{constr-1}) however, does
not yield to such simple re-expressions; but nevertheless is good enough for 
our problem at hand to be solved. To this end, we write it
in polar form: $|r^2e^{2i\alpha}-s^2e^{2i\beta}|^2 = 1$. As $r^2s^2 =4$, we may replace the
second radial parameter $s$ by the first one $r$, to write it in various forms as:
\begin{align*}
\vert r^2e^{2i\alpha}-\frac{4}{r^2}e^{2i\beta} \vert^2 = 1,\\
\Longleftrightarrow 
(r^2)^2 + (\frac{4}{r^2})^2 -8Re(e^{2i(\beta - \alpha)}) = 1,\\
\Longleftrightarrow 
\left(r^2+\frac{4}{r^2}\right)^2 -8 -8 \cos(2(\beta - \alpha)) = 1.
\end{align*}
Let us record the last equation as:
\begin{equation}\label{max}
\left(r^2+\frac{4}{r^2}\right)^2 = 1+8(1+ \cos(2(\beta -\alpha))).
\end{equation}
This only finishes expressing the constraints in polar coordinates. To 
get to the main objective here of maximizing $f$, we
likewise express it also in polar form: 
\[
f(z,w)=|\overline{p_1}z+\overline{p_2}w|^2 = |r_1e^{i\theta_1}re^{i\alpha} + r_2e^{i\theta_2}se^{i\beta} |^2.
\]
Since $rs = 2$ and likewise $r_1r_2 = 2$ as well, the above equation can be written as
\begin{eqnarray*}
\left|r_1re^{i(\alpha+\theta_1)} + (\frac{2}{r_1})(\frac{2}{r})e^{i(\theta_2 + \beta)}\right|^2
&=& r_1^2r^2 + \frac{16}{r_1^2r^2} +8Re(e^{i(\alpha+\theta_1-\theta_2-\beta)})\\
&=& r_1^2r^2 + \frac{16}{r_1^2r^2} + 8\cos(\alpha+\theta_1-\theta_2-\beta)\\
&=& (r_1^2+\frac{4}{r_1^2})(r^2+\frac{4}{r^2}) -4\left(\frac{r^2}{r_1^2}+ \frac{r_1^2}{r^2}\right) +\\
& & 8\cos(\alpha + \theta_1 - \theta_2-\beta)\\
&=& (r_1^2+\frac{4}{r_1^2})(\sqrt{1+8(1+\cos(2(\beta - \alpha)))}) -4\left(\frac{r^2}{r_1^2}+ \frac{r_1^2}{r^2}\right)\\
& & +8\cos((\beta - \alpha) - (\theta_1 - \theta_2)).
\end{eqnarray*}
If we introduce the notations: $\theta = \beta - \alpha$ and $\theta_0 = \theta_1-\theta_2$, then
our optimization problem now fully expressed in polar coordinates, consists of maximizing
\begin{equation}\label{f} 
\tilde{f}(r,\theta) = r_1^2r^2 + \frac{16}{r_1^2r^2}  +8\cos(\theta - \theta_0)
\end{equation}
subject to the constraint
\[
\left(r^2+\frac{4}{r^2}\right)^2 = 1+8(1+cos(2\theta)).
\]
Thus our original double constraints problem has now got reduced to a single constraint one.
To proceed further, let $Q(\theta):= \sqrt{1+8(1+\cos2\theta)}$. 
We can write $Q(\theta) = \sqrt{1+16\cos^2 \theta}$, whereby the just mentioned constraint can be rewritten as 
\begin{equation}\label{Q10}
r^2+\frac{4}{r^2} = Q(\theta) = \sqrt{1+16\cos^2 \theta}.
\end{equation}
It must be noted that this contains within it implicitly, certain bounds on $\theta$: a
simple application of the inequality about arithmetic and geometric means, the left hand side is
bounded below by $4$. It follows thereby that
\begin{equation} \label{Bd-on-theta}
\cos^2 \theta \geq \frac{4^2-1}{16} \; \Longleftrightarrow \; \sin^2 \theta \leq \frac{1}{16}.
\end{equation}
Let us note in passing that these bounds are indeed satisfied for $\theta=\theta_0$, due to the trivial reason 
that $p$ is already given to lie in $H_1 \cap H_2$ i.e., the coordinates of $p$ satisfy the constraints (\ref{constr-1})
and thereby all its consequences derived above. Another helpful remark to keep in the background for the 
sequel is the description of  
$H:=H_1 \cap H_2$ in polar coordinates as:
\[
H:=H_1 \cap H_2 = \Big\{ \left( (r, \alpha), (s, \beta) \right) \; : \; 
rs=2, \; r^2 + s^2 = Q(\beta - \alpha)  \Big\}.
\]
Before proceeding ahead, we would also like to record a
special case of the foregoing equation (\ref{Q10}), for later reference; namely taking $\theta = \theta_0 $ and $r = r_1$,
we have
\begin{equation}\label{Q1}
r_1^2+\frac{4}{r_1^2} = Q(\theta_0).
\end{equation}
Next we determine $r_1$ and more generally $r$ as a function of $\theta$, when $r, \theta$ are to 
vary in manner satisfying the constraints at (\ref{constr-1}). To do this
efficiently from the above data, it is convenient to consider the quadratic equation whose roots are $r^2$ and $4/r^2$ namely,
$x^2-Q(\theta)x+4 = 0$ -- this follows immediately from the fact that the sum and product of its roots is what is given above.
On the other hand by the standard formula for the roots of a quadratic, we may also express the roots as 
$x = (Q \pm \sqrt{Q^2-16})/2$. From (\ref{Q10}), we get
\[
x = \frac{Q \pm \sqrt{1+16\cos^2\theta-16}}{2} = \frac{Q(\theta) \pm \sqrt{1-16\sin^2\theta}}{2}.
\] 
But then we already know that $r^2$ and $\frac{4}{r^2}$ are the roots of the aforementioned quadratic. 
The desired expression of $r$ as a function of $\theta$, thereby follows; we record in a 
form suited for direct use in the subsequent analysis.
\begin{equation} \label{r^2-(1/r^2)-Expressions}
r^2 = \frac{1}{2}\left[\sqrt{1+16\cos^2\theta}+ \sqrt{1-16\sin^2\theta}\right],~~ \frac{4}{r^2} 
=  \frac{1}{2}\left[\sqrt{1+16\cos^2\theta}- \sqrt{1-16\sin^2\theta}\right].
\end{equation}
Let us remark in passing here that the factor $\sqrt{1-16\sin^2\theta}$ is a well-defined real number owing to the 
bounds on $\theta$ as noted in (\ref{Bd-on-theta}).
We also need to record for later reference a few particular relationships following from the foregoing pair
of expressions; firstly the difference:
\[ 
r^2-\frac{4}{r^2} = \sqrt{1-16\sin^2\theta}. 
\]
Further in particular then, this equation holds for $\theta = \theta_0 $ and $r = r_1$, in which case this reads:
\begin{equation}\label{Q2}
r_1^2-\frac{4}{r_1^2} = \sqrt{1-16\sin^2\theta_0}.
\end{equation}
Next, substituting the expressions for $r^2, \frac{4}{r^2}$ from (\ref{r^2-(1/r^2)-Expressions}) into (\ref{f}), 
we note that $\tilde{f}(r,\theta)$ {\it restricted to} $H_1 \cap H_2$ reduces to a function of $\theta$ alone,
which we call $h$:
\begin{multline}
h(\theta) \; = \tilde{f}_{\vert_{H_1 \cap H_2}}(r,\theta) = \; \frac{1}{2}\Big(r_1^2(\sqrt{1+16\cos^2\theta}+ \sqrt{1-16\sin^2\theta}) \\
+ \frac{4}{r_1^2}(\sqrt{1+16\cos^2\theta} - \sqrt{1-16\sin^2\theta})\Big) +8\cos(\theta -\theta_0). 
\end{multline}
After some simplifications, we get 
\begin{equation}\label{h-exp-simpl}
h(\theta) = \frac{1}{2}\left[ \left(r_1^2+\frac{4}{r_1^2} \right)(\sqrt{1+16\cos^2\theta}) 
 + \left(r_1^2-\frac{4}{r_1^2} \right)(\sqrt{1-16\sin^2\theta})\right] +8\cos(\theta -\theta_0).
\end{equation}
This is the simplified expression of the single variable function $h(\theta)$ that we shall presently 
put to use. As preliminary observations about $h$ that can be read-off from this expression 
would perhaps be that it is a periodic function,
indeed $2 \pi$-periodic and smooth (in-fact real analytic) on the entire real line but alas, this is 
{\it not} true: the occurrence of the term $1- 16\sin^2\theta$ under the square root forbids it
from being real-valued for all real $\theta$. Indeed, it must be remembered that it was for this reason
that we have throughout the bound on $\theta$ stipulated in (\ref{Bd-on-theta}). Actually, it
is possible to circumvent this issue simply by allowing $h$ to take complex values and fixing a 
branch of the square root on the simply connected domain $\mathbb{C} \setminus (i \mathbb{R}^-)$ where
$i \mathbb{R}^-$ is the negative imaginary axis, for instance. While this may ensures well-definedness and even 
continuity of $h$ on the real line (which will render boundedness owing to periodicity), it 
is worth noting that differentiability at those values of $\theta$ at which $\sin^2 \theta=1/16$ persists to be 
invalid. Thus, we shall be content in having the codomain of $h$ restricted to real values; 
interest in the differentiability stems from the desire to apply standard calculus techniques 
of locating maxima of the function $h$ whose points of global maximum is what we wish to examine
for our ultimate goal about retracts of $D_h$, as must not be lost sight of. We then have to 
restrict the domain of the function $h$, as mentioned above and we may do so by taking it to be its natural domain in the 
real line, namely the set of all those real numbers $\theta$ for which $h(\theta)$ also lies in the real line
(and not a non-real complex number). Owing to the `essential' periodicity of $h$ as indicated above,
we may restrict the domain of $h$ to the union of three intervals $I=I^{-} \cup I_0 \cup I^{+}$ 
where these intervals are defined as follows. First we define the middle one, namely
the interval:
\[
I_0 = \text{ connected component containing the origin in $\mathbb{R}$ of its subset } 
\Big\{ \theta \in \mathbb{R} \; : \; \sin^2 \theta \leq \frac{1}{16} \Big\}.
\]
More concisely put $I_0 = [-\epsilon_0, \epsilon_0]$ where $\epsilon_0$ stands for the 
smallest positive number satisfying $\sin^2 \epsilon_0=1/16$. The other pair of 
intervals are defined as: $I^{-}=[-\pi,-\pi + \epsilon_0]$ and $I^{+}=[\pi -\epsilon_0, \pi]$.
The boundedness of $h$ then follows by the continuity of $h$ on the compact set $I$, as already alluded to; but it 
is also possible to lay down a bound directly from the explicit expression defining $h$ in (\ref{h-exp-simpl}) which estimates
are of much use in the sequel.
\begin{equation}\label{h-bound}
\vert h(\theta) \vert \leq \frac{1}{2}\left[ \left(r_1^2+\frac{4}{r_1^2} \right) \sqrt{17} 
+ \left(r_1^2-\frac{4}{r_1^2} \right) \right] + 8.
\end{equation}
This bound of-course depends on $r_1$ which depends on $p$ but is of-course independent of $\theta$ (so, suffices
to finish the verification that $h$ is bounded by rendering an explicit bound). In this connection, it 
must more importantly be noted that the aforementioned dependence on $r_1$ means that the function $h$ depends 
the parameter $\theta_0$ and perhaps must be written $h_{\theta_0}$; of-course we drop the subscript to
prevent overloading of notations, unless emphasis enhances clarity.
Better estimates on the maximum value of $h$ will be derived but only to the extent that is 
needed for our goal. Indeed, let us recall that our goal does not really require us to determine the maximum
but rather only to ascertain whether $\theta_0$ is a point at which the maximum is 
attained. This is ofcourse just a problem of single variable calculus if
we could solve this just by applying the second derivative test to locate first the local maxima and then 
glean out of them the desired global maximum. To this end, we compute the derivative
\begin{eqnarray*}
h'(\theta) &=& \frac{1}{2}\left[(r_1^2+\frac{4}{r_1^2})\left(\frac{-32\sin\theta \cos \theta}{2\sqrt{1+16\cos^2\theta}}\right) 
 + (r_1^2-\frac{4}{r_1^2})\left(\frac{-32\sin\theta \cos \theta}{2\sqrt{1-16\sin^2\theta}}\right) \right]\\
           & & -8\sin(\theta -\theta_0)\\
               &=& -8\sin \theta \cos \theta\left[(r_1^2+\frac{4}{r_1^2})\left((\frac{1}{\sqrt{1+16\cos^2\theta}}\right)  + 
               (r_1^2-\frac{4}{r_1^2})\left(\frac{1}{\sqrt{1-16\sin^2\theta}}\right) \right]\\
               & &-8\sin(\theta -\theta_0)
\end{eqnarray*}
From (\ref{Q1}) and (\ref{Q2}),  
\begin{equation}\label{Q3}
h'(\theta_0) = -16\sin \theta_0\cos\theta_0 = -8\sin2\theta_0.
\end{equation}
We proceed next to compute the second order derivative:
\begin{multline*}
    \frac{-h''(\theta)} {8}    = \sin \theta \cos \theta \left[\left(r_1^2+\frac{4}{r_1^2}\right)\left(\frac{32\sin\theta \cos \theta}{2\sqrt{1+16\cos^2\theta}(1+16\cos^2\theta)}\right)  \right.\\
+ \left.  \left(r_1^2-\frac{4}{r_1^2}\right)\left(\frac{32\sin \theta \cos \theta}{2\sqrt{1-16\sin^2\theta}(1-16\sin^2\theta)}\right) \right] + \\
\left[\left(r_1^2+\frac{4}{r_1^2}\right)\left(\frac{1}{\sqrt{1+16\cos^2\theta}}\right)  + \left(r_1^2-\frac{4}{r_1^2}\right)\left(\frac{1}{\sqrt{1-16\sin^2\theta}}\right) \right]\left(\cos^2\theta -\sin^2\theta\right)  + \cos(\theta -\theta_0).
\end{multline*}
We simplify this a bit, as
\begin{multline*}
\frac{-h''(\theta)} {8}  = 16\sin^2 \theta \cos^2 \theta \left[\left(r_1^2+\frac{4}{r_1^2}\right)\left(\frac{1}{\sqrt{1+16\cos^2\theta}(1+16\cos^2\theta)}\right) + \right.\\
\left.  \left(r_1^2-\frac{4}{r_1^2}\right)\left(\frac{1}{\sqrt{1-16\sin^2\theta}(1-16\sin^2\theta)}\right)\right] + \\
\left[\left(r_1^2+\frac{4}{r_1^2}\right)\left(\frac{1}{\sqrt{1+16\cos^2\theta}}\right)  + \left(r_1^2-\frac{4}{r_1^2}\right)\left(\frac{1}{\sqrt{1-16\sin^2\theta}}\right) \right]\left(\cos^2\theta -\sin^2\theta\right)  + \cos(\theta -\theta_0).
\end{multline*}
\noindent Next we put $\theta = \theta_0$ in the above equation, in which case it simplifies considerably; more 
importantly, we do this only because $\theta_0$ is the point of focus for us, as already mentioned.  From (\ref{Q1}) and (\ref{Q2}), we get 
\[
\frac{-h''(\theta_0)} {8} = 16 \sin^2\theta_0 \cos^2\theta_0 \left[ \frac{1}{1+16\cos^2\theta_0} + \frac{1}{1-16\sin^2\theta_0}\right] + 2\cos2\theta_0 + 1
\]
\begin{equation}\label{Q4}
\frac{-h''(\theta_0)} {8} = 4\sin^2 2\theta_0 \left[ \frac{1}{1+16\cos^2\theta_0} + \frac{1}{1-16\sin^2\theta_0}\right] + 2\cos2\theta_0 + 1
\end{equation}
to not lose sight of our main goal in the midst of computations, we 
keep this aside for a moment, and observe that our main goal translates here to checking only whether: $\theta_0$ is a point
of global maximum for $h$ or not. This is due to the following observation: if at all $p$ is
a point of maximum for $f(z,w)=\tilde{f}(r, \theta)$, then $\theta_0$ must be a maximum point for $h$.
Now, owing to its smoothness on the open interval $I^0$ (the interior of $I$), a point of global maximum for $h$ on $I^0$
if any, is not
only a point of local maximum but also a point about which $h$ is smooth and so, the standard derivative tests of calculus are 
applicable. So, if at all $\theta_0$ is a point of maximum for $h$ on the compact 
interval $I$, we must have $h'(\theta_0) = 0$ or that $\theta_0$ is one of the two boundary points of $I$.
But then we infer from (\ref{Q3}) that: $h'(\theta_0) = 0$ implies that $\theta_0 = 0, \pm \pi/2 \text{ or } \pi$.
Thus as soon as $\theta_0$ is not one among the 
very few possibilities: $0$, $\pm \pi/2$, $\pi$, or the remaining end-points of the intervals occurring in the 
definition of $I$, it cannot 
even be a point of local extremum for $h$. Moreover, the pair of possibilities $\theta_0=\pm \pi/2$
are ruled out simply because $h(\theta_0)$ is not real and therefore not of relevance for our 
consideration of the function $h$.\\


\noindent We presently proceed to dispose off the 
cases $\theta_0 = \pm \epsilon_0$ 
where $\epsilon_0$ be the smallest
positive number such that $\sin^2\epsilon_0 = 1/16$ which means that $\pm \epsilon_0$ are the end-points
of the interval $I$;
by showing that $h_{\theta_0}(\theta)$ does not attain its (global) maximum on $I$
at $\theta=\theta_0=\epsilon_0$ -- the argument for the other case $-\theta_0$ being similar is
skipped. Now when $\theta_0 = \epsilon_0$, the expression of $h(\theta)$ simplifies a bit and reads:
\[ 
h(\theta) = 2\sqrt{1+16\cos^2\theta} + 8\cos(\theta-\epsilon_0). 
\]
Note that $h(\epsilon_0) = 2\sqrt{1+16\cos^2 \epsilon_0} + 8$. Recalling that $\cos^2 \epsilon_0 = 15/16$,
we get $h(\epsilon_0)= 16$. To show that this is not the maximum value of $h$, we show that the value at 
$\epsilon_0/2$ is strictly bigger i.e., we claim $ h(\frac{\epsilon_0}{2}) > 16$. 
To prove this, we write this claim in various equivalent forms in order to unravel it and reach an inequality that 
is readily checked to be true, beginning with explicitly writing $h(\frac{\epsilon_0}{2}) > 16$:
\[
 2\sqrt{1+16\cos^2(\frac{\epsilon_0}{2})} + 8\cos\left(\frac{\epsilon_0}{2}\right) > 16  \iff 
\sqrt{1+16\cos^2(\frac{\epsilon_0}{2})}  
 > 4\left(2- \cos\left(\frac{\epsilon_0}{2}\right)\right)
\]
which is equivalent to
\[
1+16\cos^2\left(\frac{\epsilon_0}{2}\right)
> 16\left(2- \cos\left(\frac{\epsilon_0}{2}\right)\right)^2. 
\]
Expanding the square on the right, it should be noted that the term $16\cos^2\left(\frac{\epsilon_0}{2}\right)$ actually 
cancels, leading to a reduced but of-course equivalent form:
\[
 \cos\left(\frac{\epsilon_0}{2}\right) > 1- \frac{1}{64},
\]
which uses the fact that the quantity of the left is positive owing to the definition of $\epsilon_0$.
Its precise value is not at all difficult to compute; indeed  the 
half-angle formula $\cos^2\theta = (1+\cos2\theta)/2$, leads to this 
value being determined as $\cos(\epsilon_0/2) = \sqrt{ \left(1 + \sqrt{15/16} \right)/2}$. 
So the last inequality (thereby all its equivalent forms prior to it) is equivalent to:
\[
\frac{\sqrt{15}}{4} + 1 > 2\left(1-\frac{1}{64}\right)^2 \iff \sqrt{15} > \frac{1921}{8^3} \iff 8^6 \times 15 > (1921)^2.
\]
As the last inequality is readily checked to be true indeed, this finishes the proof that $h(\epsilon_0/2) > h(\epsilon_0)$, 
thereby that $h_{\theta_0}$ does not attain its global maximum at $\epsilon_0$, in the case when $\theta_0 = \epsilon_0$.
The case $\theta_0=\pi -\epsilon_0$ and $\theta_0=-\pi + \epsilon_0$ can 
also be disposed off by similar arguments -- for example, it can be shown that $h(\pi - \epsilon_0/2) > h(\pi - \epsilon_0)$ 
-- which ascertain that $h$ does not attain (global) maximum at these `end'-points.\\

\noindent We now note that the only possible cases left are: $\theta_0 = 0$ and $\theta_0 = \pi$; 
in both cases, we see from equation (\ref{Q4}) that $h''(\theta_0)$ is strictly negative. Of-course, despite 
being the only left cases for $\theta_0$, this still only establishes that $\theta_0$ is a point of local maximum
in such cases. We now 
prove that these cases correspond to points of global maximum. 
First suppose $\theta_0=0$: this case means that the arguments in the polar coordinates
for $p=(p_1,p_2)$ are equal ($\theta_1=\theta_2$), in which case we may write: $p_1 = r_1e^{-it_0}, p_2 = r_2e^{-it_0}$, 
where $t_0 = \theta_1 = \theta_2$, wherein $r_1,r_2$ satisfy the polar analogue of the basic constraints (\ref{constr-1}) 
namely, (\ref{Q1}) .
Recall that this means that $r_1,r_2$ satisfy $r_1^2 - r_2^2 = 1$ and $r_1r_2 = 2$. From this, we 
deduce the constraint satisfied by $r_1$ namely, $r_1^2 - \frac{4}{r_1^2} = 1$ which may be 
rewritten as a quadratic equation in $r_1^2$: $r_1^4 - r_1^2 -4 = 0$. 
Solving this gives $r_1^2 = (1+\sqrt{17})/2$ and therefore, $r_2^2 = 8/(1 + \sqrt{17})$. So, we have 
\begin{equation} \label{l2-norm-p-spl}
|p_1|^2 + |p_2|^2 = r_1^2 + r_2^2 = \frac{1+\sqrt{17}}{2} + \frac{8}{1+\sqrt{17}}  = \sqrt{17}.
 \end{equation}
We shall furnish the details of the conclusion of this case. But before that
it is convenient to also deal with the case $\theta_0 = \pi$ simultaneously now, as the arguments are similar to the above case.
So, in-case $\theta_0 = \pi$
we have $\theta_2 = \theta_1 - \pi$ just by definition of $\theta_0$. The coordinates of $p$ may be 
written, in such a case, as
 $p_1 = r_1e^{-i\theta_1}, p_2 = -r_2e^{-i\theta_1}$. Note then that
 $Q(\theta_0) = \sqrt{1 + 16 \cos^2 \theta_0} = \sqrt{17}$.
Substituting $\theta_0 = \pi$ in (\ref{Q1}), we obtain $|p_1|^2 + |p_2|^2 = \sqrt{17}$ again. The 
significance of this number here is that regardless of the value of $\theta_0$, 
the bound $\vert h_{\theta_0}(\theta) \vert \leq 17$ always holds; this follows from (\ref{h-bound}), the estimate 
whose usefulness was promised there. Indeed, employing that estimate alongwith (\ref{l2-norm-p-spl})
we get $|h(\theta)| \leq \frac{1}{2}[\sqrt{17}\cdot \sqrt{17} + 1] + 8 = 17 $.
That is to say, $h$ is bounded in magnitude by $17$, which is equal to $(|p_1|^2+|p_2|^2)^2$. 
Hence, $\theta_0 = 0$ and $\theta_0 = \pi$ is a point of global maximum.  What this translates to about
our original function $f (z,w)=|\overline{p_1}z + \overline{p_2}w|^2$ from around (\ref{constr-1}) is
that: its maximum value on $H_1 \cap H_2$ and in-fact by our earlier arguments on the closure of our 
entire domain $D_h$ is $(|p_1|^2+|p_2|^2)^2$. Indeed, recalling the 
discussions and observations around (\ref{constr-1}) and (\ref{pt-in-D_h}), this is equivalent to the $\vert \lambda(z,w) \vert<1$
where $\lambda$ is the function defined prior to (\ref{pt-in-D_h}) and involved in the two
inequalities determining whether $L(z,w)$ lies in $D_h$. Indeed, the just-mentioned bound on 
$\lambda$ leads immediately to the first of these inequalities: $|(L_2(z,w))^2-(L_1(z,w))^2| < 1$; the other 
inequality at (\ref{pt-in-D_h}) is also now immediate as 
\[
\left|L_1(z,w)L_2(z,w)\right| = |(\lambda(z,w))^2p_1p_2| = |\lambda(z,w)|^2 < 1.
\]
Hence, we conclude that $L$ maps $D_h$ onto $D_h \cap {\rm span}\{p\}$,
thereby finishing the proof that $D_h \cap {\rm span}\{p\}$ is indeed a linear retract of $D_h$ for points $p$ as above.\\
\qed
\end{ex}

\noindent As the details of the foregoing example have been rather lengthy, let us summarize and record it
as a proposition.

\begin{prop} \label{D_h-retracts}
Let $D_h$ be the domain in $\mathbb{C}^2$ given by 
\[
D_h = \{ (z,w) \in \mathbb{C}^2 \; : \; \vert z^2 - w^2 \vert<1,\; \vert zw \vert<2 \}.
\]
Then: firstly, every holomorphic retract through the origin is linear. More specifically and precisely, there
are no retracts in the directions of $p \in \partial D_h \setminus H$,
where $H$ is the real analytic manifold cut out by $\vert z^2 - w^2 \vert=1$ and $\vert zw \vert=2$. 
As for points $p \in H$, 
there is a holomorphic retract in the direction of $p$ if and only if it is of the form
$(e^{it_1}r_1, e^{it_2} r_2)$ where $r_1,r_2$ are positive real numbers satisfying $r_1^2 - r_2^2=\pm 1$ and 
$r_1r_2=2$  while $t_1,t_2 \in \mathbb{R}$ with $\vert t_2 - t_1 \vert$ being either $0$ or $\pi$
(i.e., the arguments of $p_1,p_2$ are either same or differ by $\pi$); and every such retract is linear.
\end{prop}

\noindent Let us remark that there are of-course points $p$ in $H$ other than those of the 
form mentioned in the theorem and note more importantly, that those points are 
precisely the ones in the boundary of the domain $D_h$ which are farthest from the origin.
Let us briefly justify these claims. For the former, note that for any $\alpha_0$ 
satisfying $0 \leq \alpha_0 \leq 1/16$, we can 
find $p = (p_1,p_2) = \left( re^{i\theta_1}, (2/r) e^{i\theta_2}\right) \in H_1 \cap H_2$
such that $\sin^2\theta_0 = \alpha_0$, 
where $\theta_0 = \theta_1 - \theta_2$. Now choose $\theta_0 \in I$ such that $\sin^2\theta_0 = \alpha_0$. 
Write $\theta_0 = \theta_1 - \theta_2$ for some $\theta_1, \theta_2 \in [0,2\pi)$. 
Using this $\theta_0$, define 
\[ 
r^2 : = \frac{1}{2}\left[\sqrt{1+16\cos^2\theta_0}+ \sqrt{1-16\sin^2\theta_0}\right].
\]
From (\ref{max}), we get that $p = (re^{i\theta_1},\frac{2}{r}e^{i\theta_2})$ lies 
in the intersection of $H_1$ and $H_2$, finishing the proof of the first claim. As for the
second claim, note that these points 
are farthest points from origin because
we obtain the following relation from (\ref{max}), 
\[
    |p_1|^2 + |p_2|^2 = r^2 + \frac{4}{r^2} = \sqrt{1+16\cos^2\theta_0}.
\]

\begin{rem}
The foregoing arguments -- again using proposition \ref{JJ-improved} at the outset to 
reduce considerations from the generality of holomorphic retracts to linear ones -- can be adapted to determine 
precisely the retracts through the origin of other `simpler' examples of 
non-convex (but pseudoconvex) bounded Reinhardt domains
such as $\mathbb{B} \cap H^2_{1/4}$ or $B_p \cap H^2_{1/4}$, where
$H^2_{1/4}=\{(z,w) \in \mathbb{C}^2 \; : \; \vert zw \vert <1/4\}$ and
where $B_p$ denotes the standard $\ell^p$-ball with $1 \leq p < \infty$. 
The case $p =\infty$ is different as in that case (i.e., for $\Delta^2 \cap H^2_{1/4}$)
there are non-linear retracts through the origin and therefore, discussed separately below. 
The reason we chose to deal with $D_h$ extensively is not because it fails to be Reinhardt
(though that is an added virtue) but 
rather because it fails to be non-convex in a maximal manner: note that 
there does {\it not} exist any bounded domain which is non-convex at all 
boundary points, as follows by standard considerations involving farthest point(s) from a fixed point in the 
interior and appealing to the Hahn -- Banach theorem; for a balanced domain the 
next best thing to ask for is one that is non-convex at an open dense subset of its 
boundary or, that the set of convexity points is of measure zero. Both of these conditions
are fulfilled by $D_h$. Despite its simplicity, one can still ask natural questions about the 
choice of $D_h$; without elaborating on them all, let us just mention one: why not
consider $D_q \cap H^2_{1/4}$ (where $D_q$ denotes the `$\ell_q$-ball' for $0<q <1$ as before), 
as this is also a maximally non-convex domain? Our 
answer to this is that though this is easier to analyse owing to it being Reinhardt,
it turns out that through every non-smooth point of $D_q \cap H^2_{1/4}$ there
is a linear retract and is a convex boundary point whereas this is not the case for $D_h$: not every
non-smooth point of $\partial D_h$ is a point of convexity, thereby is an even better example of 
a maximally non-convex domain. Moreover, through almost
every point $\partial D_h$, there passes analytic discs which lie within $\partial D_h$
making it at the other extreme compared to convex domains of finite type which are by far, well-studied.
Let us conclude this remark by finishing off what we began it with namely, retracts
of $B_p \cap H^2_{1/4}$ for $1\leq p < \infty$ which provide simple examples of balanced domains which are neither 
convex nor maximally non-convex i.e., lie somewhere in-between these extremes (the role of 
intersection with $H^2_{1/4}$ is mainly to create this intermediate possibility by destroying 
the convexity of $B_p$ and could be replaced with other domains). The retracts 
through the origin of such domains are precisely those points of $\partial B_p$ which remain 
after the intersection with $H^2_{1/4}$ is performed. Let us also draw attention to the
fact that the case $p=1$ is allowed whereas $p = \infty$ is not because there are plenty of non-linear 
retracts (through the origin) in the polydisc.
\end{rem}

\noindent Finally, let us mention in passing from this section to the next that while 
images of retracts under biholomorphic maps are retracts, this becomes false if we
more generally consider proper holomorphic maps and expect a direct extension.
That is to say, neither the backward nor forward images of retracts under proper (holomorphic)
maps need be retracts. We may use the examples dealt with here to establish this. While $D_h$ does not admit
any biholomorphic mappings to $\Delta^2$, it does admit proper holomorphic maps. Indeed, the 
map $F(z,w)=\left( z^2-w^2, zw/2 \right)$ renders a proper map of $D_h$ onto $\Delta^2$ fixing the origin;
note that the inverse image of the linear subspace in $\Delta^2$ spanned by $(1,1)$ is not
even a submanifold of $D_h$ leave alone a retract. As for forward images, this does not
help; however, we can give even simpler examples! Consider the map $(z,w) \mapsto (z^2,w)$; it
effects a proper map
from the ellipsoid $E_4=\{(z,w) \; : \; \vert z \vert^4 + \vert w \vert^2<1\}$ to 
the Euclidean ball $\mathbb{B}$. Note that the linear subspace given by $w=z$ is a 
retract of $E_4$ (as every one-dimensional linear subspace is). However, its image
$\{z=w^2\}$ being the graph of a non-linear map is not a retract of $\mathbb{B}^2$.\\

We now recall some of the basic definitions associated with analytic polyhedra; based on section 2.3 of \cite{Forst_Stein}, \cite{petrosyan} and 
\cite{Zim_generic}. Analytic polyhedra are obtained by looking at subsets of a complex manifold $X$ occurring  
as the intersection of sublevel sets of moduli of 
holomorphic functions on $X$. Several issues arise
here: firstly if $X$ is compact, the interior of such an 
intersection is empty; even if $X$ is non-compact and
a Stein manifold, the intersection can have several
connected components. 
We get around such issues by 
phrasing the definition as follows; we note
that there are some variations in the definitions followed by different authors and we are laying the one that we shall adopt, for this reason as well. A relatively compact domain in a complex manifold $X$ of the form
\begin{equation} \label{Apoly-defn}
D = \{x \in U: |f_j(x)| < 1~ \text{for}~ j= 1,2,\ldots,m\}
\end{equation}
where the $f_j$'s are holomorphic functions
on some open subset $U \supset \overline{D}$, is said to be an \textit{analytic polyhedron}. By definition
then, $\overline{D}$ is compact and connected. We shall refer to the $f_j$'s as defining functions, though it must be noted that such a system of functions are not at all uniquely determined (as is seen immediately by taking their powers for instance); nevertheless for convenience, we shall refer to them as `the' defining functions.
By adding more functions to the list of defining functions (the $f_j$'s above), we can ensure that $F = (f_1,\ldots,f_m) : U \to \mathbb{C}^m$ embeds $D$ properly into the polydisc $\Delta^m \subset \mathbb{C}^m$. If $D$ is a domain in $\mathbb{C}^N$, this follows by considering the map $G: U \to \mathbb{C}^{m+N}$ defined by $G(z) = (f_1(z),\ldots,f_m(z),z_1/R_1,\ldots,z_N/R_N)$, where $R_1,\ldots,R_N$ are positive reals such that $D \subset \Delta^m(0;R_1,\ldots,R_N)$, which exist
by virtue of the boundedness of $D$ here. Extending this result to the setting of complex manifolds requires additional work but the result is
well-known (see section 4.5.3 of \cite{Noguchi} and section 2.3 in \cite{Forst_Stein}). If the defining (analytic) functions can be chosen so that their number equals the dimension of the manifold i.e., $m=N$, then $D$ is called a {\it special analytic polyhedron}. \\ 

Our case of first interest is $U$ being an open set in $\mathbb{C}^N$. Note then that as per the above stated definition of analytic polyhedron (which is the one that we shall adopt), $D$ is bounded; though a couple of observations in the sequel do not 
require boundedness, most of our main results do. More importantly, it may very happen that the (topological) closure $\bar{D}$ may fall short of being equal to the `algebraic/order'-closure:
\[
\overline{D}^{alg} := \{ x \in U \; : \;  \vert f_j(x) \vert \leq 1 \text{ for } j=1,2, \ldots,m \}
\]
-- for concrete examples, see \cite{closure_anal_poly}.
We shall see that when $D$ is balanced, the closures in the topological and algebraic senses coincide. Towards this and other main results, we require a couple of more definitions.
If all the defining functions $f_j$'s can be chosen to be polynomials, then $D$ is called a  \textit{polynomial polyhedron}.  Next, in all cases, we have the following.  The boundary $\partial D$ of $D$ consists of the \textit{closed faces}
$\sigma_j := \{z \in \partial D: |f_j(z)| = 1\}$ intersecting along the level-$r$ \textit{ribs} $\sigma_{j_1,\ldots,j_r} := \sigma_{j_1} \cap \ldots \cap \sigma_{j_r}$. That is, the boundary $\partial D$ can be expressed as the union of the faces; this is
based on the fact that $\overline{D} \subset \{z \in U: |f_j(z)| \leq 1~ \text{for}~ j= 1,2,\ldots,m\}$ which implies that $\partial D \subset \cup_{j=1}^{m} \{z \in U: |f_j(z)| = 1 \text{ and } \vert f_k(z) \vert \leq 1,\text{ for all } k \neq j \} =: {\partial D}^{alg}$. Associated with the faces $\sigma_j$, we need to introduce the following sets:
\[
\hat{\sigma}_j : = \{z \in \mathbb{C}^N: |f_j(z)| =1\} = Z(|f_j(z)|^2-1),
\]
where $Z(|f_j(z)|^2-1)$ denotes the zero set of $|f_j(z)|^2-1$.
\[
\tilde{\sigma}_j : = \{z \in \partial D: |f_j(z)| =1,~|f_l(z)| < 1~\text{for}~1 \leq l \leq m~\text{with}~l \neq j\}
\]
Note that $\sigma_j = \hat{\sigma}_j \cap \partial D$. The reader may find the articles \cite{Singularity}, \cite{Real_Var},\cite{Strat} to provide helpful background information concerning the structure of real analytic varieties and semi-varieties, though we do not make explicit use of their general structure results here.\\

Following the definition from \cite{Zim_generic}, the analytic polyhedron $D$ is said to be \textit{generic (non-degenerate)} if we can choose the functions $f_1,\ldots,f_m$ so that whenever $\zeta \in \sigma_{j_1,\ldots,j_r}$, the vectors 
\begin{equation*}\label{gradient}
\nabla f_{j_1}(\zeta),\ldots,\nabla f_{j_r}(\zeta)
\end{equation*}
are $\mathbb{C}$-linearly independent. We shall refer to this as the genericity condition on the $f_j$'s. The symbol $\nabla$ denotes the complex gradient i.e., the complex derivative. Though this notation is more common for real differentiable functions. We use this here follows Zimmer \cite{Zim_generic}. The context will make it clear whether we use complex or real gradient.
\begin{rem}
\begin{enumerate}
\item Every linear retract of an analytic polyhedron is again an analytic polyhedron because if $D_L$ is a linear retract of $D$ for some linear subspace $L$ of $\mathbb{C}^N$ then $D_L = \{z \in U \cap L: |f_j(z)| < 1~\text{for}~ 1 \leq j \leq m\}$ is an analytic polyhedron on $L$. 
\item Note that every analytic polyhedron $D$ in any complex manifold $X$ is Stein. Needless to say when $X = \mathbb{C}^N$, $D$ is a pseudoconvex domain.
\item Every (bounded) analytic polyhedon $D$ is hyperconvex because if $|f_j(p)| = 1$ for some $j$ then $\varphi(z) := |f_j(z)|^2-1$ is a plurisubharmonic function on $D$ and $\lim_{D \ni z \to p}\varphi(z) = 0$ (this suffices to verify hyperconvexity by the remark \ref{hyp_cnvx_rem} in the preliminaries section \ref{Hypercnvx}).
\item Product of two analytic polyhedra is again an analytic polyhedron because if $D_1 := \{z \in U_1: |f_j(z)| < 1~ \text{for}~ j =1,\ldots,m_1\}, ~~D_2 : = \{w \in U_2: |g_k(z)| < 1~ \text{for}~ k =1,\ldots,m_2\}$, where $f_j$'s are holomorphic on an open set $U_1$ containing $\overline{D}_1$ and $g_k$'s are holomorphic on an open set $U_2$ containing $\overline{D}_2$ then $D_1 \times D_2 := \{(z,w) \in U_1 \times U_2: |f_j(z)| < 1,~~~~ |g_k(w)| < 1~ \text{for}~ j =1,\ldots,m_1,~ k = 1,2,\ldots,m_2\}$. 
\end{enumerate}
\end{rem}
\begin{lem}\label{smooth=regular}
Let $D$ be a (bounded) balanced  analytic polyhedron in $\mathbb{C}^N$ given by
\[
D := \{z \in U: |f_j(z)| < 1~ \text{for}~ j= 1,2,\ldots,m\}
\] 
(where the $f_j$'s are holomorphic functions on some domain $U \subset \mathbb{C}^N$ containing $\overline{D}$). Suppose also that the system of defining functions $(f_j)$ is minimal, i.e. no function $f_j$ can be removed without changing $D$.For each $j~\text{with}~1 \leq j \leq m$, consider the analytic variety $\hat{\sigma}_j  = Z(|f_j|^2-1)$. If $p \in \tilde{\sigma}_j$  is a smooth point of $\hat{\sigma}_j$ (recall that smooth points of $\hat{\sigma}_j$ form a dense subset of $\hat{\sigma}_j$.) for some $j$, then $p$ is a regular point of $\hat{\sigma}_j$ of dimension $2N-1$. Consequently, for each $j$, $\text{dim}(\hat{\sigma}_j) = 2N-1$.
\end{lem}
\begin{proof}
Without loss of generality, we may assume that $p \in \tilde{\sigma}_1$ is a smooth point of $\hat{\sigma}_1$. We want to show that $p$ is a regular point of $\hat{\sigma}_1$ of dimension $2N-1$. To prove this by contradiction, assume that there exists an open set $\tilde{U}$ of $\mathbb{C}^N$ containing $p$ such that $\hat{U} := \tilde{U} \cap \hat{\sigma}_1$ is a smooth manifold of dimension (strictly) less than $2N-1$. It follows owing to Sard's theorem -- more precisely, by applying corollary 6.11 in \cite{J.Lee} to $\pi^r_{|_{\hat{U}}} : \hat{U} \to \partial{\mathbb{B}}$ -- that $\pi^r(\hat{U})$ is a nowhere dense subset of $\partial \mathbb{B}$. Let $q := \pi^r(p)$. Choose a sequence $q_n \in \partial \mathbb{B} \setminus \pi^r(\hat{U})$ such that $q_n$ converges to $q$. If $\alpha_n$ denotes the supremum of all positive numbers $\alpha$ such that $\alpha q_n \in D$, then $p_n := \alpha_n q_n$ lies on $\partial D$. Since $\tilde{\sigma}_1$ is an open subset of $\partial D$ containing $p$, it follows from the continuity of the map $\pi^r_{|_{\partial D}}: \partial D \to \partial \mathbb{B}$ that $p_n \in \tilde{\sigma}_1$ for sufficiently large $n$. 
Note that $\|p_n\| = |\alpha_n| \|q_n \|$. 
With 
$h$ denoting the Minkowski functional of $D$ as 
usual, we have $h(p_n)=1$, thereby:
\begin{equation}\label{Minkowski}
h(q_n)=h\big(\frac{1}{\alpha_n} p_n \big)
= \frac{1}{\alpha_n}=\frac{1}{\|p_n\|}
\end{equation}
 It now follows from the convergence of $h(q_n)$ to $h(q)$ wherein the continuity of $h$ owes to the hyperconvexity of $D$, that: $\|p_n\|$ converges to $\|p\|$. For $\delta > 0$, let $S_{\delta} := \{z \in \mathbb{C}^N: \|p\|-\delta < \|z\| < \|p\| + \delta\}$.  Now, we claim that there exists $\delta > 0$ and an open set $W \subset \pi^r(\tilde{U})$ of $\partial \mathbb{B}$ containing $q$ such that $(\pi^r)^{-1}(W) \cap S_{\delta} \subset \tilde{U}$. To prove this, consider the polar co-ordinates map $\Phi: \mathbb{C}^N \setminus \{0\} \to \mathbb{R}^+ \times \partial \mathbb{B}$ defined by $\Phi(z) = (\|z\|,z/\|z\|)$. Note that $\Phi$ is a homeomorphism and thereby $\Phi(\tilde{U})$ is an open subset of $\mathbb{R}^+ \times \partial \mathbb{B}$. This implies that there exists $\delta > 0$ and an open set $W \subset \pi^r(\tilde{U})$ such that $V := (\|p\|-\delta,\|p\| + \delta) \times W \subset \Phi(\tilde{U})$. Hence $(\pi^r)^{-1}(W) \cap S_{\delta} = \Phi^{-1}(V) \subset \tilde{U}$, which finishes the proof of the claim. Since $q_n$ converges to $q$ and $W$ is an open subset of $\partial \mathbb{B}$ containing $q$, it follows that $q_n \in W$ for large $n$. As $\pi^r(p_n) = q_n$, we have $p_n \in (\pi^r)^{-1}(W)$ for sufficiently large $n$. Furthermore, since $\|p_n\|$ converges to $\|p\|$, it follows that $p_n \in  S_{\delta}$ for sufficiently large $n$, and hence  $p_n \in  (\pi^r)^{-1}(W) \cap S_{\delta}$, which is contained in $\tilde{U}$. Thus  $p_n \in \hat{U}$ for large $n$, which contradicts the choice of $q_n = \pi^r(p_n) \notin \pi^r(\hat{U})$.
\end{proof}
\begin{prop}\label{Bal_Anal_Poly}
Let $D$ be a (bounded) balanced  analytic polyhedron in $\mathbb{C}^N$ given by
\[
D := \{z \in U: |f_j(z)| < 1~ \text{for}~ j= 1,2,\ldots,m\}
\] 
(where the $f_j$'s are holomorphic functions on some domain $U \subset \mathbb{C}^N$ containing $\overline{D}$.) Suppose also that the system of defining functions $(f_j)$ is minimal. Then we can choose the $f_j$'s to be homogeneous polynomials (so $D$ is a balanced polynomial polyhedron).
\end{prop}
\begin{proof}
First, we prove this proposition under the assumption that $D$ is a generic analytic polyhedron because the proof is more elementary in this case. We can express $D$ as $D = F^{-1}(\Delta^m)$, where $F : U \to \mathbb{C}^m$ defined by $F(z) = (f_1(z),\ldots,f_m(z))$.
Let $F(0) = (a_1,\ldots,a_m)$. 
As $a$ need not be zero, we compose $F$ with the automorphism $\varphi$ of $\Delta^m$ given by 
\[
    \varphi(z) = \left(\frac{z_1-a_1}{1-\overline{a_1}z_1},\ldots,\frac{z_m-a_m}{1-\overline{a_m}z_m}\right).
\] Then $\varphi \circ F(0) = 0$. 
After possibly shrinking $U$, to avoid $1/\overline{a_j}$ for all $j$ with $1 \leq j \leq m$, define a  holomorphic function $g_j : U \to \mathbb{C}$ by 
\[
g_j(z) = \frac{f_j(z)-a_j}{1-\overline{a_j}f_j(z)}.
\]
Since $D = (\varphi \circ F)^{-1}(\Delta^m)$, it follows that
\[
    D = \{z \in U: |g_j(z)| < 1~ \text{for}~ j =1,2,\ldots,m\}.
\]
 Moreover, $D\varphi_{|_0}$ is an invertible linear map and therefore preserves linear independence. It follows that $g_j$'s satisfy the genericity condition. In the rest of the proof, we relabel these functions as $(f_j)$
 and regard them as the defining functions for 
$D$.\\ 

Fix $z_0 \in \tilde{\sigma}_1$. We can divide the proof into the following steps. \\

\textit{Step-1:} There exists positive integer $d(z_0)$ such that for each $\lambda \in \Delta$, we have $f_1(\lambda z_0) = \lambda^{d(z_0)} B(\lambda,z_0)$, where $B(\lambda,z_0)$ is some Blaschke product on $\Delta$ with $B(0,z_0) \neq 0$.\\

Consider the holomorphic function $h$ on $\Delta$ defined by $h^{z_0}(\lambda) = f_1(\lambda z_0)$. 
Note that $h^{z_0}$ maps $\Delta$ into $\Delta$ because $D$ is a balanced hyperconvex domain. We want to show that $h^{z_0}$ is a proper map. 
To prove this, for each $k$ with $1 \leq k \leq m$, define 
\[
A_k := \{\theta \in [-\pi,\pi]: |f_k(e^{i\theta}z_0)| =1\}.
\]
As $D$ is balanced, it follows that 
\[
[-\pi,\pi] = \bigcup_{k=1}^m A_k.
\]
Now, we claim that the sets $A_k$ are pairwise disjoint. To prove this claim, consider the invertible $\mathbb{C}$-linear map $T: \mathbb{C}^N \to \mathbb{C}^N$ defined by $T(z) = e^{i\theta}z$ for some $\theta \in [-\pi,\pi]$. Note that $T$ maps $D$ into $D$, as $D$ is balanced. Since $T$ is an invertible linear transformation, $T$ maps $\partial D$ into $\partial D$. Note that $T$ preserves smooth boundary points of $D$ because if $p$ is a smooth boundary point of $D$ and $r$ be its smooth defining function near $p$ then $\tilde{r}(z) = r(e^{-i\theta}z)$ serves as a smooth defining function near $T(p)$. This finishes the proof of the claim that the sets $A_k$ are pairwise disjoint.\\

Hence $A_k^c = \cup_{j=1, j \neq k}^mA_j$. Since the $A_j$'s are closed subsets of $[-\pi,\pi]$, $A_k^c$ is closed. This implies that $A_k$ is both open and closed subset of $[-\pi,\pi]$. 
As $[-\pi,\pi]$ is connected, either $A_k = \phi$ or $[-\pi,\pi]$. 
Thus there is atleast one $j$ such that $A_j = [-\pi,\pi]$. Without loss of generality we may assume that $j=1$, i.e. $|f_1(\lambda z_0)| = 1$ whenever $|\lambda| = 1$. 
Hence, we obtain that $h^{z_0}$ is a proper map on $\Delta$. Therefore $h^{z_0}$ is of the form $h^{z_0}(\lambda) = B(\lambda, z_0)$, where $B(\lambda, z_0)$ is some finite Blaschke product.
Since $h^{z_0}(0) =0$, there exists positive integer $d(z_0)$ such that 
\begin{equation}\label{Blaschke}
    f_1(\lambda z_0) = h^{z_0}(\lambda) = B(\lambda,z_0) = \lambda^{d(z_0)} \frac{P(\lambda,z_0)}{Q(\lambda,z_0)}
\end{equation}
where $P(\lambda,z_0)$ and $Q(\lambda, z_0)$ are polynomials in $\lambda$, which does not vanishes at $\lambda =0$, and ${\rm deg}(P(\lambda,z_0)) = {\rm deg}(Q(\lambda, z_0))$. 
 Since $|f_1(z_0)| =1$, after multiplying $f_1$ by an unimodular constant, we may assume that $f_1(\lambda z_0) = B(\lambda,z_0)f_1(z_0)$.\\
 
\textit{Step-2:} There exists $z' \in \tilde{\sigma}_1$ such that $z'$ is a growth vector of $f_1$ at $0$ and $f_1$ does not vanish on $\Delta_{z'}^* := \{\lambda z': |\lambda| \leq 1,~ \lambda \neq 0\}$.\\

Let $Z_{f_1}$ denotes the zero set of $f_1$ and let $C^t := C(Z_{f_1},0)$ denotes the tangent cone to $Z_{f_1}$ at the origin.  
By theorem \ref{tangent_cone}, the set of all non-growth vectors for $f_1$ at $0$ is precisely equal to the tangent cone $C^t$.  Let $A := C^t \cap \tilde{\sigma}_1$. First, we want to show that  $A$ is a nowhere dense subset of $\tilde{\sigma}_1$. By proposition \ref{non-growth}, $C^t$ is a closed nowhere dense subset in $\mathbb{C}^N$. Consider the radial projection $\pi^r : \mathbb{C}^N \setminus \{0\} \to \partial \mathbb{B}$ defined by $\pi^r(z) = z/\|z\|$, where $\|\cdot\|$ denotes the Euclidean norm on $\mathbb{C}^N$. Note that $\pi^r(C^t) = C^t \cap \partial \mathbb{B}$ because the cone generated by $C^t \cap \partial \mathbb{B}$ is precisely equal to $C^t$. Since $dim(C^t) = 2N-2$ is less than $dim(\partial \mathbb{B})$, $C^t \cap \partial \mathbb{B}$ is a nowhere dense subset of $\partial \mathbb{B}$. Owing to the hyperconvexity of $D$, $\pi^r_{|_{\partial D}} : \partial D \to \partial \mathbb{B}$ is a homeomorphism. Hence $(\pi^r)^{-1}(C^t \cap \partial \mathbb{B})$ is a nowhere dense subset of $\partial D$. This shows that $A$ is a nowhere dense subset of $\tilde{\sigma}_1$. Let $\tau_1 := \tilde{\sigma}_1 \setminus A$. 
To prove the claim in step-2, it suffices to show that there exists $z' \in \tau_1$ such that $f_1$ does not vanish on $\Delta_{z'}^*$. 
To prove this by contradiction, assume that for every $z \in \tau_1$, $f_1$ vanishes at some point in $\Delta_{z}^*$. 
This implies that for every $z \in \tau_1$, $f_1$ vanishes at some point in $\Delta_{z}^*$. It therefore follows that $\pi^r(Z_{f_1}^*) \supseteq \pi^r(\tau_1)$, where $Z_{f_1}^* := Z_{f_1} \setminus \{0\}$. 
Denote by $Z_{f_1}^{\text{reg}}, Z_{f_1}^{\text{s}}$ respectively, the regular and singular points of $Z_{f_1}^*$. 
Hence $\pi^r(Z_{f_1}^{*} \setminus Z_{f_1}^s) \supseteq \pi^r(\tau_1) \setminus \pi^r(Z_{f_1}^s)$. By the principle of invariance of domain, it follows that $\pi^r(\tau_1)$ is an open subset of $\partial \mathbb{B}$. Thereby, in particular, its Hausdorff dimension is $2N-1$. 
By Corollary 1 of section 2.6 in \cite{Chirka}, the Hausdorff dimension of $\pi^r(Z_{f_1}^s)$ is at most $2N-4$. It follows that $f_1$ vanishes on a set of positive $(2N-1)$-dimensional Hausdorff measure. Hence $f_1$ vanishes on $U$. This contradicts the fact that $f_1$ is a non-constant function, which finishes the proof of Step-2.\\

\textit{Step-3:} $f_1$ is a homogeneous polynomial on $\mathbb{C}^N$.\\

By the above Step-2, there exists a growth vector $z' \in \tilde{\sigma}_1$  of $f_1$ at $0$ and $f_1$ does not vanish on $\Delta_{z'}^*$. By proposition \ref{non-growth}, the tangent cone $C^t$ is a closed nowhere dense subset of $\mathbb{C}^N$. 
Choose an open set $V$ around $z'$ of $\tilde{\sigma}_1$ such that $\overline{V}$ is contained in the complement of $C(Z_{f_1},0)$. Consider the punctured (complex) cone  $C_V'$ generated by $V$. Note that $\overline{C'}_V$ is contained in the complement of $C^t$. 
Now, we want to prove that there exists an open set $U'$ of $D$ containing $0$ such that $Z_{f_1}^*$ does not intersect $U' \cap C_V'$. 
To prove this by contradiction, assume that it so happens that for every neighbourhood $U'$ of the origin, $Z_{f_1}^* \cap (U' \cap C_V')$ is non-empty. 
This gives rise to the existence of a sequence $q_j \subset Z_{f_1}^* \cap (U' \cap C_V')$ such that $q_j \to 0$. Let $p_j := q_j/h(q_j)$, where $h$ is the Minkowski functional of $D$. Since $(p_j)$ is bounded, it has a convergent subsequence $(p_{j_k})$, which converges to some (growth vector) $v$ lying in $\overline{V}$. 
By definition of tangent cones, this implies that $v \in C(Z_{f_1},0)$, which contradicts $\overline{V}$ is disjoint from $C^t$. 
Therefore, there exists an open neighbourhood $U' \subset D$ of $0$ such that $Z_{f_1}^*  \cap (U' \cap C_V') = \phi$. We shall now show that $f_1$ does not vanish in outside of $U'$ if only we pass to a possibly smaller cone. To work this out rigorously, first choose $r>0$ such that $\mathbb{B}(0,r) \subset U'$. 
Denote the closed annulus $\mathbb{A}[0;r,1]$ by $B$ for short. 
Consider the function $\tilde{f}_1(\lambda,z) = f_1(\lambda z)$, which is in particular continuous on the compact set $B \times \overline{V}$. 
Since $f_1$ is a non-vanishing function on $\Delta_{z'}^*$, there exists $\epsilon > 0$ such that 
\[
|f_1(\lambda z')| > \epsilon~ \text{for all}~ \lambda \in \mathbb{A}[0;r,1].
\]
By the uniform continuity of $\tilde{f}_1$ on $B \times \overline{V}$, there exists $\delta > 0$ such that 
\[
\|\tilde{f}_1(\mu,z') - \tilde{f}_1(\lambda,z)\| < \epsilon/2~ \text{whenever}~ \|(\lambda, z) - (\mu, z')\| < \delta.
\]
In particular, if $\|z-z'\| < \delta$, then $\|\tilde{f}_1(\lambda,z') - \tilde{f}_1(\lambda,z)\| < \epsilon/2$ for all $\lambda \in B$. By triangle inequality, we have
\[
    \|\tilde{f}_1(\lambda,z)\| \geq \|\tilde{f}_1(\lambda,z')\| - \|\tilde{f}_1(\lambda,z') - \tilde{f}_1(\lambda,z)\| >  \epsilon /2
\]
whenever $\|z-z'\| < \delta$ and $z \in \overline{V}$.
Define the punctured cone
\[
C_{\delta}'(z') := \{\lambda z: \|z-z'\| < \delta,~z \in \partial C_V' \cap \partial D ,~ \lambda \in \Delta^*\}. 
\]
Then note that $f_1$ is non-vanishing on $K' := C_{\delta}'(z') \setminus U'$ and hence on all of $C_{\delta}'$ as $Z_{f_1}$ does not intersect $U'$ within $C_{\delta}'(z')$. 
Since every vector in $C_{\delta}'$ is a growth vector of $f_1$ at $0$, it follows from \ref{Blaschke} now that $f_1(\lambda z) = \lambda^{d(z')}f_1(z)$ for all $\lambda \in \Delta$. 
This shows that $d(z) = d(z')$ for all $z$ near $z'$, i.e. $d(z)$ is constant on $C_{\delta}'$. By the identity principle, $f_1(\lambda z) = \lambda^d f_1(z)$ for all $z \in U$, $\lambda \in \Delta$.
Using this functional equation, we can extend  $f_1$  to all of $\mathbb{C}^N$ by defining $f_1(\lambda z) = \lambda^{d}f_1(z)$. It follows, therefore that $f_1$ is a homogeneous entire function, thereby a homogeneous polynomial. So, $f_j$ is a homogeneous polynomial for each $j =1,2,\ldots,m$.\\

We now consider the case that $D$ is not necessarily a generic analytic polyhedron.
As in the above case, we may assume that $f_j(0) =0$ and that the collection of defining functions $(f_j)$ is minimal in the sense that no proper subset of  $(f_j)$ suffices to define $D$. The proof for this case is similar to the proof for the above case except for Step-2. Note that in this case, $\tilde{\sigma}_1$ need not be a manifold. So, we consider the analytic variety $\hat{\sigma}_1  = Z(|f_1|^2-1)$.
This analytic variety $\hat{\sigma}_1$ has a well-defined dimension $d$, which can at most be $2N-1$, where $d$ is maximal among the dimensions of smooth points, as laid down precisely by definition \ref{Def1} in the preliminaries section. 
Note that $\tilde{\sigma}_1^{reg} := \hat{\sigma}_1^{reg} \cap \tilde{\sigma}_1$ --- recall the notation and basic facts about such points in analytic variety from the preliminaries section \ref{Semi-anal} --- is a smooth submanifold of real dimension $d$ (for further details regarding the structure of analytic varieties and semi-varieties, we refer the reader to \cite{Strat}).
By the above lemma \ref{smooth=regular},  $\text{dim}(\tilde{\sigma}_1^{reg}) = 2N-1$. Now, we apply a similar argument in Step-2 to $\tilde{\sigma}_1^{reg}$ replacing $\tilde{\sigma}_1$ therein. Hence, we get the desired result of Step-2. As noted earlier, this was all that was required to finish the proof in the general case.
\end{proof}
\begin{rem}\label{cont_minkwski}
Let $D$ be any balanced analytic polyhedron. So, by the above proposition, $D := \{z \in \mathbb{C}^N: |f_j(z)| < 1~ \text{for}~ j= 1,2,\ldots,m\}$, where each $f_j$ is a homogeneous polynomial of degree $d_j := {\rm deg}(f_j)$. Define $h : \mathbb{C}^N \to \mathbb{R}$ by $h(z) := \max\{|f_j(z)|^{1/d_j}: 1 \leq j \leq m\}$. 
Note that $h$ is a continuous function on $\mathbb{C}^N$ and $h(\lambda z) = |\lambda| h(z)$.
Appealing to proposition \ref{Prop_Minkowski} (c), we conclude that $h$ is the Minkowski functional on $D$. Furthermore, by proposition  \ref{Prop_Minkowski} (d), we therefore deduce that the topological closure equals the algebraic closure: $\overline{D} = \{z \in \mathbb{C}^N: |f_j(z)| \leq 1~ \text{for}~ j= 1,2,\ldots,m\}$.
\end{rem}

\begin{lem}\label{power_linear} 
Let $f \in \mathbb{C}[z_1,\ldots,z_N]$ be a homogeneous polynomial. If $|f|$ 
is a non-zero constant on a complex hyperplane, then $f$ is a power of some linear polynomial.
\end{lem}
\begin{proof}
Assume that $|f|$ is a non-zero constant, say $c$, on some complex hyperplane $l$. Observe that this, in particular, means $l$ does not pass through the origin. Parametrize $l$ in the standard manner as $\varphi(\zeta_1,\ldots,\zeta_{N-1}) = p + \zeta_1 v_1 + \ldots + \zeta_{N-1}v_{N-1}$ for some
choice of a point $p$ on $l$ and  non-zero vectors
$v_j \in \mathbb{C}^N$. The
maximum modulus principle (for holomorphic functions of several complex variables) applied to the
holomorphic function $f \circ \varphi$ confirms that $f$ is constant ($=c$) on $l$.
This implies 
that $f \circ \varphi(\zeta)$ is a non-zero constant, where $\varphi(\zeta_1,\ldots,\zeta_{N-1}) = p + \zeta_1 v_1 + \ldots + \zeta_{N-1}v_{N-1}$ is the 
parametrization of $l$ passing through $p$ in the direction of vectors $v_j$. After multiplying by $1/c$, we may 
assume that $f$ is identically equal to $1$ on $l$. We can also express $l$ as the zero set $Z(g-1)$, where $g$ is a homogeneous linear polynomial in $\mathbb{C}[z_1,\ldots,z_N]$. Let $\tilde{f}(z_1,\ldots,z_N) := f(z_1,\ldots,z_N)-1$. Note that
$\tilde{f}$ vanishes on  $Z(g-1)$. By an application of the Hilbert Nullstellensatz, there exists $k \in \mathbb{N}$ such that $g-1$ divides $\tilde{f}^k$. 
Being a linear polynomial $g-1$ is 
irreducible in (the unique factorization domain) $\mathbb{C}[z_1,\ldots,z_N]$, thereby a prime element therein.
This implies that $g-1$ divides $\tilde{f}$. So, we write $\tilde{f} = (g-1)\tilde{h}$ for some non-zero polynomial $\tilde{h}$. Substituting $\tilde{f}$ by $(f-1)$, we obtain that
\[
f = (g-1)\tilde{h} +1 .
\]
Note that $\tilde{h}(0) = 1$ because $f(0) = 0$. So, we can write 
$\tilde{h}(z_1,\ldots,z_N) = h(z_1,\ldots,z_N) + 1$, where 
$h \in \mathbb{C}[z_1,\ldots,z_N]$ with $h(0) = 0$. 
Hence, the above equation reduces to $f = gh +g-h$. Write $h$ in its standard homogeneous decomposition: 
$h = h_1 + \ldots + h_r$, where $h_i$ is a homogeneous polynomial of degree $i$ (so the integer $r$ is one less than the degree of $f$). Note that $gh_i$ is of degree $i$ for each $i$. Thus, the above 
equation may be rewritten as
\[
    f = g(h_1 + \ldots + h_r) + g - (h_1 + \ldots + h_r) = (g-h_1) + (gh_1-h_2) + \ldots + 
    (gh_{r-1}-h_r) + gh_r
\]
As $f$ is a homogeneous polynomial of degree $r+1$, 
we infer that the first $r$ homogeneous components (i.e., all except the last) are identically zero. We 
therefore conclude
\[
h_1 = g, ~h_2 = gh_1, \ldots, \;h_r = gh_{r-1}, \; \; gh_r = f.
\]
This implies that $h = g+g^2+\ldots + g^r$ and $f = g^{r+1}$.
\end{proof}
\begin{rem}
Let us note a special case for future use. If $f$ is a homogeneous polynomial in $z$ and $w$ (i.e. $N=2$) then the level set of $|f|$ does not contain any complex line unless $f$ is a power of a linear polynomial. Consequently, a homogeneous analytic polyhedron in $\mathbb{C}^2$ is $\mathbb{C}$-extremal i.e. $\partial D$ does not contain any (affine) $\mathbb{C}$-line segments iff for each $j$, $f_j$ is not a power of a linear polynomial. 
\end{rem}

\subsection*{Proof of theorem \ref{max_noncnvx}}.
To prove this theorem by contradiction, assume that there exists a non-trivial retract $Z$ 
of $D$ passing through the origin in the direction of $p \in L := T_0Z$. Let $\rho$ 
be a retraction map from $D$ onto $Z$. By lemma \ref{Dalpha}, $T := D \rho_{|_0}$ is a 
linear projection map from $\mathbb{C}^N$ onto $L$ and $T$ maps $D$ onto $D_L := D \cap L$. Let 
$H_p := p + H_0$, where $H_0 := {\rm \ker}(T)$. Note that the affine $\mathbb{C}$-linear subspace 
$H_p$ does not intersect $D$ because if $q \in H_p \cap D$ then 
$T(q) = p$, which contradicts the fact that $T$ maps $D$ into $D$. Since $\text{dim}(\ker(T)) \geq 1$, there exist a non-zero vector $v \in \ker(T)$ such that the complex line $\ell_p := \{p+\zeta v: \zeta \in \mathbb{C}\}$ does not intersect $D$.\\

\noindent By the hypothesis, $p$ belongs to one of the open faces, i.e., $p \in \tilde{\sigma}_j$ for some $j$. This implies that $|f_j(p)| = 1$ and $|f_l(p)| < 1$ for $l \neq j$.  Consider the 
holomorphic function $\psi : \mathbb{C} \to \mathbb{C}$ defined by 
$\psi(\zeta) = f_j(p+\zeta v)$. Since $f_l$'s are continuous, there exists $r > 0$ such that $|f_l(p+\zeta v)| < 1$ for all $\zeta \in \Delta_{r}$ and all $l$ with $l \neq j$. Hence, $|f_j(p+\zeta v)| \geq 1$ for all $\zeta \in  \Delta_r$ because $\ell_p$ does not intersect $D$.  
 By the  
 minimum principle available for the non-vanishing holomorphic function $\psi$, it follows that $|\psi(\zeta)| =1$ for all $\zeta \in \Delta_r$.  This contradicts our assumption that $\partial D$ is $\mathbb{C}$-extremal.\qed
 \begin{lem}\label{homo_open_dense}
     \begin{enumerate}\rm
         \item Let $f$ be a non-constant homogeneous polynomial in $\mathbb{C}^N$. Consider the real algebraic variety $A := Z(r(z))$, where $r(z) = |f(z)|^2-c^2$ for some $c \in \mathbb{R}^+$. Then $\nabla r$ is nowhere vanishing on an open dense subset of $A$.
         \item Let $D$ be a balanced analytic polyhedron in $\mathbb{C}^N$ and let $(f_j)_{j=1}^m$ be minimal defining functions for $D$. 
         If $r_j(z) := |f_j(z)|^2-1$ then  $\nabla r_j$ is nowhere vanishing on an open dense subset of $\sigma_j$.
     \end{enumerate}
 \end{lem}
 \begin{proof}
 \begin{enumerate}
     \item Firstly note that $\nabla r(z) = |f(z)|\overline{\nabla f(z)}$. To prove (1), it suffices to show that $\nabla f$ never vanishes on an open dense subset of $A$, since $|f(z)| = 1$ for all $z \in A$. Fix $\theta \in [-\pi,\pi]$. Consider the complex analytic hypersurface 
     \[
     L_{\theta} := \{z \in \mathbb{C}^N: f(z) = ce^{i\theta}\}.
     \]
      Suppose $\nabla f(z)$ vanishes on some open subset $W := B \cap A$ of $A$.  In particular, $\nabla f$ vanishes on an open subset $V := W \cap L_{\theta}$ of $L_{\theta}$. Since $\nabla f$ is homogeneous,  $\nabla f$ vanishes on the complex cone $C_V$ generated by $V$. It follows from the Weierstrass preparation theorem that $V$ is biholomorphic to an open subset of $\mathbb{C}^{2N-2}$. Hence, $C_V$ is an open subset of $\mathbb{C}^N$. By the identity principle applied to $\nabla f$, we obtain that $\nabla f$ vanishes on $\mathbb{C}^N$. Hence, $f$ is a constant function, which is a contradiction.
     \item By the above proposition \ref{Bal_Anal_Poly}, for each $j$, $f_j$ is a homogenous polynomial in $\mathbb{C}^N$. Hence by (1), $\nabla f_j$ never vanishes on an open dense subset of $Z(|f_j|^2-1)$. This implies that $\nabla r_j$ never vanishes on an open dense subset of $\sigma_j$.
\end{enumerate}
\end{proof}
\begin{thm}\label{Weil_open_piece}
    Let $D$ be a (bounded) balanced $\mathbb{C}$-extremal analytic polyhedron in $\mathbb{C}^N$. Then $D$ does not admit any retract along any open piece of directions at the origin. 
\end{thm}
\begin{proof}
 First, we assume that $D$ is a generic analytic polyhedron. To prove this theorem by contradiction, assume that there exists a non-empty open subset $\Gamma$ of $\partial \mathbb{B}$ such that for each $v \in \Gamma$ there is a retract $Z$ of $D$ passing through the origin with $v \in T_0Z$. 
    Consider the map $\pi^r: \mathbb{C}^N \setminus \{0\} \to \partial \mathbb{B}$ defined by $\pi^r(z) = z/\|z\|$, where $\|\cdot\|$ denotes the Euclidean norm on $\mathbb{C}^N$. 
    Note that $\pi^r$ is a smooth map on $\mathbb{C}^N \setminus \{0\}$. As $D$ is bounded, it follows from the continuity of $\pi^r$ that $V := (\pi^r)^{-1}(\Gamma) \cap \partial D$ is a non-empty open subset of $\partial D$. 
 By the theorem \ref{max_noncnvx}, $V$ is contained in the union of all ribs $\sigma_{j,k}$.  
    Note that $\sigma_{j,k}$ is a smooth manifold of dimension $2N-2$ , due to the genericity condition of $D$. 
As ${\rm dim}(\sigma_{j,k}) < {\rm dim}(\partial \mathbb{B})$, it follows from Sards theorem (see specifically corollary 6.11 in \cite{J.Lee}) applied to the restriction $\pi^r_{|_{\sigma_{j,k}}}$, that $\pi^r(\sigma_{j,k})$ is a measure zero subset of $\partial \mathbb{B}$. Since $V$ is contained in the union of finitely many ribs $\sigma_{j,k}$, $\pi^r(V)$ is a measure zero subset of $\partial \mathbb{B}$.
Hence, $\pi^r(V)$ does not have an interior point in $\partial \mathbb{B}$.\\

\noindent Next note that $\pi^r(V) = \Gamma$. While $\pi^r(V) \subset \Gamma$ comes just from definition
of these sets, the reverse inclusion follows from
the hypothesis that $D$ is bounded -- indeed 
the fibers $(\pi^r)^{-1}(z)$ are rays (from the origin), every one of which have got 
to intersect $\partial D$ as the origin lies 
within $D$. Thus $\pi^r(V) = \Gamma$; but then, this contradicts the conclusion of the previous para that $\pi^r(V)$ does not have an interior point in $\partial \mathbb{B}$.\\

We now move to the case where $D$ is not necessarily a generic analytic polyhedron. By proposition \ref{Bal_Anal_Poly}, we may assume that $f_j$'s are homogeneous polynomials and form a minimal set of defining functions for the domain $D$. To prove this theorem, owing to theorem \ref{max_noncnvx} (as argued above), 
it suffices to show that all ribs $(\sigma_{j,k})$ form a nowhere dense subset of $\partial D$. For simplicity of presentation, we write the details for the case $j=1,k=2$, other cases being similar. To prove this by contradiction, assume that there exists $p \in \sigma_{1,2}$ and an open subset $B$ of $\mathbb{C}^{N}$ containing $p$ such that $V := B \cap \partial D \subset \sigma_{1,2}$. This means that for all $z \in V$, we have $|f_1(z)| = |f_2(z)| = 1$.  Note that the radial projection $\pi^r_{|_{\partial D}} : \partial D \to \partial \mathbb{B}$ is a homeomorphism. Indeed, $\pi^r_{|_{\partial D}}$ is a continuous bijection from  $\partial D$ to $\partial \mathbb{B}$ with inverse that can be explicitly written down with ease: $g(z) = z/h(z)$, whose  continuity is due to that of the Minkowski functional $h$ of our domain. As $V$ is an open subset of $\partial D$, it follows that $\pi^r(V)$ is an open subset of $\partial \mathbb{B}$. 
Though cones in general may have empty interior, this cannot happen for that spanned by $V$. To establish this, it suffices to prove that the
the punctured cone generated by $V$ namely,
$C_V' := \{tv: t \in \mathbb{R}^+,~v \in V\}$ 
is an open subset of $\mathbb{C}^N$. To do this, first observe that $C_V'$ is the same as the cone generated by $\pi^r(V)$, as follows by noting that the ray spanned by any $z \in V$ coincides with that spanned by $\pi^r(z)$. 
Now, note that the punctured cone  spanned by $\pi^r(V)$ is the image of the 
open set $\mathbb{R}^+ \times \pi^r(V) \subset \mathbb{R}^+ \times \partial \mathbb{B}$
under the (backward) polar coordinates homeomorphism
$\Psi: \mathbb{R}^+ \times \partial \mathbb{B} \to \mathbb{C}^N \setminus \{0\}$ given by 
$\Psi(r,\omega)=r\omega$. It follows thereby 
that the cone $C_V'$ is an open subset of 
$\mathbb{C}^N$, as was to be established.

Let $z$ be an arbitrary point in $C'_V$. Then there exists $z_0 \in V$ and $t \in \mathbb{R}^+$ such that $z = t z_0$. Note that $t = h(z)$ and $z_0 = z/h(z)$ because the Minkowski function $h$ is continuous. Let $d_1,d_2$ be the degrees of homogeneous polynomials $f_1,f_2$ respectively. As $|f_1(z_0)| = |f_2(z_0)| = 1$, we have 
\begin{equation*}\label{homo_Mink}
|f_2(z)| = |f_2(t z_0)| = t^{d_2}|f_2(z_0)| = t^{d_2}  ~\text{and} ~ |f_1(z)| = |f_1(t z_0)| = t^{d_1}|f_1(z_0)| = t^{d_1}.
\end{equation*}
Thereby $|f_2(z)|^{d_1} = |f_1(z)|^{d_2}$.
Since this equation holds for all $z \in V$,  by the identity principle applied to the real analytic functions $|f_1|^{d_2}, ~|f_2|^{d_1}$ gives $|f_2(z)|^{d_1} = |f_1(z)|^{d_2}$ for all
$z \in \mathbb{C}^N$. Thereby, the condition $|f_2(z)| < 1$ holds whenever $|f_1(z)| < 1$ holds i.e., we may as well drop one of these functions in defining the polyhedron. This contradicts the minimality of the defining functions for $D$, thereby completing the proof of this theorem.
\end{proof}


\subsection*{Proof of theorem \ref{prod_anal_poly} (by retracts)}
Following the proof of corollary \ref{gen-anal-poly} together with the above theorem and proposition \ref{Retrct_Dirctn}, we obtain theorem \ref{prod_anal_poly}.\\

\noindent As we have enunciated theorems involving balanced analytic polyhedra with $\mathbb{C}$-extremal boundaries, we now deal with the issue of substantiating such theorems by confirming existence of ample examples; indeed, we now discuss techniques for constructing plenty of examples of bounded balanced $\mathbb{C}$-extremal analytic polyhedra.
As noted in remark \ref{cont_minkwski}, for every such $D$, its Minkowski functional is given by $h(z) = \max \{|f_j(z)|^{1/d_j}: 1 \leq j \leq m\}$, where each $f_j$ is a homogeneous polynomial (whose degree we denote by $d_j$). Furthermore, according to proposition \ref{Prop_Minkowski} (f), $D$ is bounded if and only if $h^{-1}(0) = \{0\}$.
 Therefore, to construct such examples, it suffices to show the existence of homogeneous polynomials $f_j$ satisfying the following conditions:
\begin{enumerate}
    \item[(A)] for each $j$, $f_j$ does not reduce to a non-zero constant on any line $l$ not passing through the origin. This will give the $\mathbb{C}$-extremal condition of the theorem \ref{max_noncnvx}.
    \item[(B)]  $Z(f_1,\ldots,f_m) =\{0\}$, where $Z(f_1,\ldots,f_m)$ denotes the common zero set of $f_1,\ldots,f_m$. This corresponds to the boundedness of the domain.
\end{enumerate}

\noindent First, we provide concrete examples and methods for constructing  examples of homogeneous polynomials satisfying condition (A).\\

\textit{Example}-1: According to lemma \ref{power_linear}, every homogeneous polynomial in $\mathbb{C}[z,w]$ that is not a power of a linear polynomial satisfies condition (A).\\

\textit{Example}-2: $f(z_1,z_2,z_3) = (z_1^2 + z_2^2)(z_2^2+z_3^2)$ satisfies condition (A), which we now prove by contradiction below. 
More generally, $(az_1^2 + bz_2^2)(cz_2^2+dz_3^2)$ with $a,b,c,d \neq 0$ satisfies condition (A).\\

Suppose that $f$ does not satisfy condition (A). 
Then there exists a line $\ell$ that does not pass through the origin and also fails to satisfy condition (A). 
We can write $f(p+tv) = g_1(t)g_2(t)$, where $g_j(t) = (p_j+tv_j)^2+(p_{j+1}+tv_{j+1})^2$ for $j=1,2$.
If $f$ is a non-zero constant on $l$ then for each $j$, $g_j$ is a non-zero constant polynomial because if $g_j$ is non-constant for some $j$ then their product is also a non-constant polynomial of degree $>1$.
Now, we consider
\[
g_1(p+tv) = (p_1^2 + p_2^2) + 2(p_1v_1+p_2v_2)t + (v_1^2+v_2^2)t^2.
\]
If $g_1$ is constant on $\ell$ then $v_1^2 + v_2^2 =0$ and $p_1v_1 +p_2v_2 =0$. Multiplying $p_1^2$ on both sides of the equation $v_1^2 + v_2^2 = 0$, we obtain that $p_1^2v_1^2 + p_1^2v_2^2 = 0$. 
Since $p_1v_1 = -p_2v_2$, we have $p_2^2v_2^2 + p_1^2v_2^2 = (p_1^2 + p_2^2)v_2^2$. As $p_1^2 + p_2^2$ is not a non-zero constant, we conclude that $v_2 = 0$. 
Furthermore, as $v_1^2 = -v_2^2$, it follows that $v_1 = 0$. Next, we consider the function $g_2(t)$.
Using similar arguments, we can prove that $v_2 =v_3 =0$. This contradicts the assumption that $\ell$ is a one-dimensional line that does not pass through the origin. This shows that $f$ satisfies condition (A).\\

While the above example was a concrete one in three variables, we now formulate a general proposition to construct a fairly large collection of polynomials in any (finite) number of variables, satisfying condition (A).
\begin{prop} Let $g$ be homogeneous polynomial in $\mathbb{C}[z_1,\ldots,z_N]$.
    \begin{enumerate} 
        \rm \item[(i)] If $f$ is a homogeneous polynomial in $\mathbb{C}[z_1,\ldots,z_N]$ satisfies condition {\rm (A)}  then $h_1 = f \cdot g$ is a homogeneous polynomial satisfying condition {\rm (A)}.
        \item[(ii)] Let $f_1,\ldots,f_k$ be homogeneous polynomials in $\mathbb{C}[z_1,\ldots,z_N]$ such that its common zero set $Z(f_1,\ldots,f_k) = \{0\}$. Then $h_2 := f_1\cdots f_k\cdot g$  is a homogeneous polynomial satisfying condition {\rm (A)}.
        \end{enumerate}
\end{prop}
\begin{proof} 
To prove (i) by contradiction, assume that $h_1$ does not satisfy condition (A). 
Then there exists a line $\ell := \{p+tv: t \in \mathbb{R}\}$ such that $h_1$ is a non-zero constant on $\ell$. 
In particular, if $h_1(p+tv)$ is a non-zero constant function, then $f(p+tv)$ must also be a non-zero constant function. This contradicts our assumption that $f$ satisfies condition (A).\\

For proving (ii),consider the homogeneous polynomial $h_2(z) = f_1(z)\cdots f_k(z) \cdot g(z)$. If $h_2$ does not satisfy condition (A), then there exists a line $\ell := \{p+tv: t \in \mathbb{R}\}$, where  $v \neq 0$, such that $h_2$ is a non-zero constant on $\ell$. Then for all $t \in \mathbb{R}$, we can write
\[
h_2(p+tv) = f_1(p+tv)\cdots f_k(p+tv) \cdot g(p+tv).
\]
Since $h_2$ is a non-zero constant on $\ell$, $f_j(p+tv)$ and $g(p+tv)$ reduces to a non-zero constant function.
Note that the coefficient of the highest power of $t$ in each $f_j(p+tv)$ is $f_j(v)$ due to the homogeneity of the polynomials. Therefore, we must have $f_j(v) = 0$ for each $j = 1,2,\ldots,m$. This implies that $v \in Z(f_1,\ldots,f_m)$, contradicting our assumption that $Z(f_1,\ldots,f_m) = \{0\}$.
\end{proof}

\noindent We shall actually be constructing examples which shall satisfy the stronger condition of non-degeneracy of the defining functions at every point of all the pairwise intersections of faces,
briefly referred to as condition-(C).
\begin{lem}\label{homog_zero}
    Let $g_1,g_2,f$ be homogeneous polynomials in $\mathbb{C}[z_1,\ldots,z_N]$ with degrees $d_1,d_2,d$ respectively. 
    If $v$ is a common zero of $g_1(z),g_2(z),f(z)-e^{i\theta}$ for some $\theta \in (-\pi,\pi]$ then $e^{-i\theta/d}v$ is a common zero of $g_1(z),g_2(z),f(z)-1$.
\end{lem}
\begin{proof}
    Assume that $v$ is a common zero of $g_1(z),g_2(z),f(z)-e^{i\theta}$. This implies that $g_1(v) = g_2(v) =0$ and $f(v) = e^{i\theta}$ for some $\theta \in (-\pi,\pi]$. 
Note that $f(e^{-i\theta/d}v) = e^{-i \theta}f(v) =1$ because by the homogeneity of $f$. Since the zero set of a homogeneous polynomial is closed under scalar multiplication, $e^{-i\theta/d}v$ is a zero of $g_1$ and $g_2$. 
    Hence, we obtain that $e^{-i\theta/d}v$ is a common zero of $g_1(z),g_2(z),f(z)-1$.
\end{proof}
\noindent We shall now show that the functions $f_j$ in the foregoing examples satisfy the above mentioned condition-(C). 
Towards this, we prove that for each $j \neq k$, $\nabla f_j(z)$ and $\nabla f_k(z)$ are $\mathbb{C}$-linearly independent for each $z \in \sigma_{j,k}$.
Consider the homogeneous polynomials
\[
f_1(z_1,z_2,z_3) = z_1z_2z_3,~~~~~~~~ f_2(z_1,z_2,z_3) = (z_1^2+z_2^2)(z_2^2+z_3^2)(z_3^2+z_1^2)
\]
\[
    f_3(z_1,z_2,z_3) = (z_1+z_2-z_3)(z_1-z_2+z_3)(-z_1+z_2+z_3)
\]
Their gradients are given as follows:
\begin{equation*}
\nabla f_1(z) = (z_2z_3,z_1z_3,z_1z_2)
\end{equation*}
\begin{multline*}
    \nabla f_2(z) = \left(2z_1(z_2^2+z_3^2)(2z_1^2+z_2^2+z_3^2),2z_2(z_1^2+z_3^2)(z_1^2+2z_2^2+z_3^2),\right.\\\left.2z_3(z_1^2+z_2^2)(z_1^2+z_2^2+2z_3^2)\right)
\end{multline*}
\begin{multline*}
\nabla f_3(z) = \left(-3z_1^2 +z_2^2 +z_3^2+2z_1z_3-2z_2z_3+2z_1z_2,z_1^2-3z_2^2+z_3^2+\right.\\\left.2z_2z_3-2z_1z_3+2z_1z_2,z_1^2+z_2^2-3z_3^2-2z_1z_2+2z_1z_3+2z_2z_3\right)
\end{multline*}
First, we prove that $\nabla f_1(z)$ and $\nabla f_2(z)$ are linearly independent for each $z \in \sigma_{1,2}$. It suffices to show that for every $\theta_1,\theta_2 \in (-\pi,\pi]$, $\nabla f_1(z)$ and $\nabla f_2(z)$ are linearly independent for each $z \in \sigma_{1,2}$ with $f_1(z) = e^{i \theta_1}$ and $f_2(z) = e^{i\theta_2}$. Initially, we shall show this for $\theta_1= \theta_2 =0$. Let $z$ be a point in $\sigma_{1,2}$ such that $f_1(z) = 1$ and $f_2(z) =1$.
Suppose $\nabla f_1(z)$ is a scalar multiple of $\nabla f_2(z)$. Then 
\[
    \frac{\partial f_2(z)/\partial z_1}{\partial f_1(z)/\partial z_1} = \frac{\partial f_2(z)/\partial z_2}{\partial f_1(z)/\partial z_2} = \frac{\partial f_2(z)/\partial z_3}{\partial f_1(z)/\partial z_3}  
\]
The above equation reduces to the following equations: 
\[
    g_1(z_1,z_2,z_3) = 0,~~g_2(z_1,z_2,z_3) = 0
\]
where 
\begin{eqnarray*}
    g_1(z_1,z_2,z_3) &=& \frac{\partial f_2}{\partial z_1}\frac{\partial f_1}{\partial z_2} - \frac{\partial f_2}{\partial z_2}\frac{\partial f_1}{\partial z_1}\\ 
                     &=& z_3(z_1^4z_2^2 + 2z_1^4z_3^2 - z_1^2z_2^4 + z_1^2z_2^4 - 2z_2^4z_3^2 - z_2^2z_3^4),\\
    g_2(z_1,z_2,z_3) &=& \frac{\partial f_2}{\partial z_2}\frac{\partial f_1}{\partial z_3} - \frac{\partial f_2}{\partial z_3}\frac{\partial f_1}{\partial z_2}\\
                     &=& z_1(-z_1^4z_2^2 + z_1^4z_3^2 - 2z_1^2z_2^4 + 2z_1^2z_3^4 - z_2^4z_3^2 + z_2^2z_3^4).
\end{eqnarray*}
Now, we consider the common zeros of the equations $g_1 =0, g_2 =0, f_2 =1$. Note that the common zero set of these equations is exactly the same as the variety of the ideal $I = \langle g_1,g_2,f_2-1\rangle$. 
Note that $\sqrt{I} = \mathbb{C}[z_1,z_2,z_3]$\footnote{We can compute radical of $I$ using the Macaulay2 software.}. Therefore, the functions $g_1,g_2, f_2-1$ have no common zero in $\mathbb{C}^3$.
Fix $\theta \in (-\pi,\pi]$. We want to show that there is no common zero for the functions $g_1,g_2,f_2 - e^{i\theta}$. 
Let $v$ be a common zero of these functions. By the above lemma \ref{homog_zero}, $e^{-i\theta/6}v$ is a  common zero of $g_1,g_2,f_2-1$. This contradicts the above fact that there is no common zero for $g_1,g_2,f_2-1$.
This shows that $\nabla f_1(z)$ and $\nabla f_2(z)$ are linearly independent for each $z \in \sigma_{1,2}$.\\

Next, we prove that $\nabla f_2(z)$ and $\nabla f_3(z)$ are linearly independent for each $z \in \sigma_{2,3}$. Let $z$ be a point in $\sigma_{1,2}$ such that $f_2(z) = 1$ and $f_3(z) =1$.
If $\nabla f_2(z)$ is a scalar multiple of $\nabla f_3(z)$, then the ratios of each component of $\nabla f_2$ and $\nabla f_3$ are the same. Hence, we obtain the following equations.
\[
    h_1(z_1,z_2,z_3) = 0,~~h_2(z_1,z_2,z_3) = 0
\]
where 
\begin{multline*}
    h_1(z_1,z_2,z_3) = \frac{\partial f_2}{\partial z_1}\frac{\partial f_3}{\partial z_2} - \frac{\partial f_2}{\partial z_2}\frac{\partial f_3}{\partial z_1}\\ 
                     = 3z_1^6z_2 - 2z_1^5z_2z_3 + 2z_1^5z_3^2 + 9z_1^4z_2^3 - 2z_1^4z_2^2z_3 + 9z_1^4z_2z_3^2 - 4z_1^4z_3^3 - 9z_1^3z_2^4 - 6z_1^3z_2^2z_3^2 + 3z_1^3z_3^4 + \\2z_1^2z_2^4z_3 + 6z_1^2z_2^3z_3^2 + 3z_1^2z_2z_3^4 - 2z_1^2z_3^5 - 3z_1z_2^6 + 2z_1z_2^5z_3 - 9z_1z_2^4z_3^2 - 3z_1z_2^2z_3^4 + z_1z_3^6\\ - 2z_2^5z_3^2 + 4z_2^4z_3^3 - 3z_2^3z_3^4 + 2z_2^2z_3^5 - z_2z_3^6
\end{multline*}
\begin{multline*}
    h_2(z_1,z_2,z_3) = \frac{\partial f_2}{\partial z_2}\frac{\partial f_3}{\partial z_3} - \frac{\partial f_2}{\partial z_3}\frac{\partial f_3}{\partial z_2}\\
                     = z_1^6z_2 - z_1^6z_3 - 2z_1^5z_2^2 + 2z_1^5z_3^2 + 3z_1^4z_2^3 + 3z_1^4z_2^2z_3 - 3z_1^4z_2z_3^2 - 3z_1^4z_3^3 - 4z_1^3z_2^4+ 4z_1^3z_3^4\\ + 2z_1^2z_2^5 + 9z_1^2z_2^4z_3 - 6z_1^2z_2^3z_3^2 + 6z_1^2z_2^2z_3^3 - 9z_1^2z_2z_3^4 - 2z_1^2z_3^5 - 2z_1z_2^5z_3 - 2z_1z_2^4z_3^2\\ + 2z_1z_2^2z_3^4 + 2z_1z_2z_3^5 + 3z_2^6z_3 + 9z_2^4z_3^3 - 9z_2^3z_3^4 - 3z_2z_3^6
                 \end{multline*}
We want to find the common zeros of $h_1 =0, h_2 =0, f_2 =1$. Consider the ideal $J = \langle h_1,h_2,f_2-1\rangle$. 
Note that $\sqrt{J} = \mathbb{C}[z_1,z_2,z_3]$. So, we conclude that there is no common zero for the functions $h_1,h_2,f_2-1$.
Using similar arguments to those employed in the proof of the linear independence of $\nabla f_1(z), \nabla f_2(z)$ that there is no common zero for the functions $h_1,h_2,f_2 = e^{i\theta}$, where $\theta$ is some fixed real number in $(-\pi,\pi]$. 
This shows that $\nabla f_2(z)$ and $\nabla f_3(z)$ are linearly independent for each $z \in \sigma_{2,3}$.\\

Finally, we prove that $\nabla f_1(z)$ and $\nabla f_3(z)$ are linearly independent for each $z \in \sigma_{1,3}$. Let $z$ be a point in $\sigma_{1,2}$ such that $f_1(z) = 1$ and $f_3(z) =1$.
If $\nabla f_1(z)$ and $\nabla f_3(z)$ are linearly dependent, then the ratios of each component of $\nabla f_1$ and $\nabla f_3$ are the same. Hence, we obtain the following equations.
\[
    k_1(z_1,z_2,z_3) = 0,~~k_2(z_1,z_2,z_3) = 0
\]
where 
\begin{eqnarray*}
    k_1(z_1,z_2,z_3) &=& \frac{\partial f_1}{\partial z_1}\frac{\partial f_3}{\partial z_2} - \frac{\partial f_1}{\partial z_2}\frac{\partial f_3}{\partial z_1}\\
                     &=& z_3(-3z_1^3+z_1^2z_2+2z_1^2z_3-z_1z_2^2+z_1z_3^2+3z_2^3-2z_2^2z_3-z_2z_3^2)\\
    k_2(z_1,z_2,z_3) &=& \frac{\partial f_1}{\partial z_2}\frac{\partial f_3}{\partial z_3} - \frac{\partial f_1}{\partial z_3}\frac{\partial f_3}{\partial z_2}\\
                     &=& z_1(z_1^2z_2-z_1^2z_3+2z_1z_2^2-2z_1z_3^2-3z_2^3+z_2^2z_3-z_2z_3^2+3z_3^3)
\end{eqnarray*}
Now, we look at the common zeros of $k_1 =0, k_2 =0, f_1 -1$. Consider the ideal $K = \langle k_1,k_2,f_1-1\rangle$. 
Note that $\sqrt{K} = \mathbb{C}[z_1,z_2,z_3]$. So, we conclude that there is no common zero for the functions $k_1,k_2,f_1-1$.
Using similar arguments to those employed in the proof of the linear independence of $\nabla f_1(z), \nabla f_2(z)$ that there is no common zero for the functions $h_1,h_2,f_2 = e^{i\theta}$, where $\theta$ is some fixed real number in $(-\pi,\pi]$. 
This shows that $\nabla f_1(z)$ and $\nabla f_3(z)$ are linearly independent for each $z \in \sigma_{1,3}$.\\
 
\noindent We would like to conclude this section by laying down at least one concrete example of a
balanced analytic polyhedron whose boundary is not only $\mathbb{C}$-extremal but also 
has the feature of 
non-degeneracy of the defining functions {\it throughout} (unlike the example $D_h$ in $\mathbb{C}^2$ that we dealt with earlier) the various open faces and all pairwise intersections of the closed faces. Namely, the analytic polyhedron 
\[
D = \{(z_1,z_2,z_3) \in \mathbb{C}^3: |f_i(z_1,z_2,z_3)| < 1~\text{for}~i=1,2,3\}
\]
where $f_1(z) = z_1z_2z_3,f_2(z) = (z_1^2+z_2^2)(z_2^2+z_3^2)(z_3^2+z_1^2), f_3(z) = (z_1-z_2+z_3)(z_1+z_2-z_3)(-z_1+z_2-z_3)$.

\section{Retracts of Unbounded and topologically non-trivial domains -- 
Proofs of theorems \ref{Hartogs triangle} -- \ref{RiemSurf x HyperbolicMfld}} \label{Unbdd & Top non-trivial} \label{Products-unbdd-topnontrivial}

\noindent As mentioned in the introduction until this stage, we have for the most part,
restricted attention to 
retracts of bounded domains; though the domains in the previous section allow topological non-triviality,
as we considered `analytic complements' $D \setminus A$ at the beginning therein,
the boundedness assumption prevailed. To go beyond, to more general settings, we require 
the following lemma:
\begin{lem}\label{retrct_XtimesY}
Let $X$ and $Y$ be complex manifolds. If $F : X \times Y  \to X \times Y$ 
is a holomorphic retraction map 
from $X \times Y$ onto $Z \times \{c\}$ for some $c \in Y$, then $Z$ is a retract of $X$.
\end{lem}

\begin{proof}
Write $F(x,y) = (f(x,y),c)$. By comparing the first components of $F \circ F = F$ (the defining equation 
of a retraction map), we get 
\begin{equation}\label{Z}
f(f(x,y),c) = f(x,y) 
\end{equation}
Define $G : X \to Z$ by $G(x) = f(x,c)$. Note that
$G(G(x)) = G(f(x,c)) = f\left(f(x,c),c)\right)$. 
From (\ref{Z}), we get $ G(G(x))= f(x,c) = G(x)$. So, $G \circ G = G$ i.e., $G$ is a retraction map from $X$ onto
$Z$, thereby $Z$ is a retract of $X$.
\end{proof}

\noindent We are now ready to detail the proofs of various theorems dealing with unbounded and topologically 
non-trivial domains (and manifolds) as stated in the introduction, starting with that of 
theorem (\ref{1st-unbdd & top-nontrivial}).\\

\begin{proof}[Proof of theorem \ref{1st-unbdd & top-nontrivial}]
\noindent Let $Z$ be a retract of $\mathbb{C} \times D$ passing 
through $p := (a,c)$.  Let 
$F: \mathbb{C} \times D \to \mathbb{C} \times D$ be a holomorphic retraction map such 
that ${\rm image}(F) = Z$. Write $F(z,w) = \left(f_1(z,w),f_2(z,w) \right)$.
\smallskip\\
\noindent We start by considering the second component of $F$.
For each fixed $w \in D$, $f^2_w: \mathbb{C} \to D$ defined by $f^2_w(z) = f_2(z,w)$ 
is a bounded holomorphic map on $\mathbb{C}$. By Liouville's theorem, 
$f^2_w$ is a constant function and this implies that $f_2(z,w)$ is independent 
of $z$. So $f_2(z,w) = g(w)$ for some holomorphic map $g$ 
on $D$. Note that $g$ is a retraction map on $D$ such 
that $g(c) = c$. So either
$g \equiv c$, or $g \equiv I$, the identity map on $D$
and accordingly, we are led to breaking into $2$ cases as dealt with below.\\

\noindent \underline{Case-1:} $g \equiv c$.
In this case, $F(z,w) = (f_1(z,w),c)$.
By comparing the first component of $F \circ F = F$, we get $f_1(f_1(z,w),c) = f_1(z,w)$.
Note that $F$ is a holomorphic retraction map 
on $\mathbb{C} \times D$ onto $Z_1 \times \{c\}$, 
where $Z_1 := {\rm image}(f_1)$ on $\mathbb{C} \times D$. By lemma \ref{retrct_XtimesY}, $Z_1$ is a retract of $\mathbb{C}$ passing through $a$. Hence either $Z_1 = \mathbb{C}$, or $Z_1=\{a\}$. It follows that $Z = \mathbb{C} \times \{c\}$ or $Z = \{p\}$.\\

\noindent\underline{Case-2:} $g \equiv I$, the identity map on $D$. In this 
case, $F(z,w) = (f_1(z,w),w)$. For each fixed $w \in D$, we
then consider the holomorphic function $f^1_{w}: \mathbb{C} \to \mathbb{C}$ 
defined by $f^1_{w}(z) = f_1(z,w)$. The fact that $F$ is a retraction map i.e., $F \circ F = F$
implies in particular that $f_1\left(f_1(z,w),w\right) = f_1(z,w)$. This translates to the following
for the function $f^1_w$:
\[
f^1_{w} \circ f^1_{w}(z) = f^1_{w}(f_1(z,w)) = f_1(f_1(z,w),w) = f_1(z,w) = f^1_{w}(z).
\]
In short: $f^1_{w} \circ f^1_{w} = f^1_{w}$, thereby $f^1_{w}$ is a retraction map on $\mathbb{C}$.
Consequently, either $f^1_{w} \equiv I$, the identity map on $\mathbb{C}$, or 
$f^1_{w}$ is a constant 
function. This means that either $f_1(z,w) = z$ or  $f_1(z,w)$ is 
a function of $w$ alone, i.e. $f_1(z,w) = f(w)$ for 
some holomorphic function $f$ on $D$. So, we
conclude that $Z = \mathbb{C} \times D$, or $Z = \{(f(w),w)  :  w \in D\}$.
\end{proof}

\noindent Next we shall detail the complete determination of the retraction maps of $\mathbb{C} \times \Delta$.    
\begin{proof}[Proof of theorem \ref{C-minus-Delta}:] 
Let $F: \mathbb{C} \times \Delta \to \mathbb{C} \times \Delta$ be a retraction map 
such that ${\rm image}(F) = Z$ and $F(0) = 0$. Write $F(z,w) = (f(z,w),g(z,w))$. For 
each fixed $w$, $g_w: \mathbb{C} \to \Delta$ defined by $g_w(z) = g(z,w)$ is a 
bounded holomorphic map on $\mathbb{C}$. By Liouville's theorem, $g_w$ is a 
constant function and this implies that $g(z,w)$ is independent of $z$. 
So $g(z,w) = \tilde{g}(w)$, where $\tilde{g}$ is a holomorphic function on $\Delta$. 
Note that $\tilde{g}$ is a retraction map on $\Delta$ such 
that $\tilde{g}(0) = 0$. So $\tilde{g} \equiv 0$ or $\tilde{g} \equiv I$, the identity map on $\Delta$.\\

\underline{Case-1}: $\tilde{g} \equiv 0$. 
In this case, $F(z,w) = (f(z,w),0)$. The fact that $F$ is a retraction map  i.e. $F \circ F = F$ 
implies that $f(f(z,w),0) = f(z,w)$. Since $f$ has a single power series representation centred at origin, we can write 
\[
f(z,w) = P_1(z,w) + P_2(z,w) + P_3(z,w) +\ldots ,
\]
where $P_i(z,w)$ is a homogeneous polynomial in $z,w$ of degree $i$.\\
Since $f(f(z,w),0) = f(z,w)$, we get
\[
f(P_1(z,w) + P_2(z,w) + P_3(z,w) + \ldots ,0) = P_1(z,w) + P_2(z,w) + P_3(z,w) + \ldots.
\]
\begin{multline}\label{Q}
P_1(P_1 + P_2 + P_3 +\ldots, 0) + P_2(P_1 + P_2 + P_3 + \ldots,0) + P_3(P_1 + P_2 + P_3 + \ldots,0)\\
+ \ldots =  P_1(z,w) + P_2(z,w) + P_3(z,w) + \ldots.
\end{multline}
Equating degree one terms on both sides, we get $P_1(P_1,0) = P_1(z,w)$. 
Either $P_1 \equiv 0$ or $P_1(z,w) =z + bw$ for some $b \in \mathbb{C}$ and accordingly split into two subcases.\\

\noindent \underline{Subcase-(1a)}: $P_1 \equiv 0$. 
From (\ref{Q}), we get 
\[
P_2(P_2+P_3+ \ldots,0) + P_3(P_2 + P_3 + \ldots,0)
    + \ldots =  P_2(z,w) + P_3(z,w) + \ldots.
\]
The least degree term on left side is $P_2(P_2,0)$, whose degree is $4$ and the least degree term 
on right side is $P_2(z,w)$, whose degree is $2$. So we 
conclude that $P_2 \equiv 0$. Continuing in this manner, we get $P_i \equiv 0$ for all $i=1,2, \ldots$. Thereby
we conclude in this subcase that $F \equiv 0$. Hence $Z = \{(0,0)\}$.\\

\noindent \underline{Subcase-(1b):} $P_1(z,w) = z+bw$.
From (\ref{Q}), we get 
\[
P_2(P_1 + P_2 + P_3 + \ldots,0) + P_3(P_1 + P_2 + P_3 + \ldots,0)
    + \ldots =  0.
\]
Equating degree $2$ terms on both sides, we get $P_2(P_1,0) = 0$. This implies that  
$P_2(z,w) = w(Q_1(z,w))$, where $Q_1(z,w)$ is a polynomial in $z$ and $w$. Substituting 
$P_2$ in \ref{Q}, we get
\[
P_3(P_1 + P_2 + P_3 + \ldots,0) + P_4(P_1 + P_2 + P_3 + \ldots)
    + \ldots =  0
\]
Next, equating degree $3$ terms on both sides, we get $P_3(P_1,0) = 0$. This 
implies that $P_3(z,w) = w(Q_2(z,w))$, where $Q_2(z,w)$ is a polynomial in $z$ and $w$.
Continuing in this way, we get $P_i(z,w) = w(Q_{i-1}(z,w))$, where $Q_{i-1}$ is a polynomial in $z$ and $w$. 
Finally $f(z,w) = z + w(h(z,w))$, where $h$ is some holomorphic map on $\mathbb{C} \times \Delta$. In this subcase, 
$F(z,w) = (z+w(h(z,w)),0)$. Hence $Z = \mathbb{C} \times \{0\}$.\\

\underline{Case-2:} $\tilde{g} \equiv I$.
    In this case,  $F(z,w) = (f(z,w),w)$. Since $F \circ F = F$, we get $f(f(z,w),w) = f(z,w)$.\\
Since $f$ has a single power series expansion centred at origin, we can write
\[f(z,w) = P_1(z,w) + P_2(z,w) + P_3(z,w) +\ldots ,\]
where $P_i(z,w)$ is a homogeneous polynomial in $z,w$ of degree $i$\\
    Since $f(f(z,w),w) = f(z,w)$, we get
    \[f(P_1(z,w) + P_2(z,w) + P_3(z,w) + \ldots ,w) = P_1(z,w) + P_2(z,w) + P_3(z,w) + \ldots \]
    \begin{multline}\label{P}
    P_1(P_1 + P_2 + P_3 +\ldots, w) + P_2(P_1 + P_2 + P_3 + \ldots,w) + P_3(P_1 + P_2 + P_3 + \ldots,w)\\
    + \ldots =  P_1(z,w) + P_2(z,w) + P_3(z,w) + \ldots
\end{multline}
Comparing the smallest degree term on both sides, we get $P_1(P_1(z,w),w) =P_1(z,w) $. Write $P_1(z,w) = az + bw$ for some $a,b \in \mathbb{C}$. From the above relation, we get $P_1(az + bw,w) = az+bw$. This implies that $a(az+bw) + bw = az + bw$. After rearranging we get  $a^2z + b(a+1)w = az + bw$. Equating the coefficient of $z$ and $w$, we get $a^2 = a$ and $b(a+1) = b$. From this we can conclude either $a = 0$ or $1$. So we divide into two cases.\\

\noindent \underline{Subcase-(2a)}: $a =1$.
From the relation $b(a+1) = b$, we get $b = 0$. So $P_1(z,w) = z$.\\
From (\ref{P}), we get 
\[ (z + P_2 + P_3 + \ldots) + P_2(z + P_2 + P_3 + \ldots,w) + P_3(z + P_2 + P_3 + \ldots, w) + \ldots = z + P_2 + P_3 + \ldots \]
After simplifying, above equation reduces to
    \[ P_2(z + P_2 + P_3 + \ldots,w) + P_3(z + P_2 + P_3 + \ldots, w)+  \ldots = 0 .\]
Again equating smallest degree on both sides, we get $P_2 \equiv 0$.\\
Continuing in this manner, we get $P_i \equiv 0$ for $ i = 1,2,\ldots$. Hence $f_1(z,w) = z$.\\
In this subcase, $F(z,w) = (z,w)$. 
Hence we get $Z = \mathbb{C} \times \Delta$.\\

\noindent \underline{Subcase-(2b)}: $a = 0$.
In this subcase, we get $P_1(z,w) = bw$.\\
From (\ref{P}), we get
\[ P_2(bw + P_2 + P_3 +\ldots,w) + P_3(bw + P_2 + P_3 + \ldots,w) + \ldots = P_2 + P_3 + \ldots\]
Comparing the smallest degree terms, we get $P_2(bw,w) = P_2(z,w)$. This implies $P_2$ is a function of $w$ alone.\\
(\ref{P}) reduces to 
\[ P_3(bw + P_2 + P_3 + \ldots,w ) + P_4(bw + P_2 + P_3 + \ldots, w) \ldots = P_3 + P_4 + \ldots \] 
After equating the smallest degree, we get $P_3$ is a function of $w$ alone.\\
Continuing this manner, we get $P_i$ is a function of $w$ alone. So $f(z,w) = h(w)$, where $h(w) = P_1(w) + P_2(w) + P_3(w) + \ldots$\\
\smallskip
In this subcase,  $F(z,w) = \left(h(w),w\right)$. Hence we get $Z = \{(h(w),w): w \in \Delta \}$.\\
Since $\mathbb{C} \times \Delta$ is homogeneous, upto a conjugation of automorphism,  
every non-trivial retraction map is of the form $F(z,w) = \left(z+wh(z,w),0\right)$, where $h \in \mathcal{O}(\mathbb{C} \times \Delta)$ or $F(z,w) = (h(w),w)$, where $h \in \mathcal{O}(\Delta)$
\end{proof}

\noindent Next, we turn to the proofs of theorems dealing with domains or manifolds with much greater topological
complexity.

\begin{proof}[Proof of proposition \ref{discrete-removed}]
Let $\widetilde{Z}$ be a retract of $(\mathbb{C} \setminus A) \times D$ passing 
through an arbitrary point $(a,b) \in (\mathbb{C} \setminus A) \times D$. 
Let $F: (\mathbb{C} \setminus A) \times D \to (\mathbb{C} \setminus A) \times D$ 
be a retraction map such that ${\rm image}(F) = \widetilde{Z}$. Write $F(z,w) = (f(z,w),g(z,w))$.
For each fixed $w \in D$, $g_w: (\mathbb{C} \setminus A) \to D$ defined 
by $g_w(z) = g(z,w)$ is a bounded holomorphic map on $(\mathbb{C} \setminus A)$. Applying Riemann removable 
singularities theorem to each component 
of $g_w$, $g_w$ can extend holomorphically as a map $g_w': \mathbb{C} \to \overline{D}$. By 
Liouville's theorem, $g_w'$ is a constant function. In particular, $g_w$ is constant function 
and thus implies that $g(z,w)$ is independent of $z$. So $g(z,w) = \tilde{g}(w)$ for some 
holomorphic map $\tilde{g}$ on $D$. From the definition of retraction map, we 
get $F \circ F = F$. This implies in particular that 
$f\left(f(z,w),\tilde{g}(w)\right) = f(z,w)$ and $\tilde{g}(\tilde{g}(w)) = \tilde{g}$. 
Since $F(a,b) = (a,b)$, we get $\tilde{g}(b) = b$. Note that $\tilde{g}$ is a 
retraction map on $D$ such that $\tilde{g}(b) = b$. For each fixed  $w \in D$, we get a 
holomorphic map $f_{w}: (\mathbb{C} \setminus A) \to (\mathbb{C} \setminus A)$ 
defined by $f_{w}(z) = f(z,\tilde{g}(w))$. Note that
\[ 
f_{w} \circ f_{w}(z) = f_{w}\left(f(z,\tilde{g}(w))\right) = 
f\left(f(z,\tilde{g}(w)),\tilde{g}(\tilde{g}(w)\right) = f(z,\tilde{g}(w)) = f_{w}(z). 
\]
This shows that $f_{w}$ is a retraction map on $(\mathbb{C} \setminus A)$. Since every domain in $\mathbb{C}$, the only retraction maps are  identity and constant function, either $f_w \equiv I$, the identity map on $(\mathbb{C} \setminus A)$ or $f_w \equiv$ constant function (depends on $w$), hence implies that $f_w$ is independent of $z$. Hence for all $z,w \in (\mathbb{C} \setminus A) \times D$, we have 
\begin{equation}\label{H}
f(z,\tilde{g}(w)) = H(\tilde{g}(w)) \quad \text{for some holomorphic function}~ H ~ \text{on}~ D\end{equation}
\begin{center} or \end{center}
\begin{equation}\label{K}
    f(z,\tilde{g}(w)) = z.
\end{equation}

\noindent Define $Z := {\rm image}(\tilde{g})$. 
According to (\ref{H}) and (\ref{K}), we divide into $2$ cases.\\ 

\textit{Case-1:}  $f(z,\tilde{g}(w)) = H(\tilde{g}(w))$ for all $(z,w) \in (\mathbb{C} \setminus A) \times D$\\

We claim that $\widetilde{Z} = \{(H(w),w): w \in Z\}$. To prove this claim, 
observe that for each $w \in Z$, we have $\tilde{g}(w) = w$ and $f(H(w),w) = f(H(w),\tilde{g}(w)) = H(w)$ . This implies that $\{(H(w),w): w \in Z\} \subset \widetilde{Z}$. For 
attaining converse implication, let $(z_0,w_0)$ be an arbitrarily point in $\widetilde{Z}$. Since $\tilde{g}(w_0) = w_0$ and $f(z_0,w_0) = z_0$ . Hence we get $z_0 = f(z_0,w_0) = f(z_0,\tilde{g}(w_0)) = H(\tilde{g}(w_0)) = H(w_0)$. This implies that $(z_0,w_0) = \left(H(w_0),w_0\right)$. Hence $\widetilde{Z} \subset \{\left (H(w), w\right): w \in Z\}$, and this proves the claim.\\  

\textit{Case-2:}  $f(z,\tilde{g}(w)) = z$ for all $(z,w) \in (\mathbb{C} \setminus A) \times D.$\\

In this case, we claim that $\widetilde{Z} = (\mathbb{C} \setminus A) \times Z$. To prove this claim, 
note that each $(z_0,w_0) \in (\mathbb{C} \setminus A) \times Z$, $\tilde{g}(w_0) = w_0$ and $f(z_0,\tilde{g}(w_0)) = f(z_0,w_0) = z_0$. Hence $(\mathbb{C} \setminus A) \times Z \subset \widetilde{Z}$. Conversely assume that $(z_0,w_0)$ be an arbitrary point in $ \widetilde{Z}$. Note that $\widetilde{g}(w_0) = w_0 \in Z$ and $f(z_0,w_0) = z_0$. This implies that $(z_0,w_0) \in (\mathbb{C} \setminus A) \times Z$, which proves our claim.\\

\noindent Finally we get that every retract $\widetilde{Z}$ is of the form
\[ 
\widetilde{Z} = \{(f(w),w) : w \in Z\}
\] 
where $Z$ is a retract of $D$ and $f$ is a holomorphic map on $Z$
\textit{or}, of the form $(\mathbb{C} \setminus A) \times Z$.

\end{proof}

\noindent We shall next record the proof of theorem \ref{RiemSurf x HyperbolicMfld}, 
whose verifying details will be shortened as most of its arguments are contained in the prequel. 
To explicitly spell out the special case of domains contained in Euclidean space: taking 
$X$ to be the simplest non-compact Riemann surface 
namely $\mathbb{C}$, we have a description of all the retracts of $\mathbb{C} \times D$
where $D$ is any hyperbolic domain in $\mathbb{C}^N$. Indeed as mentioned earlier, this class contains domains with 
a wide variety of fundamental (or for that matter higher homotopy and homology) groups.

\begin{proof}[Proof of theorem \ref{RiemSurf x HyperbolicMfld}:]

(i)~   All arguments used in the proof of theorem \ref{1st-unbdd & top-nontrivial} holds 
for this case too except for proving $g_w$ is independent of $z$. So, we have only to prove 
that $g_w$ is independent of $z$.\\
Let $Z$ be a retract of $X \times Y$ passing through $p := (a,c)$, where $(a,c) \in X \times Y$. 
Let $F: X \times Y \to X \times Y$ be a holomorphic retraction map such that ${\rm image}(F) = Z$. 
Write $F(z,w) = (f(z,w),g(z,w))$ in standard component notations.\\
 \noindent For each fixed $w \in Y$, $g_w: X \to Y$ defined by $g_w(z) = g(z,w)$ is a 
holomorphic map from a compact Riemann surface $X$ to non-compact Riemann surface $Y$. 
If $g_w$ is non-constant map, then by open mapping theorem, $g_w(X)$ is open and 
by continuity of $g_w$, $g_w(X)$ is compact, hence closed. Since, $X$ is connected,
$g_w$ is surjective. This does not holds because $Y$ is non-compact. Hence we conclude
that $g_w$ is a constant function and this implies that $g(z,w)$ is independent of $z$.
So $g(z,w) = \tilde{g}(w)$ for some holomorphic map $\tilde{g}$ on $Y$. \\

(ii)~ The proof for this case follows similar line of arguments used in the proof of 
proposition \ref{discrete-removed}. So, we have only to prove that $g_w$ is independent of $z$.\\

\noindent Let $\widetilde{Z}$ be a retract of $X \times Y$ passing through $p := (a,c)$, 
where $(a,c) \in X \times Y$. Let $F: X \times Y \to X \times Y$ be a holomorphic retraction 
map such that ${\rm image}(F) = \widetilde{Z}$. Write $F(z,w) = (f(z,w),g(z,w))$. For each 
fixed $w \in Y$, consider then the map $g_w: X \to Y$ defined by $g_w(z) = g(z,w)$, which is a holomorphic map 
from a non-hyperbolic Riemann surface $X$ to hyperbolic complex manifold $Y$. Let $d_X, d_Y$ 
denotes Kobayashi hyperbolic pseudo-distance on $X$ and $Y$ respectively. By the contractive 
property of hyperbolic pseudodistance under holomorphic maps, we get 
$d_Y(g_w(x_1), g_w(x_2)) \leq d_X(x_1,x_2)$. Since $X$ is non-hyperbolic 
Riemann surface, $d_X \equiv 0$. By hyperbolicity of $Y$, we get $g_w(x_1) = g_w(x_2)$. Since this is true for every $x_1, x_2 \in X$, we get $g_w$ is a constant 
function for each fixed $w \in Y$. This implies that $g(z,w)$ is independent of $z$. 
\end{proof}

\noindent We shall now detail the proof of theorem \ref{Hartogs triangle} about retracts of the
Hartogs triangle and its analytic complements. Firstly, a very quick proof is as follows. Making the identification
$\Omega_H \simeq \Delta \times \Delta^*$, retracts of the Hartogs triangle 
can be deduced as an immediate corollary of lemma \ref{Retr-of-anal-complements} taking
$D$ to be just the bidisc $\Delta^2$.  An alternative direct albeit longer proof is as follows.

\begin{proof}[Proof of theorem \ref{Hartogs triangle}]
Let $F$ be a non-constant holomorphic retraction map on $(\Delta \times \Delta^*)\setminus A$ such that $Z = {\rm image}(F)$.  
Write $F(z,w) = \left(F_1(z,w),F_2(z,w)\right)$, where 
$F_1 : (\Delta \times \Delta^*) \setminus A \to \Delta$ and $F_2 : (\Delta \times \Delta^*) \setminus A \to \Delta^*$ 
are holomorphic functions. Since $A$ is a proper analytic subset of $\Delta \times \Delta^*$ and by Riemann removable 
singularities theorem, $F_1, F_2$ can be extended holomorphically as a map 
$\widetilde{F}_1, \widetilde{F}_2 : \Delta \times \Delta^* \to \overline{\Delta}$ respectively. 
Write $\Delta \times \Delta^* = \Delta^2 \setminus A_0$, where $A_0 = \{(z,w) \in \Delta^2: w = 0 \}$. Note that $A_0$ is a proper analytic subset of $\Delta^2$. Applying the same argument, $\widetilde{F}_1, \widetilde{F}_2$ can be extended to a holomorphic map from $\Delta^2$ to $\overline{\Delta}$.
By a little abuse of notation, we view $\widetilde{F}_1, \widetilde{F}_2$ are holomorphic maps from $\Delta^2$ to $\Delta$. 
We now observe that the images of these maps actually does not contain any point 
from $\partial \Delta$ i.e., they map into $\Delta$; this follows
by the open mapping theorem (non-constant holomorphic functions 
from domains in $\mathbb{C}^N$ are open maps into $\mathbb{C}$;). Note that atleast one of $\widetilde{F}_1$ or  $\widetilde{F}_2$
are non-constant! So, $\widetilde{F}_1, \widetilde{F}_2$ are holomorpic maps from $\Delta^2$  to $\Delta$. 
Thereby, $F$ extends holomorphically as a map $\widetilde{F} : \Delta^2 \to \Delta^2$;
elaborately for the sequel, 
$\widetilde{F}(z,w) = \left(\widetilde{F}_1(z,w), \widetilde{F}_2(z,w)\right)$ to $\Delta^2$.\\
Next we want to prove that $\widetilde{F}$ is a retraction map on $\Delta^2$. By hypothesis, $F \circ F \equiv F$ on $(\Delta \times \Delta^*)\setminus A$. 
Let $(z_0,w_0) \in A$ be an arbitrary point. Pick any sequence 
$(z_n,w_n) \in (\Delta \times \Delta^*)\setminus A$ converging to $(z_0,w_0)$. So,
$z_n \to z_0$ and $w_n \to w_0$. Since $(z_n,w_n) \in (\Delta \times \Delta^*)\setminus A$ 
and $\widetilde{F} \equiv F$ on $(\Delta \times \Delta^*)\setminus A$, we have $\widetilde{F} \circ \widetilde{F}(z_n,w_n) = \widetilde{F}(z_n,w_n)$. By continuity of $\widetilde{F}$, as $n \to \infty$ we get $\widetilde{F} \circ \widetilde{F}(z_0,w_0) = \widetilde{F}(z_0,w_0)$. Since $(z_0,w_0)$ is an arbitrary point, we get the relation $\widetilde{F} \circ \widetilde{F} \equiv \widetilde{F}$ on $\Delta^2$. This proves that $\widetilde{F}$ is  a holomorphic retraction on $\Delta^2$.\\

\noindent Let $\widetilde{Z} := {\rm image}(\widetilde{F})$. By theorem 3 in \cite{Heath_Suff}, $\widetilde{Z}$ is either  of the form $\{(z,f(z)): z \in \Delta\}$ for some holomorphic self map $f$ on $\Delta$ or $\{(g(w),w): w \in \Delta\}$ for some holomorphic self map $g$ on $\Delta$. Since $Z \subset (\Delta \times \Delta^*)\setminus A$ and $ Z \subset \widetilde{Z}$ ($\widetilde{F}$ is a holomorphic extension of $F$), we get $Z$ is either of the form $\{(z,f(z)): z \in \Delta\} \cap ((\Delta \times \Delta^*) \setminus A)$ for some holomorphic map $f: \Delta \to \Delta$ or $\{(g(w),w): w \in \Delta^*\} \cap ((\Delta \times \Delta^*) \setminus A)$ for some holomorphic map $g: \Delta \to \Delta$.
\end{proof}

\section{Retracts of $\mathbb{C}^2$; proof of theorem \ref{polyretract-of-C^2}} \label{sectn-polyretract-of-C^2}

\noindent To begin with a few words in general about the non-trivial (holomorphic) retracts of the complex Euclidean space $\mathbb{C}^N$ of any given dimension $N$, note that all its retracts are contractible complex submanifolds which are non-compact as $\mathbb{C}^N$
does not admit any non-trivial compact submanifold. Apart from being Stein submanifolds, they are non-hyperbolic; indeed, they
are `Liouville' submanifolds as their Kobayashi metric vanishes identically (due to lemma \ref{contr}). In-particular, we see as an immediate 
consequence of the uniformization theorem that any one-dimensional retract $Z$ of $\mathbb{C}^N$ is biholomorphic 
to $\mathbb{C}$; thereby, $Z \times \mathbb{C}^{N-1}$ is biholomorphic to
$\mathbb{C}^N$. Thus, we may decompose $\mathbb{C}^N$ as $Z \times \mathbb{C}^{N-1}$ as well, though it may be asked as to 
why we may want such a decomposition. More importantly, our purpose in proving that retracts of the various kinds
of {\it bounded} balanced domains earlier, are graphs of holomorphic maps, was to cast the retracts in as simple a 
`normal form' as possible. It is natural to first examine if this as {\it literally} true for $\mathbb{C}^N$. 
Towards this, we first give an example of a holomorphic retract 
of $\mathbb{C}^2=\mathbb{C} \times \mathbb{C}$ which can
neither be realized as the graph over the first factor nor 
the second; moreover, it is also not a product of retracts -- taking graphs and products of retracts are the two canonical 
ways in which retracts of product domains arise and this provides
an example that cannot be realized as arising directly from 
either of these ways. This is discussed in the following

\begin{prop}\label{not_graph}
The one-dimensional complex submanifold of $\mathbb{C}^2$ given 
by 
\[ Z:=  \{(e^z,ze^{-z}): z \in \mathbb{C} \} \] 
is a holomorphic retract of $\mathbb{C}^2$
that can neither be written as a graph of any holomorphic map on the first factor 
nor the second; indeed, 
it is the image of the holomorphic retraction map $f$ on $\mathbb{C}^2$ defined 
by $f(z,w) = (e^{zw},zw~ e^{-zw})$ and $Z$ can in-fact not be written as a graph of any 
holomorphic map on any one-dimensional linear subspace of $\mathbb{C}^2$.
\end{prop}

\begin{proof} It is easily first noted that $Z$ cannot be written as $\{(z,g(z)): z \in \mathbb{C}\}$ for 
some entire function $g$ on $\mathbb{C}$ because the first projection $\pi_1(z,g(z)) = z$ maps 
{\it onto} $\mathbb{C}$ whereas $\pi_1(e^z,ze^{-z}) = e^z$ is not onto. Next, we shall establish that 
$Z$ cannot be written in the form $\{(g(w),w): w \in \mathbb{C}\}$ either.
To prove this by contradiction, suppose $Z = \{(g(w),w): w \in \mathbb{C}\}$ for some entire function $g$. 
This implies that $g$ satisfies $e^{wg(w)} = g(w)$. Hence we get $(e^{g(w)})^w = g(w)$. Therefore, if we 
define $h: \mathbb{C} \to \mathbb{C}^*$ by $h(w) = e^{g(w)}$, it follows that: $h$ satisfies $(h(w))^w = g(w)$.
Note that $h ~\text{and} ~ g$ are surjective holomorphic functions from $\mathbb{C}$ to $\mathbb{C}^*$.\\

\noindent Let $L_1$ be an arbitrary ray emanating from origin  and $z$ an arbitrary point of 
$\mathbb{C} \setminus L_1$. Just by the closedness of $L_1$ in 
$\mathbb{C}$, we may pick an open disc $D \subset \mathbb{C} \setminus L_1$. As 
$\mathbb{C} \setminus L_1$ is simply connected, there exists atleast one (indeed infinitely many)  
branch of the (complex) logarithm on $\mathbb{C} \setminus L_1$; pick one of them and 
call it $\ell_1$. Choose a ray $L_2$ starting from origin 
such that $L_2$ intersects $D$ . Note that $L_2$ divides $D$ into two regions  which shall be 
referred to by
$D^+ ~ \text{and}~ D^-$. As with $\ell_1$, we likewise
pick and denote by $\ell_2$, a branch of logarithm on $\mathbb{C} \setminus L_2$. \\

\noindent Since $h(w)^w = g(w)$ for all $w \in \mathbb{C}$, we have  $w\ell_1(h(w)) = \ell_1(g(w))~ \text{for all}~ w \in D^+ \cup D^-$   which
implies that $h(w) = e^{\frac{\ell_1(g(w))}{w}}$. Similarly, we 
get $h(w) = e^{\frac{\ell_2(g(w))}{w}}~ \text{for all}~ w \in D^+ \cup D^-$.\\
It then follows that on $D^+ \cup D^-$, 
\[ 
e^{\frac{\ell_1(g(w))}{w}} = e^{\frac{\ell_2(g(w))}{w}}.
\]
This implies that for all $w \in D^+ \cup D^- $, we have 
\begin{equation} \label{log} 
\ell_1(g(w)) - \ell_2(g(w)) = 2n \pi iw.
\end{equation} 
But $\ell_1(g(w)) - \ell_2(g(w))$ is an integer multiple of $2 \pi i$ for 
any $w \in D^+ \cup D^-$. So, (\ref{log}) holds only for $n = 0$. Hence 
we get $\ell_1(g(w)) = \ell_2(g(w)) ~ \text{for all}~ w \in D^+ \cup D^-$. This 
implies that $\ell_2$ is a branch of logarithm on $\mathbb{C}^*$. This 
contradicts the fact that there  is no branch of 
logarithm on $\mathbb{C}^*$ and completes the proof.
\end{proof}

We note that the above retract is however,
straightenable and thereby 
conjugate to a graph -- we say that a
complex submanifold $M$ of $\mathbb{C}^N$
of dimension $k$
is straightenable or rectifiable
if we can apply an automorphism of
$\mathbb{C}^N$ which maps $M$ onto 
$\mathbb{C}^k$. Getting 
examples of retracts which are not conjugate
to graphs is more non-trivial but
it is not difficult to derive it 
from the 
work of Forstneric -- Globevnik -- Rosay
\cite{Forstn}.
To briefly recall their main result,
let us recall as they do in the 
introduction of \cite{Forstn},
that Abhyankar and Moh in \cite{Abynkr} and M.Suzuki in \cite{Suzuki} 
proved the following: if 
$P$ is a polynomial embedding of $\mathbb{C}$ into 
$\mathbb{C}^2$, then there exists a polynomial automorphism $\psi$ 
of $\mathbb{C}^2$ such that 
$\psi \circ P(\mathbb{C}) = \mathbb{C}
\times \{0\}$. This result is not true for holomorphic embeddings 
of $\mathbb{C}$ in $\mathbb{C}^2$ or for that
matter in $\mathbb{C}^N$ for any $N \geq 2$. By the first proposition in \cite{Forstn}, 
there exists a
 proper holomorphic embedding $H:\mathbb{C} \to \mathbb{C}^N$ such
 that for no automorphism $\psi$ of $\mathbb{C}^N$, 
 we can attain: 
 $\psi \circ H(\mathbb{C}) = \mathbb{C}\times\{0\} \subset \mathbb{C}^N)$. For the embedding provided 
 by that proposition, we prove here 
 that ${\rm image}(H)$ is indeed a retract
  of $\mathbb{C}^N$.
Towards this, 
the manner in which this proper
holomorphic embedding $H$ 
is defined and attained in the proof of proposition $2$ of their paper \cite{Forstn},
is now very briefly recalled.
Start with the standard 
embedding $H_0:\mathbb{C} \to \mathbb{C}^N$ defined
 by $H_0(\zeta) = (\zeta,0)$. 
 For $j \geq 1$, let $H_j := \psi_j \circ
  H_{j-1}$, where $\psi_j$ 
  is some automorphism of
  $\mathbb{C}^N$. 
  These
 $(H_j)$'s constitute 
 a sequence of proper embeddings
 of $\mathbb{C}$ satisfying 
some properties which we do not need to repeat here. 
Taking their uniform
limit on compacts $H :=  \lim H_j$ 
yields the desired embedding; for further details,
we refer the reader to \cite{Forstn} but this recap
will suffice for our aforementioned purpose which we finish off in the next paragraph.\\

\noindent Denote by $\rho_0: \mathbb{C}^N \to 
\mathbb{C}^N$, {\it the simplest} of standard one-dimensional
retraction maps namely, $\rho_0(z,w) = (z,0)$, where $z \in 
\mathbb{C}$ and $w \in \mathbb{C}^{N-1}$. Note that 
${\rm image} (H_0) = {\rm image}(\rho_0)$.
For $j \geq 1$, let
$\rho_j: \mathbb{C}^N \to 
\mathbb{C}^N$ be defined by $\rho_j := \psi_j \circ H_j \circ \psi_j^
{-1}$, where $\psi_j$ is the automorphism mentioned in the above 
construction. Observe that each $\rho_j$ is a 
holomorphic retraction 
with ${\rm image}(\rho_j) = {\rm image}(H_j)$.
Note that the $\rho_j$'s converge 
uniformly on compacts. If we
denote the limit
map by $\rho$ then, $\rho$ is also 
a retraction -- this follows
by passing to the limit
in the defining equation of a retraction
for our specific map $\rho_j$ 
(i.e., writing $\rho_j \circ \rho_j=\rho_j$ elaborately),
wherein taking a `triple limit' is
justified owing to the uniform convergence.
Finally, we note that the limit
of this retraction map $\rho$ satisfies
 ${\rm image}(H) = {\rm image}(\rho)$. 
 Hence we have verfied the 
 existence of a one-dimensional 
 holomorphic retract $Z$ of $\mathbb{C}^N$ which is not 
 straightenable, thereby not conjugate to a 
 graph as well.\\

\noindent Owing to the foregoing impossibility of
attaining a rectifiability result even for one-dimensional
holomorphic retracts of $\mathbb{C}^N$, 
it seems too difficult to obtain a complete characterization of holomorphic retracts, indeed even for $\mathbb{C}^2$ which is unknown. However it is possible to deduce a complete characterization of 
polynomial retracts on $\mathbb{C}^2$ from known results about the polynomial 
ring in two complex variables and 
this yields in-particular, that all 
polynomial retracts are rectifiable -- this result is not at all new. However, even this deduction is not entirely trivial and 
needs some work which was done
by Shpilrain and Yu 
in \cite{Shpilrain -- Yu} with the key ingredients being the 
theorem 3.5 in \cite{Cost} and the Abhyankar -- Moh theorem (main theorem of \cite{Abynkr}). We shall show an alternative route
without these ingredients as discussed below. \\

\noindent An important point to note here is that though the Abhyankar -- Moh 
theorem is a result that renders the rectifiability of polynomial
embeddings of $\mathbb{C}$, it does not follow as an immediate 
corollary that polynomial retracts can be straightened. Firstly, 
though retracts are embedded complex
submanifolds wherein `embedded' here is in the sense
of differential geometry, it is far from clear that a (non-trivial) holomorphic 
retract of $\mathbb{C}^2$ is the image of an embedding of $\mathbb{C}$ -- 
this is true but seems to require the uniformization theorem. Furthermore,
if the retract is a polynomial retract, it is not at all immediate that 
the embedding given by the uniformization theorem can be taken to be
a polynomial one. In-fact, the converse namely, that the image 
of an embedding in the sense of the Abhyankar -- Moh theorem is indeed
a (polynomial) retract is not clear without invoking that theorem.
Notwithstanding all this, it is possible to give 
a direct, self-contained and elementary proof for this (known) result
concerning the rectifiability of polynomial retracts of $\mathbb{C}^2$; 
in-particular, neither using the result of Costa and nor
the Abhyankar -- Moh theorem, as we do in the 
next subsection.\\

\subsection{Proof of theorem \ref{polyretract-of-C^2}}
Since $Z$ is a polynomial retract of $\mathbb{C}^2$,  
there exists an idempotent map $F: \mathbb{C}^2 \to \mathbb{C}^2$ whose 
(both) components are (holomorphic) polynomials and 
${\rm image}(F) = Z$; we may assume from the
outset via a translation, that $Z$ passes through the
origin, so that $F(0)=0$ in-particular.
Again, by a linear change of variables, we 
may further assume 
$DF_{\vert_0}(z,w) = (z,0)$. So, we can write  
$F(z,w) = (z,0) + (P_1(z,w),P_2(z,w))$ for some 
$P_1,P_2 \in \mathbb{C}[z,w]$ each of which is either 
zero or have the property that
all its monomials are of degree at-least two.\\

\noindent We shall now
proceed to proving that $P_1 \equiv 0$ or $P_2 \equiv 0$.
To prove this by contradiction, suppose that neither $P_1$ nor $P_2$ 
are identically zero. Let $P_1^{(r)},P_2^{(k)}$ be the homogeneous constituents 
of highest degree in the homogeneous decomposition of $P_1$ and 
$P_2$ respectively. 
Indeed, let us 
introduce the notations for the homogeneous
decompositions that we will need
to work with in the 
sequel:
\begin{align*}
P_1(z,w) &= P_1^{(2)}(z,w) + \ldots + P_1^{(r)}(z,w)\\ 
P_2(z,w) &= P_2^{(2)}(z,w) + \ldots + P_2^{(k)}(z,w),
\end{align*}
where $P_1^{(i)},P_2^{(j)}$ 
satisfy the following conditions:
\begin{itemize}
\item Either $P_1^{(i)} \equiv 0$ or $P_1^{(i)}$ is a homogeneous 
polynomial of degree $i$ 
for $i = 2,\ldots,r$.  
\item Either $P_2^{(j)} \equiv 0$ or $P_2^{(j)}$ is a homogeneous polynomial of degree $j$ for $j = 2,\ldots,k$.
\end{itemize}

\noindent To prove that $Z$ can be straightened, it suffices 
to show that by a change of variables,
we can reduce the arbitrary polynomial retraction map 
$F$ to a corresponding `normal form'.
To do so, we need to solve the equation of idempotence:
$F(F(z,w)) = F(z,w)$; indeed, in the notations for the components
of $F$ already introduced, this leads to the following identities 
for the afore-mentioned
polynomials:
\begin{multline}\label{P'} 
P_1^{(2)} \left(z+P_1^{(2)} +\ldots+P_1^{(r)}, \; 
P_2^{(2)} +\ldots+P_2^{(k)} \right)
\\
\; \Large\text{+} \; 
P_1^{(3)} \left( z+P_1^{(2)} +\ldots+P_1^{(r)}, P_2^{(2)}+\ldots+P_2^{(k)} \right)\\ 
\hspace{-2in}\Large\Large\text{+} \; \cdots \cdots  \; \\
        \Large\text{+}\;
        P_1^{(r)}\left(z+P_1^{(2)} +\ldots+P_1^{(r)},\; P_2^{(2)}+\ldots+P_2^{(k)}\right)
       \; \; \Large\text{=} \;\; \Large\text{0}
\end{multline}
and
\begin{multline}\label{P''} 
P_2^{(2)} \left(z+P_1^{(2)} +\ldots+P_1^{(r)}, \; 
P_2^{(2)} +\ldots+P_2^{(k)} \right)
\\
\; \Large\text{+} \; 
P_2^{(3)} \left( z+P_1^{(2)} +\ldots+P_1^{(r)}, P_2^{(2)}+\ldots+P_2^{(k)} \right)\\ 
\hspace{-2in}\Large\Large\text{+} \; \cdots \cdots  \; \\
 \hspace{2in}       \Large\text{+}\;
  P_2^{(k)}\left(z+P_1^{(2)} +\ldots+P_1^{(r)},\; P_2^{(2)}+\ldots+P_2^{(k)}\right)\\
  \Large{\Large\text{=}}\;\;
  P_2^{(2)}(z,w) + \ldots + P_2^{(k)}(z,w)
\end{multline}

\noindent We now divide the arguments into two cases depending 
on whether $r = k$ or not.\\

\textbf{Case-1:} $r \neq k$.  
In this case, note that all the highest degree terms of 
left hand side of (\ref{P'}) are contained within 
$P_1^{(r)}\left( P_1^{(r)},P_2^{(k)} \right)$. 
Write $P_1^{(r)}(z,w) = \sum_{i+j = r} a_{ij}z^iw^j$.
Hence 
\[
P_1^{(r)} \left(P_1^{(r)},P_2^{(k)}\right) = 
\sum_{i+j=r}a_{ij}(P_1^{(r)})^i(P_2^{(k)})^j.
\]
Each summand in the above sum is a 
homogeneous polynomial of degree 
$ri+jk = ri+(r-i)k$, 
which we may write as: $rk + i(r-k)$. 
Let $N(i) := rk +i(r-k)$. 
Note that on the ordered finite set of 
integers:
$\left( 0,1,\ldots,r \right)$,
the function $N(i)$ is strictly monotonic, as $r \neq k$. 
Thus $N(i)$ attains maximum at a unique 
point $i_0$. 
So, the highest degree monomials 
of $P_1^{(r)}\left(P_1^{(r)},P_2^{(k)} \right)$ are 
all contained within the homogeneous polynomial 
$a_{i_0j_0}\left(P_1^{(r)}\right)^{i_0} \left(P_2^{(k)}\right)^{j_0}$.
Equating the highest degree terms of 
the left hand and right hand sides 
of (\ref{P'}), 
we obtain: 
$a_{i_0j_0}\left(P_1^{(r)}\right)^{i_0}\left(P_2^{(k)}\right)^{j_0} \equiv 0$. 
Since $a_{i_0j_0} \neq 0$, we have $\left(P_1^{(r)}\right)^{i_0}\left(P_2^{(k)}\right)^{j_0} 
\equiv 0$. This implies by the identity principle 
that $P_1^{(r)} \equiv 0$ or $P_2^{(k)} \equiv 0$,  
which contradicts our assumption that $P_1\cdot P_2 \not\equiv 0$. 
Thus, either $P_1 \equiv 0$ or $P_2 \equiv 0$, contradicting the 
assumption for the case being dealt with here.\\

\textbf{Case-2:} $r = k$.  
In this case, note that $P_1^{(k)}\left(P_1^{(k)},P_2^{(k)}\right) \equiv 0$ 
as right hand side of 
(\ref{P'}) is identically zero. 
We also have $P_2^{(k)}\left(P_1^{(k)},P_2^{(k)}\right) \equiv 0$,
this time owing to the fact that
the highest degree of right 
hand side of 
(\ref{P''}) is at-most $k$ whereas
the degree of every monomial in 
the polynomial $P_2^{(k)}\left(P_1^{(k)},P_2^{(k)}\right)$ is $k^2$ 
unless this polynomial vanishes identically. Consider then the holomorphic map $F_h: 
\mathbb{C}^2 \to \mathbb{C}^2$ defined by 
$F_h(z,w) = \left(P(z,w),Q(z,w)\right)$, 
where $P(z,w) = P_1^{(k)}(z,w)$ and 
$Q(z,w) = P_2^{(k)}(z,w)$. 
First we consider the sub-case that 
${\rm image}(F_h)$ has non-empty 
interior. Since 
$P(P,Q) \equiv 0, Q(P,Q) \equiv 0$, it 
follows by 
the identity principle that $P$ and $Q$ are identically zero, 
which contradicts our assumption that 
$P_1,P_2 \not\equiv 0$.\\

\noindent Next we deal with the other sub-case that 
${\rm image}(F_h)$ has empty interior. First we 
want to prove that there exists $\lambda_0 \in \mathbb{C}$ such 
that $Q(z,w) = \lambda_0 P(z,w)$. Let $S$ denote the zero set of 
$F_h$. If $(\alpha,\beta) \in S$, then for each $\lambda \in 
\mathbb{C}$, $\lambda(\alpha,\beta) \in S$ -- this follows from the fact that $S$ is the zero set of a map whose components are both homogeneous polynomials
of the {\it same} degree and here is where the 
hypothesis of `case-2' is used.
Thus,
$S$ is an union of  complex lines in $\mathbb{C}^2$. 
Let $S':= \pi(S)$, where 
$\pi: (\mathbb{C}^2 \setminus \{0\}) \to 
\mathbb{CP}^1$ is the canonical quotient map.
We observe that $S'$ is 
a finite set in $\mathbb{CP}^1$. 
To verify this, let $[z : w] \in S'$. Without loss of 
generality, assume that $z \neq 0$. Dividing by $z^k$ throughout 
each of the equations $P(z,w) = 0$ and $Q(z,w) = 0$, we 
get one variable polynomials $f,g$ both of degree at-most 
$k$, such that:
$f(w/z) = 0$ and $g(w/z) = 0$.
It follows that
the values of the ratio $w/z$ such that it is a root of $f$ 
(likewise of $g$) are at-most $k$ many; consequently, that 
$S'$ is a finite set in 
$\mathbb{CP}^1$ and in-particular that
$S'$ is a closed subset of $\mathbb{CP}^1$.\\ 

\noindent As $F_h(\lambda z, 
\lambda w) = \lambda^k F_h(z,w)$, the map  
 $F_h$ on $\mathbb{C}^2$ induces a holomorphic 
map $\tilde{F_h} : \mathbb{CP}^1 \setminus S' \to \mathbb{CP}^1$ 
defined by $\tilde{F_h}([z : w]) = [P(z,w) : Q(z,w)]$. Note that 
$\tilde{F_h}$ is a rational map from $\mathbb{CP}^1$ into 
$\mathbb{CP}^1$. Hence $\tilde{F_h}$ induces a regular map from 
$\mathbb{CP}^1$ into $\mathbb{CP}^1$. Since $\tilde{F_h}$ is not 
surjective, $\tilde{F_h}$ is a constant map. This implies that 
$\rm{image}(F_h)$ is a one-dimensional complex line passing 
through origin. This shows that there exists $\lambda_0 \in 
\mathbb{C}$ such that $Q(z,w) = \lambda_0 P(z,w)$. We shall now
use this to reduce the present case to the former. To this end,
we apply an affine-linear transformation of coordinates given 
by $\psi(z,w) = (z,w-\lambda_0z)$. Its inverse has a similar form; indeed, 
$\psi^{-1}(z,w) = (z,w+\lambda_0z)$. Under this 
change of variables, our retraction map $F$ is 
transformed to its conjugate by $\psi$, namely
$G:=\psi \circ F \circ \psi^{-1}$, again a polynomial 
retraction map on $\mathbb{C}^2$. As we shall need it below to make a
careful observation, we compute its
explicit expression:\\
$\psi \circ F \circ \psi^{-1}(z,w) = \psi \circ F(z,w+\lambda_0z)$
\begin{equation*}
    \begin{split}
= \psi\left(z+P_1^{(2)}(z,w+\lambda_0z)+\ldots+P_1^{(k)}(z,w+\lambda_0z), P_2^{(2)}(z,w+\lambda_0z) +\ldots  \right.\\ 
\left. +P_2^{(k)}(z,w+\lambda_0z)\right)
\end{split}
\end{equation*}
\begin{equation*}
\begin{split}
= \left(z+P_1^{(2)}(z,w+\lambda_0z)+\ldots+P_1^{(k)}(z,w+\lambda_0z), ( P_2^{(2)}(z,w+\lambda_0z)+\ldots \right.\\
\left. +P_2^{(k)}(z,w+\lambda_0z)) - \lambda_0( z+P_1^{(2)}(z,w+\lambda_0z)+\ldots+P_1^{(k)}(z,w+\lambda_0z))\right)
\end{split}
\end{equation*}
Denote the {\it analogues} of the 
component polynomials (as with the case of $F$) in the above by 
$P_1'$ and $P_2'$; or in other words, precisely so that
\[
G(z,w) \;=\; \left( z + P'_1(z,w),\; P'_2(z,w) \right).
\]
Then notice that unlike $F$, the 
degrees of the components of $G$ are not both equal.
Indeed, observe that
$\deg(P'_1) = \deg(P_1)$ but $\deg(P'_2) < \deg(P_2)$, as 
follows by observing the highest degree homogeneous constituents
in the expressions occurring in the last equation 
(particularly the second component of $G$), and 
recalling that $P_2^{(k)} = \lambda_0P_1^{(k)}$. 
It follows therefore that all the
hypotheses required for case-1 above, is 
satisfied by the components of  
$G$ allowing us to appeal to the arguments therein
to conclude that either $P'_1 \equiv 0$ or $P'_2 \equiv 0$.
As $P'_1 = P_1 \circ \psi^{-1}$ and $\psi$ is 
an automorphism of $\mathbb{C}^2$, we conclude that the 
former case $P'_1 \equiv 0$ implies that $P_1 \equiv 0$
as well. 
The latter case $P'_2 \equiv 0$, however does not 
lead to a similar conclusion; in this case, recall that
relation between $P_2'$ and $P_1,P_2$ is given by
\[
P_2'(z,w) \; = \; P_2 \circ \psi^{-1}(z,w) \; - \; 
\lambda_0\left( z + P'_1(z,w) \right)
\]
Thus, the latter case $P'_2 \equiv 0$ means:
\[
P_2 \left( \psi^{-1}(z,w) \right) \; - \;
\lambda_0 P_1 \left( \psi^{-1}(z,w) \right) 
\; = \; \lambda_0 z
\]
It is then clear that this equation cannot hold thereby
giving the required contradiction to conclude 
that: the case $P_2' \equiv 0$ cannot happen; indeed, 
the right side of the above equation has no
non-linear terms whereas
the left hand side of the above has no linear terms, 
so that such an equation can at-best hold only 
when both sides are zero which in-turn can 
happen only when $\lambda_0=0$,
a contradiction (recall that $\lambda_0\neq0$).
Thus, the only conclusion of case-2 is that $P_1 \equiv 0$.\\

\noindent The common conclusion of  the foregoing cases $1$ and $2$, is that either $P_1 \equiv 0$ or $P_2 \equiv 0$. Now, we
unravel what each of these
possibilities means for our retraction map $F$.\\

\noindent If $P_1 \equiv 0$, then the equation 
(\ref{P''}) reduces to 
\begin{equation}\label{P'=0}    
P_2^{(2)}\left(z, P_2^{(2)} + \ldots +P_2^{(k)}\right) + \cdots + 
P_2^{(k)}\left(z,P_2^{(2)} + \ldots +P_2^{(k)}\right) 
\;=\; P_2^{(2)}(z,w) + \cdots + P_2^{(k)}(z,w)
\end{equation}
Examine the left hand side for {\it all} monomials of the lowest degree and
observe that it may altogether expressed as: $P_2(z,0)$,
provided of-course that this is non-zero, else the 
aforementioned lowest degree is higher than two.
Therefore, equating the lowest degree term on both sides of the above equation 
(\ref{P'=0}), we get that $P_2^{(2)}(z,0) = P_2^{(2)}(z,w)$, 
unless it is the case that $P_2 \equiv 0$. In both cases,  
we conclude that $P_2^{(2)}(z,w)$ is
independent of $w$ i.e., a polynomial
function of the single variable $z$ (possibly zero). 
Further in any case,
by canceling-off these entities from both sides,
equation (\ref{P'=0}) reduces  to
\begin{equation}
P_2^{(3)}\left(z, P_2^{(2)}+\ldots+P_2^{(k)}\right) + \cdots + 
P_2^{(k)}\left(z,P_2^{(2)}+ \ldots +P_2^{(k)}\right) 
\;=\;
P_2^{(3)}(z,w) + \cdots+ P_2^{(k)}(z,w).
\end{equation}
Again comparing the lowest degree terms on both sides of the above 
equation, we obtain that $P_2^{(3)}$ is at-most a function of $z$ alone i.e., either a cubic polynomial in $z$ or the 
zero polynomial. Continuing 
this process, we finally get that 
$P_2^{(2)},P_2^{(3)},\ldots,P_2^{(k)}$ are at-most
functions of $z$ alone. This shows 
that $P_2(z,w)$ is a function of $z$ alone. 
Hence, we have proven that $F$ is
of the form $F(z,w) = (z,\varphi(z))$ for some 
$\varphi \in \mathbb{C}[z]$; specifically,
$\varphi(z) = P_2(z,0)$. Finally here, if we conjugate our
retraction map $F$ by 
another automorphism -- which we call $\psi$ again -- 
$\psi(z,w) = (z,w-\varphi(z))$ and 
denote the resulting retraction again by $F$,
we obtain that $F$ is the 
standard orthogonal projection onto 
the $z$-axis with ${\rm image}(F) = \{(\zeta,0) 
: \zeta \in \mathbb{C}\}$, rendering the 
required rectification result for both the 
retraction map as well as the retract.\\

\noindent If $P_2 \equiv 0$, then the equation 
(\ref{P'}) becomes 
\begin{equation}\label{P''=0}   
P_1^{(2)}\left(z+P_1^{(2)} +\ldots+P_1^{(r)}, 0 \right)  
        + \cdots + 
P_1^{(r)}\left(z+P_1^{(2)} +\ldots+P_1^{(r)},0 \right) 
\;= \; 0.
\end{equation}
Equating the smallest degree terms of both sides 
of the above equation (\ref{P''=0}), we get that $P_1^{(2)}(z,0) = 0$.
This means that all the monomials in $P_1^{(2)}$ are divisible
by $w$, thereby
$P_1^{(2)}(z,w) = wQ^{(1)}(z,w)$, for some polynomial 
$Q^{(1)}(z,w)$ of degree one. 
Using this back in the above equation (\ref{P''=0}), we 
may reduce it to: 
\begin{equation*}
P_1^{(3)} \left(z+P_1^{(2)} + \ldots +P_1^{(r)}, 0 \right)
+ \cdots + 
P_1^{(r)}\left(z+P_1^{(2)} +\ldots+P_1^{(r)},0 \right) \; = \; 0.
\end{equation*}
Again comparing lowest degree terms on both sides, we 
get that 
$P_1^{(3)}(z,w) \equiv wQ^{(2)}(z,w)$ 
for some polynomial $Q^{(2)}(z,w)$ 
of degree two. 
Continuing this process, finally we obtain that for each 
$i = 2,\ldots,r$, there exists a polynomial  
$Q^{(i-1)}(z,w)$ of degree $(i-1)$ such that
$P_1^{(i)}(z,w) = wQ^{(i-1)}(z,w)$. Therefore, we obtain this time
that without any further change of variables, our
retraction map is precisely of the form
\[ 
F(z,w) = (z+wQ(z,w),0).
\]
for some $Q \in \mathbb{C}[z,w]$ which vanishes at the origin; 
specifically in the above notations,
$Q(z,w) = Q^{(1)}(z,w) + \ldots + Q^{(r-1)}(z,w)$. 
Hence, ${\rm image}(F) = \{(\zeta,0) : \zeta \in \mathbb{C}\}$.
\qed\\

\noindent We can also derive the precise form of all 
polynomial retraction mappings of $\mathbb{C}^2$, as a corollary of proof of the 
foregoing theorem \ref{polyretract-of-C^2}.
\begin{cor}
If $F: \mathbb{C}^2 \to \mathbb{C}^2$ be a polynomial retraction 
map, then after a translation and a
(isotropic) scaling, $F$ has one of the following forms:
\begin{enumerate}
\rm \item $F(z,w) = \left(d(az+bw)-b\varphi(az+bw),-c(az+bw) +a\varphi(az+bw)\right)$,\\ 
where $a,b,c,d \in \mathbb{C}$ with $ad-bc =1$ and $\varphi$ is a polynomial of one variable with $\varphi(0) = \varphi'(0)=0$.
\rm \item $F(z,w) = \left(dL(z,w),-cL(z,w)\right)$,\\
where $a,b,c,d \in \mathbb{C}$ with $ad-bc =1$,  $L(z,w) = az+bw+(cz+dw)Q(az+bw,cz+dw)$ and $Q(z,w)$ is a polynomial of degree greater than one.
\end{enumerate}
\end{cor}

\noindent Finally, as mentioned above, we may use the above 
analysis to give an elementary proof of Costa's result 
on the form of the polynomial ring/algebra in two variables. 
To this end, we begin by recalling some definitions; 
a good exposition of associated matter 
can be found in [\cite{Shplrn}].

\begin{defn}
Let $K[x, y]$ be the polynomial algebra in two variables over a field $K$ of 
characteristic $0$. 
An idempotent $K$-algebra homomorphism $\varphi$ from
$K[x,y]$ into itself is called a retraction (or a projection) map on $K[x,y]$. The image of 
such a homomorphism is called a retract of (the polynomial ring/algebra) $K[x,y]$. Thus,
a subalgebra $R$ of $K[x, y]$ is called a \textit{retract} if there 
is an idempotent homomorphism 
$\varphi: K[x, y] \rightarrow$ $K[x, y]$ such that $\varphi(K[x, y])=R$.
\end{defn}
\noindent To provide just one concrete example to illustrate the above definition, 
take $p = x + x^2y$; then $K[p]$ is a retract of $K[x,y]$ 
(the retraction map here, is given by $x \mapsto p$ , $y \mapsto 0$). Of-course as always,
before proceeding to such examples, one can cite the trivial ones, namely $K[x,y]$ or $(0)$; a retract 
of $K[x,y]$ which is neither of these two extremes, is referred to as a proper retract in this context.
In \cite{Cost}, Costa proved that every proper retract of $K[x,y]$ has the 
form $K[p]$ for some polynomial $p \in K[x,y]$. We shall be using this result in the sequel.\\

\noindent It is helpful in all of what follows, to first record the elementary fact as to
how polynomial retraction maps on $\mathbb{C}^2$ transform
to (or induce) a retraction map on the polynomial ring $\mathbb{C}[z,w]$, as discussed next. 
\begin{lem}\label{Poly_Retrctn}
A holomorphic map $\rho: \mathbb{C}^2 \to \mathbb{C}^2$ given by 
$\rho(z,w) = (p(z,w),q(z,w))$ where $p,q \in \mathbb{C}[z,w]$ is a polynomial retraction 
if and only if there 
exists a retraction map 
$\hat{\rho} : \mathbb{C}[z,w] \to \mathbb{C}[z,w]$  which take the 
standard generators $z,w$ to the polynomials $p(z,w)$ and $q(z,w)$.
\end{lem}
\begin{proof}
    Assume that $\rho: \mathbb{C}^2 \to \mathbb{C}^2$ is a (holomorphic) retraction map 
    on $\mathbb{C}^2$ such that  its components  
    namely $p(z,w)$, $q(z,w)$ are polynomials. Define $\hat{\rho}: \mathbb{C}[z,w] \to \mathbb{C}[z,w]$ by $z \mapsto p(z,w),~ w \mapsto q(z,w)$. Since $\rho \circ \rho = \rho$, we 
 have $p(p(z,w),q(z,w)) = p(z,w)$ and 
$q(p(z,w),q(z,w)) = q(z,w)$. This implies that  $\hat{\rho} \circ \hat{\rho} = \hat{\rho}$ i.e., $\hat{\rho}$ is a retraction map.\\ 

\noindent Conversely, assume $\hat{\rho}: \mathbb{C}[z,w] \to \mathbb{C}[z,w]$ 
is a retraction map. Then there exists  $p(z,w),\\ q(z,w) \in \mathbb{C}[z,w]$ such 
that  $\hat{\rho}(z) = p(z,w)$ and  $\hat{\rho}(w) =q(z,w)$.  Let the 
subalgebra generated by $p(z,w)$ and $q(z,w)$, which is a retract of the polynomial algebra
$\mathbb{C}[z,w]$ be denoted by $R$.
Define $\rho: \mathbb{C}^2 \to \mathbb{C}^2$  by $\rho(z,w) = (p(z,w),q(z,w))$. 
Since $p,q \in R$,  $\hat{\rho}(p) = p$ and  $\hat{\rho}(q) = q$. This implies 
that $p(p(z,w),q(z,w)) =p(z,w)$ and $q(p(z,w),q(z,w)) = q(z,w)$.  But this just means that
$\rho \circ \rho =\rho$ i.e., $\rho$ is a polynomial retraction mapping on $\mathbb{C}^2$.\\
\end{proof}

Now we can prove theorem 3.5 in \cite{Cost} as a corollary of the above theorem \ref{polyretract-of-C^2}.
\begin{cor}
    If $R$ is a non-trivial retract of $\mathbb{C}[z,w]$, then 
    $R = \mathbb{C}[r]$ for some $r \in \mathbb{C}[z,w]$.
\end{cor}
\begin{proof}
Since $R$ is a non-trivial retract of $\mathbb{C}[z,w]$, there 
exists an idempotent homomorphism $\tilde{\rho} : \mathbb{C}[z,w] 
\to \mathbb{C}[z,w]$ such that $\text{image}(\tilde{\rho}) = R$. 
Let $p(z,w) := \tilde{\rho}(z)$ and $q(z,w) := \tilde{\rho}(w)$.
By lemma \ref{Poly_Retrctn}, $\tilde{\rho}$ induces a polynomial 
retraction map $\rho: \mathbb{C}^2 \to \mathbb{C}^2$ defined by 
$\rho(z,w) = (p(z,w),q(z,w))$. If we denote $\text{image}(\rho)$  by $Z$ then by above theorem \ref{polyretract-of-C^2}, there exists a polynomial 
automorphism $\psi$ on $\mathbb{C}^2$ such that $\psi(Z) = 
\{(\zeta,0): \zeta \in \mathbb{C}\}$. Let $\hat{\rho} := 
\psi \circ \rho \circ \psi^{-1}$ and $\hat{Z} := \{(\zeta,0): \zeta 
\in \mathbb{C}\}$. Note that $\hat{\rho}$ is a polynomial 
retraction map on $\mathbb{C}^2$ and $\text{image}(\hat{\rho}) = 
\hat{Z}$. Consider the induced homomorphism $\psi^*$ on 
$\mathbb{C}[z,w]$ defined by $\psi^*(z) = \psi_1(z,w)$ and $\psi^*(w) = 
\psi_2(z,w)$, where $\psi_1, \psi_2$ are the components of $\psi$.
It is easy to see that the surjectivity of $\psi$ implies the injectivity of $\psi^*$. While
it is true that the injectivity of $\psi$
implies the surjectivity of $\psi$ is true, it is
far more non-trivial. As our purpose here 
is to give a self-contained proof, 
we proceed without splitting it to directly 
verify that $\psi^*$ is an automorphism on $\mathbb{C}[z,w]$. 
Let $\phi(z,w) = (\phi_1(z,w),\phi_2(z,w))$ denote the inverse of 
$\psi(z,w)$. Consider its induced homomorphism,
$\phi^*: \mathbb{C}[z,w] \to \mathbb{C}[z,w]$
defined by $\phi^*(z) = \phi_1(z,w)$ and $\phi^*(w) = \phi_2(z,w)$. 
To prove $\psi^*$ is an automorphism, it suffices to show that 
$\psi^*$ and $\phi^*$ are inverses to each other. Note that 
\[
\phi^*\circ \psi^*(z) = \phi^*(\psi_1(z,w)) = \psi_1(\phi_1(z,w),\phi_2(z,w)) 
\]
and 
\[
\phi^*\circ \psi^*(w) = \phi^*(\psi_2(z,w)) = \psi_2(\phi_1(z,w),\phi_2(z,w)) 
\]
Since $\psi \circ \phi(z,w) = (z,w)$, we obtain that 
$\phi^*\circ \psi^*(z) = z$ and 
$\phi^*\circ \psi^*(w) = w$. Similarly, we get that $\psi^*\circ \phi^*(z) 
=z$ and  $\psi^*\circ \phi^*(w) =w$. 
 This shows that $\psi^*\circ \phi^*$ and  $\psi^*\circ \phi^*$
fixes the generators $z,w$. As $\psi^*$ and $\phi^*$ are ring homomorphisms, 
we obtain that $\phi^* \circ \psi^*$ and $\psi^* \circ \phi^*$ are the 
identity maps on $\mathbb{C}[z,w]$, finishing 
the verification that $\psi^*$ is an automorphism.
Getting back to the proof proper (of this corollary), since $\text{image}(\hat{\rho}) = \hat{Z}$, 
there exists $s(z,w) \in \mathbb{C}[z,w]$ such that 
$\hat{\rho}(z,w) = (s(z,w),0)$. To conclude, we only need to consider the idempotent 
homomorphism $\rho^*:\mathbb{C}[z,w] \to \mathbb{C}[z,w]$ defined 
by $\rho^*(z,w) := \psi^* \circ \tilde{\rho} \circ \phi^*$; and, Note that $\text{image}(\rho^*) = \mathbb{C}[s]$. This implies that 
$\psi^*(R) = \mathbb{C}[s]$ and hence we obtain that $R = 
\mathbb{C}[r]$, where $r(z,w) = \phi^*(s(z,w))$.
\end{proof}

\noindent It is not clear if the above results can be extended
to the case when the number of variables is greater than two, even 
if we continue restricting to the polynomial case. 
We hope to be able to take this up in a future article soon.

\subsection{Proof of theorem \ref{polyretract-of-C^2}} We now deal with the other already known proof. This is essentially an exposition of
the proof of theorem $1.1$  of 
\cite{Shpilrain -- Yu} with greater details.
Let $Z$ be a non-trivial polynomial retract of $\mathbb{C}^2$. There exists a retraction 
map $\rho: \mathbb{C}^2 \to \mathbb{C}^2$ of the form 
$\rho(z,w) = (p(z,w),q(z,w))$, where $p,q \in \mathbb{C}[z,w]$ such 
that $\rho(\mathbb{C}^2) = Z$. By lemma \ref{Poly_Retrctn}, $\rho$ induces a retraction 
map  $\hat{\rho}: \mathbb{C}[z,w] \to \mathbb{C}[z,w]$ such that 
$\hat{\rho}(z) = p(z,w)$ and  $\hat{\rho}(w) = q(z,w)$,  the components of the retraction mapping
$\rho$ on $\mathbb{C}^2$. Note that  ${\rm image}(\hat{\rho}) = \mathbb{C}[p,q]$, the 
subalgebra generated by $p,q$ which we shall denote by $R$.\\

\noindent By  theorem 3.5 in \cite{Cost}, $\mathbb{C}[p,q] = \mathbb{C}[r]$, 
 for some $r = r(z,w) \in R \subset \mathbb{C}[z,w]$. Hence we may write $p = q_1(r),~ q= q_2(r)$, 
 for some $q_1,q_2 \in \mathbb{C}[z]$ i.e., $q_1,q_2$ are polynomials of a single variable.
Let us write out this factorization for our retraction mapping $\rho$ on $\mathbb{C}^2$, a bit 
elaborately for convenience in the computations that follow, as
\[
\rho(z,w) = \left( q_1(r(z,w)),q_2(r(z,w)) \right).
\]

Suppose $q_1$ and $q_2$ have degree $n \geq 1$ and $m \geq 1$ respectively (the case where atleast
one of them has degree zero i.e., is constant, is simpler and dealt with during the course of concluding the
proof). 
It follows by the Abhyankar -- Moh theorem  (main theorem of \cite{Abynkr}) that: either $n$ divides $m$ or $m$ 
divides $n$. Suppose without loss of generality that $\deg(q_1) = k\cdot \deg(q_2)$ for some integer $k \geq 1$.\\

\noindent Consider  then the automorphism $\psi: \mathbb{C}^2 \to \mathbb{C}^2$ defined 
by $\psi(z,w) = (z-cw^k,w)$, where $c$ is chosen so that leading term of the polynomials $q_1$
and $cq_2^k$ in $q_1 -cq_2^k$ cancel out. The point of considering this polynomial automorphism
$\psi$ is due to its effect in rendering a strict reduction of degree of the polynomial $q_1$ when 
we pass from our retraction map $\rho$ to its conjugate by $\psi$, explained as follows: 
\begin{eqnarray*}
\psi \circ \rho \circ \psi^{-1}(z,w) &=& \psi \circ \rho (z+cw^k,w) = \psi(q_1(r(z+cw^k,w)),q_2(r(z+cw^k,w)))\\
&=& (q_1(r(z+cw^k,w)) -cq_2^k(r(z+cw^k,w)),q_2(r(z+cw^k.w)))\\
&=& (\tilde{q_1}(\tilde{r}(z,w)), \tilde{q_2}(\tilde{r}(z,w)))
\end{eqnarray*}
where $\tilde{q_1} = q_1 -cq_2^k,~ \tilde{q_2} = q_2,~ \tilde{r}(z,w) = r(z+cw^k,w)$. Note 
that $\tilde{q_1}(\tilde{r}),~ \tilde{q_2}(\tilde{r})$ generates $\mathbb{C}[r]$.\\

\noindent Observe that the degree of the single variable polynomial $\tilde{q}_1$ is 
strictly less than that of $q_1$; we shall call the polynomials 
$\tilde{q}_1, \tilde{q}_2$ now simply by $q_1,q_2$ and repeat the above process of passing from 
one retraction to another through conjugation by the (new and appropriate) automorphism 
$\psi$. So, continuing this process , we eventually arrive within finitely many steps, at a pair of polynomials
such that the following happens. One of  them is a 
constant 
polynomial $c$, and the other polynomial, denoted by $h$,  is such that
$\mathbb{C}[h] = \mathbb{C}[r]$ (i.e. $h = c_1r + c_2$). 
 The composition of the various (finitely many) automorphisms 
 involved in the aforementioned process, gives a polynomial automorphism $\phi$ of $\mathbb{C}^2$ such that  
\[ 
\left(\phi \circ \rho \circ \phi^{-1}\right)(z,w) \; = \; (c_1r(z,w) + c_2, c).
\]
Let $\tilde{\rho} := \phi \circ \rho \circ \phi^{-1}$, which of-course is a 
retraction map on $\mathbb{C}^2$. After conjugating it by an affine automorphism 
namely $\tilde{\psi}(z,w) = (z,w-c)$, we can ensure that the retraction map -- which 
we shall call $\tilde{\rho}$ again -- has a particularly simple
form wherein the second component is just zero i.e.,
$\tilde{\rho}(z,w)$ is of the form $(s(z,w),0)$ for some $s(z,w) \in \mathbb{C}[z,w]$.
Since $\tilde{\psi} \circ \tilde{\rho} \circ \tilde{\psi}^{-1}$ is a retraction map, 
it follows that $s(s(z,w),0) = s(z,w)$.
Writing $s(z,0) = a_0 + a_1z + a_2z^2+ \cdots +a_kz^k$, we then have:
\[ 
s(s(z,w),0) = a_0 + a_1s(z,w) + \cdots +a_ks(z,w)^k = s(z,w). 
\]
 Comparing the highest degree monomials on both sides (of the second equality above), we first deduce 
that: $a_2 = \cdots =a_k = 0$ and $a_1=0$; it is then immediate that $a_0$ must also be zero.
Thus, $s(z,0)=z$ meaning that among the monomials not divisible by $w$, there are no monomials other than $z$. 
Hence, we conclude that $s$ must be of the form: $s(z,w) = z + w\cdot l(z,w)$, 
for some $l(z,w) \in \mathbb{C}[z,w]$. Therefore in all, it follows from the foregoing arguments that
there exists a polynomial automorphism  
$\Phi$ on $\mathbb{C}^2$ such that 
\[
\tilde{\rho}(z,w) \; = \; \left(\Phi \circ \rho \circ \Phi^{-1}\right)(z,w) \; = \; \left( s(z,w),0\right).
\]
Thus, after applying the polynomial change of coordinates given by $\Phi$, the retract $Z$ is transformed 
into
\[
\Phi(Z) = \{(\zeta,0): \zeta \in \mathbb{C}\},
\]
finishing the proof of the theorem that all non-trivial polynomial retracts of $\mathbb{C}^2$ can be 
rectified.
\qed\\
\\


\begin{bibdiv}
\begin{biblist}

\bib{Baronti-Pappini}{article}{
    AUTHOR = { Marco Baronti and Papini Pier Luigi},
     TITLE = {Norm-one projections onto subspaces of {$l_p$}},
   JOURNAL = {Ann. Mat. Pura Appl. (4)},
  FJOURNAL = {Annali di Matematica Pura ed Applicata. Serie Quarta},
    VOLUME = {152},
      YEAR = {1988},
     PAGES = {53--61},
      ISSN = {0003-4622},
   MRCLASS = {46B25},
  MRNUMBER = {980971},
MRREVIEWER = {H.\ E.\ Lacey},
       DOI = {10.1007/BF01766140},
       URL = {https://doi.org/10.1007/BF01766140},
}

\bib{Beata-survey}{article}{
    AUTHOR = {Randrianantoanina, Beata},
     TITLE = {Norm-one projections in {B}anach spaces},
      NOTE = {International Conference on Mathematical Analysis and its
              Applications (Kaohsiung, 2000)},
   JOURNAL = {Taiwanese J. Math.},
  FJOURNAL = {Taiwanese Journal of Mathematics},
    VOLUME = {5},
      YEAR = {2001},
    NUMBER = {1},
     PAGES = {35--95},
      ISSN = {1027-5487,2224-6851},
   MRCLASS = {46-03 (01A60 46B45 46E30)},
  MRNUMBER = {1816130},
       DOI = {10.11650/twjm/1500574888},
       URL = {https://doi.org/10.11650/twjm/1500574888},
}


\bib{BBMV-union}{article}{
    AUTHOR = {G. P. Balakumar{,} Diganta Borah{,} Prachi Mahajan and
    Kaushal Verma} 
     TITLE = {Limits of an increasing sequence of complex manifolds},
   JOURNAL = {Ann. Mat. Pura Appl. (4)},
  FJOURNAL = {Annali di Matematica Pura ed Applicata. Series IV},
    VOLUME = {202},
      YEAR = {2023},
    NUMBER = {3},
     PAGES = {1381--1410},
      ISSN = {0373-3114,1618-1891},
   MRCLASS = {32F45},
  MRNUMBER = {4576945},
       DOI = {10.1007/s10231-022-01285-9},
       URL = {https://doi.org/10.1007/s10231-022-01285-9},
}

\bib{Bracci-Saracco}{article}{
    AUTHOR = {F. Bracci and A. Saracco},
     TITLE = {Hyperbolicity in unbounded convex domains},
   JOURNAL = {Forum Math.},
  FJOURNAL = {Forum Mathematicum},
    VOLUME = {21},
      YEAR = {2009},
    NUMBER = {5},
     PAGES = {815--825},
      ISSN = {0933-7741,1435-5337},
   MRCLASS = {32Q45 (32F17 32F45)},
  MRNUMBER = {2560392},
MRREVIEWER = {Simone\ Diverio},
       DOI = {10.1515/FORUM.2009.039},
       URL = {https://doi.org/10.1515/FORUM.2009.039},
}

\bib{Encyclo-Vol6}{book}{
     TITLE = {Several complex variables. {VI}},
    SERIES = {Encyclopaedia of Mathematical Sciences},
    VOLUME = {69},
    EDITOR = {W. Barth and R. Narasimhan},
      NOTE = {Complex manifolds,
              A translation of {\cyr Sovremennye problemy matematiki.
              Fundamental\cprime nye napravleniya, Tom} 69, Akad. Nauk SSSR,
              Vsesoyuz. Inst. Nauchn. i Tekhn. Inform., Moscow},
 PUBLISHER = {Springer-Verlag, Berlin},
      YEAR = {1990},
     PAGES = {viii+310},
      ISBN = {3-540-52788-5},
   MRCLASS = {32-06},
  MRNUMBER = {1095088},
}

\bib{Brocker-Janich-difftop-bk}{book}{
    AUTHOR = {Br\"{o}cker Theodor  and Klaus J\"{a}nich},
     TITLE = {Introduction to differential topology},
      NOTE = {Translated from the German by C. B. Thomas and M. J. Thomas.},
 PUBLISHER = {Cambridge University Press, Cambridge-New York},
      YEAR = {1982},
     PAGES = {vii+160},
      ISBN = {0-521-24135-9;0-521-28470-8},
   MRCLASS = {58-01 (57Rxx)},
  MRNUMBER = {674117},
}

\bib{Cartan-on-retracts}{article}{
    AUTHOR = {Cartan Henri},
     TITLE = {Sur les r\'{e}tractions d'une vari\'{e}t\'{e}},
   JOURNAL = {C. R. Acad. Sci. Paris S\'{e}r. I Math.},
  FJOURNAL = {Comptes Rendus des S\'{e}ances de l'Acad\'{e}mie des Sciences.
              S\'{e}rie I. Math\'{e}matique},
    VOLUME = {303},
      YEAR = {1986},
    NUMBER = {14},
     PAGES = {715},
      ISSN = {0249-6291},
   MRCLASS = {58B10 (58C30)},
  MRNUMBER = {870703},
MRREVIEWER = {Efstathios\ Vassiliou},
}


\bib{Rud_book_Fn_theory_unitball}{book}{
AUTHOR = {Rudin, Walter},
TITLE = {Function theory in the unit ball of {$\Bbb C^n$}},
SERIES = {Classics in Mathematics},
NOTE = {Reprint of the 1980 edition},
PUBLISHER = {Springer-Verlag, Berlin},
YEAR = {2008},
PAGES = {xiv+436},
ISBN = {978-3-540-68272-1},
}

\bib{Lindstraus}{book} {
    AUTHOR = {J. Lindenstrauss and L. Tzafriri},
     TITLE = {Classical {B}anach spaces. {I}},
    SERIES = {Ergebnisse der Mathematik und ihrer Grenzgebiete, Band 92},
      NOTE = {Sequence spaces},
 PUBLISHER = {Springer-Verlag, Berlin-New York},
      YEAR = {1977},
     PAGES = {xiii+188},
      ISBN = {3-540-08072-4},
}




\bib{Heath_Suff}{article} {
    AUTHOR = {L. F. Heath and T. J. Suffridge},
     TITLE = {Holomorphic retracts in complex {$n$}-space},
   JOURNAL = {Illinois J. Math.},
  FJOURNAL = {Illinois Journal of Mathematics},
    VOLUME = {25},
      YEAR = {1981},
    NUMBER = {1},
     PAGES = {125--135},
      ISSN = {0019-2082},
}

\bib{Kaup-Upm}{article}{
    AUTHOR = {W. Kaup and H. Upmeier},
     TITLE = {Banach spaces with biholomorphically equivalent unit balls are
              isomorphic},
   JOURNAL = {Proc. Amer. Math. Soc.},
  FJOURNAL = {Proceedings of the American Mathematical Society},
    VOLUME = {58},
      YEAR = {1976},
     PAGES = {129--133},
      ISSN = {0002-9939,1088-6826},
   MRCLASS = {32M05 (46B05 58C10)},
  MRNUMBER = {422704},
MRREVIEWER = {Ng\^{o} Van Qu\^{e}},
       DOI = {10.2307/2041372},
       URL = {https://doi.org/10.2307/2041372},
}

\bib{Metrc_dynmcl_aspct}{book} {
     TITLE = {Metrical and dynamical aspects in complex analysis},
    SERIES = {Lecture Notes in Mathematics},
    VOLUME = {2195},
    EDITOR = {Blanc-Centi{,} L\'{e}a},
      NOTE = {Papers based on lectures from the CNRS's Thematic School held
              in Lille, 2015,
              CEMPI Series},
 PUBLISHER = {Springer, Cham},
      YEAR = {2017},
     PAGES = {xiv+171},
      ISBN = {978-3-319-65836-0; 978-3-319-65837-7},
MRREVIEWER = {Kyle Edward Kinneberg},
}    

\bib{Mazet_Rextrm}{article} {
    AUTHOR = {P. Maz\'{e}t},
     TITLE = {Principe du maximum et lemme de {S}chwarz \`a valeurs
              vectorielles},
   JOURNAL = {Canad. Math. Bull.},
  FJOURNAL = {Canadian Mathematical Bulletin. Bulletin Canadien de
              Math\'{e}matiques},
    VOLUME = {40},
      YEAR = {1997},
    NUMBER = {3},
     PAGES = {356--363},
      ISSN = {0008-4395},
       DOI = {10.4153/CMB-1997-042-9},
       URL = {https://doi.org/10.4153/CMB-1997-042-9},
}
		
\bib{Vigue}{incollection} {
    AUTHOR = {J. P. Vigu\'{e}},
     TITLE = {Fixed points of holomorphic mappings},
 BOOKTITLE = {Complex geometry and analysis ({P}isa, 1988)},
    SERIES = {Lecture Notes in Math.},
    VOLUME = {1422},
     PAGES = {101--106},
 PUBLISHER = {Springer, Berlin},
      YEAR = {1990},
       DOI = {10.1007/BFb0089409},
       URL = {https://doi.org/10.1007/BFb0089409},
}
\bib{Meggn}{book} {
    AUTHOR = {R. E. Megginson},
     TITLE = {An introduction to {B}anach space theory},
    SERIES = {Graduate Texts in Mathematics},
    VOLUME = {183},
 PUBLISHER = {Springer-Verlag, New York},
      YEAR = {1998},
     PAGES = {xx+596},
      ISBN = {0-387-98431-3},
       DOI = {10.1007/978-1-4612-0603-3},
       URL = {https://doi.org/10.1007/978-1-4612-0603-3},
}
\bib{Vesntn}{incollection} {
    AUTHOR = {Edoardo Vesentini},
     TITLE = {Complex geodesics and holomorphic maps},
 BOOKTITLE = {Symposia {M}athematica, {V}ol. {XXVI} ({R}ome, 1980)},
     PAGES = {211--230},
 PUBLISHER = {Academic Press, London-New York},
      YEAR = {1982},
}
				
\bib{Pranav_JJ}{article} {
    AUTHOR = {Pranav Haridas{} and Jaikrishnan Janardhanan },
     TITLE = {A note on the smoothness of the {M}inkowski function},
   JOURNAL = {Complex Var. Elliptic Equ.},
  FJOURNAL = {Complex Variables and Elliptic Equations. An International
              Journal},
    VOLUME = {66},
      YEAR = {2021},
    NUMBER = {4},
     PAGES = {541--545},
      ISSN = {1747-6933},
       DOI = {10.1080/17476933.2020.1727898},
       URL = {https://doi.org/10.1080/17476933.2020.1727898},
}
\bib{Chirka}{book} {
    AUTHOR = {E. M. Chirka},
     TITLE = {Complex analytic sets},
    SERIES = {Mathematics and its Applications (Soviet Series)},
    VOLUME = {46},
      NOTE = {Translated from the Russian by R. A. M. Hoksbergen},
 PUBLISHER = {Kluwer Academic Publishers Group, Dordrecht},
      YEAR = {1989},
     PAGES = {xx+372},
      ISBN = {0-7923-0234-6},
       DOI = {10.1007/978-94-009-2366-9},
       URL = {https://doi.org/10.1007/978-94-009-2366-9},
}
\bib{Abate_Isometry}{article} {
    AUTHOR = {Abate,Marco,and  Jean-Pierre Vigu\'{e}},
     TITLE = {Isometries for the {C}arath\'{e}odory metric},
   JOURNAL = {Proc. Amer. Math. Soc.},
  FJOURNAL = {Proceedings of the American Mathematical Society},
    VOLUME = {136},
      YEAR = {2008},
    NUMBER = {11},
     PAGES = {3905--3909},
      ISSN = {0002-9939},
MRREVIEWER = {Marek Jarnicki},
       DOI = {10.1090/S0002-9939-08-09391-X},
       URL = {https://doi.org/10.1090/S0002-9939-08-09391-X},
}
\bib{Simh}{article} {
    AUTHOR = {R. R. Simha},
     TITLE = {Holomorphic mappings between balls and polydiscs},
   JOURNAL = {Proc. Amer. Math. Soc.},
  FJOURNAL = {Proceedings of the American Mathematical Society},
    VOLUME = {54},
      YEAR = {1976},
     PAGES = {241--242},
      ISSN = {0002-9939},
       DOI = {10.2307/2040793},
       URL = {https://doi.org/10.2307/2040793},
}
\bib{Abate_Horosphre}{article} {
    AUTHOR = {Marco Abate},
     TITLE = {Horospheres and iterates of holomorphic maps},
   JOURNAL = {Math. Z.},
  FJOURNAL = {Mathematische Zeitschrift},
    VOLUME = {198},
      YEAR = {1988},
    NUMBER = {2},
     PAGES = {225--238},
      ISSN = {0025-5874},
MRREVIEWER = {M. Herv\'{e}},
       DOI = {10.1007/BF01163293},
       URL = {https://doi.org/10.1007/BF01163293},
}


\bib{Barth}{article}{
    AUTHOR = {T. J. Barth},
     TITLE = {The {K}obayashi indicatrix at the center of a circular domain},
   JOURNAL = {Proc. Amer. Math. Soc.},
  FJOURNAL = {Proceedings of the American Mathematical Society},
    VOLUME = {88},
      YEAR = {1983},
    NUMBER = {3},
     PAGES = {527--530},
      ISSN = {0002-9939,1088-6826},
   MRCLASS = {32F15 (32A07 32H15)},
  MRNUMBER = {699426},
MRREVIEWER = {Peter\ Pflug},
       DOI = {10.2307/2045007},
       URL = {https://doi.org/10.2307/2045007},
}

\bib{Bdfrd}{article} {
    AUTHOR = {E. Bedford},
     TITLE = {On the automorphism group of a {S}tein manifold},
   JOURNAL = {Math. Ann.},
  FJOURNAL = {Mathematische Annalen},
    VOLUME = {266},
      YEAR = {1983},
    NUMBER = {2},
     PAGES = {215--227},
      ISSN = {0025-5831},
MRREVIEWER = {B. Gilligan},
       DOI = {10.1007/BF01458443},
       URL = {https://doi.org/10.1007/BF01458443},
}

\bib{Bh-Bo-Su}{article}{
    AUTHOR = {{,}G. Bharali, D. Borah, and S. Gorai},
     TITLE = {The {S}queezing {F}unction: {E}xact {C}omputations, {O}ptimal
              {E}stimates, and a {N}ew {A}pplication},
   JOURNAL = {J. Geom. Anal.},
  FJOURNAL = {Journal of Geometric Analysis},
    VOLUME = {33},
      YEAR = {2023},
    NUMBER = {12},
     PAGES = {383},
      ISSN = {1050-6926,1559-002X},
   MRCLASS = {32F27 (32F45 32H35 53C35)},
  MRNUMBER = {4655067},
       DOI = {10.1007/s12220-023-01439-y},
       URL = {https://doi.org/10.1007/s12220-023-01439-y},
}


\bib{Mazet_Vigue_fxdpnt}{article} {
    AUTHOR = {Mazet Pierre and Jean-Pierre Vigu\'{e} },
     TITLE = {Points fixes d'une application holomorphe d'un domaine born\'{e}
              dans lui-m\^{e}me},
   JOURNAL = {Acta Math.},
  FJOURNAL = {Acta Mathematica},
    VOLUME = {166},
      YEAR = {1991},
    NUMBER = {1-2},
     PAGES = {1--26},
      ISSN = {0001-5962},
MRREVIEWER = {Marco Abate},
       DOI = {10.1007/BF02398882},
       URL = {https://doi.org/10.1007/BF02398882},
}   
\bib{Thai_Fxdpnt_cnvx}{article} {
    AUTHOR = {Do Duc Thai},
     TITLE = {The fixed points of holomorphic maps on a convex domain},
   JOURNAL = {Ann. Polon. Math.},
  FJOURNAL = {Annales Polonici Mathematici},
    VOLUME = {56},
      YEAR = {1992},
    NUMBER = {2},
     PAGES = {143--148},
      ISSN = {0066-2216},
MRREVIEWER = {Marco Abate},
       DOI = {10.4064/ap-56-2-143-148},
       URL = {https://doi.org/10.4064/ap-56-2-143-148},
}

\bib{GZ}{article}{
      title={Holomorphic retracts in the Lie ball and the tetrablock}, 
      author={Gargi Ghosh and Włodzimierz Zwonek},
      year={2024},
      eprint={arXiv:2406.18396},
      archivePrefix={arXiv},
      primaryClass={math.CV},
      url={https://arxiv.org/abs/2406.18396}, 
}


\bib{Kosnski}{article} {
    AUTHOR = {Kosi\'{n}ski, \L ukasz ,and John E. McCarthy },
     TITLE = {Norm preserving extensions of bounded holomorphic functions},
   JOURNAL = {Trans. Amer. Math. Soc.},
  FJOURNAL = {Transactions of the American Mathematical Society},
    VOLUME = {371},
      YEAR = {2019},
    NUMBER = {10},
     PAGES = {7243--7257},
      ISSN = {0002-9947},
MRREVIEWER = {J. E. Pascoe},
       DOI = {10.1090/tran/7597},
       URL = {https://doi.org/10.1090/tran/7597},
}
   
\bib{Jrncki_invrnt_dst}{book} {
    AUTHOR = {M. Jarnicki and P. Pflug},
     TITLE = {Invariant distances and metrics in complex analysis},
    SERIES = {De Gruyter Expositions in Mathematics},
    VOLUME = {9},
   EDITION = {extended},
 PUBLISHER = {Walter de Gruyter GmbH \& Co. KG, Berlin},
      YEAR = {2013},
     PAGES = {xviii+861},
      ISBN = {978-3-11-025043-5; 978-3-11-025386-3},
MRREVIEWER = {Viorel V\^{a}j\^{a}itu},
       DOI = {10.1515/9783110253863},
       URL = {https://doi.org/10.1515/9783110253863},
}
\bib{Bhnblst}{article} {
    AUTHOR = {Frederic Bohnenblust},
     TITLE = {A characterization of complex {H}ilbert spaces},
   JOURNAL = {Portugal. Math.},
  FJOURNAL = {Portugaliae Mathematica},
    VOLUME = {3},
      YEAR = {1942},
     PAGES = {103--109},
      ISSN = {0032-5155},
MRREVIEWER = {N. Dunford},
}    
\bib{Vigue_fxdpnt_cnvx}{incollection} {
    AUTHOR = {Vigu\'{e}, Jean-Pierre},
     TITLE = {Fixed points of holomorphic mappings in a bounded convex
              domain in {${\bf C}^n$}},
 BOOKTITLE = {Several complex variables and complex geometry, {P}art 2
              ({S}anta {C}ruz, {CA}, 1989)},
    SERIES = {Proc. Sympos. Pure Math.},
    VOLUME = {52},
     PAGES = {579--582},
 PUBLISHER = {Amer. Math. Soc., Providence, RI},
      YEAR = {1991},
MRREVIEWER = {Marco Abate},
}
\bib{Kakutni}{article} {
    AUTHOR = {S. Kakutani},
     TITLE = {Some characterizations of {E}uclidean space},
   JOURNAL = {Jpn. J. Math.},
  FJOURNAL = {Japanese Journal of Mathematics},
    VOLUME = {16},
      YEAR = {1939},
     PAGES = {93--97},
      ISSN = {0075-3432},
MRREVIEWER = {F. J. Murray},
       DOI = {10.4099/jjm1924.16.0\_93},
       URL = {https://doi.org/10.4099/jjm1924.16.0_93},
}
\bib{Kuczm_Rtrct_polybll}{article} {
    AUTHOR = {Kuczumow, Tadeusz},
     TITLE = {Holomorphic retracts of polyballs},
   JOURNAL = {Proc. Amer. Math. Soc.},
  FJOURNAL = {Proceedings of the American Mathematical Society},
    VOLUME = {98},
      YEAR = {1986},
    NUMBER = {2},
     PAGES = {374--375},
      ISSN = {0002-9939},
MRREVIEWER = {S. Swaminathan},
       DOI = {10.2307/2045715},
       URL = {https://doi.org/10.2307/2045715},
}


\bib{Ian_Grhm}{article} {
    AUTHOR = {Graham, Ian},
     TITLE = {Intrinsic measures and holomorphic retracts},
   JOURNAL = {Pacific J. Math.},
  FJOURNAL = {Pacific Journal of Mathematics},
    VOLUME = {130},
      YEAR = {1987},
    NUMBER = {2},
     PAGES = {299--311},
      ISSN = {0030-8730},
MRREVIEWER = {Valentin Zdravkov Hristov},
       URL = {http://projecteuclid.org/euclid.pjm/1102690179},
}

\bib{Kodama}{article} {
    AUTHOR = {Kodama, Akio},
     TITLE = {On boundedness of circular domains},
   JOURNAL = {Proc. Japan Acad. Ser. A Math. Sci.},
  FJOURNAL = {Japan Academy. Proceedings. Series A. Mathematical Sciences},
    VOLUME = {58},
      YEAR = {1982},
    NUMBER = {6},
     PAGES = {227--230},
      ISSN = {0386-2194},
MRREVIEWER = {Alan T. Huckleberry},
       URL = {http://projecteuclid.org/euclid.pja/1195515968},
}
	
\bib{Lempert-FundArt}{article}{
    AUTHOR = {L. Lempert},
     TITLE = {Holomorphic retracts and intrinsic metrics in convex domains},
   JOURNAL = {Anal. Math.},
  FJOURNAL = {Analysis Mathematica},
    VOLUME = {8},
      YEAR = {1982},
    NUMBER = {4},
     PAGES = {257--261},
      ISSN = {0133-3852,1588-273X},
   MRCLASS = {32H15 (32H20)},
  MRNUMBER = {690838},
MRREVIEWER = {Jacob\ Burbea},
       DOI = {10.1007/BF02201775},
       URL = {https://doi.org/10.1007/BF02201775},
}		
		
		
\bib{Shplrn}{incollection} {
    AUTHOR = {Shpilrain, Vladimir},
     TITLE = {Combinatorial methods: from groups to polynomial algebras},
 BOOKTITLE = {Groups {S}t. {A}ndrews 1997 in {B}ath, {II}},
    SERIES = {London Math. Soc. Lecture Note Ser.},
    VOLUME = {261},
     PAGES = {679--688},
 PUBLISHER = {Cambridge Univ. Press, Cambridge},
      YEAR = {1999},
       DOI = {10.1017/CBO9780511666148.029},
       URL = {https://doi.org/10.1017/CBO9780511666148.029},
}
		
\bib{Cost}{article} {
    AUTHOR = {Costa, Douglas L.},
     TITLE = {Retracts of polynomial rings},
   JOURNAL = {J. Algebra},
  FJOURNAL = {Journal of Algebra},
    VOLUME = {44},
      YEAR = {1977},
    NUMBER = {2},
     PAGES = {492--502},
      ISSN = {0021-8693},
MRREVIEWER = {J. W. Brewer},
       DOI = {10.1016/0021-8693(77)90197-1},
       URL = {https://doi.org/10.1016/0021-8693(77)90197-1},
}

\bib{Rossi}{article}{
    AUTHOR = {Rossi, Hugo},
     TITLE = {Vector fields on analytic spaces},
   JOURNAL = {Ann. of Math. (2)},
  FJOURNAL = {Annals of Mathematics. Second Series},
    VOLUME = {78},
      YEAR = {1963},
     PAGES = {455--467},
      ISSN = {0003-486X},
   MRCLASS = {32.49 (57.34)},
  MRNUMBER = {162973},
MRREVIEWER = {R.\ C.\ Gunning},
       DOI = {10.2307/1970536},
       URL = {https://doi.org/10.2307/1970536},
}



\bib{Abynkr}{article} {
    AUTHOR = {Abhyankar, Shreeram S. ,and Tzuong Tsieng Moh },
     TITLE = {Embeddings of the line in the plane},
   JOURNAL = {J. Reine Angew. Math.},
  FJOURNAL = {Journal f\"{u}r die Reine und Angewandte Mathematik. [Crelle's
              Journal]},
    VOLUME = {276},
      YEAR = {1975},
     PAGES = {148--166},
      ISSN = {0075-4102},
MRREVIEWER = {O.-H. Keller},
}


\bib{Fost}{article} {
    AUTHOR = {John Erik Forn\ae ss and Nessim Sibony.},
     TITLE = {Increasing sequence of complex manifolds},
 JOURNAL = {Math. Ann.},
  FJOURNAL = {Mathematische Annalen},
    VOLUME = {255},
      YEAR = {1981},
    NUMBER = {3},
     PAGES = {351-360},
 }

 \bib{Shpilrain -- Yu}{article} {
    AUTHOR = {Vladimir Shpilrain  and Jie-Tai Yu},
     TITLE = {Polynomial retracts and the {J}acobian conjecture},
   JOURNAL = {Trans. Amer. Math. Soc.},
  FJOURNAL = {Transactions of the American Mathematical Society},
    VOLUME = {352},
      YEAR = {2000},
    NUMBER = {1},
     PAGES = {477--484},
      ISSN = {0002-9947,1088-6850},
   MRCLASS = {14R15 (13F20)},
  MRNUMBER = {1487631},
MRREVIEWER = {Amartya\ K.\ Dutta},
       DOI = {10.1090/S0002-9947-99-02251-5},
       URL = {https://doi.org/10.1090/S0002-9947-99-02251-5},
}
 
\bib{Suffridge-Ball}{article}{
    AUTHOR = {T. J. Suffridge},
     TITLE = {Common fixed points of commuting holomorphic maps of the
              hyperball},
   JOURNAL = {Michigan Math. J.},
  FJOURNAL = {Michigan Mathematical Journal},
    VOLUME = {21},
      YEAR = {1974},
     PAGES = {309--314},
      ISSN = {0026-2285,1945-2365},
   MRCLASS = {46G20},
  MRNUMBER = {367661},
MRREVIEWER = {E.\ L.\ Stout},
       URL = {http://projecteuclid.org/euclid.mmj/1029001354},
}


\bib{Hua}{book}{
  title={Harmonic analysis of functions of several complex variables in the classical domains},
  author={L. K. Hua},
  number={6},
  year={1963},
  publisher={American Mathematical Soc.}
}

\bib{Mok}{article}{
  title={Holomorphic retractions of bounded symmetric domains onto totally geodesic complex submanifolds},
  author={Mok, Ngaiming},
  journal={Chinese Annals of Mathematics, Series B},
  volume={43},
  number={6},
  pages={1125--1142},
  year={2022},
  publisher={Springer}
}

\bib{oka_0}{book}{
 AUTHOR = {Nishino, Toshio},
     TITLE = {Function theory in several complex variables},
    SERIES = {Translations of Mathematical Monographs},
    VOLUME = {193},
      NOTE = {Translated from the 1996 Japanese original by Norman Levenberg
              and Hiroshi Yamaguchi},
 PUBLISHER = {American Mathematical Society, Providence, RI},
      YEAR = {2001},
     PAGES = {xiv+366},
      ISBN = {0-8218-0816-8},
   MRCLASS = {32-02 (32-01)},
  MRNUMBER = {1818167},
MRREVIEWER = {Steven\ George\ Krantz},
       DOI = {10.1090/mmono/193},
       URL = {https://doi.org/10.1090/mmono/193},
}
\bib{Whitney}{book}{
    AUTHOR = {Whitney, Hassler},
     TITLE = {Complex analytic varieties},
 PUBLISHER = {Addison-Wesley Publishing Co., Reading, Mass.-London-Don
              Mills, Ont.},
      YEAR = {1972},
     PAGES = {xii+399},
   MRCLASS = {32-XX},
  MRNUMBER = {387634},
MRREVIEWER = {H.\ Cartan},
}
\bib{JJ}{article} {
    title={Proper holomorphic mappings of balanced domains in $\mathbb{C}^N$ },
    author={Jaikrishnan Janardhanan },
  journal={Mathematische Zeitschrift},
  volume={280},
  pages={257--268},
  year={2015},
  publisher={Springer}
}
\bib{HanPeters-Zeager}{article}{
    AUTHOR = {Peters, Han,and Crystal Zeager},
     TITLE = {Tautness and {F}atou components in {${\Bbb P}^2$}},
   JOURNAL = {J. Geom. Anal.},
  FJOURNAL = {Journal of Geometric Analysis},
    VOLUME = {22},
      YEAR = {2012},
    NUMBER = {4},
     PAGES = {934--941},
      ISSN = {1050-6926,1559-002X},
   MRCLASS = {37F10 (32F45 32H50)},
  MRNUMBER = {2965356},
MRREVIEWER = {Marco\ Abate},
       DOI = {10.1007/s12220-011-9221-0},
       URL = {https://doi.org/10.1007/s12220-011-9221-0},
}
\bib{Dineen}{article} {
    AUTHOR = {Se\'{a}n Dineen{,} and Richard M. Timoney },
     TITLE = {Complex geodesics on convex domains},
 BOOKTITLE = {Progress in functional analysis ({P}e\~{n}\'{\i}scola, 1990)},
    SERIES = {North-Holland Math. Stud.},
    VOLUME = {170},
     PAGES = {333--365},
 PUBLISHER = {North-Holland, Amsterdam},
      YEAR = {1992},
      ISBN = {0-444-89378-4},
   MRCLASS = {46G20 (32H15 46B45 53C22)},
  MRNUMBER = {1150757},
MRREVIEWER = {M.\ Nishihara},
       DOI = {10.1016/S0304-0208(08)70330-X},
       URL = {https://doi.org/10.1016/S0304-0208(08)70330-X},
}
\bib{Lemp_geodesc}{article}{,
    AUTHOR = {Lempert, L.},
     TITLE = {La m\'etrique de {K}obayashi et la repr\'esentation des
              domaines sur la boule},
   JOURNAL = {Bull. Soc. Math. France},
  FJOURNAL = {Bulletin de la Soci\'et\'e{} Math\'ematique de France},
    VOLUME = {109},
      YEAR = {1981},
    NUMBER = {4},
     PAGES = {427--474},
      ISSN = {0037-9484},
   MRCLASS = {32H15},
  MRNUMBER = {660145},
MRREVIEWER = {M.\ Skwarczy\'nski},
       URL = {http://www.numdam.org/item?id=BSMF_1981__109__427_0},
}
\bib{Vesentini_1}{article} {
    AUTHOR = {Vesentini, Edoardo},
     TITLE = {Complex geodesics},
   JOURNAL = {Compositio Math.},
  FJOURNAL = {Compositio Mathematica},
    VOLUME = {44},
      YEAR = {1981},
    NUMBER = {1-3},
     PAGES = {375--394},
      ISSN = {0010-437X,1570-5846},
   MRCLASS = {32H15 (46G20)},
  MRNUMBER = {662466},
MRREVIEWER = {Lawrence\ Gruman},
       URL = {http://www.numdam.org/item?id=CM_1981__44_1-3_375_0},
}
\bib{Abate_Itrtn}{book} {
    AUTHOR = {Abate, Marco},
     TITLE = {Iteration theory of holomorphic maps on taut manifolds},
    SERIES = {Research and Lecture Notes in Mathematics. Complex Analysis
              and Geometry},
 PUBLISHER = {Mediterranean Press, Rende},
      YEAR = {1989},
     PAGES = {xvii+417},
   MRCLASS = {32H50 (30D05 32-02 32H15)},
  MRNUMBER = {1098711},
MRREVIEWER = {Adib\ A.\ Fadlalla},
}

\bib{Jrnck_Frst_stps}{book}{
  title={First steps in several complex variables: Reinhardt domains},
  author={M. Jarnicki and P. Pflug},
  volume={7},
  year={2008},
  publisher={European Mathematical Society}
}


\bib{Zim}{article}{,
    AUTHOR = {Zimmer, Andrew},
     TITLE = {Kobayashi hyperbolic convex domains not biholomorphic to
              bounded convex domains},
   JOURNAL = {Math. Z.},
  FJOURNAL = {Mathematische Zeitschrift},
    VOLUME = {300},
      YEAR = {2022},
    NUMBER = {2},
     PAGES = {1905--1916},
      ISSN = {0025-5874,1432-1823},
   MRCLASS = {32F45},
  MRNUMBER = {4363801},
MRREVIEWER = {Armen\ Edigarian},
       DOI = {10.1007/s00209-021-02858-9},
       URL = {https://doi.org/10.1007/s00209-021-02858-9},
}
\bib{Abe}{article}{,
    AUTHOR = {Abe, Makoto},
     TITLE = {On the {N}ebenh\"ulle},
   JOURNAL = {Mem. Fac. Sci. Kyushu Univ. Ser. A},
  FJOURNAL = {Memoirs of the Faculty of Science. Kyushu University. Series
              A. Mathematics},
    VOLUME = {36},
      YEAR = {1982},
    NUMBER = {2},
     PAGES = {181--184},
      ISSN = {0373-6385,1883-2172},
   MRCLASS = {32D05},
  MRNUMBER = {676798},
MRREVIEWER = {J.\ Kajiwara},
       DOI = {10.2206/kyushumfs.36.181},
       URL = {https://doi.org/10.2206/kyushumfs.36.181},
}

\bib{Persson}{article}{,
    AUTHOR = {Persson, H\r{a}kan},
    TITLE = {On Stein Neighborhood Bases and the Nebenh$\ddot{u}$lle},
  JOURNAL = {M.S. Thesis (2010), Ume\r{a} University, Sweden},
}
\bib{Forstn}{article}{,
    AUTHOR = { Franc Forstneric{,} Josip Globevnik  and Jean-Pierre Rosay},
     TITLE = {Nonstraightenable complex lines in {$\mathbb{C}^2$}},
   JOURNAL = {Ark. Mat.},
  FJOURNAL = {Arkiv f\"or Matematik},
    VOLUME = {34},
      YEAR = {1996},
    NUMBER = {1},
     PAGES = {97--101},
      ISSN = {0004-2080,1871-2487},
   MRCLASS = {32H35 (32E30 32H02)},
  MRNUMBER = {1396625},
MRREVIEWER = {Eric\ Bedford},
       DOI = {10.1007/BF02559509},
       URL = {https://doi.org/10.1007/BF02559509},
}
\bib{Suzuki}{article}{,
author={Suzuki, Masakazu},
  title={Propri{\'e}t{\'e}s topologiques des polyn{\^o}mes de deux variables complexes, et automorphismes alg{\'e}briques de l'espace {$\bold C^2$}},
  journal={Journal of the Mathematical Society of Japan},
  volume={26},
  number={2},
  pages={241--257},
  year={1974},
  publisher={The Mathematical Society of Japan}
}
\bib{Giles}{article}{,
    AUTHOR = {Giles, J. R.},
     TITLE = {Classes of semi-inner-product spaces},
   JOURNAL = {Trans. Amer. Math. Soc.},
  FJOURNAL = {Transactions of the American Mathematical Society},
    VOLUME = {129},
      YEAR = {1967},
     PAGES = {436--446},
      ISSN = {0002-9947,1088-6850},
   MRCLASS = {46.15},
  MRNUMBER = {217574},
MRREVIEWER = {N.\ T.\ Peck},
       DOI = {10.2307/1994599},
       URL = {https://doi.org/10.2307/1994599},
}
\bib{Faulk}{article}{,
    AUTHOR = {G.D. Faulkner and J. E. Huneycutt },
     TITLE = {Orthogonal decomposition of isometries in a {B}anach space},
   JOURNAL = {Proc. Amer. Math. Soc.},
  FJOURNAL = {Proceedings of the American Mathematical Society},
    VOLUME = {69},
      YEAR = {1978},
    NUMBER = {1},
     PAGES = {125--128},
      ISSN = {0002-9939,1088-6826},
   MRCLASS = {47A65},
  MRNUMBER = {463954},
MRREVIEWER = {T.\ Ando},
       DOI = {10.2307/2043205},
       URL = {https://doi.org/10.2307/2043205},
}
\bib{Lamp}{article}{,
    AUTHOR = {John Lamperti},
     TITLE = {On the isometries of certain function-spaces},
   JOURNAL = {Pacific J. Math.},
  FJOURNAL = {Pacific Journal of Mathematics},
    VOLUME = {8},
      YEAR = {1958},
     PAGES = {459--466},
      ISSN = {0030-8730,1945-5844},
   MRCLASS = {46.00},
  MRNUMBER = {105017},
MRREVIEWER = {M.\ G.\ Arsove},
       URL = {http://projecteuclid.org/euclid.pjm/1103039892},
}
\bib{Lacey}{book}{,
    AUTHOR = {Lacey, H. Elton},
     TITLE = {The isometric theory of classical {B}anach spaces},
    SERIES = {Die Grundlehren der mathematischen Wissenschaften},
    VOLUME = {Band 208},
 PUBLISHER = {Springer-Verlag, New York-Heidelberg},
      YEAR = {1974},
     PAGES = {x+270},
   MRCLASS = {46EXX (46BXX)},
  MRNUMBER = {493279},
MRREVIEWER = {James\ Hagler},
}
\bib{J.Lee}{book}{,
AUTHOR = {Lee, John M.},
     TITLE = {Introduction to smooth manifolds},
    SERIES = {Graduate Texts in Mathematics},
    VOLUME = {218},
   EDITION = {Second},
 PUBLISHER = {Springer, New York},
      YEAR = {2013},
     PAGES = {xvi+708},
      ISBN = {978-1-4419-9981-8},
   MRCLASS = {58-01 (53-01 57-01)},
  MRNUMBER = {2954043},
}

\bib{Zim_generic}{article}{,
  title={Generic analytic polyhedron with a non-compact automorphism group},
  author={Zimmer, Andrew},
  journal={Indiana University Mathematics Journal},
  pages={1299--1326},
  year={2018},
  publisher={JSTOR}
}
\bib{barrier}{article}{,
  title={Local vs. global hyperconvexity, tautness or $ k $-completeness for unbounded open sets in $\mathbb{C}^n$},
  author={Nikolai Nikolov and Peter Pflug},
  journal={Annali della Scuola Normale Superiore di Pisa-Classe di Scienze},
  volume={4},
  number={4},
  pages={601--618},
  year={2005}
}
\bib{Kobyshi}{book}{,
    AUTHOR = {Kobayashi, Shoshichi},
     TITLE = {Hyperbolic complex spaces},
    SERIES = {Grundlehren der mathematischen Wissenschaften [Fundamental
              Principles of Mathematical Sciences]},
    VOLUME = {318},
 PUBLISHER = {Springer-Verlag, Berlin},
      YEAR = {1998},
     PAGES = {xiv+471},
      ISBN = {3-540-63534-3},
   MRCLASS = {32H20 (32H15 32H25 32H30)},
  MRNUMBER = {1635983},
MRREVIEWER = {William\ A.\ Cherry},
       DOI = {10.1007/978-3-662-03582-5},
       URL = {https://doi.org/10.1007/978-3-662-03582-5},
}

\bib{Forst_Stein}{book}{,
   AUTHOR = {Forstneri\v c, Franc},
     TITLE = {Stein manifolds and holomorphic mappings},
    VOLUME = {56},
 PUBLISHER = {Springer, Heidelberg},
      YEAR = {2011},
     PAGES = {xii+489},
      ISBN = {978-3-642-22249-8; 978-3-642-22250-4},
   MRCLASS = {32-02 (32E10 32E30 32H02 32Q28 32Q55)},
  MRNUMBER = {2975791},
MRREVIEWER = {Finnur\ L\'arusson},
       DOI = {10.1007/978-3-642-22250-4},
       URL = {https://doi.org/10.1007/978-3-642-22250-4},
}

\bib{Hyp-cnvx}{book}{,
 AUTHOR = {Adachi, Masanori},
     TITLE = {On a hyperconvex manifold without non-constant bounded
              holomorphic functions},
 BOOKTITLE = {Geometric complex analysis},
    SERIES = {Springer Proc. Math. Stat.},
    VOLUME = {246},
     PAGES = {1--10},
 PUBLISHER = {Springer, Singapore},
      YEAR = {2018},
      ISBN = {978-981-13-1672-2; 978-981-13-1671-5},
   MRCLASS = {32U10 (32A99)},
  MRNUMBER = {3923213},
MRREVIEWER = {Guokuan\ Shao},
       DOI = {10.1007/978-981-13-1672-2\_1},
       URL = {https://doi.org/10.1007/978-981-13-1672-2_1},
}
\bib{petrosyan}{article}{,
  title={Uniform approximation by polynomials on real non-degenerate Weil polyhedron.},
  author={Petrosyan, A.I.},
  journal={Acta Mathematica Universitatis Comenianae. New Series},
  volume={76},
  number={2},
  pages={173--178},
  year={2007},
  publisher={Comenius University Press}
}
\bib{closure_anal_poly}{article}{,
    AUTHOR = {de Paepe, P. J.},
     TITLE = {Closure of open analytic polyhedra},
   JOURNAL = {Compositio Math.},
  FJOURNAL = {Compositio Mathematica},
    VOLUME = {28},
      YEAR = {1974},
     PAGES = {333--341},
      ISSN = {0010-437X,1570-5846},
   MRCLASS = {32E05},
  MRNUMBER = {352533},
MRREVIEWER = {Andrew\ Markoe},
}
\bib{henkin}{article}{,
    AUTHOR = {Henkin, G. M.},
     TITLE = {An analytic polyhedron is not holomorphically equivalent to a
              strictly pseudoconvex domain},
   JOURNAL = {Dokl. Akad. Nauk SSSR},
  FJOURNAL = {Doklady Akademii Nauk SSSR},
    VOLUME = {210},
      YEAR = {1973},
     PAGES = {1026--1029},
      ISSN = {0002-3264},
   MRCLASS = {32F15},
  MRNUMBER = {328125},
MRREVIEWER = {Raymond\ O.\ Wells, Jr.},
}
\bib{Strat}{article}{,
  title={A geometric proof of the existence of Whitney stratifications},
  author={Vadim Yur'evich{,} Kaloshin},
  journal={Moscow Mathematical Journal},
  volume={5},
  number={1},
  pages={125--133},
  year={2005},
  publisher={Независимый Московский университет--МЦНМО}
}
\bib{Noguchi}{book}{,
    AUTHOR = {Noguchi, Junjiro},
     TITLE = {Analytic function theory of several variables},
      NOTE = {Elements of Oka's coherence},
 PUBLISHER = {Springer, Singapore},
      YEAR = {2016},
     PAGES = {xvi+397},
      ISBN = {978-981-10-0289-2; 978-981-10-0291-5},
   MRCLASS = {32-02 (14F05 32Axx 32Cxx 32Exx 32Txx)},
  MRNUMBER = {3526579},
MRREVIEWER = {Eduardo\ S.\ Zeron},
       DOI = {10.1007/978-981-10-0291-5},
       URL = {https://doi.org/10.1007/978-981-10-0291-5},
}
\bib{Real_Var}{book}{,
     TITLE = {Ordered fields and real algebraic geometry},
    EDITOR = {Dubois, D. W.},
      NOTE = {Papers from the conference held at the University of Colorado,
              Boulder, Colo., July 4--8, 1983,
              Rocky Mountain J. Math. {\bf 14} (1984), no. 4},
 PUBLISHER = {Rocky Mountain Mathematics Consortium, Phoenix, AZ},
      YEAR = {1984},
     PAGES = {729--992},
      ISSN = {0035-7596,1945-3795},
   MRCLASS = {12-06 (14-06)},
  MRNUMBER = {773111},
}
\bib{Singularity}{incollection}{,
    AUTHOR = {Massey, David B. and L\^e, D\~ung Tr\'ang},
     TITLE = {Notes on real and complex analytic and semianalytic
              singularities},
 BOOKTITLE = {Singularities in geometry and topology},
     PAGES = {81--126},
 PUBLISHER = {World Sci. Publ., Hackensack, NJ},
      YEAR = {2007},
      ISBN = {978-981-270-022-3; 981-270-022-6},
   MRCLASS = {32S60 (32B20)},
  MRNUMBER = {2311485},
MRREVIEWER = {Jean-Paul\ Brasselet},
       DOI = {10.1142/9789812706812\_0003},
       URL = {https://doi.org/10.1142/9789812706812_0003},
}
\end{biblist}
\end{bibdiv}
\end{document}